\documentclass[12pt,fullpage]{amsart}

\usepackage{amssymb,amsmath,amsthm,amsfonts}
\usepackage[mathscr]{euscript}
\usepackage[utf8]{inputenc}
\usepackage{dsfont}
\usepackage{mathrsfs}
\usepackage{esint}
\usepackage{graphicx}
\usepackage{subfig}
\usepackage{enumitem}
\usepackage{cases}
\usepackage{geometry}
\renewcommand{\leq}{\leqslant}
\renewcommand{\le}{\leqslant}
\renewcommand{\geq}{\geqslant}
\renewcommand{\ge}{\geqslant}

\usepackage{color}
\definecolor{citation}{rgb}{0.2,0.5,0.2}
\definecolor{formula}{rgb}{0.1,0.2,0.5}
\definecolor{url}{rgb}{0,0.2,0.7}
\usepackage[colorlinks=true,linkcolor=formula,urlcolor=url,citecolor=citation]{hyperref}

\newtheoremstyle%
   {olivier}%       %1 Nom
   {1.5ex}%          %2 Espace avant
   {1.5ex}%          %3 Espace après
   {\sl}%      %4 forme des caractères
   {11.5pt}%          %5 indentation
   {\bf\sc}   %6 Style de l'entête
   {.~---}%             %7
   {1ex}%      %8 Retour à la ligne après le titre
   {}%             %9 Comme dans plain ?

\newtheoremstyle%
   {oliv0}%       %1 Nom
   {1.5ex}%          %2 Espace avant
   {1.5ex}%          %3 Espace après
   {\sl}%      %4 forme des caractères
   {11.5pt}%          %5 indentation
   {\bf\sc}   %6 Style de l'entête
   {~---}%             %7
   {1ex}%      %8 Retour à la ligne après le titre
   {}%             %9 Comme dans plain ?

\newtheoremstyle%
   {clovis}%           %1 Nom
   {1.5ex}%            %2 Espace avant
   {1.5ex}%            %3 Espace après
   {\normalfont}%      %4 forme des caractères
   {11.5pt}%             %5 indentation
   {\sl}               %6 Style de l'entête
   {.~---}%            %7  %\,---\kern.2em\ignorespaces
   {1ex}%              %8 Retour à la ligne après le titre
   {}%                 %9 Comme dans plain ?

\newtheorem{theo}{Theorem}[section]
\newtheorem{cor}[theo]{Corollary}
\newtheorem{lemma}[theo]{Lemma}
\newtheorem{prop}[theo]{Proposition}
\newtheorem{conj}[theo]{Conjecture}

\newtheorem{claim}[theo]{Claim}

\theoremstyle{definition}
\newtheorem{de}[theo]{Definition}
\theoremstyle{fact}

\theoremstyle{remark}
\newtheorem{remark}[theo]{Remark}

\newtheorem{example}[theo]{Example}
\numberwithin{equation}{section}

\def\R {\mathbb{R}}

\def\N {\mathbb{N}}
\def\S {\mathbb{S}}
\def\Z {\mathbb{Z}}

\def\j {\mathcal{J}}

\def\p {\mathcal{P}}
\def\eps{\varepsilon}
\def\O{\Omega}

\newcommand{\opp}[1]{\p[{#1}]}

\newlength{\defbaselineskip}
\def\restriction#1#2{\mathchoice
              {\setbox1\hbox{${\displaystyle #1}_{\scriptstyle #2}$}
              \restrictionaux{#1}{#2}}
              {\setbox1\hbox{${\textstyle #1}_{\scriptstyle #2}$}
              \restrictionaux{#1}{#2}}
              {\setbox1\hbox{${\scriptstyle #1}_{\scriptscriptstyle #2}$}
              \restrictionaux{#1}{#2}}
              {\setbox1\hbox{${\scriptscriptstyle #1}_{\scriptscriptstyle #2}$}
              \restrictionaux{#1}{#2}}}
\def\restrictionaux#1#2{{#1\,\smash{\vrule height .8\ht1 depth .85\dp1}}_{\,#2}}
\allowdisplaybreaks[4]

\begin{document}

\title[Propagation phenomena with nonlocal diffusion]{Propagation phenomena with nonlocal diffusion in presence of an obstacle}

\author[Julien Brasseur]{Julien Brasseur}
\address{\'Ecole des Hautes \'Etudes en Sciences Sociales, PSL Research University, CNRS, Centre d'Analyse et de Math\'ematique Sociales, Paris, France}
\email{julien.brasseur@ehess.fr, julienbrasseur@wanadoo.fr}
\author[J\'er\^{o}me Coville]{J\'er\^ome Coville}
\address{BioSP, INRA, 84914, Avignon, France}
\email{jerome.coville@inra.fr}

\begin{abstract}
%...
We consider a nonlocal semi-linear parabolic equation on a connected exterior domain of the form $\mathbb{R}^N\setminus K$, where
$K\subset\mathbb{R}^N$ is a compact ``obstacle". The model we study is motivated by applications in biology and takes into
account long range dispersal events that may be anisotropic depending on how a given population
perceives the environment. %To formulate properly our model we develop a new theoretical framework.
To formulate this in a meaningful manner, we introduce a new theoretical framework which is of both mathematical and biological interest.
The main goal of this paper is to construct an entire solution that behaves like a planar travelling wave as $t\to-\infty$
and to study how this solution propagates depending on the shape of the obstacle. We show that whether the solution
recovers the shape of a planar front in the large time limit is equivalent to whether a certain Liouville type
property is satisfied. We study the validity of this Liouville type property and we extend some
previous results of Hamel, Valdinoci and the authors. Lastly, we show that the entire solution is a generalised transition front.
\end{abstract}


%\keywords{Besov spaces, restriction to almost every lines, generalized smoothness.}

\maketitle

\tableofcontents

\section{Introduction}

Since the seminal works of Fisher \cite{Fisher1937}, Kolmogorov, Petrovskii and Piskunov \cite{Kolmogorov1937} on the propagation of advantageous genes in an homogeneous population, %the description of propagation phenomena using %simple mechanistic
reaction-diffusion models have been extensively used to study the complex dynamics arising in nature \cite{Berestycki2017,Cantrell2004,Kanel1961,Kawasaki1997,Murray1993,Okubo2002}.
One of the main success of this type of modelling is the notion of ``travelling waves" that has emerged from it, which has provided a rich and flexible theoretical framework to analyse the underlying dynamics of the problem considered.\par

In the past two decades, reaction-diffusion models involving more realistic descriptions of spatial interactions as well as of the environment have been considered to analyse a wide range of problems from ecology \cite{Kawasaki1997,Kinezaki2003, Okubo2002,Shigesada1986}, combustion theory \cite{Kagan2000,Kagan1998,Sivashinsky2002} or phase transition in heterogeneous medium \cite{Coville2010a,Dirr2006}. This has considerably increased our understanding of the impact of the time and spatial heterogeneities of the environment on propagation phenomena. In turn, this profusion of work has led to the introduction of new notions of travelling waves generalising the traditional notion of planar wave and reflecting the essential properties of the environment \cite{Berestycki2002,Berestycki2012a,Berestycki2009d,Berestycki2005b,Berestycki1992,Bonnet1999,Kinezaki2003,Matano2003,Nadin2009,Nolen2009,Shen2004,Shigesada1986,Xin2000}.
In particular, notions such as pulsating fronts, random fronts or conical (or curved) fronts have been introduced to analyse propagation phenomena occurring in time and/or space periodic environments \cite{Berestycki2002,Berestycki2005b,Matano2006,Nadin2009,Xin2000}, or random ergodic environments \cite{Matano2003,Nolen2009,Shen2004}, or to study propagation phenomena with some geometrical constraints  \cite{Berestycki2016b,Bonnet1999,Ninomiya2005}. It turns out that almost all of these new notions fall into the definition of generalised transition wave recently introduced by Berestycki and Hamel in \cite{Berestycki2007b}, see also \cite{Berestycki2012a,Hamel2016}.\par
It is worth mentioning that the complexity of propagation phenomena may come from either heterogeneous interactions (heterogeneous diffusion and reaction) or the geometry of the domain where the equation is defined (cylinder with rough boundary or domain with a complex structure). In the latter case, new phenomena are observed such as the pinning  of fronts.
We point the interested reader to \cite{Berestycki2017,Nadin2018,Xin2000} and references therein for a more thorough description of the state of the art on propagation phenomena in the context of reaction-diffusion equations.\par

Propagation phenomena can also be observed using other types of models, in particular nonlocal models that take into account long range dispersal phenomena (which are commonly observed in ecology, see \cite{Bartumeus2007,Bartumeus2009,Cain,Chapman}). For example, planar fronts \cite{Alfaro2017,Bates1997,Carr2004,Chen1997,Coville2007d,Coville2008a,Coville2007a}, pulsating fronts \cite{Coville2013,Fang2015,Rawal2015,Shen2012a} and generalised transition waves \cite{Berestycki2017a,Lim2016,Shen2017,Shen2017a,Shen2019} have been constructed for integro-differential models of the form
\begin{equation}\label{nonloc}
\partial_t u(t,x)=\mathcal{J}\ast u(t,x) -u(t,x)+f(t,x,u(t,x)) \,{\mbox{ for }} (t,x)\in \R\times\R^N,
\end{equation}
where  $f$ is a classical bistable or monostable nonlinearity, $\j$ is a nonnegative probability density function and $\ast$ is the standard convolution operator given by
$$ \mathcal{J}\ast u(x):=\int_{\R^N}\mathcal{J}(x-y)\hspace{0.05em}u(y)\hspace{0.1em}\mathrm{d}y. $$
However, to the best of our knowledge, there are no results dealing with the impact of the geometry on the large time dynamics of such nonlocal semi-linear equation, and only linear versions of \eqref{nonloc} seem to have been considered, see \cite{cortazar2012,cortazar2016}.\par

In the spirit of the pioneer work of Berestycki, Hamel and Matano \cite{Berestycki2009d}, we analyse here the effect of the geometry of the domain on the propagation phenomena for an adapted version of \eqref{nonloc} on exterior domains.
Precisely, given a compact set $K\subset\R^N$ with nonempty interior such that the exterior domain $\Omega:=\R^N\setminus K$ is connected, we are interested
in the properties and large time behaviour of the entire solutions $u$ to the following nonlocal semi-linear parabolic problem
\begin{align}
\partial_tu=Lu+f(u) \,{\mbox{ a.e. in }} \R\times\overline{\Omega}, \tag{P}\label{P}
\end{align}
where $L$ is the nonlocal diffusion operator given by
\begin{align*}
Lu(x):=\int_{\R^N\setminus K}J(\delta(x,y))\hspace{0.05em}(u(y)-u(x))\hspace{0.1em}\mathrm{d}y.
\end{align*}
Here, $J$ is a nonnegative kernel, $f$ is a ``bistable" nonlinearity and $\delta:\overline{\Omega}\times\overline{\Omega}\to[0,\infty)$ is a distance on $\overline{\Omega}$ that behaves locally like the Euclidean distance (precise assumptions on $J$, $f$ and $\delta$ will be given later on, see Subsection~\ref{SE:ASSUMP}).\par

The problem \eqref{P} can be seen as a nonlocal version of the reaction-diffusion problem studied by Berestycki, Hamel and Matano in \cite{Berestycki2009d}, namely
\begin{align}
\label{eqlocale}
\left\{\begin{array}{rl}
\partial_tu=\Delta u+f(u) & \text{in }\R\times\Omega,\vspace{3pt}\\
\nabla u\cdot\nu=0 & \text{on }\R\times\partial\Omega.
\end{array}
\right.
\end{align}
There, they show that for any unit vector $e\in \mathbb{S}^{N-1}$ (where $\mathbb{S}^{N-1}$ denotes the unit sphere of $\R^N$), there exists a generalised transition wave in the direction $e$ solution to \eqref{eqlocale}, i.e. for any  $e\in \mathbb{S}^{N-1}$, there exists an entire solution, $u(t,x)$, to \eqref{eqlocale} defined for all $t\in \R$ and all $x\in \overline{\Omega}$ that satisfies $0<u(t,x)<1$ for all $(t,x)\in\R\times\overline{\Omega}$ and such that
$$ |u(t,x)-\phi(x\cdot e+ct)|\underset{t\to-\infty}{\longrightarrow}0 \, {\mbox{ uniformly in }}x\in\overline{\Omega}, $$
where $(\phi,c)$ is a planar travelling wave of speed $c>0$. That is, $(\phi,c)$ is the unique (up to shift) increasing solution to
\begin{align*}
\left\{
\begin{array}{l}
c\,\phi'=\phi''+f(\phi)\text{ in }\R,\vspace{3pt}\\
\displaystyle\lim_{z\to +\infty}\phi(z)=1,\, \lim_{z\to -\infty}\phi(z)=0.
\end{array}
\right.
\end{align*}
Moreover, they prove that there exists a classical solution, $u_\infty$, to
\begin{align}
\left\{
\begin{array}{rl}
\Delta u_\infty+f(u_\infty)=0 & \text{in }\overline{\Omega},\vspace{3pt}\\
\nabla u_\infty\cdot\nu=0 & \text{on }\partial \Omega, \vspace{3pt}\\
0< u_\infty\leq 1 & \text{in }\overline{\Omega}, \vspace{3pt}\\
u_\infty(x)\to1 & \text{as }|x|\to+\infty,
\end{array} \label{EQ:BHM0}
\right.
\end{align}
which they show corresponds to the large time limit of $u(t,x)$ in the sense that
$$ |u(t,x)-u_\infty(x)\hspace{0.1em}\phi(x\cdot e+ct)|\underset{t\to\infty}{\longrightarrow}0 \, {\mbox{ uniformly in }}x\in\overline{\Omega}. $$
In addition, they were able to classify the solutions $u_\infty$ to \eqref{EQ:BHM0} with respect to the geometry of $K$.
Precisely, they proved that if the obstacle $K$ is either starshaped or directionally convex (see \cite[Definition 1.2]{Berestycki2009d}), then the solutions $u_\infty$ to \eqref{EQ:BHM0} are actually identically equal to $1$ in the whole set $\overline{\Omega}$.
This remarkable rigidity property was further extended to more complex obstacles by Bouhours in \cite{Bouhours2015} who showed a sort of ``stability" of this result with respect to small regular perturbations of the obstacle.
Yet, this phenomenon does not hold in general. Indeed, Berestycki \emph{et al.} \cite{Berestycki2009d} proved that when the domain is no longer starshaped nor directionally convex but merely simply connected (see \cite{Berestycki2009d}), then \eqref{EQ:BHM0} admits nontrivial solutions with $0<u_\infty<1$ in $\overline{\Omega}$, thus implying that the disturbance caused by the obstacle may remain forever depending on the geometry of $K$.\par

Our main objective in this article is to construct such an entire solution for the problem \eqref{P} and to study its main properties with respect to the geometry of the domain.

 \subsection{Biological motivation}

Before stating our main results, let us first discuss the relevance of this type of model.  To this end, let us go back to the very description of population dispersal.
For it, let us denote by $u(t,x)$ the density of the population at time $t$ and location $x$.
Moreover, let us discretize uniformly the domain $\Omega$ into small cubes of volume $|\Delta x_i|$ centered at points $x_i\in\Omega$, and the time into discrete time steps $\Delta t$. Then, following Huston \emph{et al.} \cite{Hutson2003}, we can describe the evolution of the population in terms of the exchange of individuals between sites. Namely, for a site $x_i$, the total number of individuals $N(t,x_i)$ will change during the time step $\Delta t$ according to
\begin{align*}
\frac{N(t+\Delta t,x_i)-N(t,x_i)}{\Delta t}= \frac{N_{i \leftarrow} -N_{i\to }}{\Delta t},
\end{align*}
where $N_{i \leftarrow}$ and $N_{i\to }$ denote the total number of individuals reaching and leaving the site $x_i$, respectively.
Since $N(t,x_i)= u(t,x_i)|\Delta x_i|$, this can be rewritten
$$ \frac{u(t+\Delta t,x_i)-u(t,x_i)}{\Delta t}|\Delta x_i|=\gamma\hspace{0.1em}|\Delta x_i|\sum_{j=-\infty}^{+\infty}\big(\mathcal{J}(x_i,x_j)\hspace{0.05em}u(t,x_j)- \mathcal{J}(x_j,x_i)\hspace{0.05em}u(t,x_i)\big)|\Delta x_j|, $$
where $\mathcal{J}(x_i,x_j)$ denotes the rate of transfer of individuals from the site $x_i$ to the site $x_j$ and $\gamma$ denotes a dispersal rate (or diffusion coefficient).\par

In ecology, understanding the structure of the rate of transfer $\mathcal{J}(x_i,x_j)$ is of prime interest as it is known to condition some important feature of the dispersal of the individuals \cite{Clobert2012,Langlois2001,Nathan2012,Robledo-Arnuncio2014}. For example, this rate can reflect some constraints of the environment on the capacity of movement of the individuals \cite{Clobert2012,Cortazar2007,Etherington2016,Graves2014,Schurr2005} and/or incorporate important features that are biologically/ecologically relevant such as  a dispersal budget \cite{Berestycki2016a,Hutson2003} or a more intrinsic description of the landscape such as its connectivity, fragmentation, anisotropy or other particular geometrical structure   \cite{Adriaensen2003,Clobert2012,Etherington2016,Fagan2002,Ricketts2001,Taylor1993,Tischendorf2000}.\par

Here, we are particularly interested in the impact that the geometry of $\Omega$ can have on this rate. A natural assumption is to consider that $\mathcal{J}(x_i,x_j)$ depends on the ``effective distance'' between $x_i$ and $x_j$.
The perception of the environment being a characteristic trait of a given species (as observed in \cite{Frantz2012}), this notion of ``effective distance'' will then change depending on the species considered.\par

Let us consider, for instance, an habitat consisting of a uniform field with, in the middle of it, a circular pond, e.g. $\O:=\R^2\setminus \overline{B_1}$ where $B_1$ denotes the unit disk of $\R^2$. One can then imagine that, for some species having a high dispersal capacity (as, for example, bees \cite{Pokorny2015}), the pond will not be considered as an obstacle in the sense that it does not affect their displacement (since the individuals can easily ``jump'' from one side of the pond to another). On the contrary, for other species, such as many land animals, this pond will \emph{actually} be seen as a physical dispersal barrier. Whence, to go from one side of the pond to another they will have to circumvent it. So, in this case, the metric considered to evaluate this ``effective distance'' has to reflect such type of behaviour (see e.g. \cite{Schurr2005} for an illustrative example).\par

\begin{figure}
\centering
\includegraphics[scale=0.25]{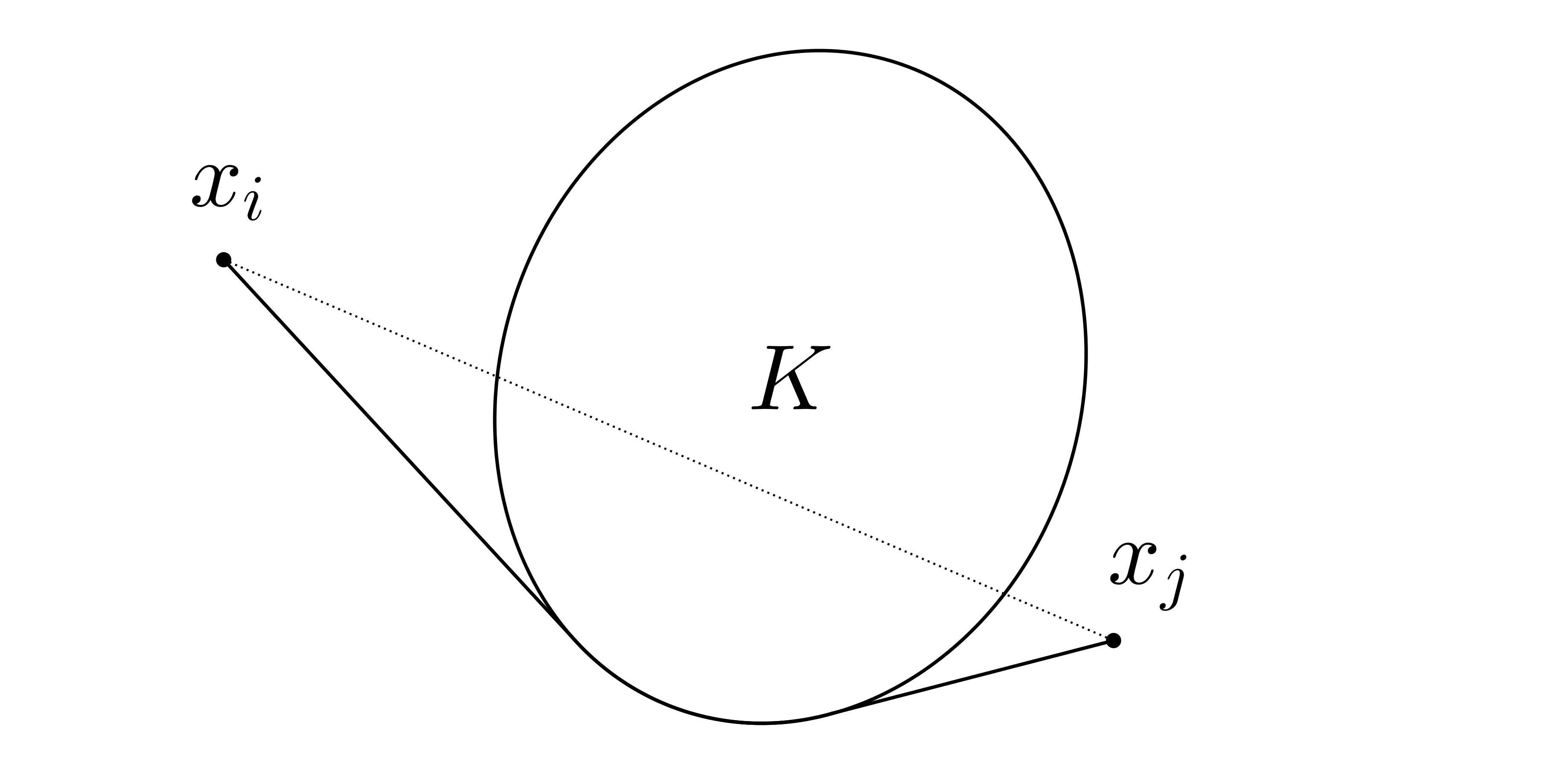}
\caption{The geodesic distance (continuous line)}  and the Euclidean distance (dashed line) between $x_i,x_j\in\R^N\setminus K$. \label{fig0}
\end{figure}

A way to understand the impact of the landscape on the movement of the individuals is to use a ``least cost path" modelling \cite{Adriaensen2003,Etherington2016,Sutherland2015}. The metric related to this geographic concept can then serve as a prototype for the metric used to evaluate $\j(x_i,x_j)$. The idea behind the ``least cost path" concept is to assign to each path taken to join one site to another some costs related to a predetermined constraint and try to minimize the costs. This notion can then be related to the notion of geodesic path on a smooth surface where the costs then reflect some geometrical aspect of the landscape.
Following this idea, it is then natural to consider the ``effective distance'' as some geodesic distance $\delta$ reflecting how the geometry of the landscape is perceived by the species considered and to take $\j(x_i,x_j)=J(\delta(x_i,x_j))$, where $J$ is a function encoding the probability to make a jump of length $\delta(x_i,x_j)$. In the above example, the appropriate distance will then be either the standard Euclidean distance (i.e. $\delta(x_i,x_j)=|x_i-x_j|$) or the geodesic distance defined in the perforated domain $\Omega$.\par

Since diffusion is classically accompanied by demographic variations (which we may suppose to be described by a nonlinear function $f$ of the density of population), by letting $|\Delta x_j|\to0^+$ and $\Delta t\to0^+$, we then formally get
$$\partial_t u(t,x)=\gamma\left(\int_{\Omega}J(\delta(x,y))\hspace{0.05em}u(t,y)\hspace{0.1em}\mathrm{d}y - u(t,x)\int_{\Omega}J(\delta(y,x))\hspace{0.1em}\mathrm{d}y\right)+f(u(t,x)), $$
which thereby yields equation \eqref{P}, up to an immaterial constant $\gamma$.

It is worth mentioning that, although the description of the rate of transfer using a geodesic distance is well-known in the ecology community \cite{Etherington2016,Sutherland2015}, to our knowledge, this the first time that such concept has been formalised mathematically in the framework of nonlocal reaction-dispersal equations to describe the evolution of a population living in a domain and having a long distance dispersal strategy.

The mathematical framework we propose goes far beyond the situation we analyse here. Indeed, the model  \eqref{P} is quite natural and  well-posed as soon as  a geodesic-type distance, which we will refer to as ``quasi-Euclidean" (see Definition~\ref{quasieuclid}), can be defined on the domain $\O$ considered, allowing thus to handle domains with very complex geometrical structure (such graph trees, which are particularly pertinent in conservation biology for the help they can provide in the understanding of the impact of blue and green belts in urban landscapes \cite{deOliveira2014,Xiu2016}). As we will see, our setting also allows to model an extremely wide class of biologically relevant ``effective distances" (see Remark~\ref{RE:QUASI}).

\subsection{Notations and definitions}\label{SE:NOTATIONS}

Before we set our main assumptions, we need to introduce some necessary definitions.\par

We begin with the metric framework on which we will work.
\begin{de}
Let $x,y\in\R^N$. We call a \emph{path connecting} $x$ \emph{to} $y$ any continuous piecewise $C^1$ function $\gamma:[0,1]\to\R^N$ with $\gamma(0)=x$ and $\gamma(1)=y$ and we denote by $\mathrm{length}(\gamma)$ its length. The set of all such paths is denoted by $H(x,y)$.
\end{de}
\begin{de}\label{quasieuclid}
Let $E\subset\R^N$. A \emph{quasi-Euclidean distance} on $E$ is a distance $\delta$ on $E$ such that $\delta(x,y)=|x-y|$ if $[x,y]\subset E$ and $\delta(x,y)\geq|x-y|$ for all $x,y\in E$. We denote by $\mathcal{Q}(E)$ the set of all quasi-Euclidean distances on $E$.
\end{de}
\begin{example}\label{EX:QUASI}
The geodesic distance $d_{E}$ on a set $E$, defined by
\begin{align}
d_{E}(x,y):=
\left\{
\begin{array}{cl}
\displaystyle\inf_{\gamma\in H(x,y) \atop \gamma\subset \overline{E}}\,\mathrm{length}(\gamma) & \text{ if } x,y \text{ belong to the same connected component}, \\
+\infty & \text{ otherwise,}
\end{array} \nonumber
\right.
\end{align}
is a nontrivial quasi-Euclidean distance. If $d_F$ is the geodesic distance on a set $F\supset E$, then its restriction $\restriction{d_F}{E}$ to $E\times E$ is another nontrivial example of quasi-Euclidean distance on $E$. Moreover, since $\mathcal{Q}(E)$ is a convex set, one may obtain other examples of such distances by convex combination of the previous examples and/or the Euclidean distance.
%Other examples of nontrivial quasi-Euclidean distances can be obtained by convex combination of $d_{E}(\cdot,\cdot)$ and the Euclidean distance.
\end{example}
\begin{remark}
If $E$ is convex, then the Euclidean distance is the only quasi-Euclidean distance.
\end{remark}
\begin{remark}\label{RE:QUASI}
Roughly speaking, a quasi-Euclidean distance can be interpreted as the length of a path connecting two points and which behaves locally like the Euclidean distance. In fact, the condition $\delta(x,y)\geq|x-y|$ can be equivalently rephrased by saying that, for any two points $x,y\in E$, there exists a path $\gamma\in H(x,y)$ (which is \emph{not} compelled to stay in $E$) connecting $x$ to $y$ and such that $\delta(x,y)=\mathrm{length}(\gamma)$. Biologically speaking, it provides a natural and flexible tool to model the ``effective distance" between two locations. It can account for a wide range of situations, for example it can model a population whose individuals can jump through some obstacles (say small ones) and not through others (say large ones), or through portions of an obstacle, as well as all the intermediary situations.
\end{remark}
\begin{de}\label{DE:CovProp}
Let $E\subset\R^N$ be a connected set and let $\delta\in\mathcal{Q}(\overline{E})$. Let $J:[0,\infty)\to[0,\infty)$ be a measurable function with $|\mathrm{supp}(J)|>0$. For any $x\in\overline{E}$, we define $\Pi_0(J,x):=\{x\}$ and
$$ \Pi_{j+1}(J,x):=\bigcup_{z\in\Pi_j(J,x)}\!\!\mathrm{supp}\left(J(\delta(\cdot,z))\right) \,{\mbox{ for any }}j\geq0. $$
We say that the metric space $(E,\delta)$ has the $J$-\emph{covering property} if
$$ \overline{E}=\bigcup_{j\geq0}\,\Pi_j(J,x) \,{\mbox{ for every }}x\in\overline{E}. $$
\end{de}
\begin{remark}\label{RE:COV:PROP}
If $E$ is a connected set and if $\delta$ is the Euclidean distance, then the above property is automatically satisfied (see Proposition~\ref{CovEucl} in the Appendix). Moreover, if $E=\R^N\setminus K$ for some compact convex set $K\subset\R^N$ with $C^2$ boundary and if $\mathrm{supp}(J)$ contains a nonempty open set (e.g. if $J$ is continuous), then $(E,\delta)$ has the $J$-covering property for any $\delta\in\mathcal{Q}(\overline{E})$ (see Proposition~\ref{COV:CVX} in the Appendix).
\end{remark}
Let us also list in this subsection a few notations and definitions used in the paper:
\vskip 0.2cm
\begin{tabular}{rl}
$|E|$ & is the Lebesgue measure of the measurable set $E$; \\
$\mathds{1}_E$ & is the characteristic function of the set $E$; \\
$\mathbb{S}^{N-1}$ & is the unit sphere of $\R^N$; \\
$B_R$ & is the open Euclidean ball of radius $R>0$ centered at the origin; \\
$B_R(x)$ & is the open Euclidean ball of radius $R>0$ centered at $x\in\R^N$; \\
$\mathcal{A}(R_1,R_2)$ & is the open annulus $B_{R_2}\setminus\overline{B_{R_1}}$; \\
$\mathcal{A}(x,R_1,R_2)$ & is the open annulus $x+\mathcal{A}(R_1,R_2)$; \\
$g\ast h$ & is the convolution of $g$ and $h$; \\
$\Delta_h^2$ & is the operator given by $\Delta_h^2f(x)=f(x+h)-2f(x)+f(x-h)$; \\
%$A\bigtriangleup B$ & is the symmetric difference symbol: $A\bigtriangleup B=(A\setminus B)\cup (B\setminus A)$; \\
$\lfloor x\rfloor$ & is the integral part of $x\in\R$, i.e. $\lfloor x\rfloor=\sup\{k\in\Z;k\leq x\}$.
\end{tabular} \\

Given $E\subset\R^N$ and $p\in[1,\infty]$, we denote by $L^p(E)$ the Lebesgue space of (equivalence classes of) measurable functions $g$ for which the $p$-th power of the absolute value is Lebesgue integrable when $p<\infty$ (resp. essentially bounded when $p=\infty$). When the context is clear, we will write $\|g\|_p$ instead of $\|g\|_{L^p(E)}$. The set $L^\infty(E)\cap C(E)$ of bounded continuous functions on $E$ will be denoted by $C_b(E)$. Given $\alpha\in(0,1)$ and $p\in[1,\infty]$, $B_{p,\infty}^\alpha(\R^N)$ stands for the Besov-Nikol'skii space consisting in all measurable functions $g\in L^p(\R^N)$ such that
$$ [g]_{B_{p,\infty}^\alpha(\R^N)}:=\sup_{h\ne0}\frac{\|g(\cdot+h)-g\|_{L^p(\R^N)}}{|h|^\alpha}<\infty. $$
Note that, when $p=\infty$, the space $B_{\infty,\infty}^\alpha(\R^N)$ coincides with the classical H\"older space~$C^{0,\alpha}(\R^N)$.
For a set $E\subset\R^N$ and $g:E\to\R$, we set
$$[g]_{C^{0,\alpha}(E)}:=\sup_{x,y\in E,\,x\neq y}\frac{|g(x)-g(y)|}{|x-y|^{\alpha}}.$$
Moreover, given $(k,\alpha)\in\N\times(0,1)$, $(E,F)\subset \R\times\R^N$ and a function $g:E\times F\to\R$, we say that $g\in C^k(E,C^{0,\alpha}(F))$ if, for all $(t,x)\in E\times F$, it holds that
$$ g(\cdot,x)\in C^k(E)\ \text{ and }\ g(t,\cdot)\in C^{0,\alpha}(F). $$

For our purposes, we need to introduce a new function space, closely related to $B_{p,\infty}^\alpha(\R^N)$.
\begin{de}
Let $E\subset\R^N$ be a measurable set and let $\delta$ be a distance on $E$. Let $\alpha\in(0,1)$ and $p\in[1,\infty)$. We call $\mathbb{B}_{p,\infty}^\alpha(E;\delta)$ the space of functions $g:\R_+\to\R$ such that $g_{\text{rad}}\in L^p(\R^N)$ where $g_{\text{rad}}(x):=g(|x|)$ and such that
$$ [g]_{\mathbb{B}_{p,\infty}^\alpha(E;\delta)}:=\sup_{x_1,x_2\in E,\,x_1\neq x_2}\frac{\left\|g(\delta(x_1,\cdot))-g(\delta(x_2,\cdot))\right\|_{L^p(E)}}{|x_1-x_2|^\alpha}<\infty. $$
\end{de}
\begin{remark}\label{RE:BESOV}
If $E=\Omega=\R^N\setminus K$ for some compact set $K\subset\R^N$, if $\delta\in\mathcal{Q}(\overline{\Omega})$ and if $g\in\mathbb{B}_{p,\infty}^\alpha(\Omega;\delta)$ has compact support, then $g_{\mathrm{rad}}\in B_{p,\infty}^\alpha(\R^N)$. Moreover, if $\delta$ is the Euclidean distance, then it also holds that
$$ RB_{p,\infty}^\alpha(\R^N):=\big\{g~\text{s.t.}~g_{\text{rad}}\in B_{p,\infty}^\alpha(\R^N)\big\}\subset \mathbb{B}_{p,\infty}^\alpha(\Omega;\delta). $$
However, in general, $\mathbb{B}_{p,\infty}^\alpha(\Omega;\delta)$ and $RB_{p,\infty}^\alpha(\R^N)$ are distinct function spaces.
\end{remark}

\subsection{Assumptions}\label{SE:ASSUMP} Let us now specify the assumptions made all along this paper. Throughout the paper we will always assume that
\begin{align}
\begin{array}{l}
K\subset\R^N \text{ is a compact set, that }\Omega:=\R^N\setminus K \text{ is connected and that }\delta\in\mathcal{Q}(\overline{\Omega}).
\end{array}
\label{C1}
\end{align}
As already mentioned above, we will suppose that $f$ is of ``bistable" type. More precisely, we will assume that $f:[0,1]\to\R$ is such that
\begin{align}
\left\{
\begin{array}{l}
\exists\,\theta\in(0,1),\  f(0)=f(\theta)=f(1)=0,\  f<0\text{ in }(0,\theta),\  f>0\text{ in }(\theta,1), \vspace{3pt}\\
f\in C^{1,1}([0,1]),\  f'(0)<0, \ f'(\theta)>0\  \text{and }f'(1)<0.
\end{array}
\right. \label{C2}
\end{align}
Also, we suppose that $J:[0,\infty)\to[0,\infty)$ is a \emph{compactly supported} measurable function with $|\mathrm{supp}(J)|>0$ such that
\begin{align}
\left\{
\begin{array}{l}
(\Omega,\delta)\text{ has the }J\text{-covering property}, \vspace{1pt}\\
\displaystyle\int_{\R^N}J_{\text{rad}}(z)\hspace{0.1em}\mathrm{d}z=1 \text{ where } J_{\text{rad}}(z):=J(|z|), \vspace{1pt}\\
\forall\, x_1\in\overline{\Omega},\ \displaystyle\lim_{x_2\to x_1}\|J(\delta(x_1,\cdot))-J(\delta(x_2,\cdot))\|_{L^1(\Omega)}=0, \vspace{1pt}\\
\mathcal{J}^\delta\in L^\infty(\Omega) \text{ where } \mathcal{J}^\delta(x):=\displaystyle\int_{\Omega}J(\delta(x,z))\hspace{0.1em}\mathrm{d}z.
\end{array}
\right. \label{C3}
\end{align}
%\begin{align}
%\left\{
%\begin{array}{l}
%(\Omega,\delta)\text{ has the }J\text{-covering property}, \vspace{2pt}\\
%J_{\text{rad}}\text{ has unit mass, where } J_{\text{rad}}(z):=J(|z|), \vspace{2pt}\\
%\mathcal{J}^\delta\in L^\infty(\Omega), \text{ where } \mathcal{J}^\delta(x):=\|J(\delta(x,\cdot))\|_{L^1(\Omega)}, \vspace{2pt}\\
%\forall\, x_1\in\overline{\Omega},\ \displaystyle\lim_{x_2\to x_1}\|J(\delta(x_1,\cdot))-J(\delta(x_2,\cdot))\|_{L^1(\Omega)}=0.
%\end{array}
%\right. \label{C3}
%\end{align}
Biologically speaking, the first assumption in \eqref{C3} means that if $\delta$ reflects how the individuals of a given species move in the environment given by $\Omega$ and if $J(\delta(x,y))$ represents their dispersal rate,
%the rate of dispersal is given by the kernel
then the individuals can reach any point of $\Omega$ in a finite number of ``jumps" regardless of their initial position.
Mathematically speaking, it ensures that the strong maximum principle holds (as will be made clear throughout the paper).
As for the last two assumptions, they are essentially meant to ensure that $\mathcal{J}^\delta\in C_b(\overline{\Omega})$. They are satisfied if, for instance, either $\delta$ is the Euclidean distance or $J$ is non-increasing and $J\in \mathbb{B}_{1,\infty}^\alpha(\Omega;\delta)$.\par

Lastly, we require the datum $(J,f)$ to be such that there exist an increasing function $\phi\in C(\R)$ and a speed $c>0$ satisfying
\begin{align}\left\{
\begin{array}{c}
c\,\phi'=J_1\ast \phi-\phi+f(\phi) \text{ in }\R, \vspace{3pt}\\
\displaystyle\lim_{z\to +\infty}\phi(z)=1,\, \lim_{z\to -\infty}\phi(z)=0.
\end{array}
\right. \label{C4}
\end{align}
where $J_1$ is the nonnegative even kernel given by:
\begin{align}
J_1(x):=\int_{\R^{N-1}} J_{\text{rad}}(x,y')\hspace{0.1em}\mathrm{d} y'. \label{J1}
\end{align}
\begin{remark}\label{rem:boot}
Notice that \eqref{C4} implies that $0<\phi<1$ and that $\phi\in C^{0,1}(\R)$. Actually, the fact that $f\in C^{1,1}([0,1])$ (as imposed by assumption \eqref{C2}) guarantees that $\phi\in C^{2}(\R)$ (as can be seen by a classical bootstrap argument).
\end{remark}
\begin{remark}
Although this is well-known (see e.g. \cite{Bates1997,Chen1997,Coville2007d,Yagisita2009}), it is worth mentioning that \eqref{C4} is not an empty assumption. For example, it is satisfied if, in addition to \eqref{C2} and \eqref{C3}, the following assumptions are made:
\begin{align}
J_{\text{rad}}\in W^{1,1}(\R^N), \,\, \, \max_{[0,1]}\hspace{0.1em}f'<1\  \text{ and }   \int_0^1f(s)\hspace{0.1em}\mathrm{d}s>0. \label{C5}
\end{align}
See also \cite[Section 2.4]{Brasseur2019} for additional comments on the matter.
\end{remark}

%%%%%%%%%%%%%%%%%%%%%
%%%%%%%%%%%%%%%%%%%%%

\section{Main results}

The results of Berestycki, Hamel and Matano for the classical problem \eqref{eqlocale} say that there exists an entire solution $u(t,x)$ that behaves like a planar wave as $t\to-\infty$ and as a planar wave multiplied by $u_\infty(x)$ as $t\to+\infty$, where $u_\infty$ solves \eqref{EQ:BHM0}. Moreover, they were able to classify the solutions to \eqref{EQ:BHM0} with respect to the geometry of $K$, providing us with a good insight on how  the latter influences the large time dynamics.\par

Our goal here is to obtain corresponding results for the nonlocal problem \eqref{P}.
In the first place, we will prove that there exists an entire solution to \eqref{P} with analogous properties as in the classical case.
Then, we will study more precisely how the geometry of $K$ affects its large time behaviour and we will prove
that this question is equivalent to investigating under which circumstances a certain Liouville type property holds.

\subsection{General existence results}
Our first main result deals with the existence and uniqueness of an entire (i.e. time-global) solution to problem \eqref{P}. %By ``entire solution" we mean a time-global solution.
\begin{theo}[Existence of an entire solution]\label{TH:EXIST:ENTIRE}
Assume \eqref{C1},~\eqref{C2},~\eqref{C3},~\eqref{C4}~and let $(\phi,c)$ be as in \eqref{C4}. Suppose that $J\in \mathbb{B}_{1,\infty}^\alpha(\Omega;\delta)$ for some $\alpha\in(0,1)$ and that
\begin{align}
\max_{[0,1]}f'<\inf_{\Omega}\mathcal{J}^\delta, %\text{ and that $J\in \mathbb{B}_{1,\infty}^\alpha(\Omega;\delta)$, for some $\alpha\in(0,1)$}.
\label{fJdelta-pos}
\end{align}
Then, there exists an entire solution $u\in C^2(\R,C^{0,\alpha}(\overline{\Omega}))$ to \eqref{P} satisfying $0<u<1$ and $\partial_tu>0$ in $\R\times\overline{\Omega}$. Moreover,
\begin{align}
|u(t,x)-\phi(x_1+ct)|\underset{t\to-\infty}{\longrightarrow}0 \, {\mbox{ uniformly in }}x\in\overline{\Omega}. \label{unique}
\end{align}
Furthermore, \eqref{unique} determines a unique bounded entire solution to \eqref{P}.
If, in addition, \eqref{C5} holds, then there exists a continuous solution, $u_\infty:\overline{\Omega}\to(0,1]$, to
\begin{align}
\left\{
\begin{array}{r l}
Lu_\infty+f(u_\infty)=0 & \text{in }\overline{\Omega}, \vspace{3pt}\\
u_\infty(x)\to1 & \text{as }|x|\to\infty,
\end{array} \tag{$P_\infty$}\label{Pinf}
\right.
\end{align}
such that
\begin{align}
|u(t,x)-u_\infty(x)\hspace{0.1em}\phi(x_1+ct)|\underset{t\to+\infty}{\longrightarrow} 0 \, {\mbox{ locally uniformly in }}x\in\overline{\Omega}. \label{asympt:large:time}
\end{align}
\end{theo}
\begin{figure}[!ht]
\centering
\subfloat[$t=50$]{\includegraphics[width=4cm,height=2.7cm]{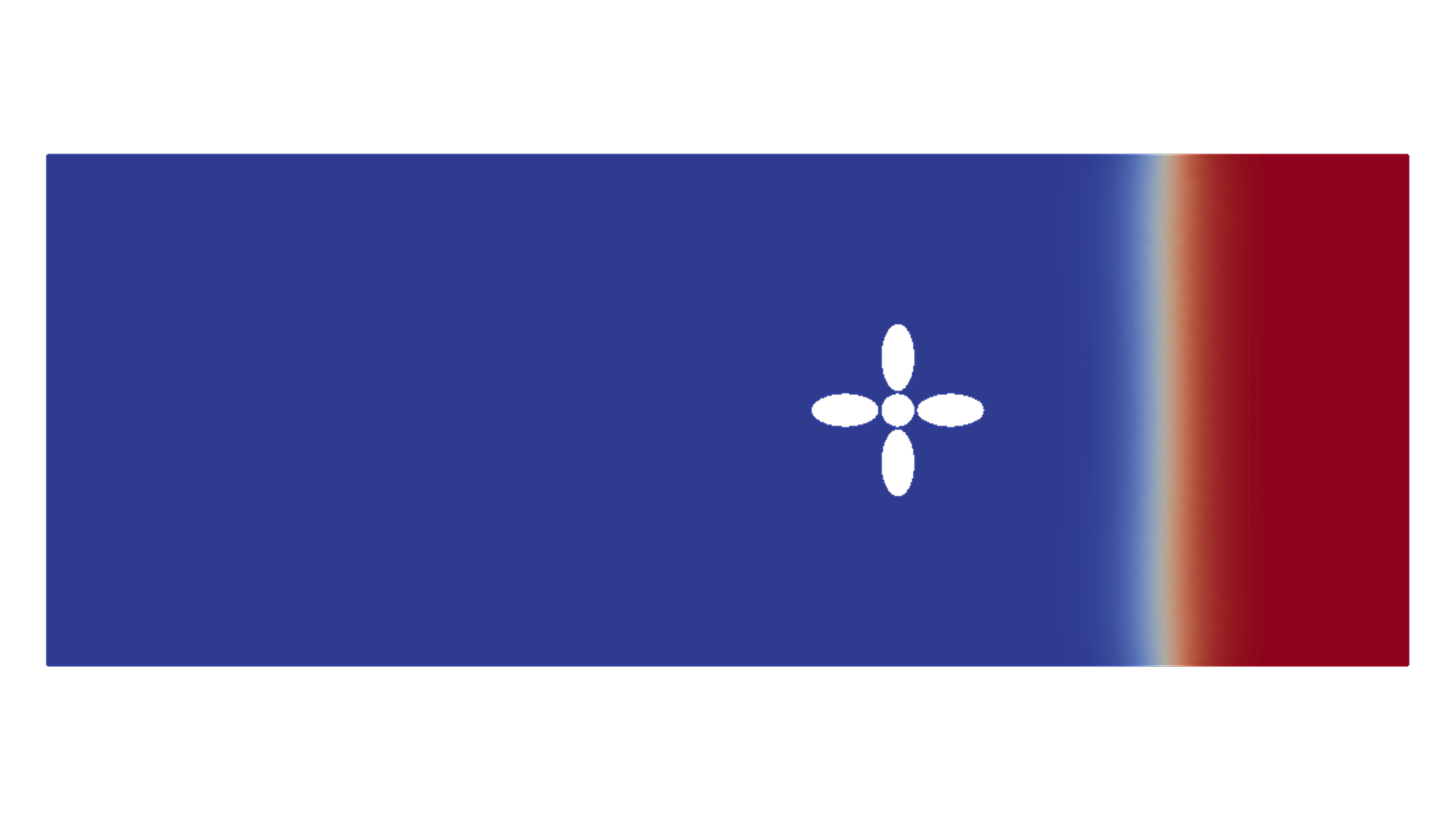}}
\subfloat[$t=100$]{\includegraphics[width=4cm,height=2.7cm]{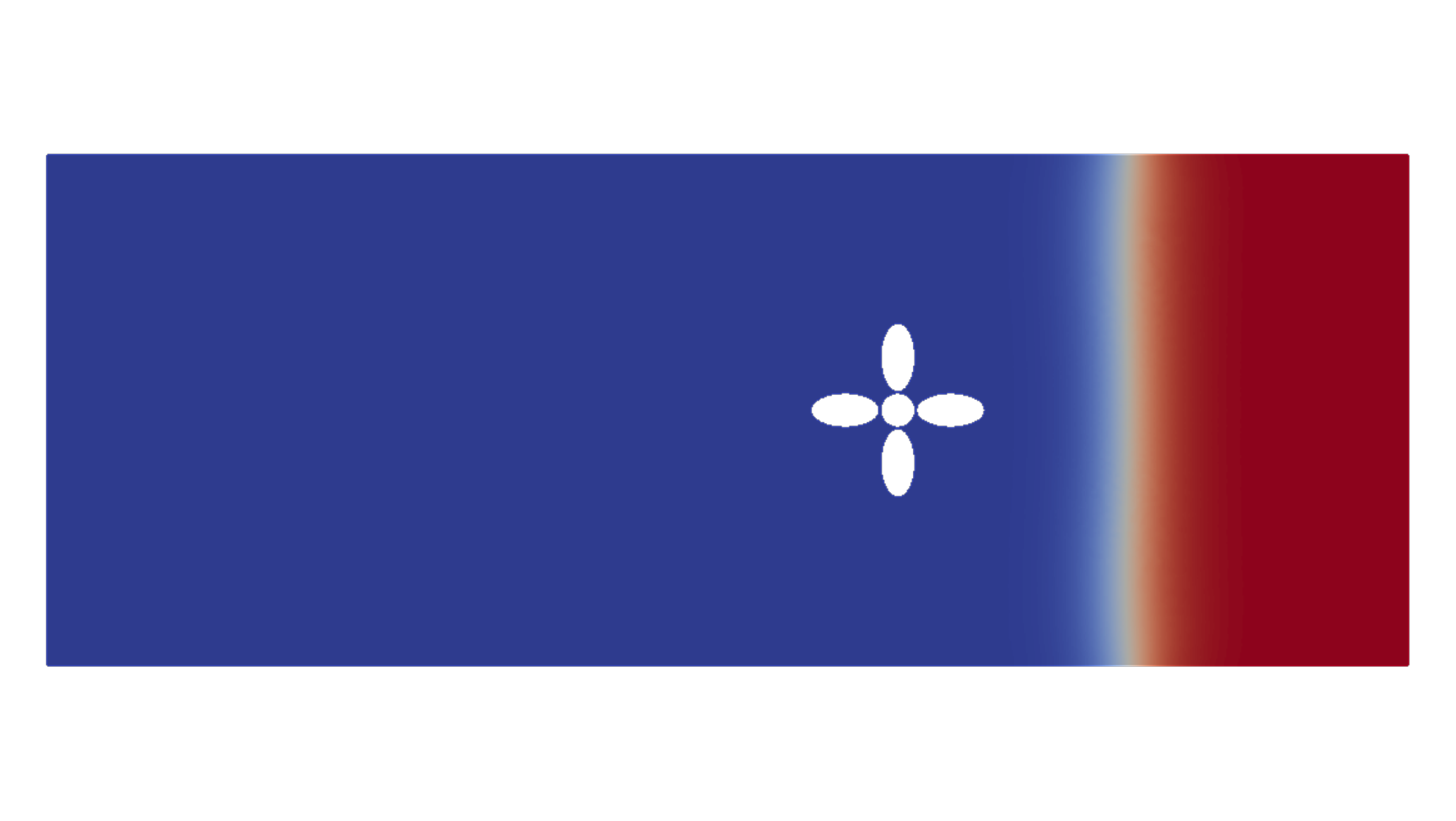}}
%\subfloat[$t=300$]{\includegraphics[width=3.2cm, height=2cm]{test-150.pdf}}
\subfloat[$t=200$]{\includegraphics[width=4cm,height=2.7cm]{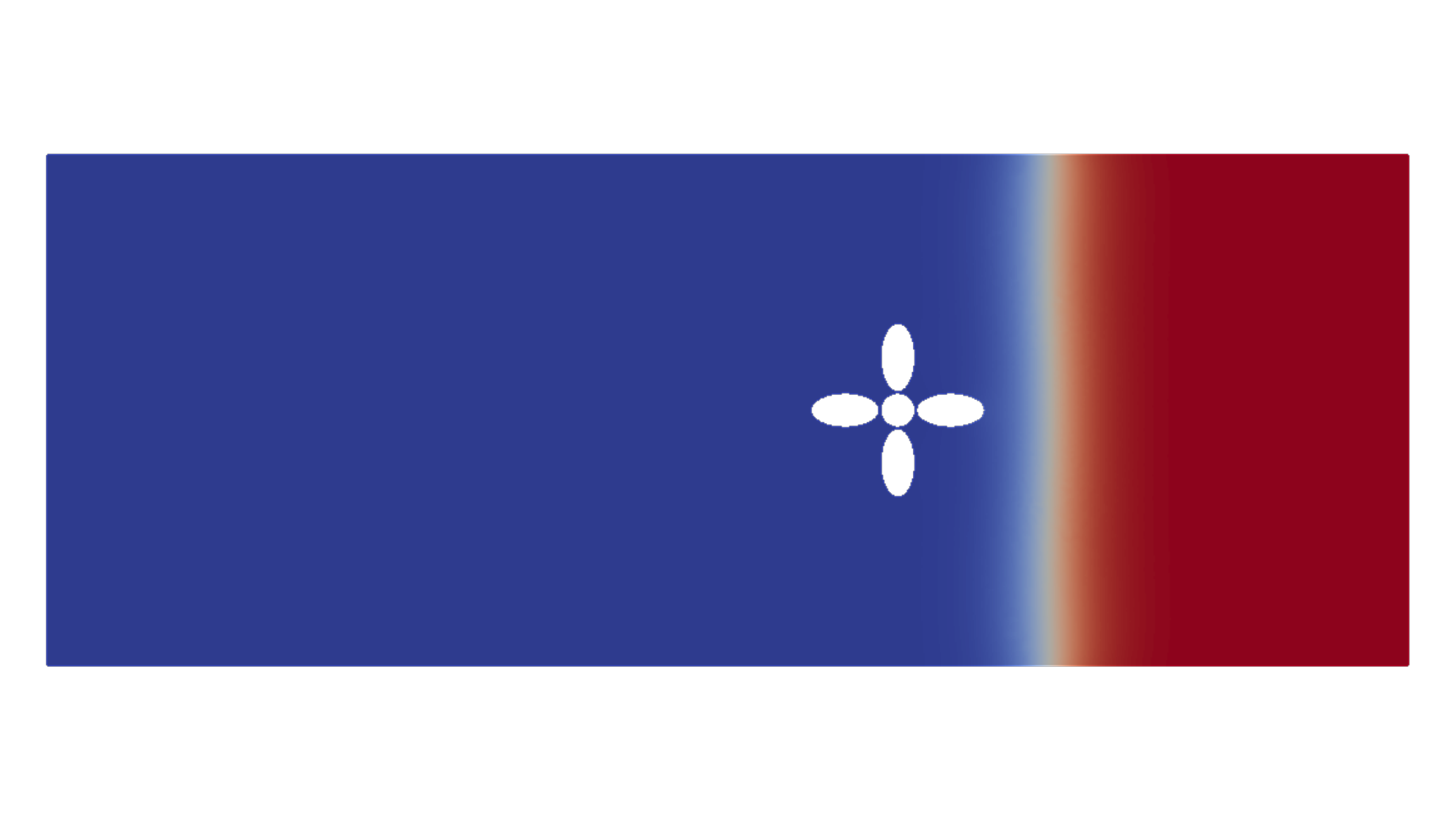}}
\subfloat[$t=300$]{\includegraphics[width=4cm,height=2.7cm]{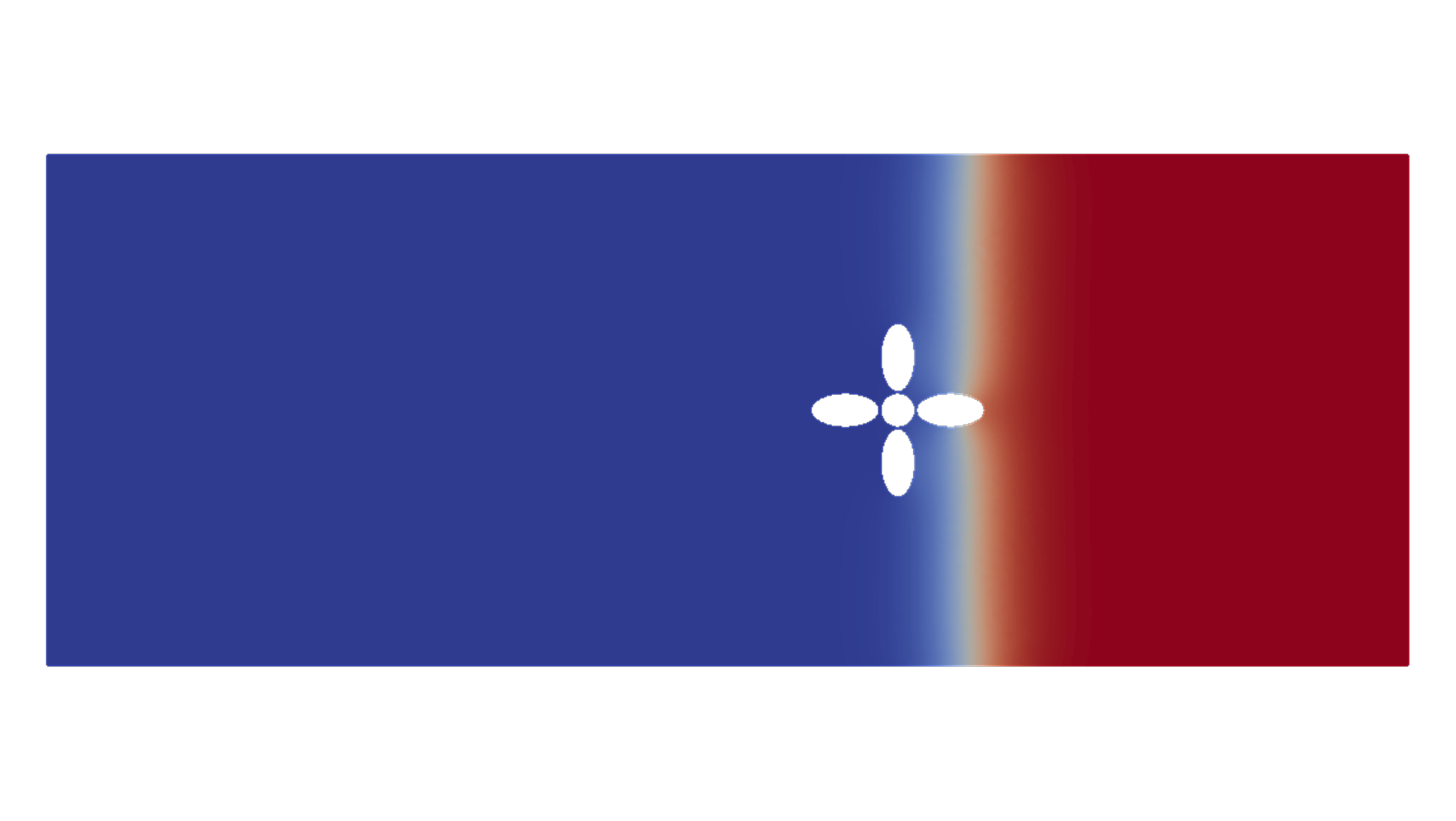}}

\subfloat[$t=400$]{\includegraphics[width=4cm,height=2.7cm]{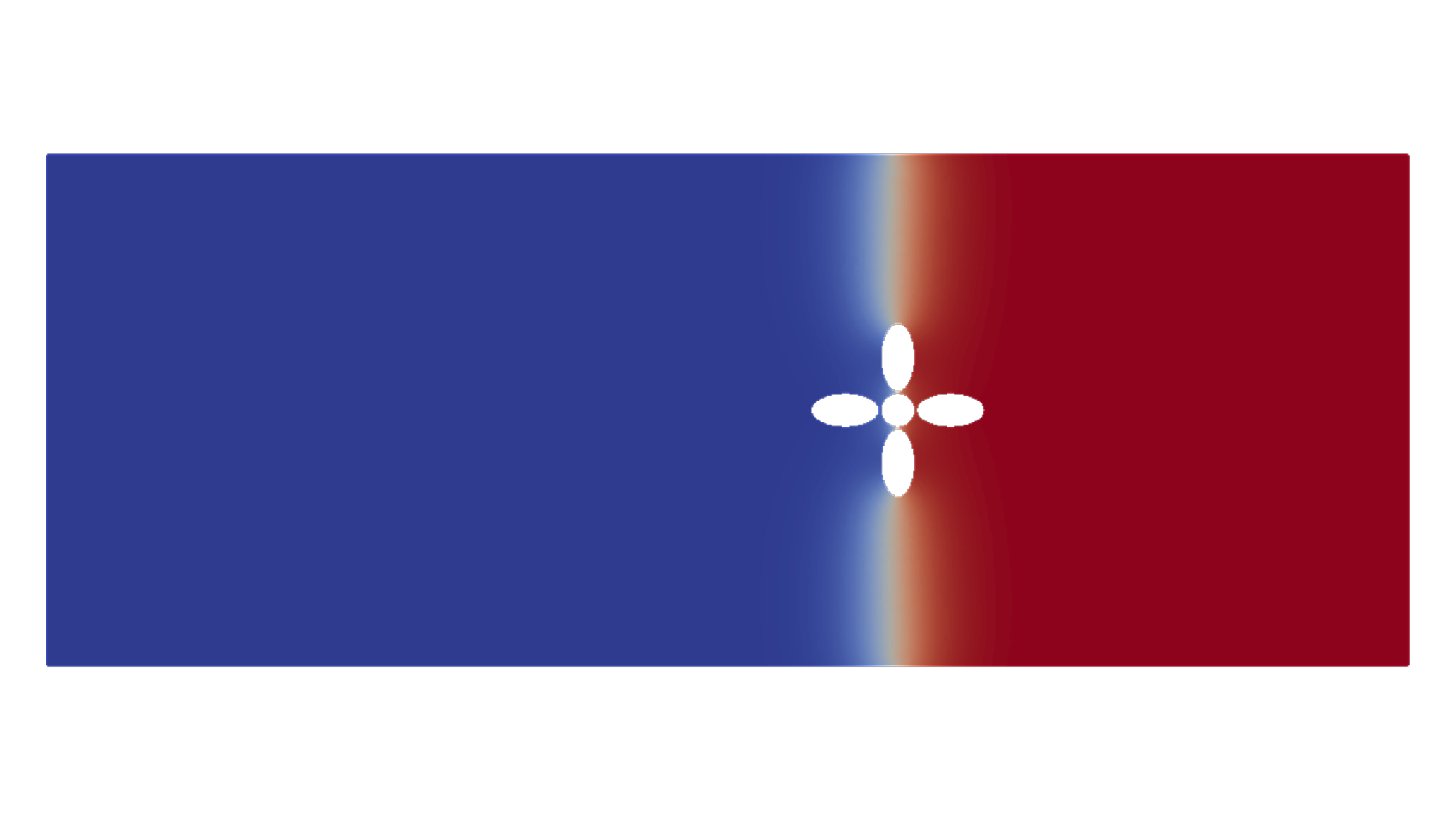}}
%\subfloat[$t=150$]{\includegraphics[width=3.2cm]{./simu/test-350.pdf}}
\subfloat[$t=450$]{\includegraphics[width=4cm,height=2.7cm]{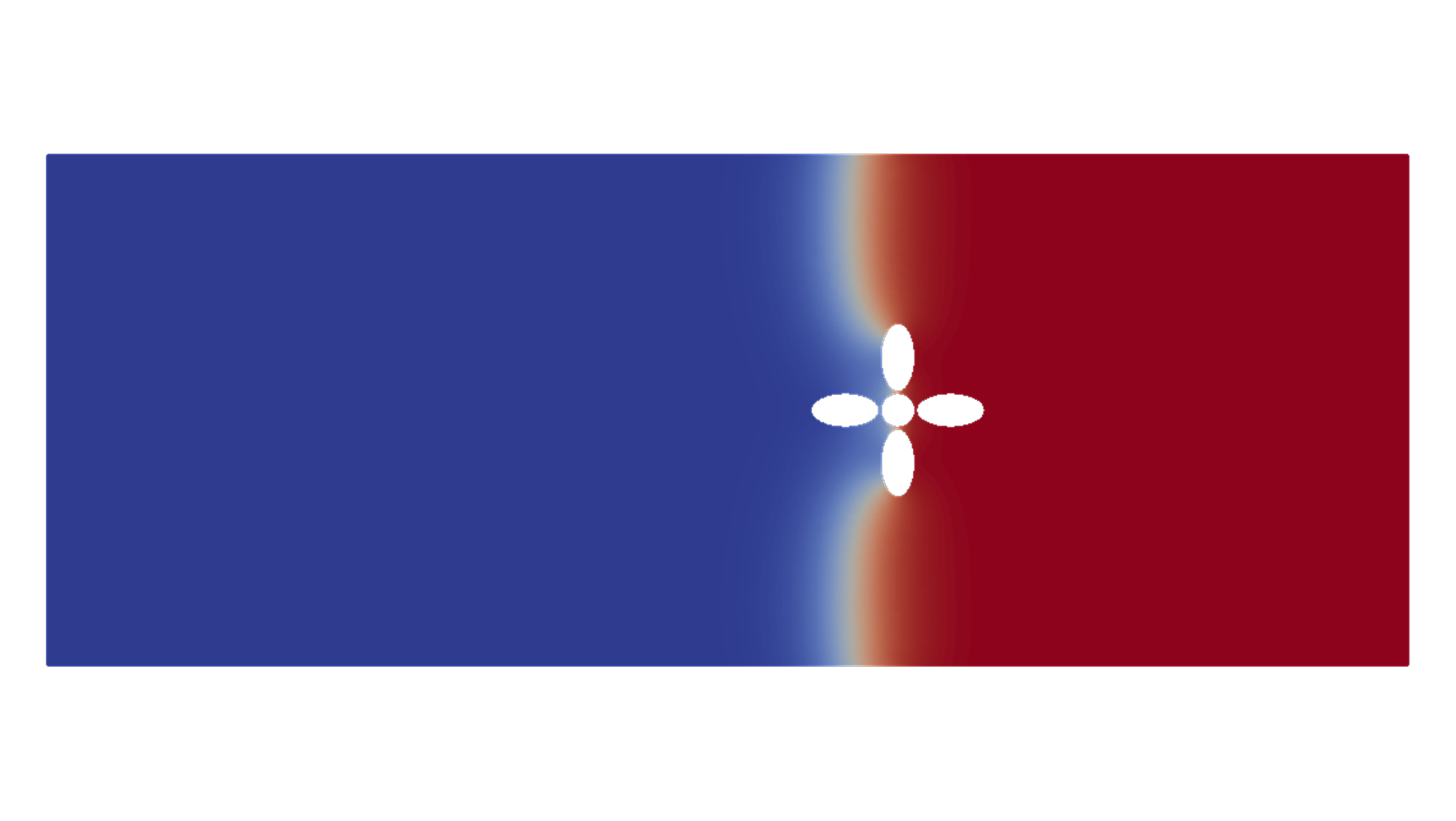}}
\subfloat[$t=500$]{\includegraphics[width=4cm,height=2.7cm]{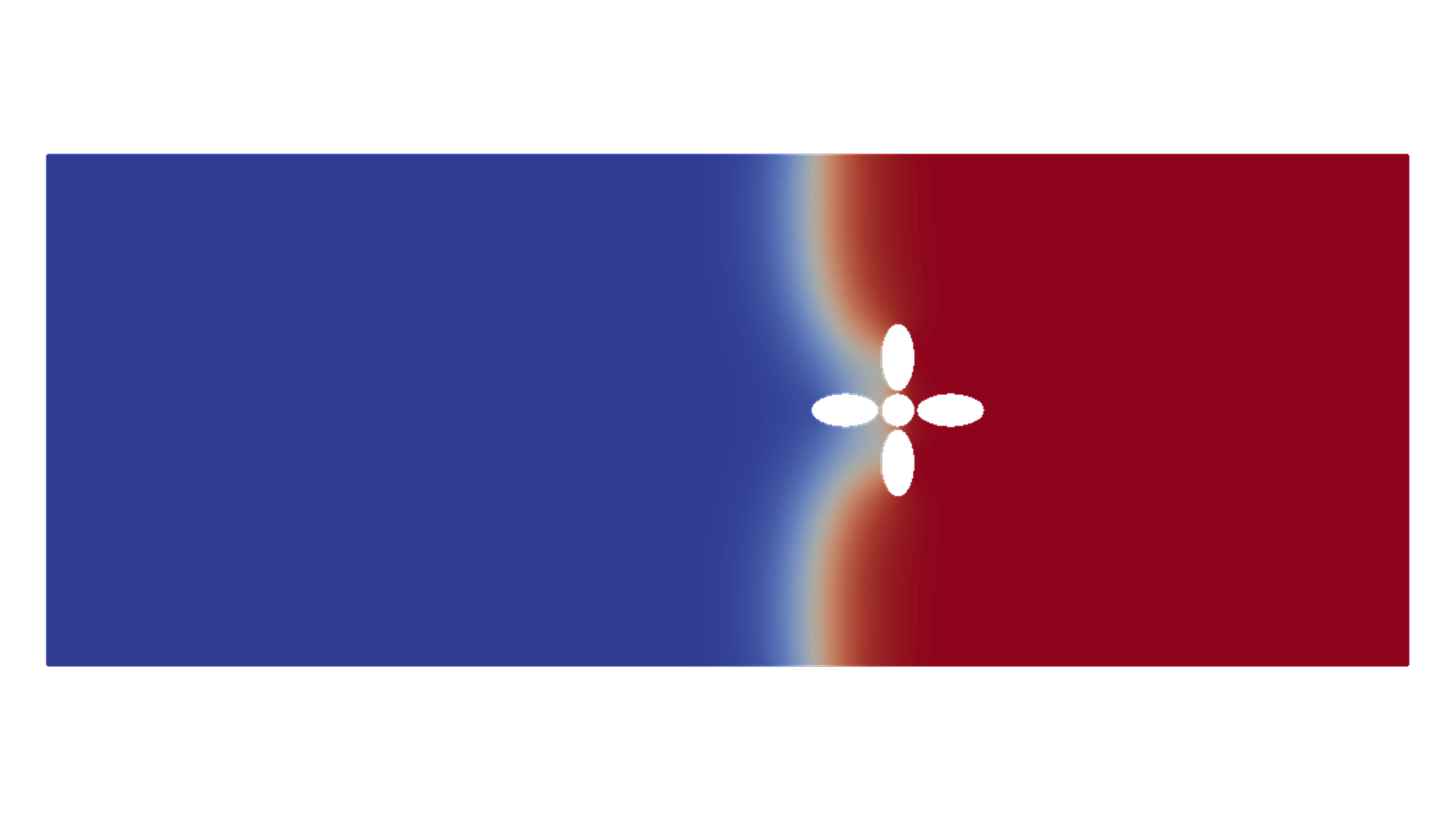}}
\subfloat[$t=550$]{\includegraphics[width=4cm,height=2.7cm]{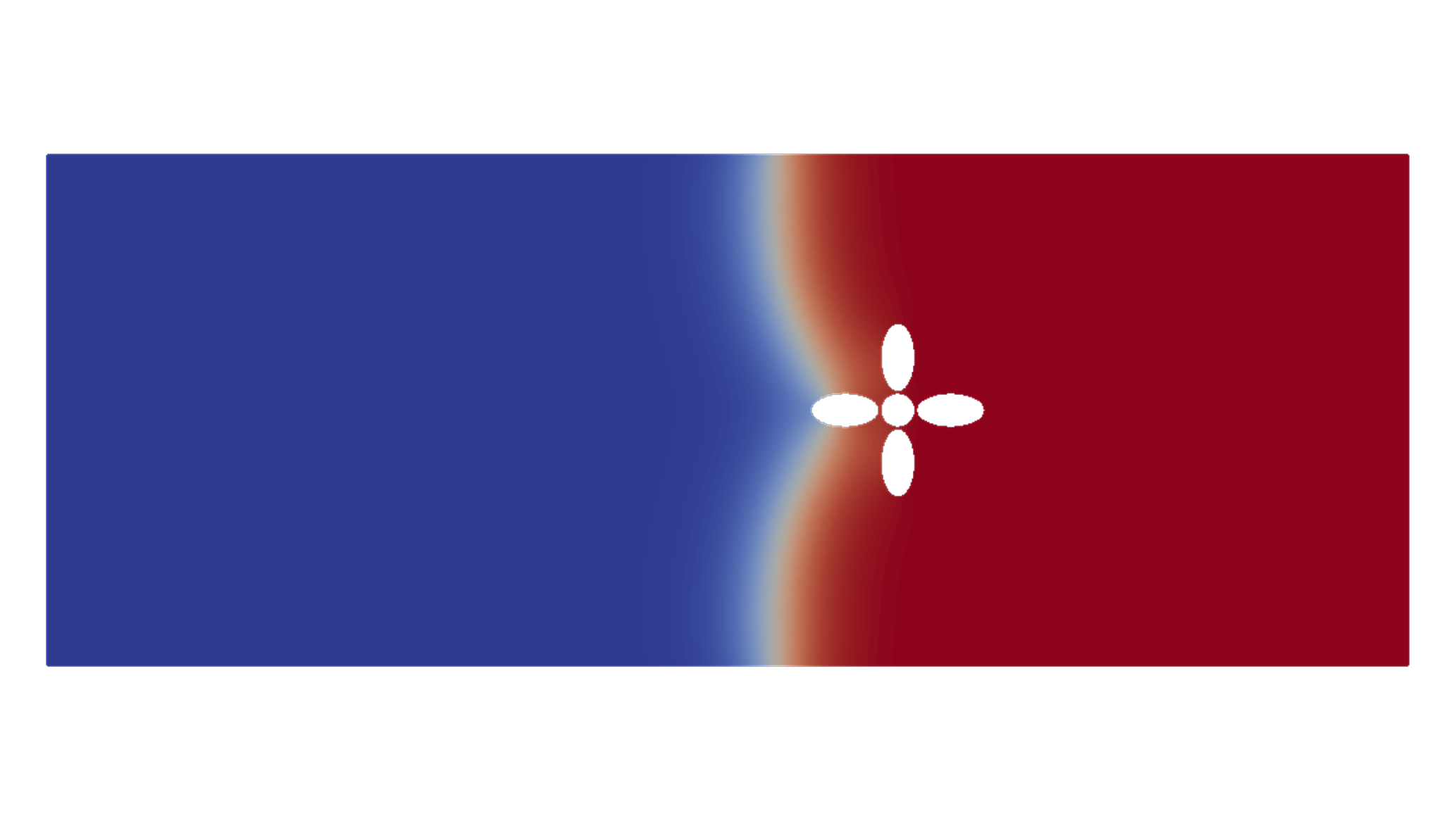}}

%\subfloat[$t=600$]{\includegraphics[width=3.2cm]{./simu/test-550.pdf}}
\subfloat[$t=600$]{\includegraphics[width=4cm,height=2.7cm]{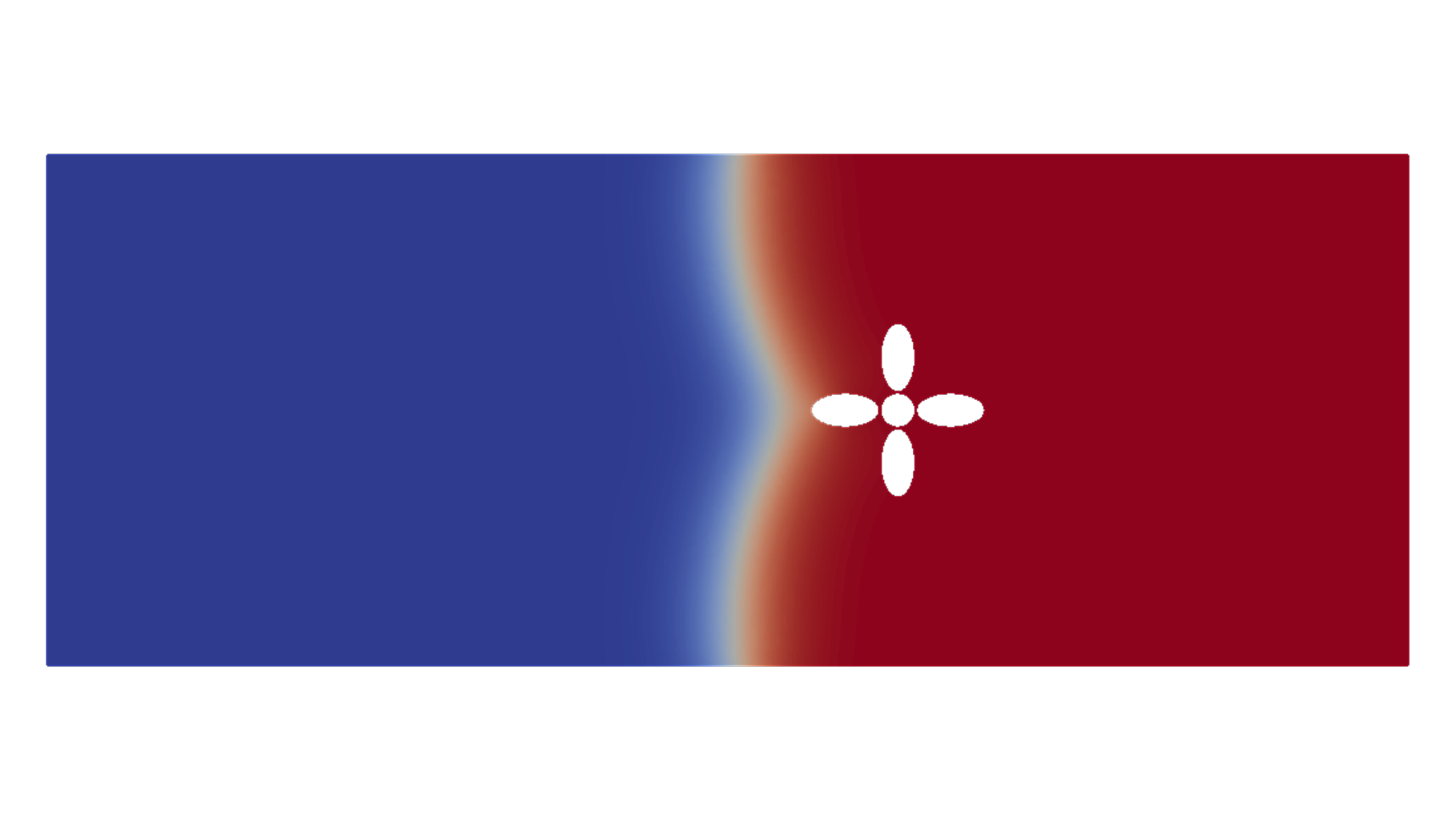}}
\subfloat[$t=650$]{\includegraphics[width=4cm,height=2.7cm]{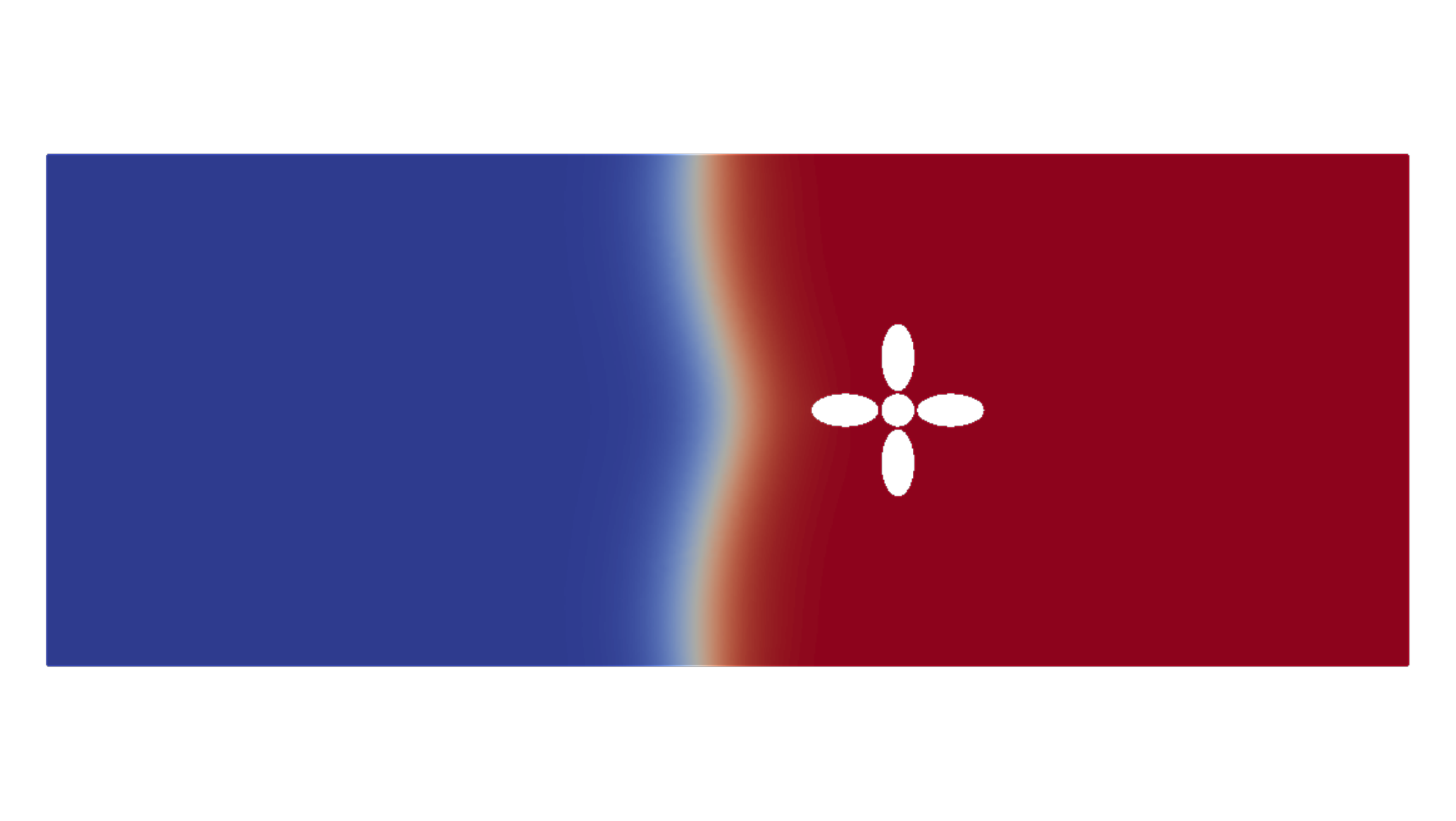}}
\subfloat[$t=700$]{\includegraphics[width=4cm,height=2.7cm]{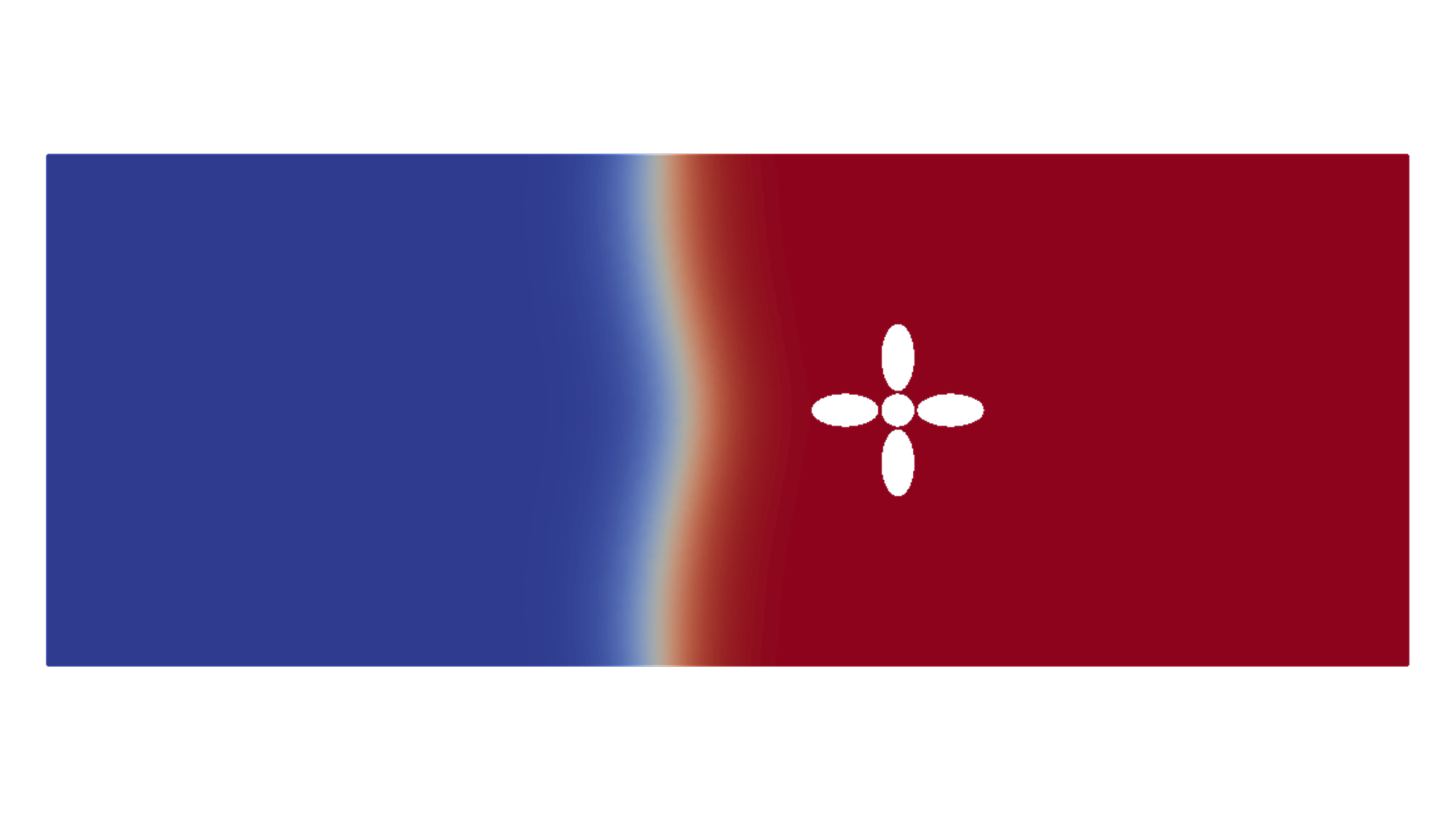}}
\subfloat[$t=750$]{\includegraphics[width=4cm,height=2.7cm]{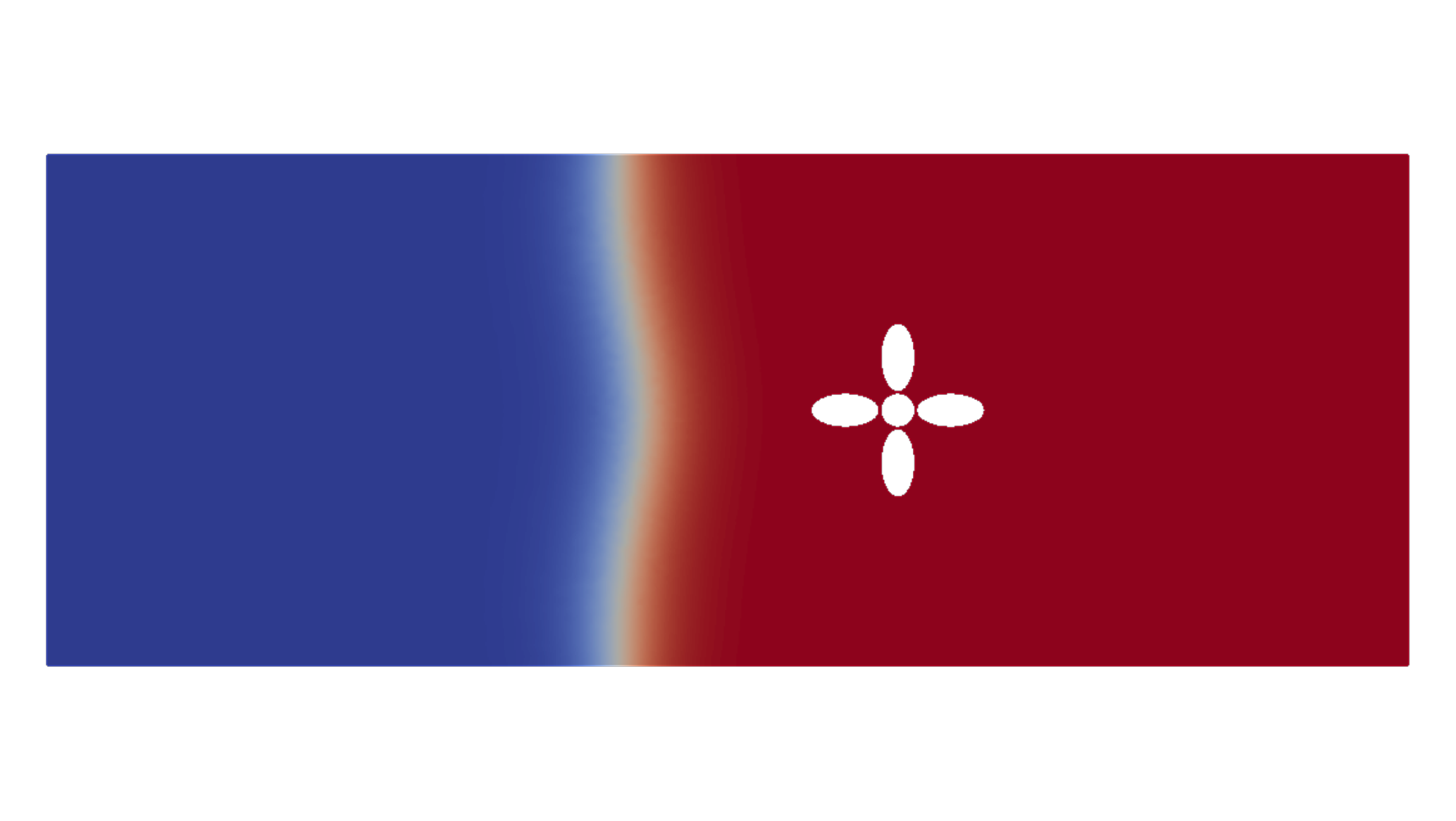}}
\caption{\footnotesize Numerical approximation of the solution of problem \eqref{P} at different times, starting from a Heaviside type initial density. For the simulation, $J(x)\sim e^{-|x|^2}\mathds{1}_{B_1}(x)$, the distance $\delta$ is the Euclidean distance and the obstacle $K$ is the union of the unit disk and four ellipsoids. On the domain $\O:=[30,-50]\times[-15,15]\setminus K$ we perform an IMEX Euler scheme in time combined with a finite element method in space with a time step of $0.075$. We observe that the solution behaves like a generalised transition wave.}
\end{figure}
\begin{remark}
We have stated, for simplicity, the existence of an entire solution that propagates in the direction $e_1=(1,0,\cdots,0)$. However, this restriction is immaterial and our arguments also yield that, for every $e\in\S^{N-1}$, there exists an entire solution propagating in the direction $e$ and satisfying the same properties as above.
\end{remark}
\begin{remark}
A consequence of the uniqueness part of Theorem \ref{TH:EXIST:ENTIRE} is that the entire solution $u(t,x)$ shares the same symmetry as $K$ in the hyperplane $\{x_1\}\times\R^{N-1}$. More precisely, if $\mathscr{T}$ is an isometry of $\R^{N-1}$ such that $(x_1,\mathscr{T}x')\in\overline{\Omega}$ for any $(x_1,x')\in\overline{\Omega}$, then
$$ u(t,x_1,\mathscr{T}x')=u(t,x_1,x') \, {\mbox{ for all }}(t,x)\in\R\times\overline{\Omega}. $$
\end{remark}
\begin{remark}
If $(\Omega,\delta)$ does not have the $J$-covering property, then we still have the existence of an entire solution satisfying \eqref{unique}, but we only have that $0\leq u(t,x)\leq 1$ and $\partial_tu(t,x)\geq0$ for any $(t,x)\in\R\times\overline{\Omega}$ (as opposed to the strict inequalities in Theorem~\ref{TH:EXIST:ENTIRE}). Moreover, the uniqueness may \emph{fail} because the strong maximum principle does \emph{not} hold in this case.
\end{remark}

\subsection{Large time behaviour}

As in the local case \cite{Berestycki2009d}, the large time behaviour of $u(t,x)$ depends on the geometry of $K$. Hamel, Valdinoci and the authors have shown in \cite{Brasseur2019} that, if $\delta$ is the Euclidean distance and $K$ is convex, then the problem \eqref{Pinf} admits a Liouville type property: namely, the only possible solution to \eqref{Pinf} is the trivial solution $u_\infty\equiv1$. We prove that this fact can be extended to arbitrary quasi-Euclidean distances (up to a slight additional assumption on $J$), which then results in the following theorem:
\begin{theo}\label{TH:LIOUVILLE}
Suppose all the assumptions of Theorem~\ref{TH:EXIST:ENTIRE} and that $K\subset\R^N$ is convex.
If $\delta(x,y)\not\equiv|x-y|$ suppose, in addition, that $J$ is non-increasing.
Then, there exists a unique entire solution $u(t,x)$ to \eqref{P} in $\overline{\Omega}$ such that $0<u(t,x)<1$ and $\partial_tu(t,x)>0$ for all $(t,x)\in\R\times\overline{\Omega}$ and
$$ |u(t,x)-\phi(x_1+ct)|\underset{t\to\pm\infty}{\longrightarrow} 0 \, {\mbox{ locally uniformly in }}x\in\overline{\Omega}. $$
\end{theo}
In other words: if the obstacle $K$ is convex, then the entire solution $u(t,x)$ to \eqref{P} will eventually recover the shape of the planar travelling wave $\phi(x_1+ct)$ as $t\to+\infty$, i.e. the presence of an obstacle will not alter the large time behaviour of the solution $u(t,x)$. This is a consequence of the fact that \eqref{Pinf} satisfies a Liouville type property, see Figure \ref{fig-simu-convex}.
\begin{figure}[!ht]
\centering
\subfloat[$t=0$]{\includegraphics[height=2.25cm, width=3.9cm]{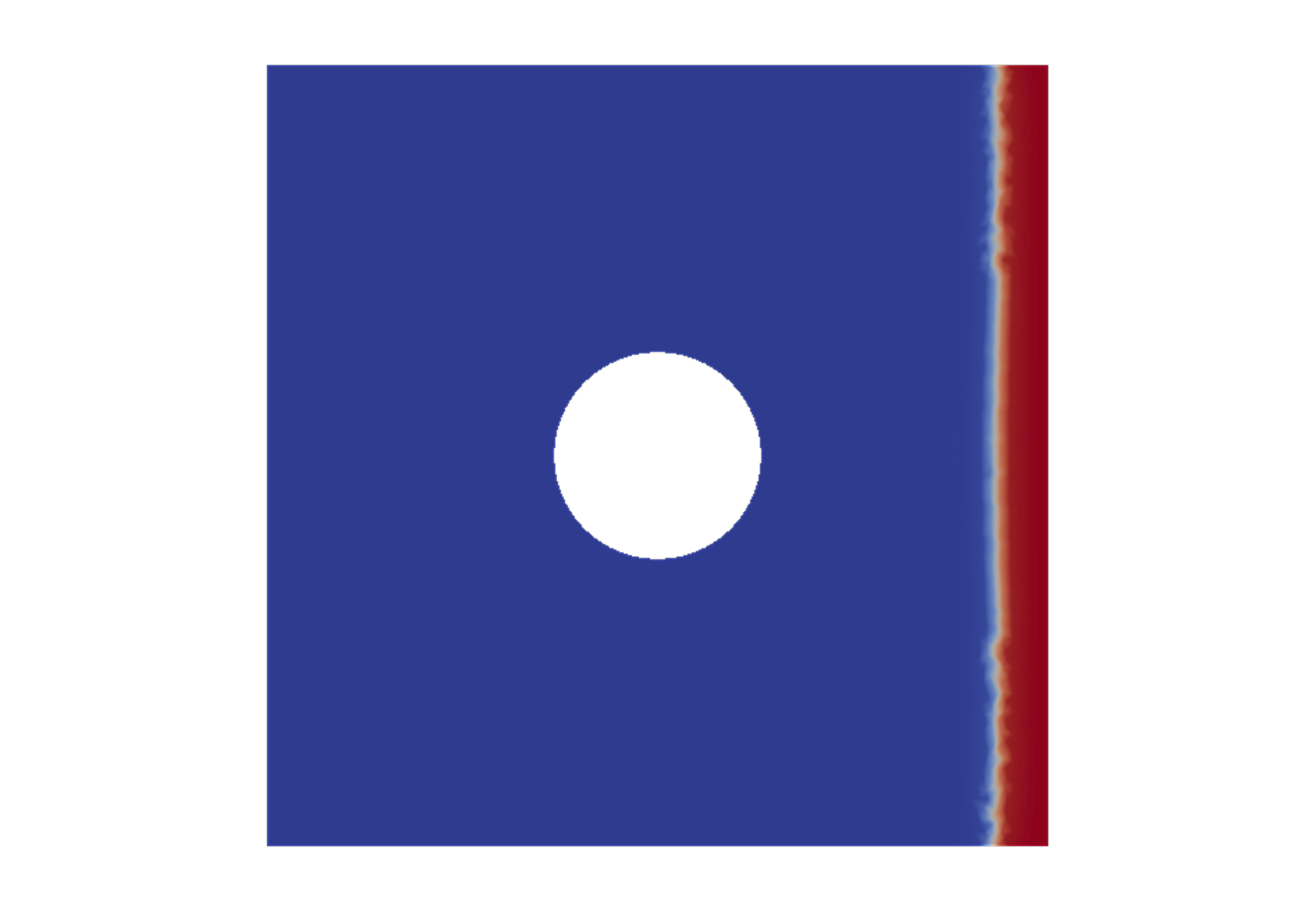}}
\subfloat[$t=150$]{\includegraphics[height=2.25cm,width=3.9cm]{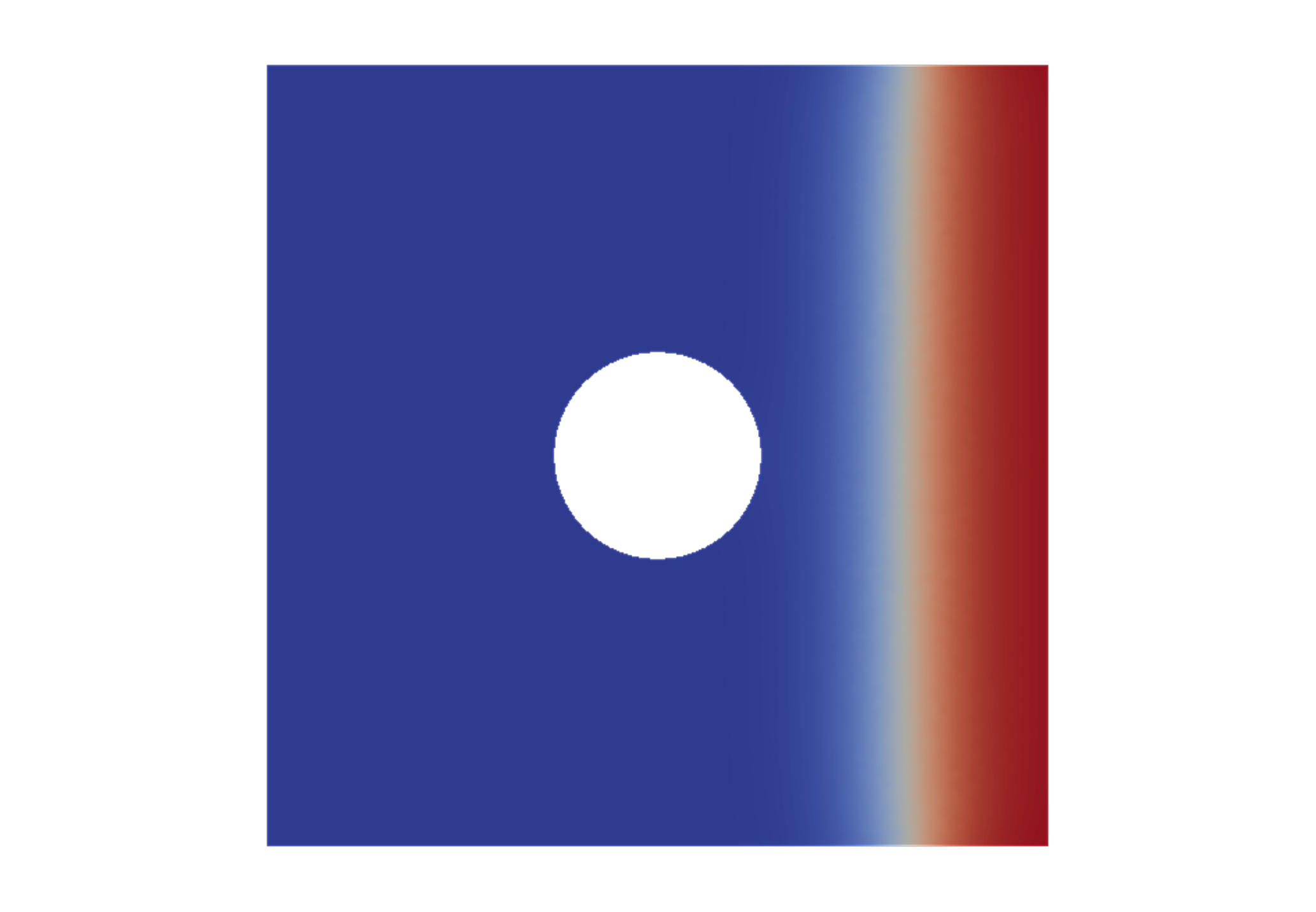}}
\subfloat[$t=300$]{\includegraphics[height=2.25cm,width=3.9cm]{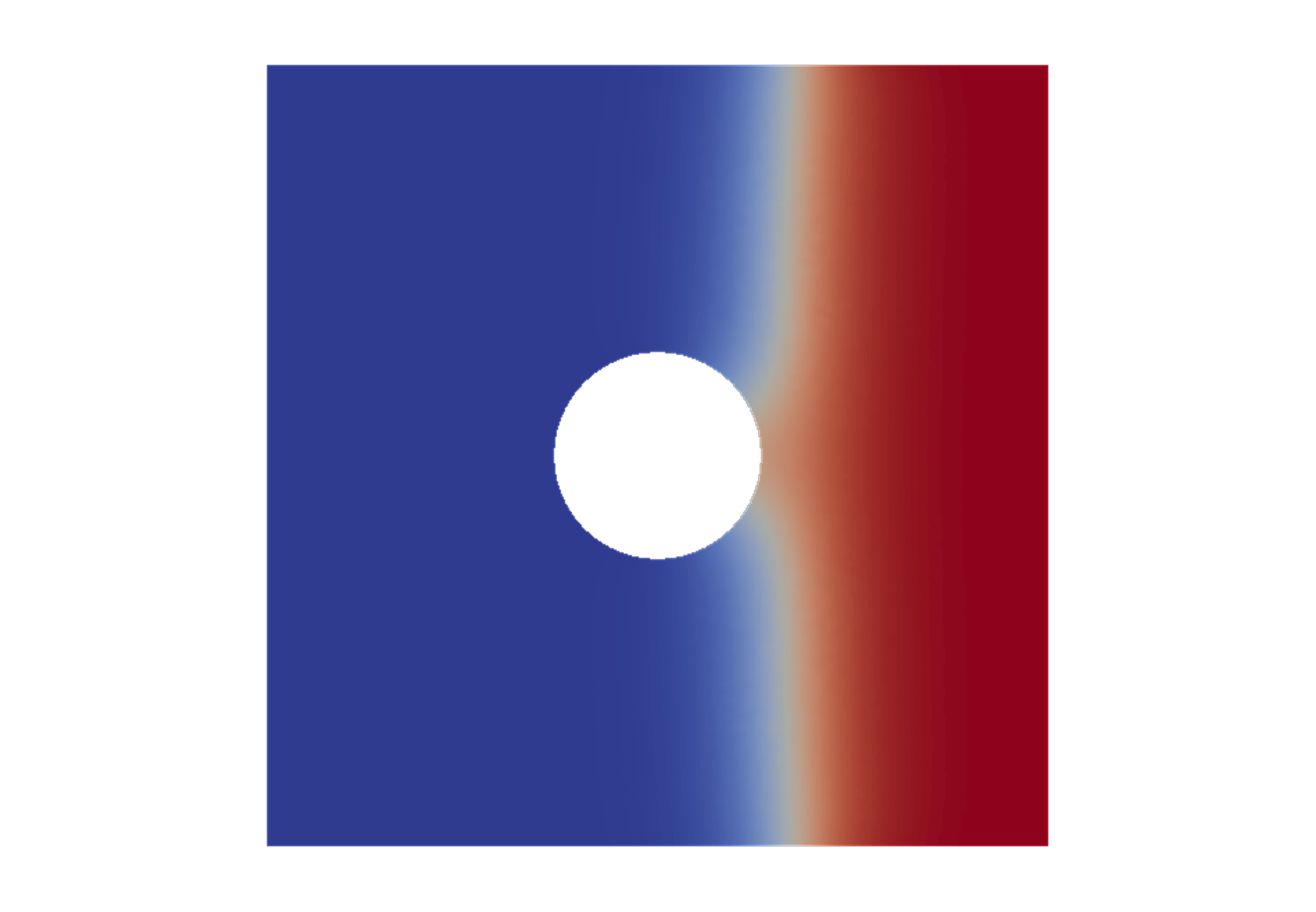}}
\subfloat[$t=450$]{\includegraphics[height=2.25cm,width=3.9cm]{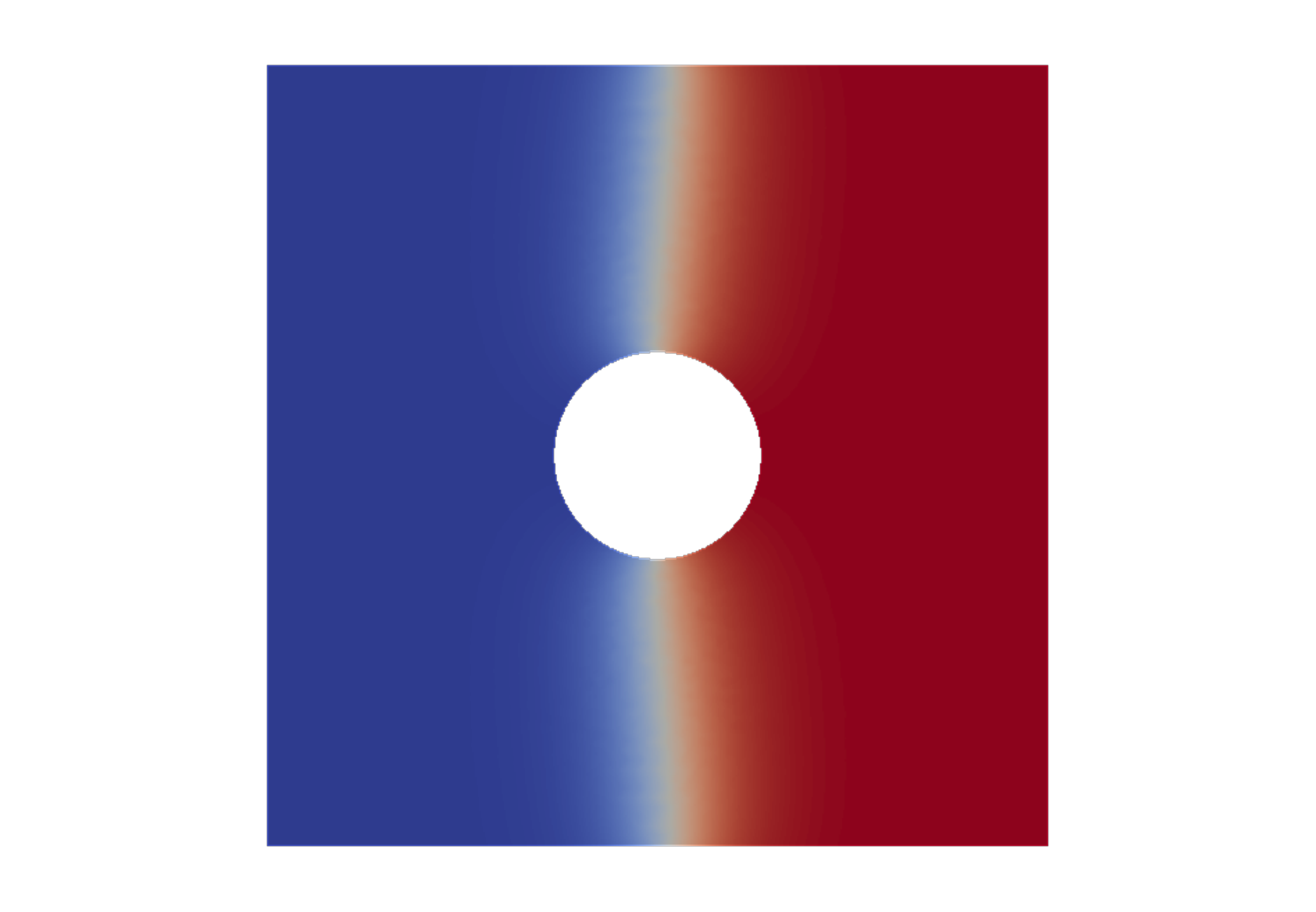}}

\subfloat[$t=600$]{\includegraphics[height=2.25cm,width=3.9cm]{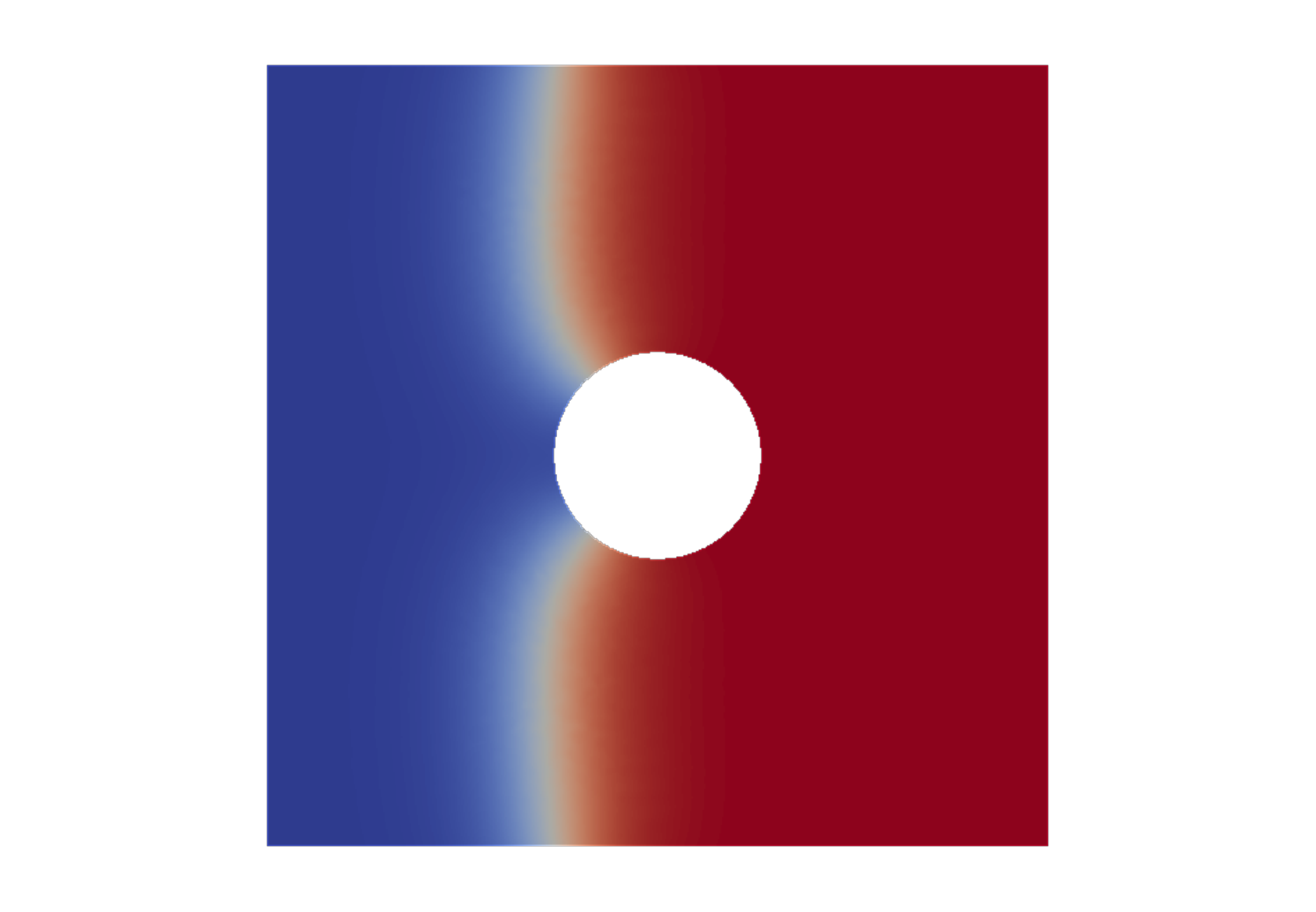}}
\subfloat[$t=750$]{\includegraphics[height=2.25cm,width=3.9cm]{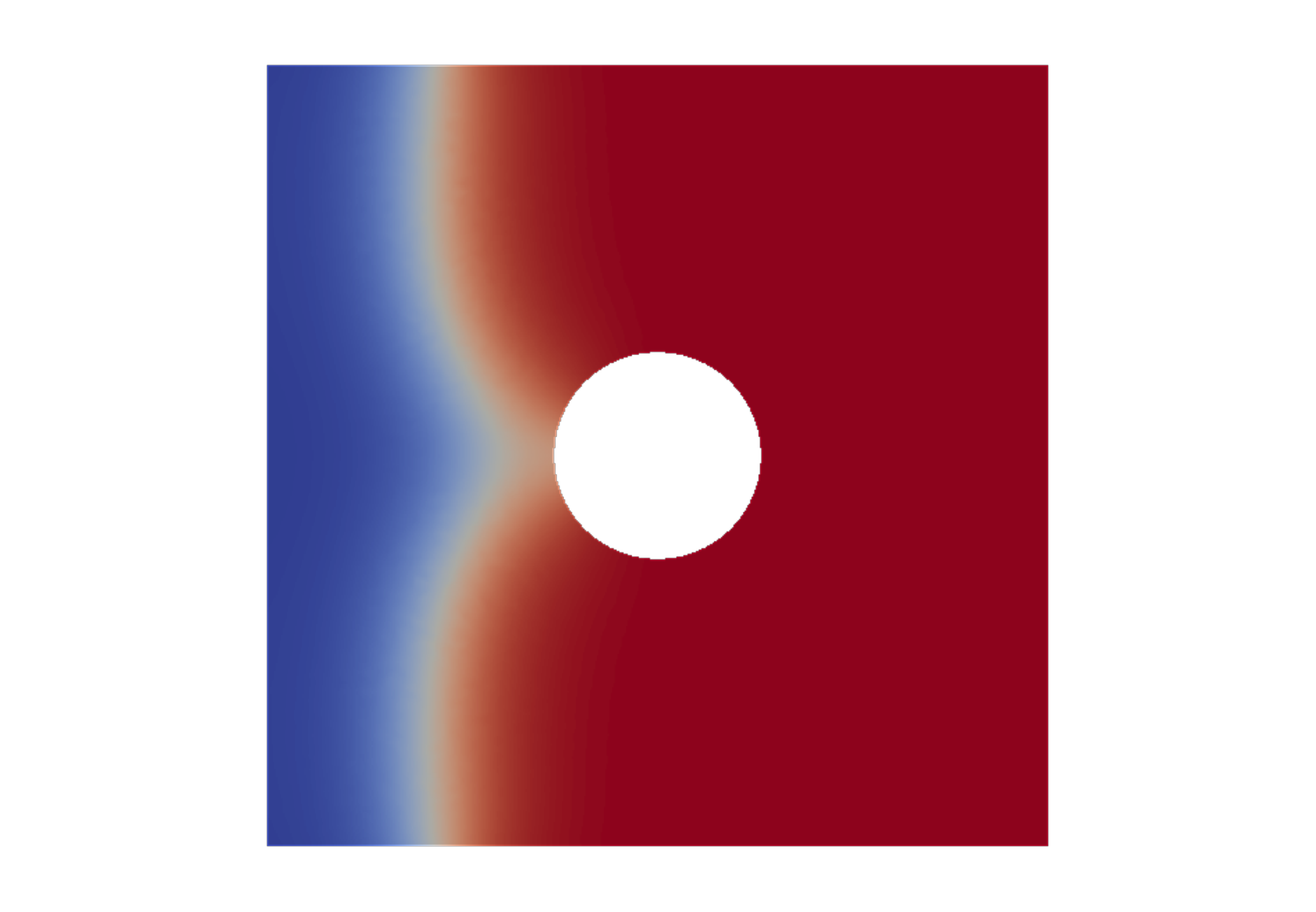}}
\subfloat[$t=900$]{\includegraphics[height=2.25cm,width=3.9cm]{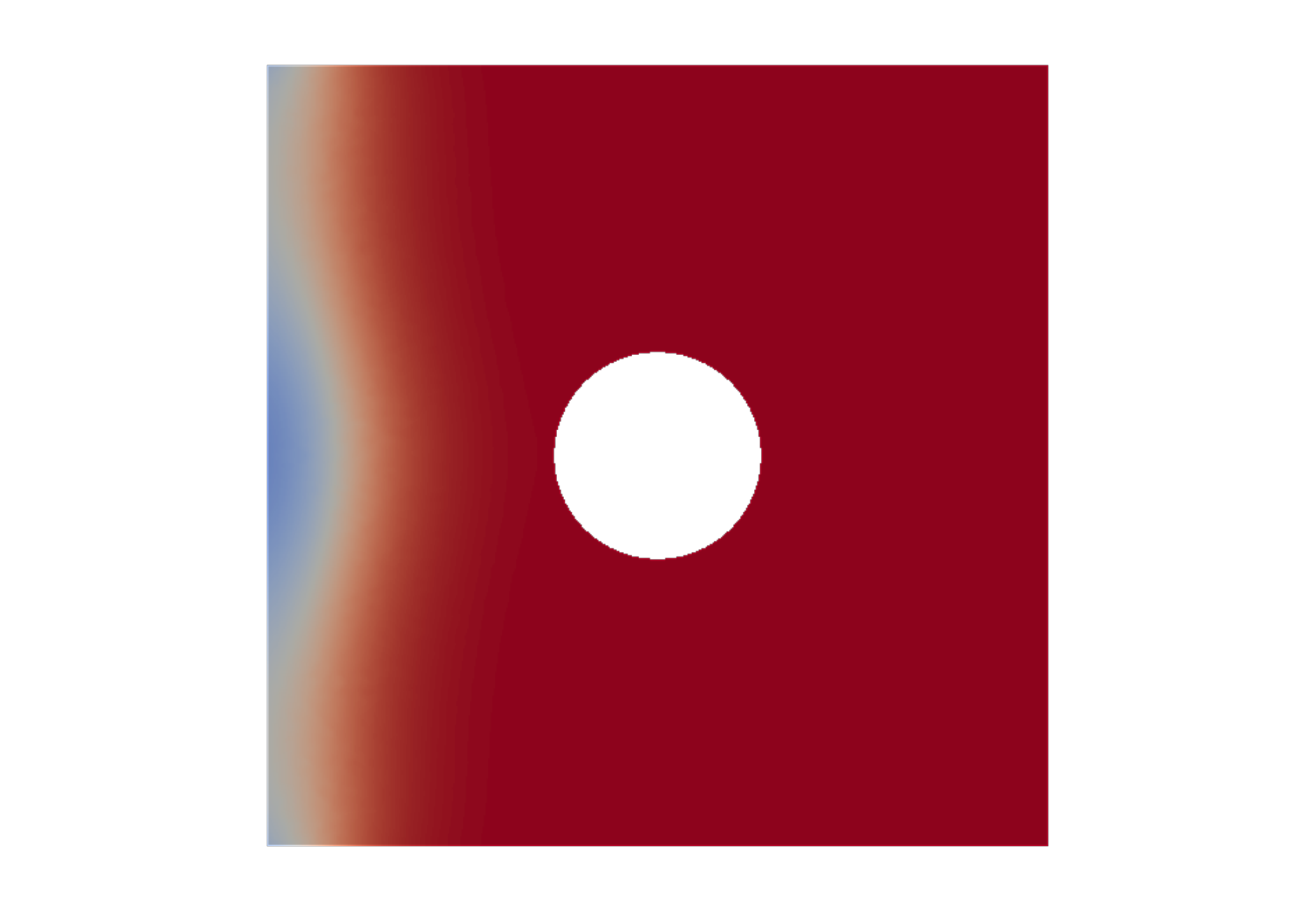}}
\subfloat[$t=1000$]{\includegraphics[height=2.25cm,width=3.9cm]{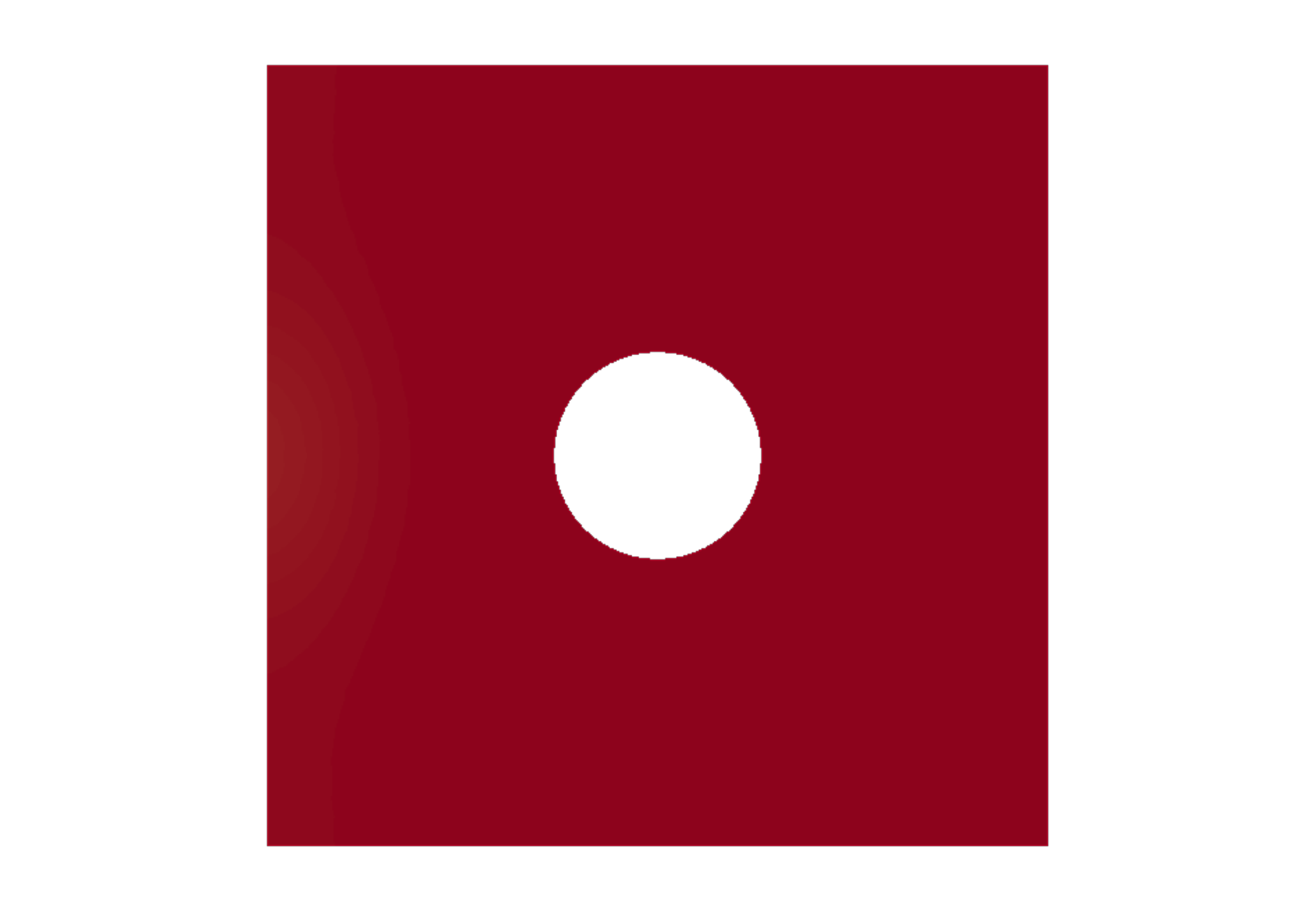}}
\caption{\footnotesize Numerical approximation of the solution of problem \eqref{P} at different times, starting from a Heaviside type initial density. For the simulation, $J(x)\sim e^{-|x|^2}\mathds{1}_{B_1}(x)$, the distance $\delta$ is the Euclidean distance and the obstacle $K$ is a disk of radius $4$. On the domain $\O:=[-15,15]^2\setminus K$ we perform an IMEX Euler scheme in time combined with a finite element method in space with a time step of $0.05$. We observe that the solution converges to a trivial asymptotic profile as $t\to\infty$, namely 1.}\label{fig-simu-convex}
\end{figure}

However, the authors have shown in \cite{Brasseur2018} that there exist obstacles $K$ as well as a datum $(J,f)$ for which this property is \emph{violated}, i.e. such that \eqref{Pinf} admits a non-trivial solution $\widetilde{u}_\infty\in C(\overline{\Omega})$ with $0<\widetilde{u}_\infty<1$ in $\overline{\Omega}$. Hence, the picture described at Theorem~\ref{TH:LIOUVILLE} \emph{cannot} be expected for general obstacles.
Nevertheless, this does not immediately imply that the solution $u_\infty$ to \eqref{Pinf} arising in Theorem~\ref{TH:EXIST:ENTIRE} is not constant.
We prove that, whether the unique entire solution $u(t,x)$ to \eqref{P} satisfying \eqref{unique} recovers the shape of the planar travelling wave $\phi(x_1+ct)$ as $t\to+\infty$ is \emph{equivalent} to the question of whether \eqref{Pinf} satisfies the Liouville type property. Precisely,
\begin{theo}\label{NonLiouville2}
Suppose all the assumptions of Theorem~\ref{TH:EXIST:ENTIRE}.
Let $u(t,x)$ be the unique bounded entire solution to \eqref{P} satisfying \eqref{unique}. Let $u_\infty\in C(\overline{\Omega})$ be the solution to \eqref{Pinf} such that \eqref{asympt:large:time} holds, i.e. such that
$$ |u(t,x)-u_\infty(x)\hspace{0.1em}\phi(x_1+ct)|\underset{t\to+\infty}{\longrightarrow} 0  {\mbox{ locally uniformly in }}x\in\overline{\Omega}.$$
Then, $u_\infty\equiv1$ in $\overline{\Omega}$ if, and only if, \eqref{Pinf} satisfies the Liouville property.
\end{theo}
As a consequence of Theorem~\ref{NonLiouville2} and of \cite[Theorems 1.1, 1.3]{Brasseur2018} we obtain
\begin{cor}\label{COR:CONTR}
There exist a smooth, simply connected, non-starshaped compact set $K\subset\R^N$, a quasi-Euclidean distance $\delta\in\mathcal{Q}(\overline{\Omega})$ and a datum $(J,f)$ satisfying all the assumptions of Theorem~\ref{TH:EXIST:ENTIRE}, such that the unique bounded entire solution $u(t,x)$ to \eqref{P} satisfying \eqref{unique} does not recover the shape of a planar travelling wave in the large time limit, that is
$$ |u(t,x)-u_\infty(x)\hspace{0.1em}\phi(x_1+ct)|\underset{t\to+\infty}{\longrightarrow} 0  {\mbox{ locally uniformly in }}x\in\overline{\Omega}, $$
where $u_\infty\in C(\overline{\Omega})$ is a solution to \eqref{Pinf} such that $0<u_\infty<1$ in $\overline{\Omega}$.
\end{cor}
\begin{remark}
The distance $\delta\in\mathcal{Q}(\overline{\Omega})$ in Corollary~\ref{COR:CONTR} may be chosen to be either the Euclidean or the geodesic distance, see \cite{Brasseur2018}. See Figure~\ref{FIG:CONTREEX} for an example illustrating the conclusion of Corollary~\ref{COR:CONTR}. The obstacle that is pictured is the same as the one we constructed in \cite{Brasseur2018}.
\end{remark}
\begin{figure}[!ht]
\centering
\subfloat[$t=0$]{\includegraphics[height=2cm,width=4cm]{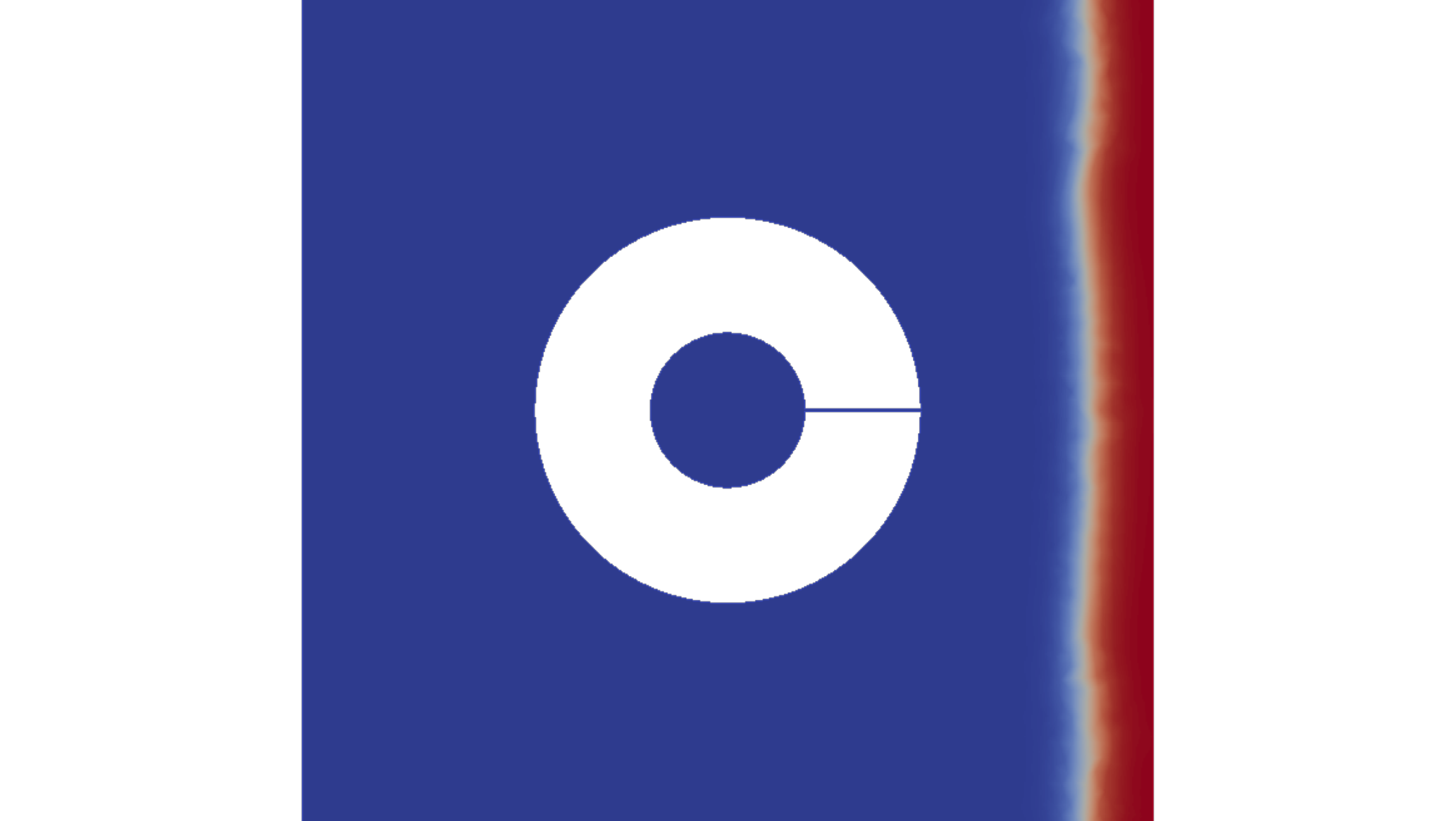}}
\subfloat[$t=40$]{\includegraphics[height=2cm,width=4cm]{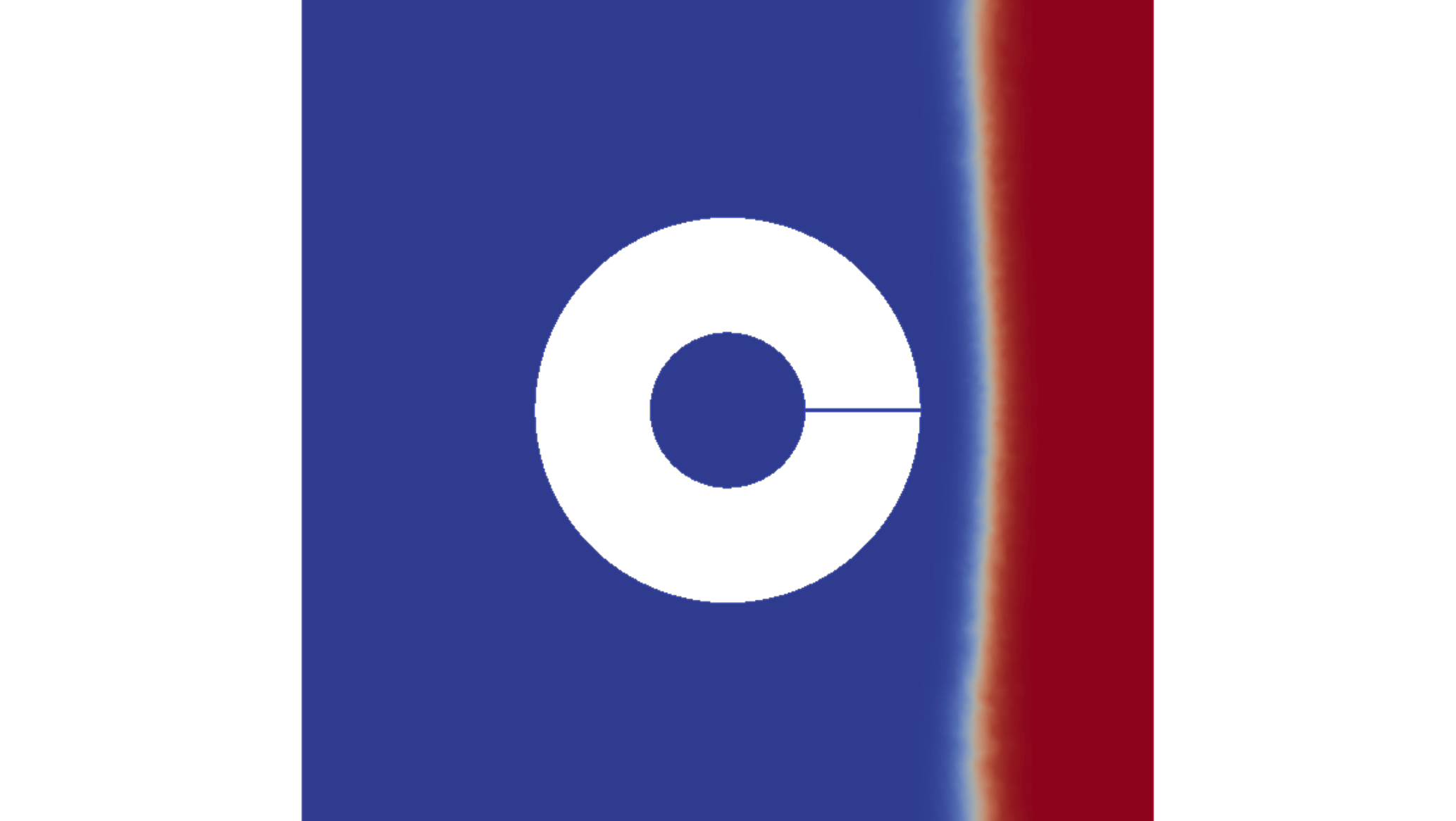}}
\subfloat[$t=80$]{\includegraphics[height=2cm,width=4cm]{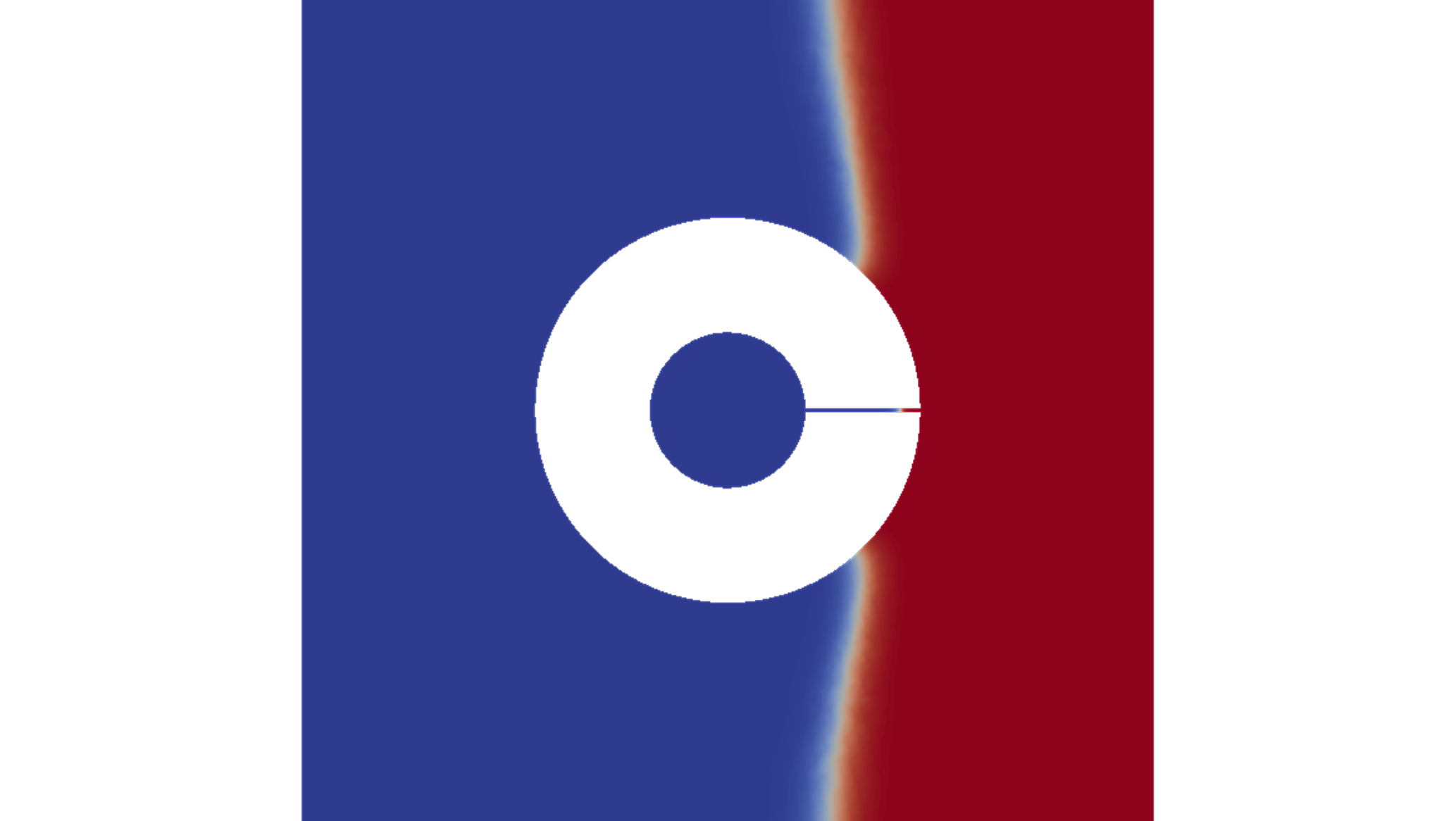}}
\subfloat[$t=120$]{\includegraphics[height=2cm,width=4cm]{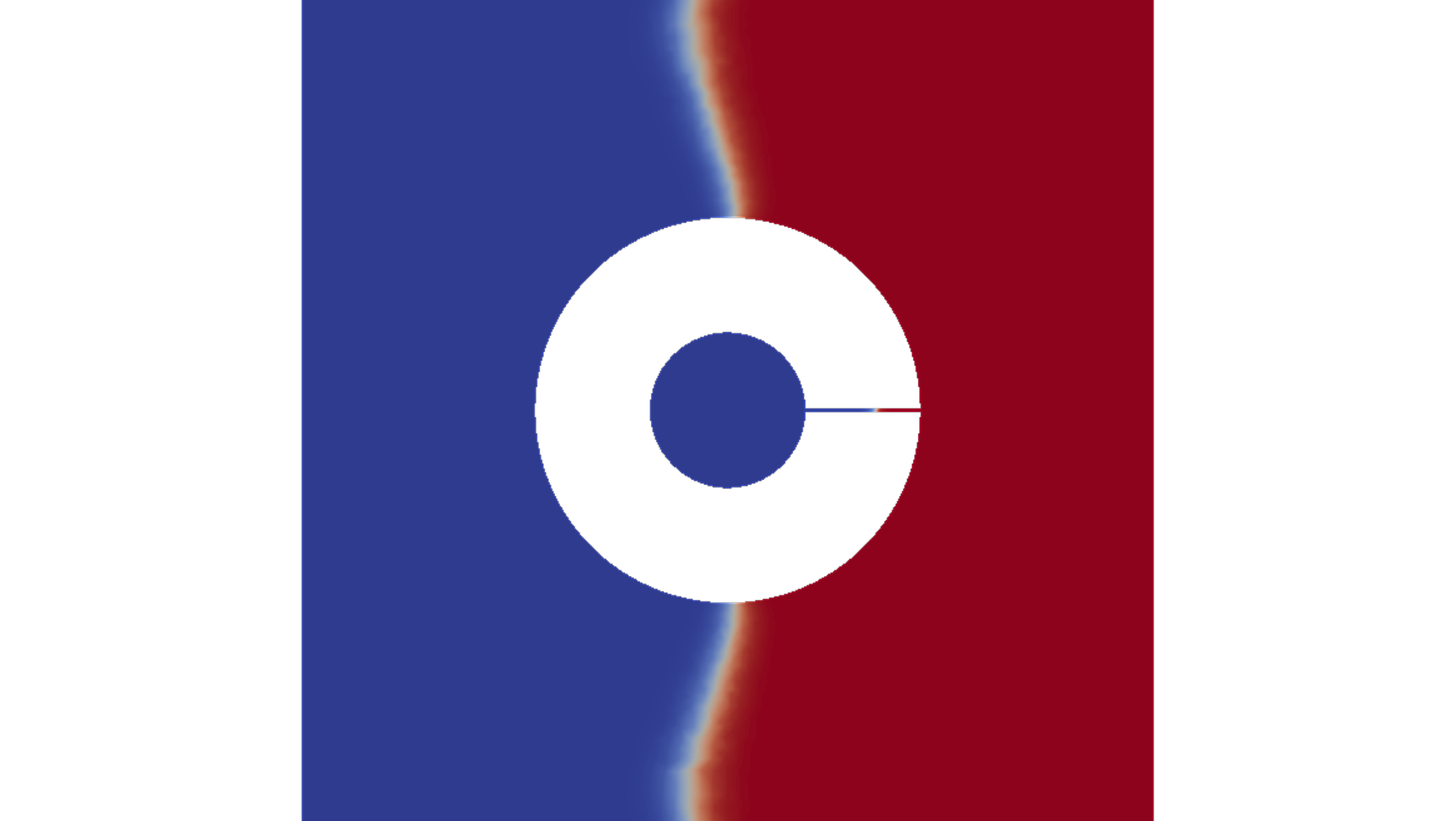}}

\subfloat[$t=160$]{\includegraphics[height=2cm,width=4cm]{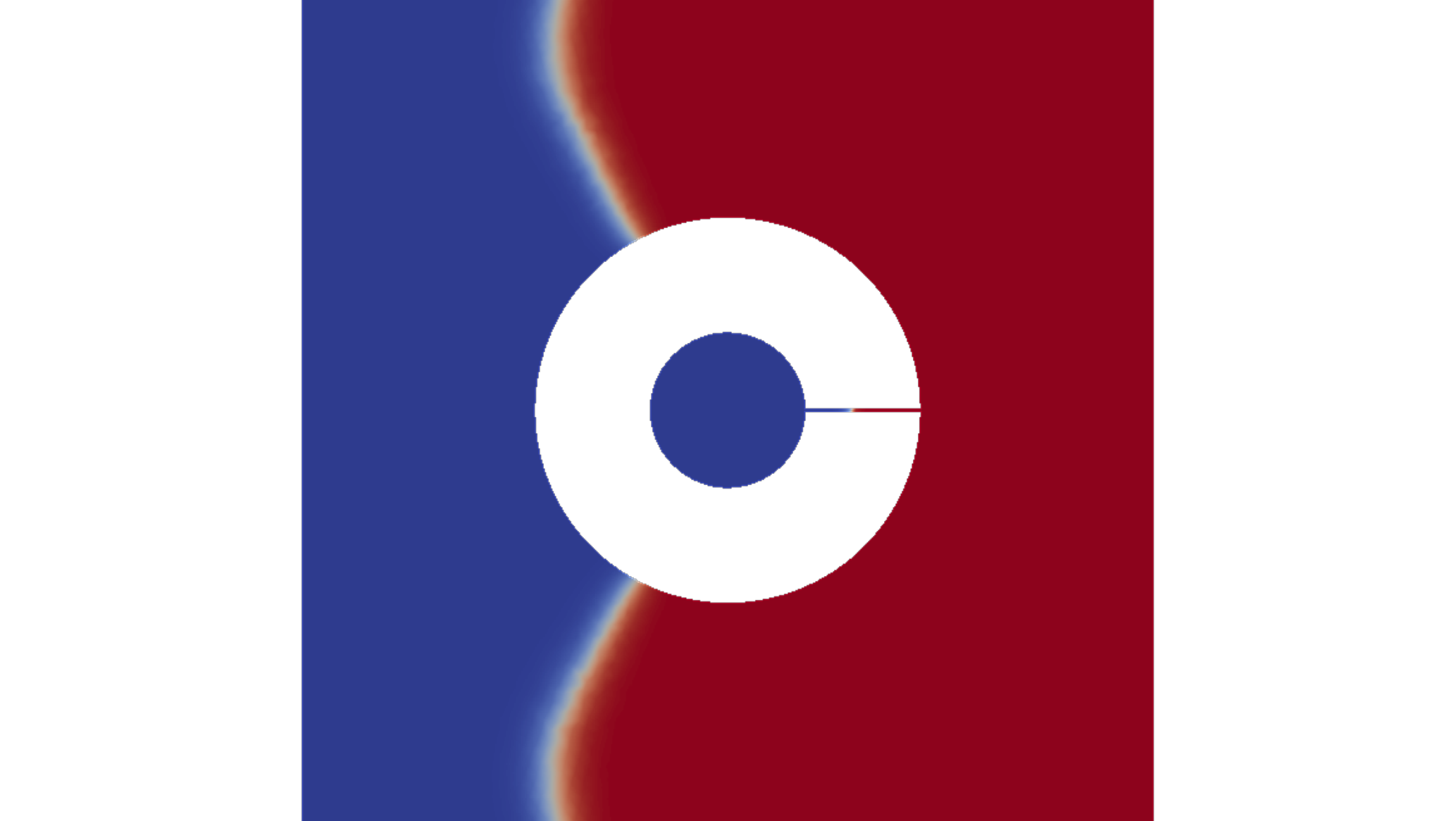}}
\subfloat[$t=200$]{\includegraphics[height=2cm,width=4cm]{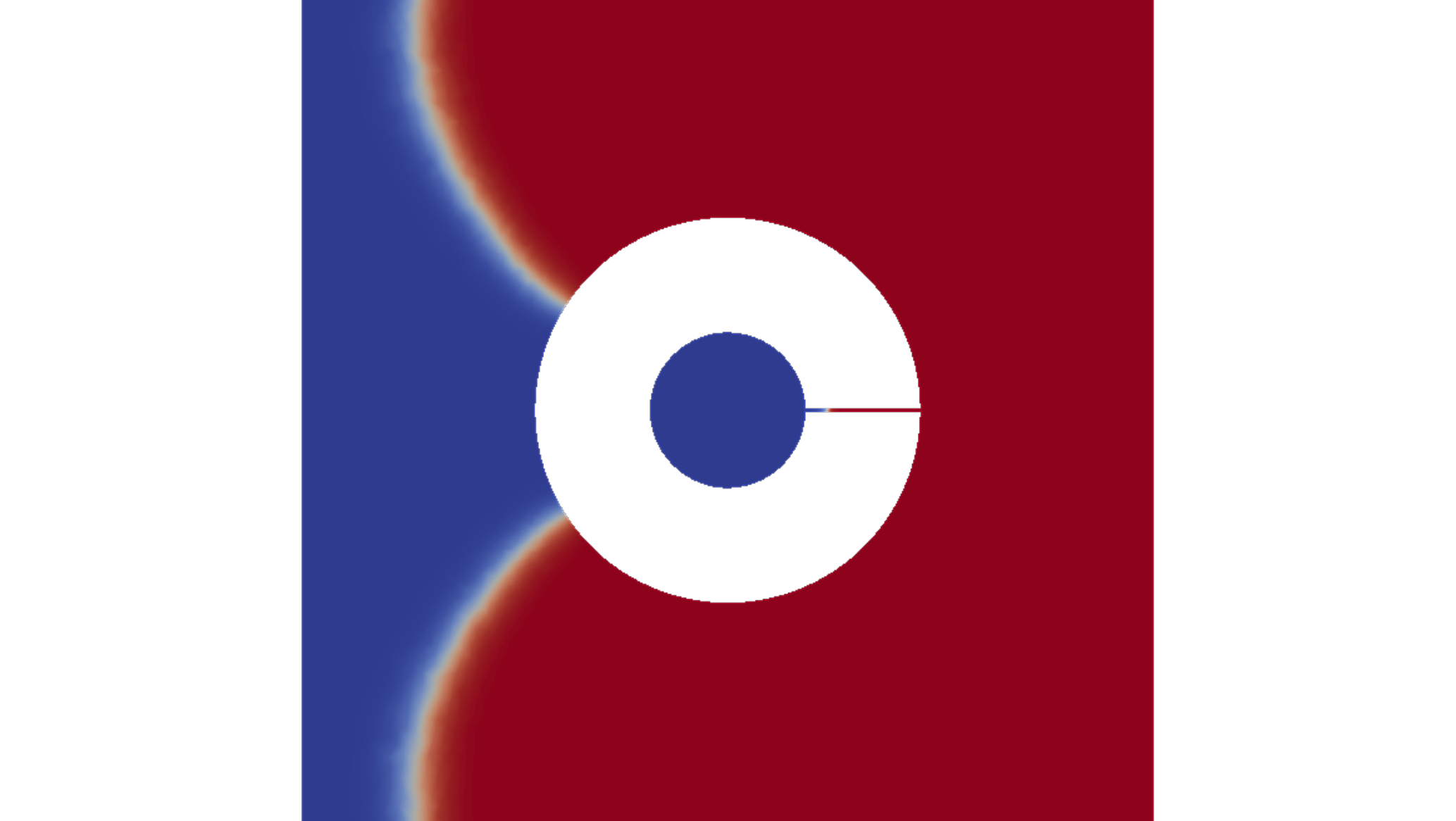}}
\subfloat[$t=240$]{\includegraphics[height=2cm,width=4cm]{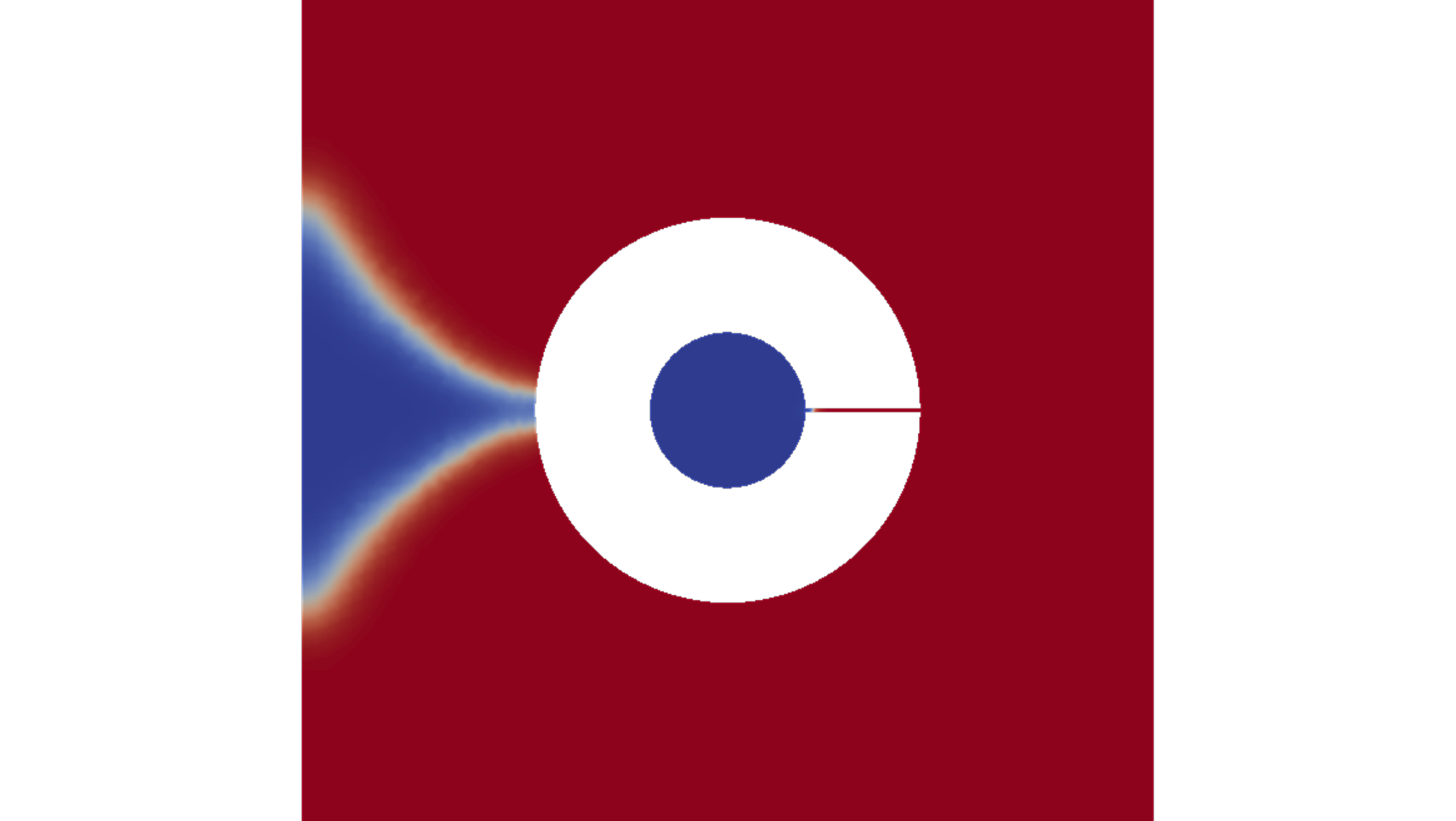}}
\subfloat[$t=280$]{\includegraphics[height=2cm,width=4cm]{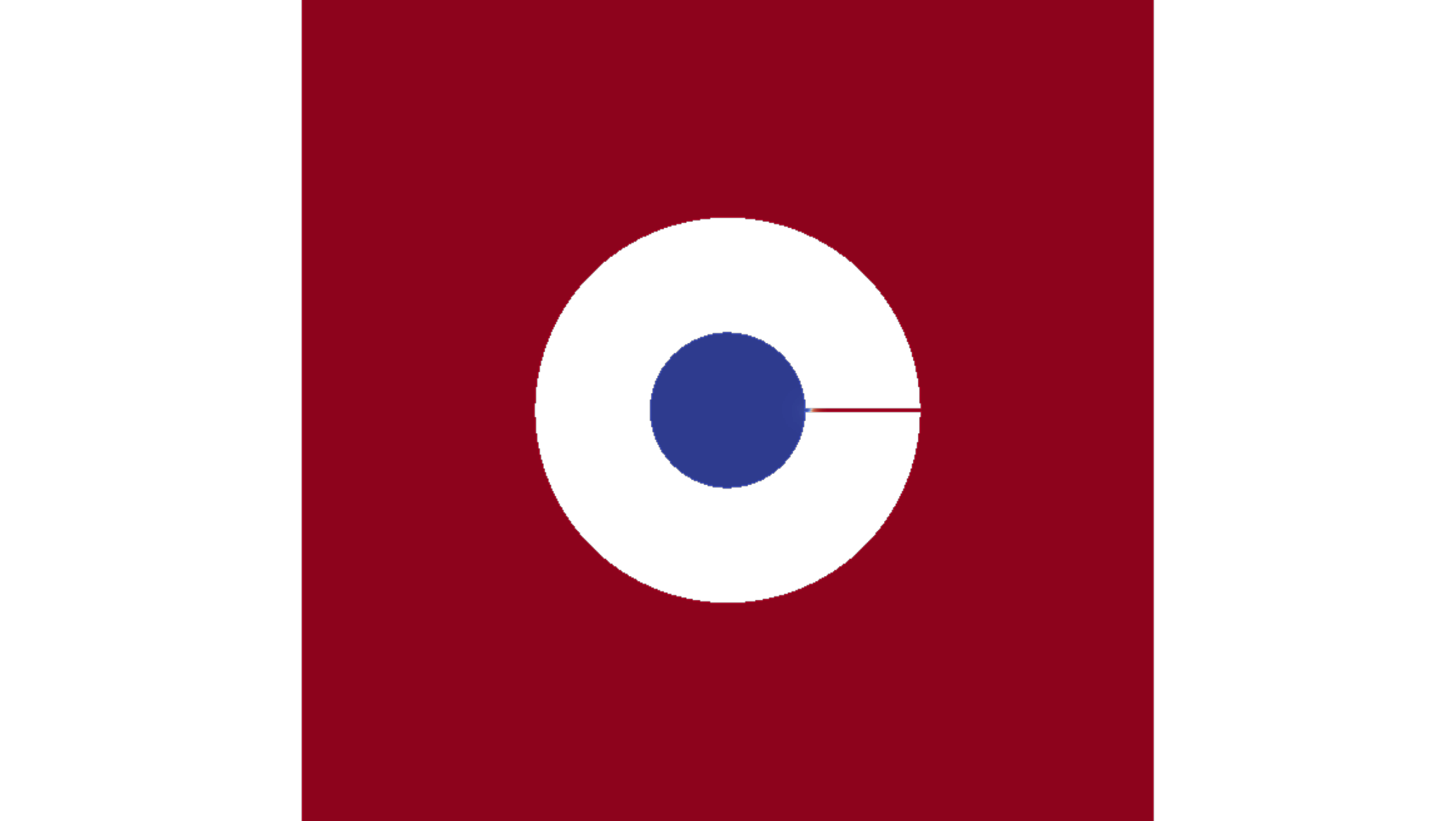}}
\caption{\footnotesize Numerical approximation of the solution of problem \eqref{P} at different times, starting from a Heaviside type initial density. For the simulation, $J(x)\sim e^{-|x|^2}\mathds{1}_{B_1}(x)$, the distance $\delta$ is the Euclidean distance and the obstacle $K$ is the annulus $\mathcal{A}(2,5)$ to which we have removed a small channel to make its complement connected. On the domain $\O:=[-11,11]^2\setminus K$, we perform an IMEX Euler scheme in time combined with a finite element method in space with a time step of $0.1$. We observe that the solution converges to a non-trivial asymptotic profile as $t\to\infty$.}\label{FIG:CONTREEX}
\end{figure}
\begin{remark}
If the convergence in \eqref{asympt:large:time} was known to be uniform in space, then the local uniform convergence in Theorems~\ref{TH:LIOUVILLE}-\ref{NonLiouville2} and in Corollary~\ref{COR:CONTR} could be replaced by a uniform convergence without modification in the proofs.
\end{remark}

\subsection{Global mean speed and transition fronts}

As we have stated in the previous subsection, the propagation may not always be \emph{complete} depending on the shape of the obstacle. That is, it may happen that the solution $u_\infty$ to the stationary problem \eqref{Pinf} arising in \eqref{asympt:large:time} is not identically $1$. An important question that remains to be addressed is the characterisation of the speed at which this solution propagates.
We prove that the entire solution $u(t,x)$ to \eqref{P} is a generalised transition front in the sense of Berestycki-Hamel \cite{Berestycki2007b} and that its global mean speed coincides with that of the planar front given by \eqref{C4}.

\begin{theo}\label{TH:FRONT}
Assume all the assumptions of Theorem~\ref{TH:EXIST:ENTIRE}. Let $u(t,x)$ be the unique entire solution to \eqref{P} satisfying \eqref{unique} and let $u_\infty\in C(\overline{\Omega})$ be the solution to \eqref{Pinf} such that \eqref{asympt:large:time} holds.
Then, $u(t,x)$ is a generalised transition almost-planar invasion front between $0$ and $u_\infty$ with global mean speed $c$, in the sense that
$$ \sup_{(t,x)\in\R\times\overline{\Omega},\,x_1+ct\geq A}|u(t,x)-u_\infty(x)|\underset{A\to\infty}{\longrightarrow} 0\ \text{ and } \sup_{(t,x)\in\R\times\overline{\Omega},\,x_1+ct\leq -A} u(t,x)\underset{A\to\infty}{\longrightarrow} 0. $$
\end{theo}
\begin{remark}
Notice that the presence of an arbitrary quasi-Euclidean distance, which introduces anisotropic features to the diffusion as well as a higher sensitivity to the geometry, does not slow down the propagation.
\end{remark}
%\begin{remark}
%Interestingly, our proof is different than the one obtained by Berestycki \emph{et al.} \cite{Berestycki2009d} for the classical problem \eqref{eqlocale}. In fact, our arguments also work in the local case, providing thus an alternative proof for this case.
%\end{remark}
%\clearpage
\subsection{Some open problems}

Prior to proving our main results, let us mention some open questions which, we believe, are of interest.

We have considered, for simplicity, the case of compactly supported kernels only.
It would be of interest to investigate the validity of our results without this assumption. Although this remains an open question, we believe that our results remain true at least if $J$ decays faster than any exponential. However, if $J$ is too heavily tailed, then the asymptotic growth or decay of $\phi(x_1+ct)$ may no longer be of exponential type and new tools would be needed to conclude. Nevertheless, let us mention that Hamel, Valdinoci and the authors proved in \cite{Brasseur2019} that the Liouville type property holds for convex obstacles and the Euclidean distance as long as there exists an increasing subsolution to $J_1\ast\phi-\phi+f(\phi)=0$ in $\R$ connecting the two stable states $0$ and $1$, which is known to be the case even if $J$ decays slower than any exponential (e.g. if \eqref{C5} holds and $J_{1}$ admits a first order moment).

Another interesting problem, which would be of both mathematical and biological interest, is the study of the influence of the quasi-Euclidean distance on the large time dynamics. For example: given an obstacle $K$ and a datum $(J,f)$, may it happen that the solution successfully invades the whole of $\overline{\Omega}$ when $\delta$ is the Euclidean distance but fails to do so when $\delta$ is the geodesic distance?
Numerical simulations suggest that this phenomenon could indeed occur (see Figure \ref{FIG:CONTREEXG}), which leads us to formulate the following conjecture:
\begin{conj}\label{CONJ:delta}
There exists a compact obstacle $K\subset\R^N$ such that $\R^N\setminus K$ is connected, as well as a datum $(J,f)$ satisfying all the assumptions of Theorem~\ref{TH:EXIST:ENTIRE} such that $u_\infty\equiv 1$ in $\overline{\Omega}$ if $\delta$ is the Euclidean distance and such that $0<u_\infty<1$ in $\overline{\Omega}$ if $\delta$ is the geodesic distance (or any other quasi-Euclidean distance distinct from the Euclidean distance).
\end{conj}

\begin{figure}[!ht]
\centering
\begin{tabular}{c|c}
\subfloat[$t=300$]{\includegraphics[height=2.45cm]{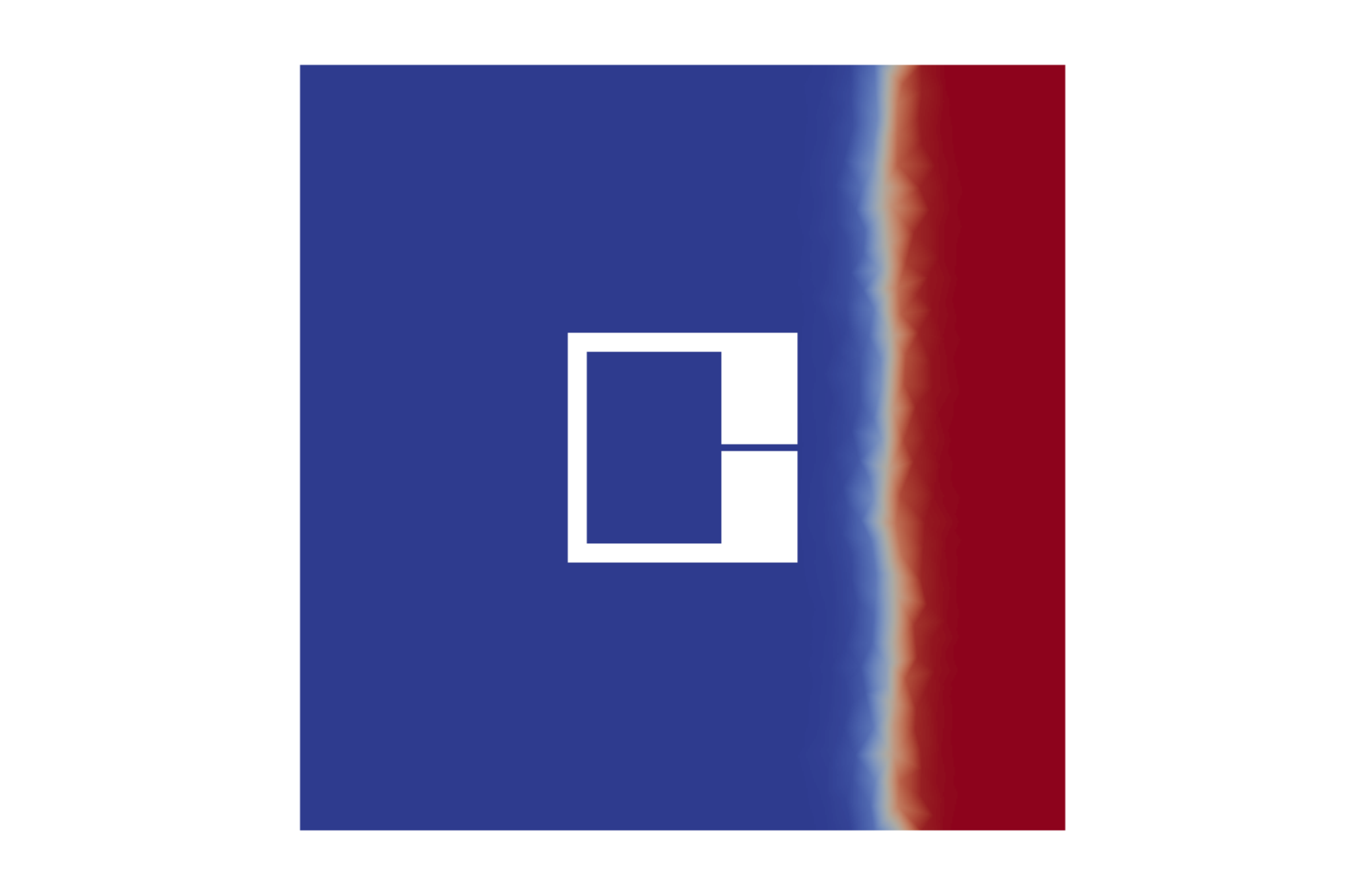}} \subfloat[$t=700$]{\includegraphics[height=2.45cm]{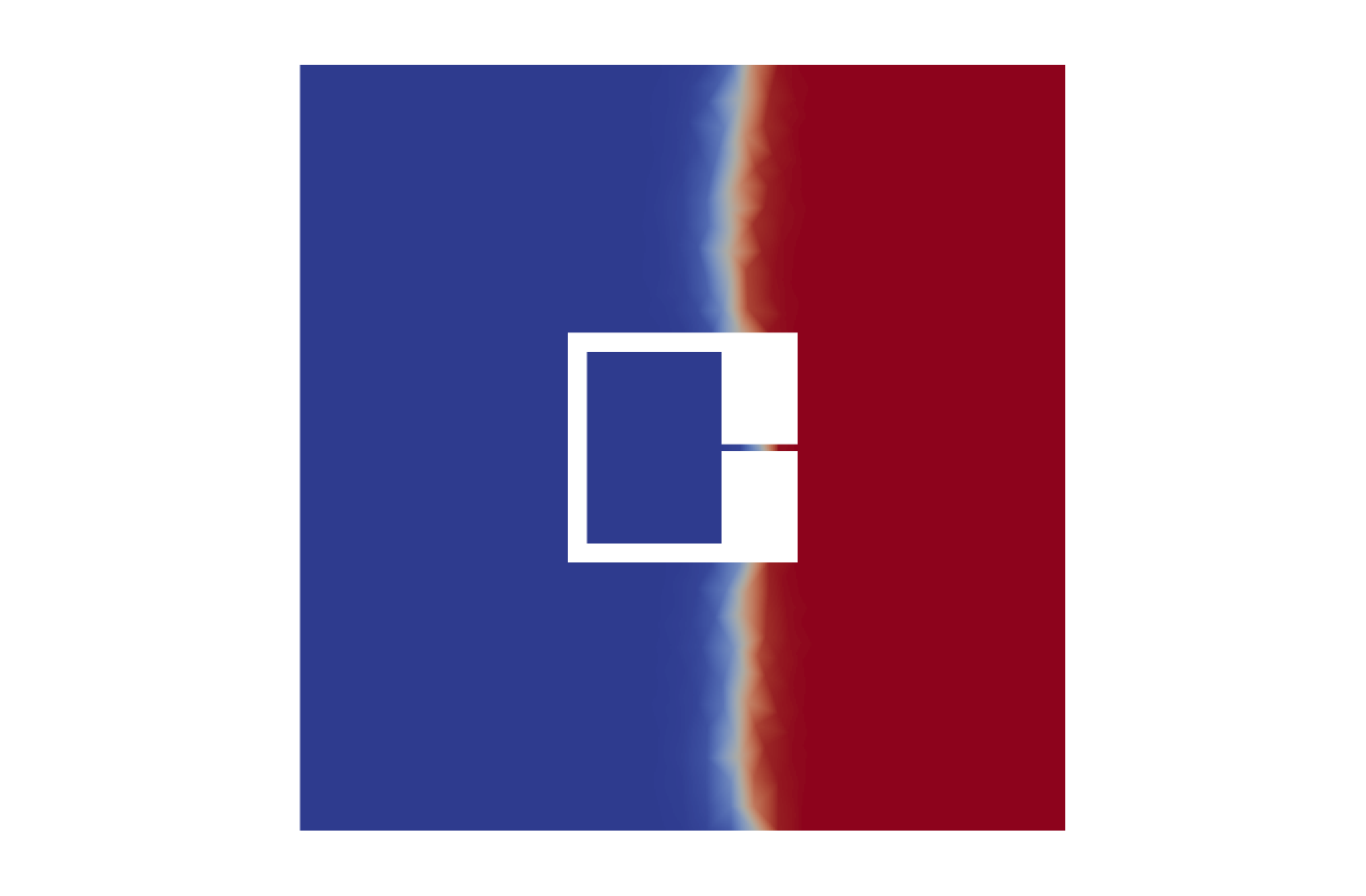}} & \subfloat[$t=300$]{\includegraphics[height=2.45cm]{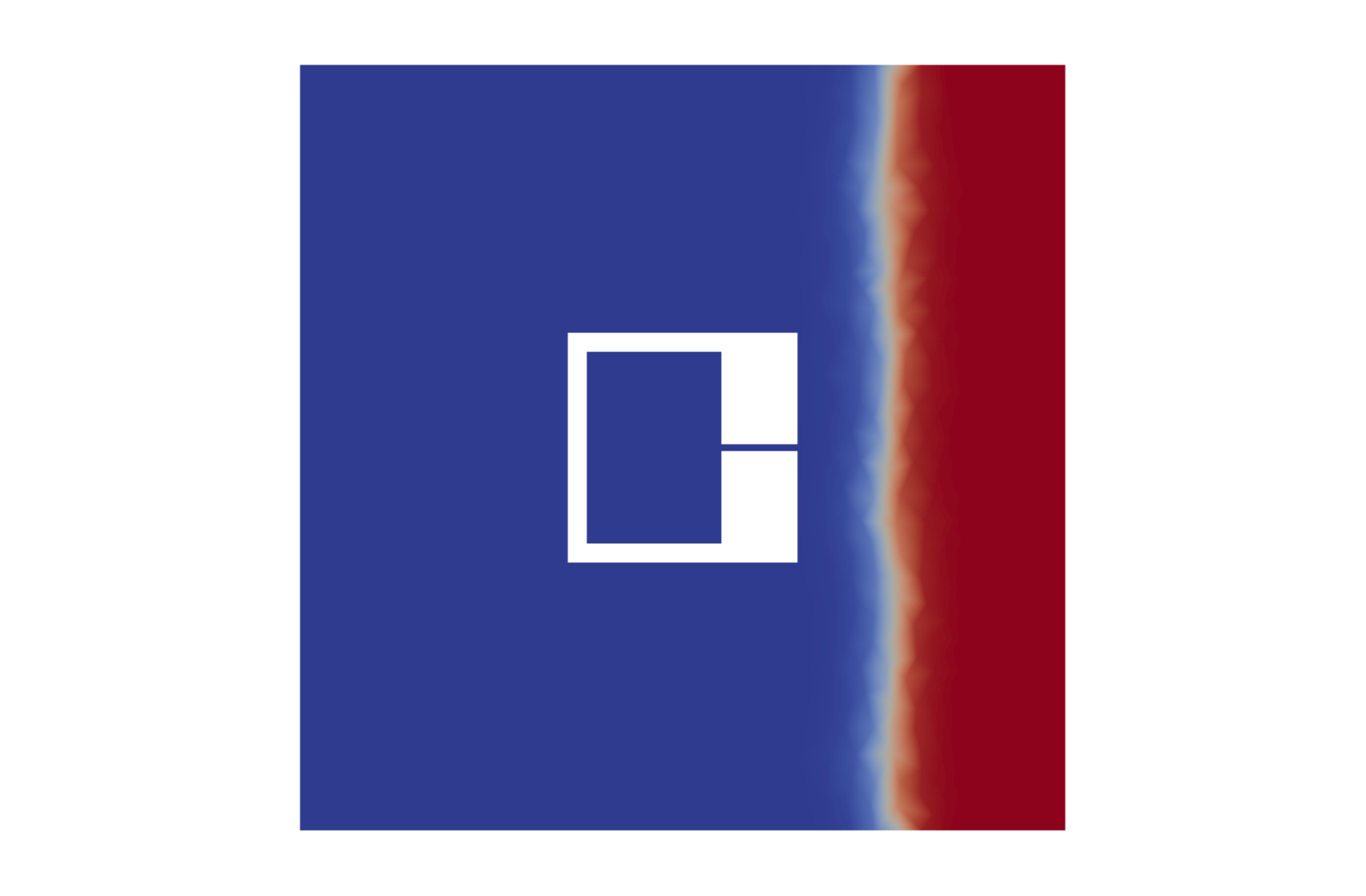}} \subfloat[$t=700$]{\includegraphics[height=2.45cm]{solg-t-700.pdf}} \\
\subfloat[$t=1100$]{\includegraphics[height=2.45cm]{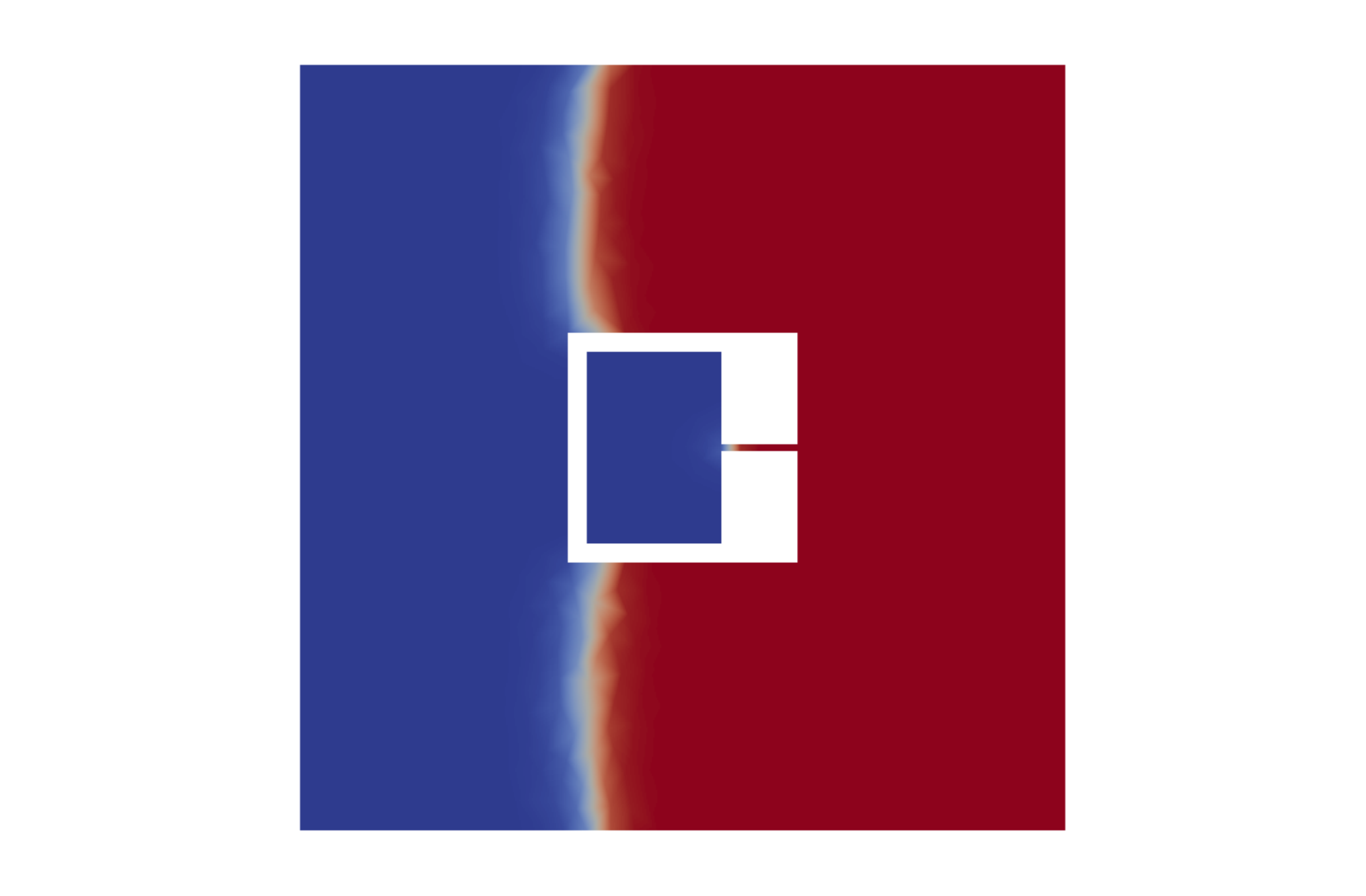}} \subfloat[$t=1200$]{\includegraphics[height=2.45cm]{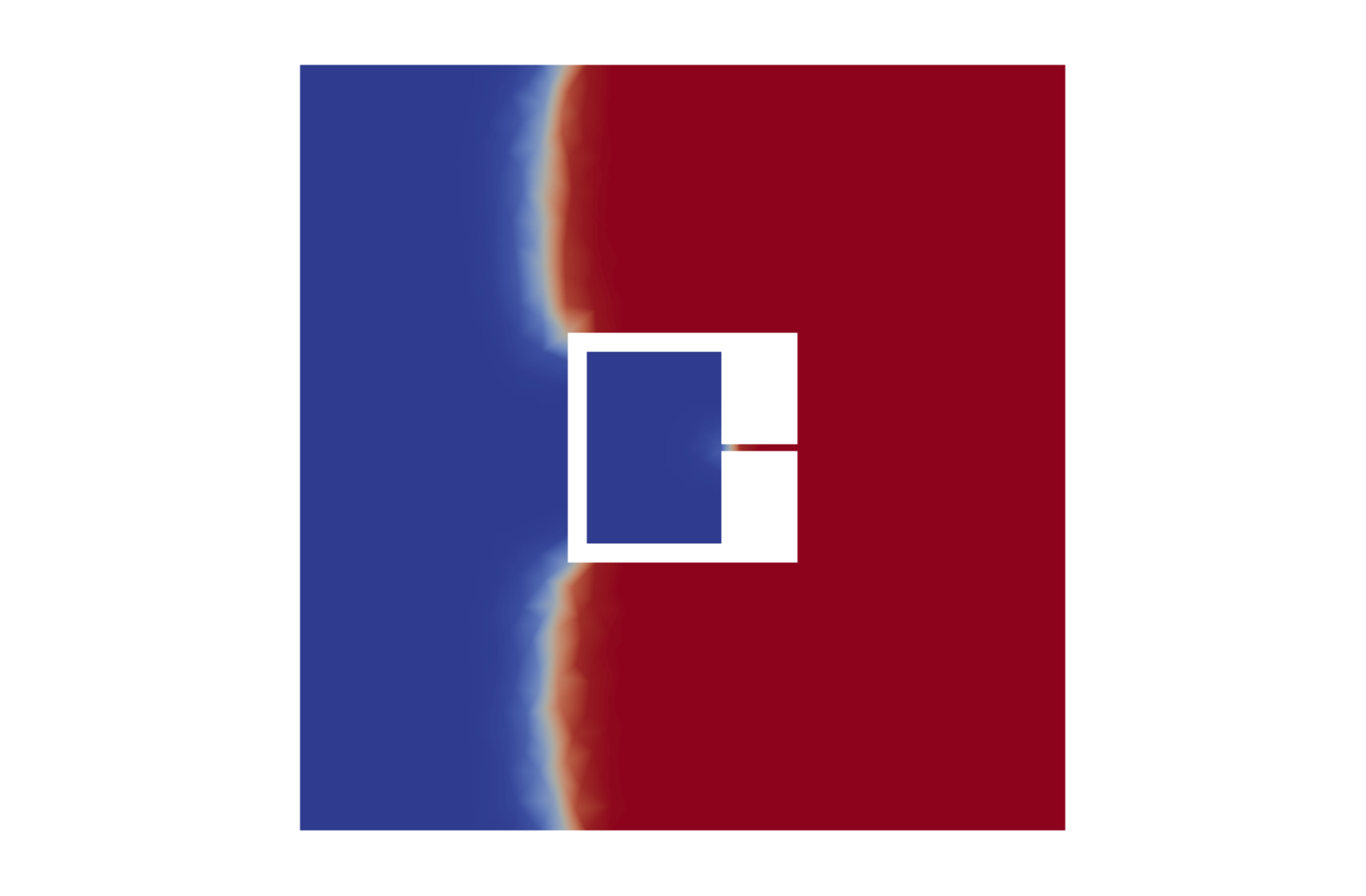}} & \subfloat[$t=1100$]{\includegraphics[height=2.45cm]{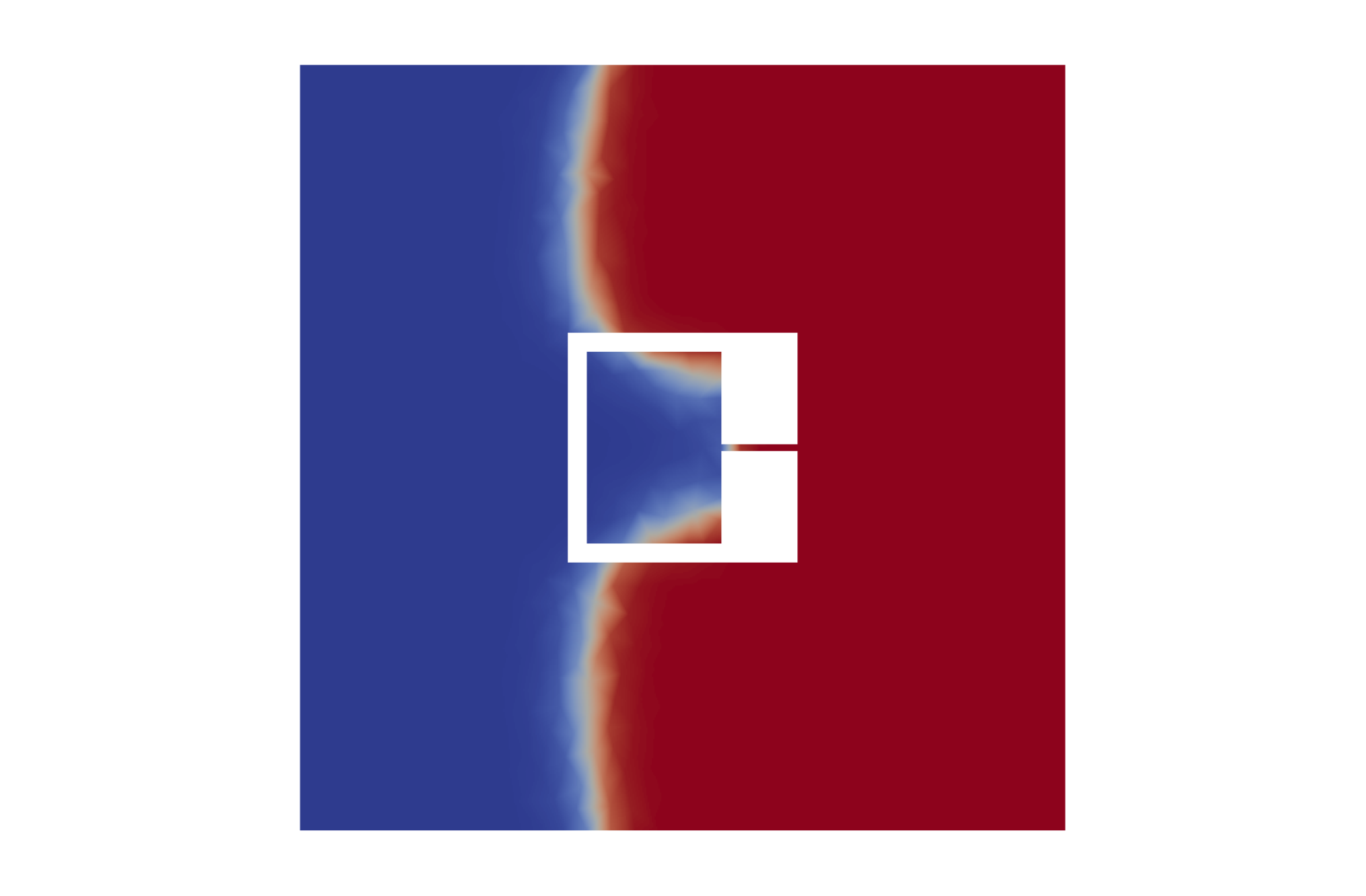}}
\subfloat[$t=1200$]{\includegraphics[height=2.45cm]{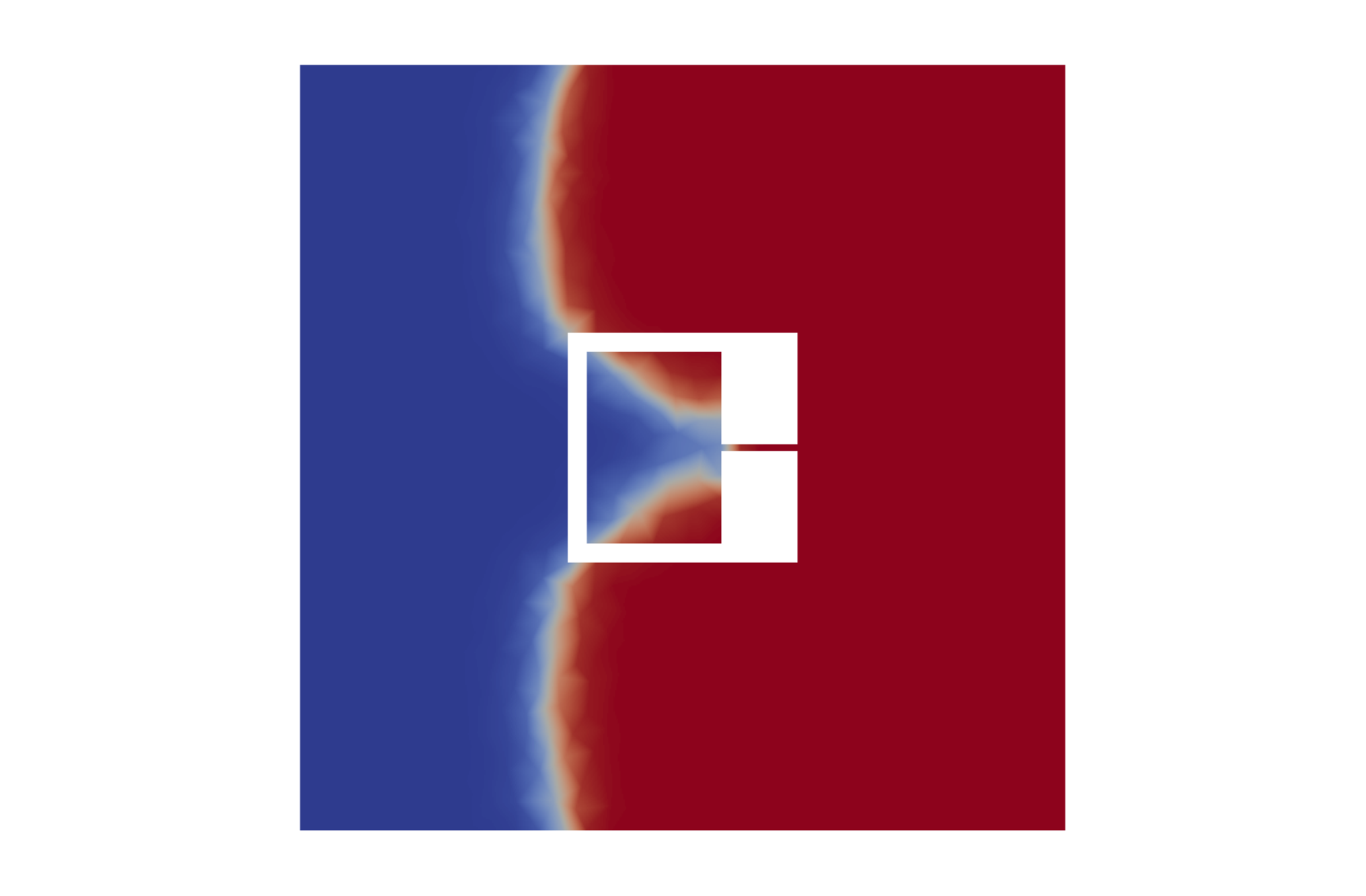}} \\
\subfloat[$t=1300$]{\includegraphics[height=2.45cm]{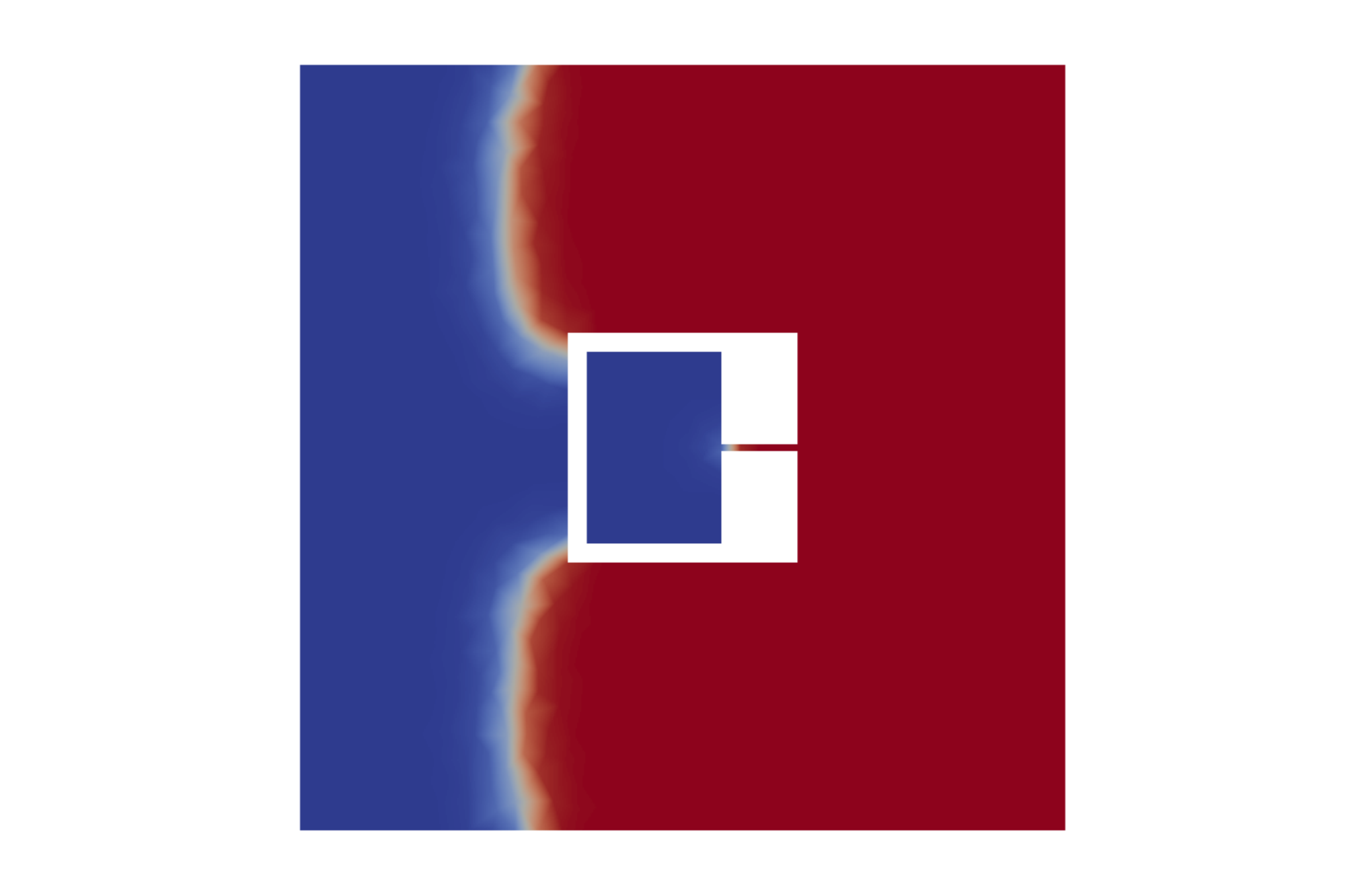}}
\subfloat[$t=1500$]{\includegraphics[height=2.45cm]{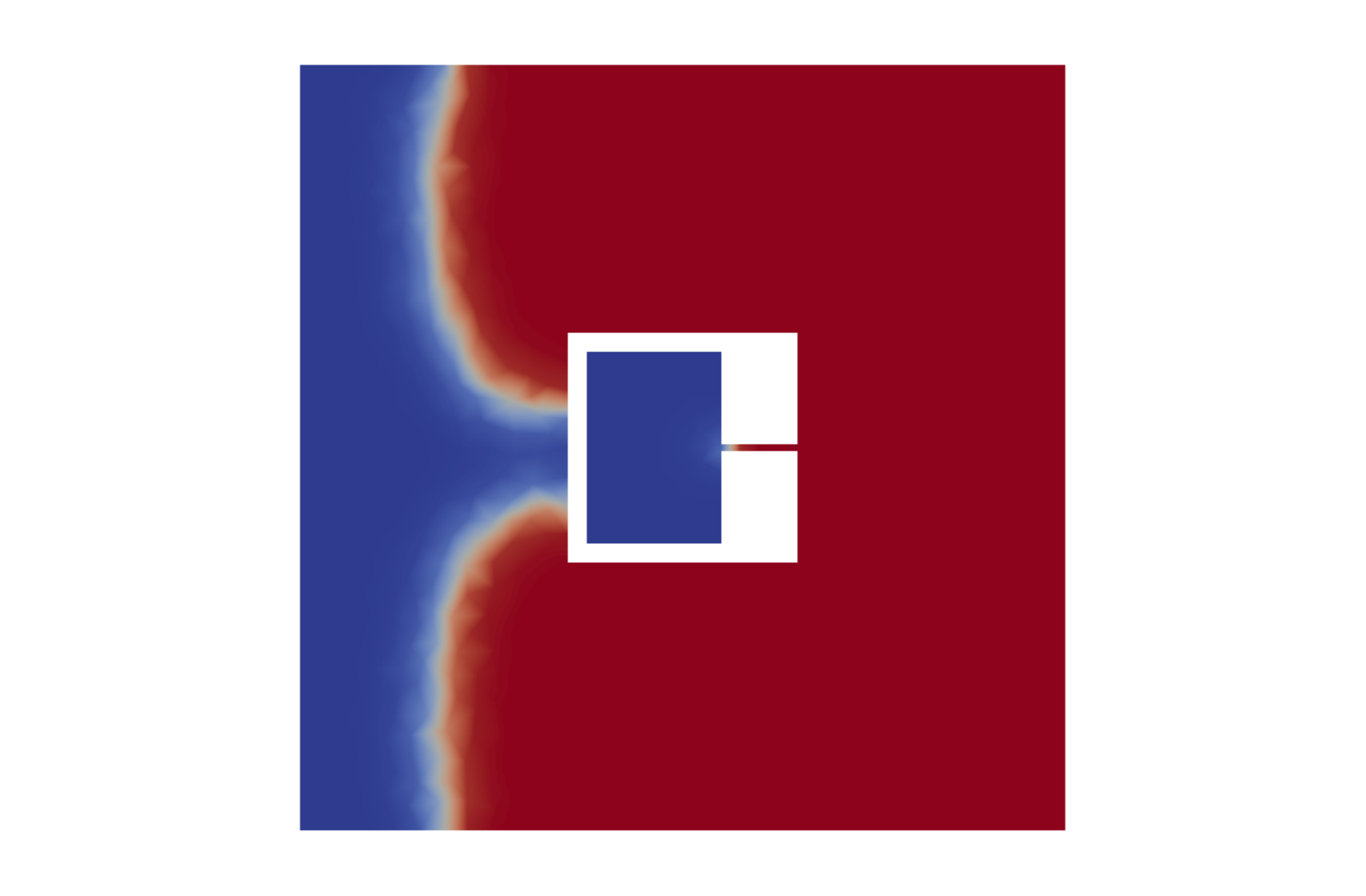}}
&
\subfloat[$t=1300$]{\includegraphics[height=2.45cm]{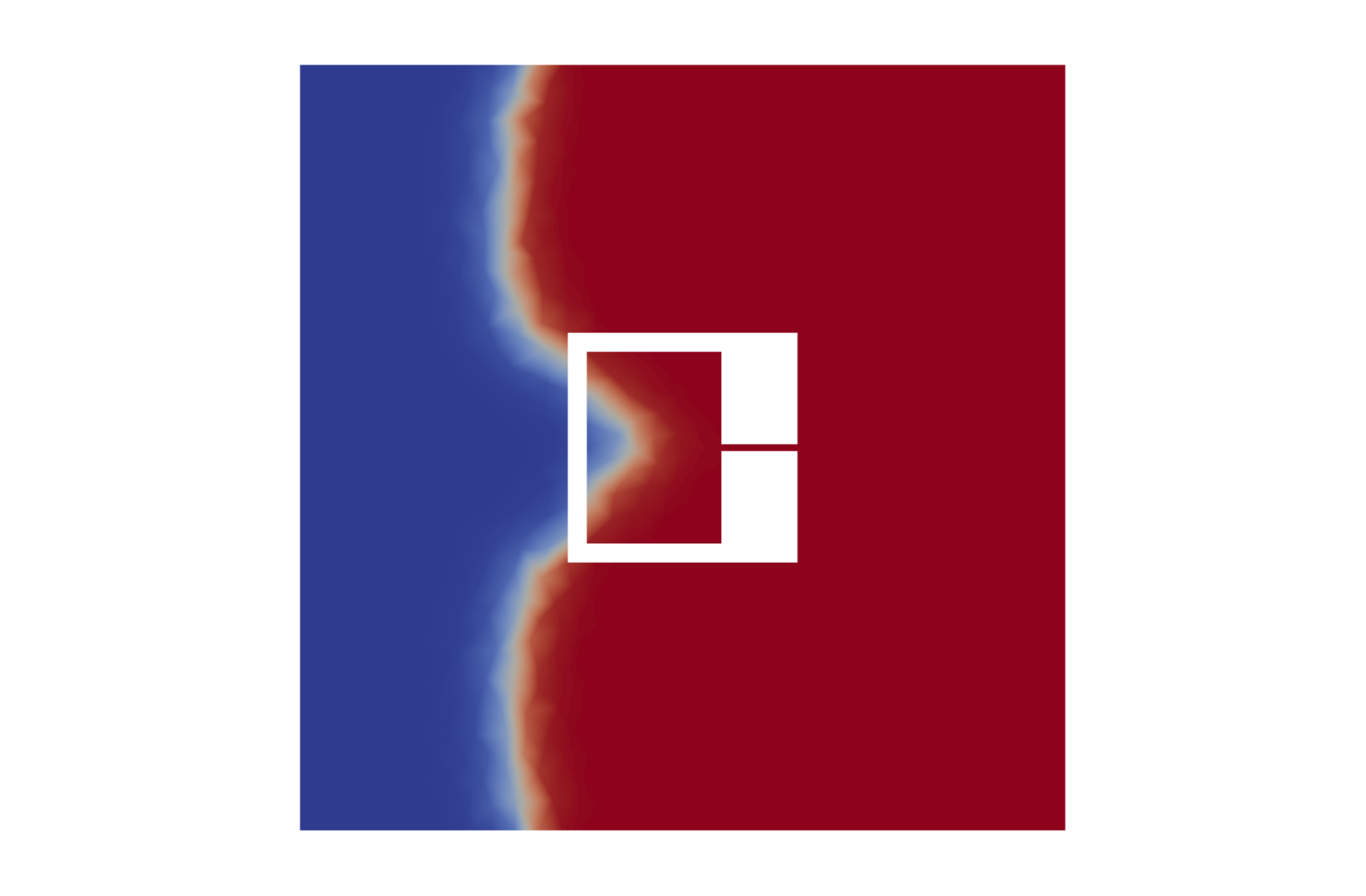}}
\subfloat[$t=1500$]{\includegraphics[height=2.45cm]{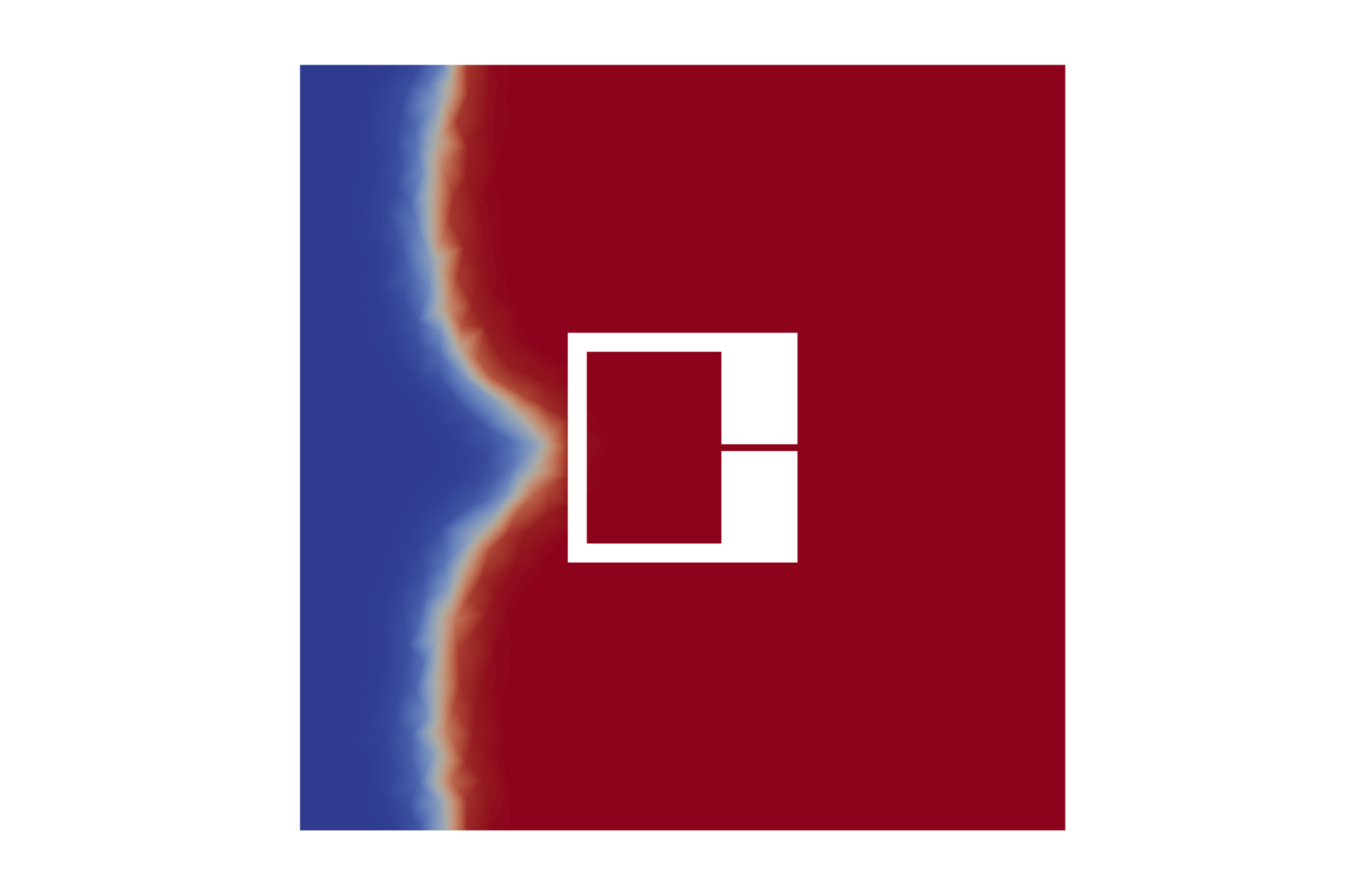}} \\
\subfloat[$t=1900$]{\includegraphics[height=2.45cm]{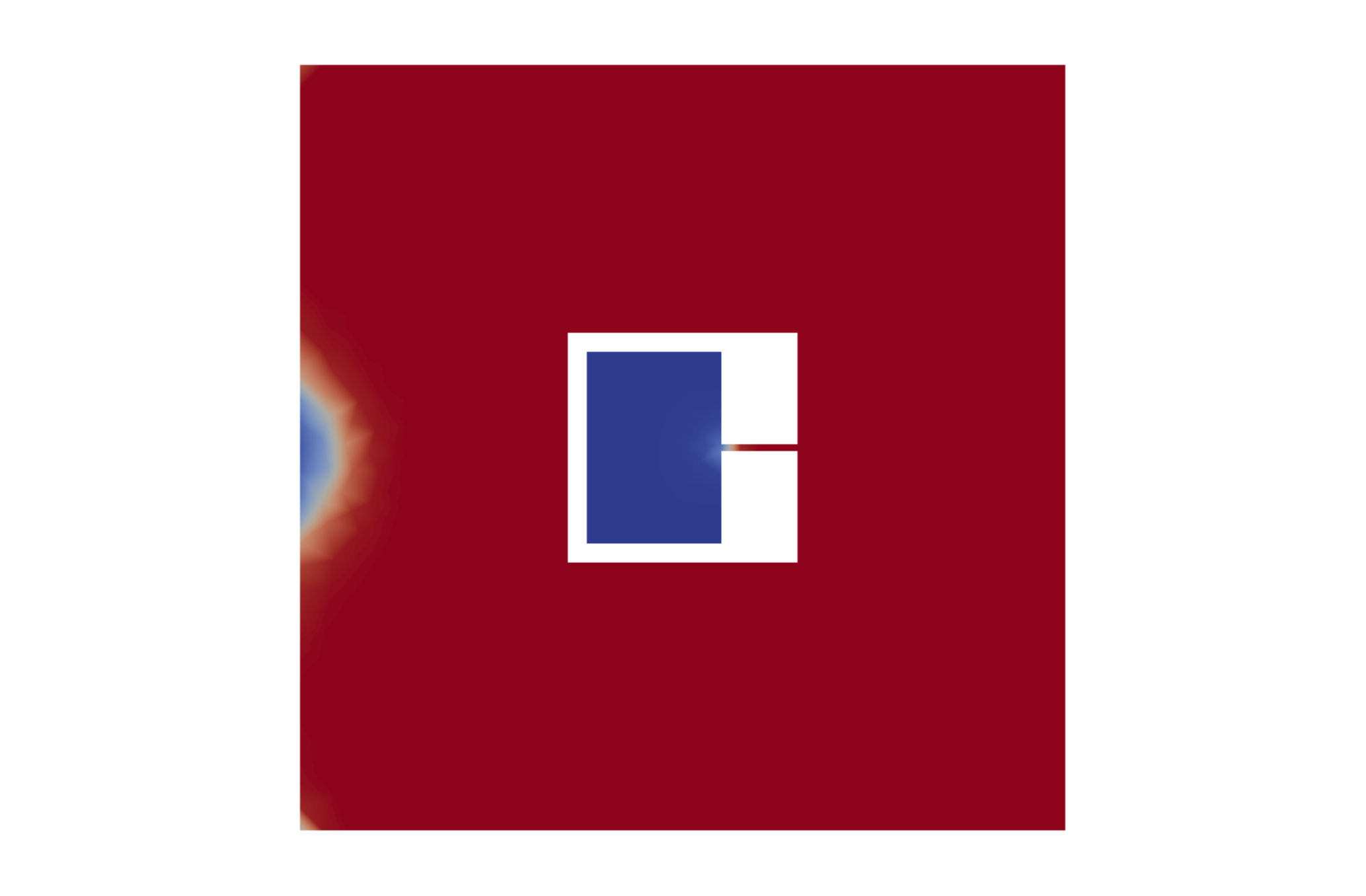}}
\subfloat[$t=2000$]{\includegraphics[height=2.45cm]{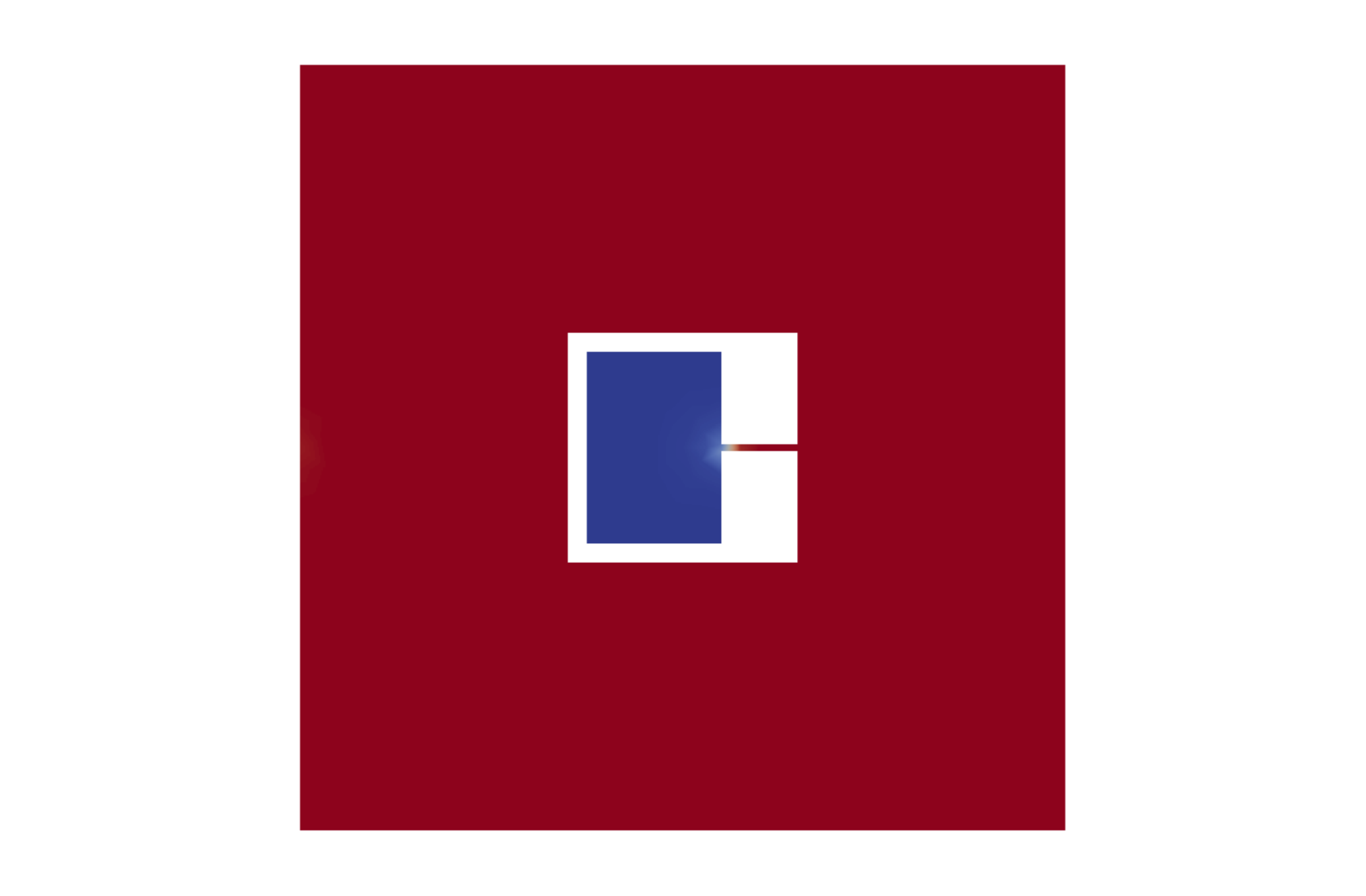}}
&
\subfloat[$t=1900$]{\includegraphics[height=2.45cm]{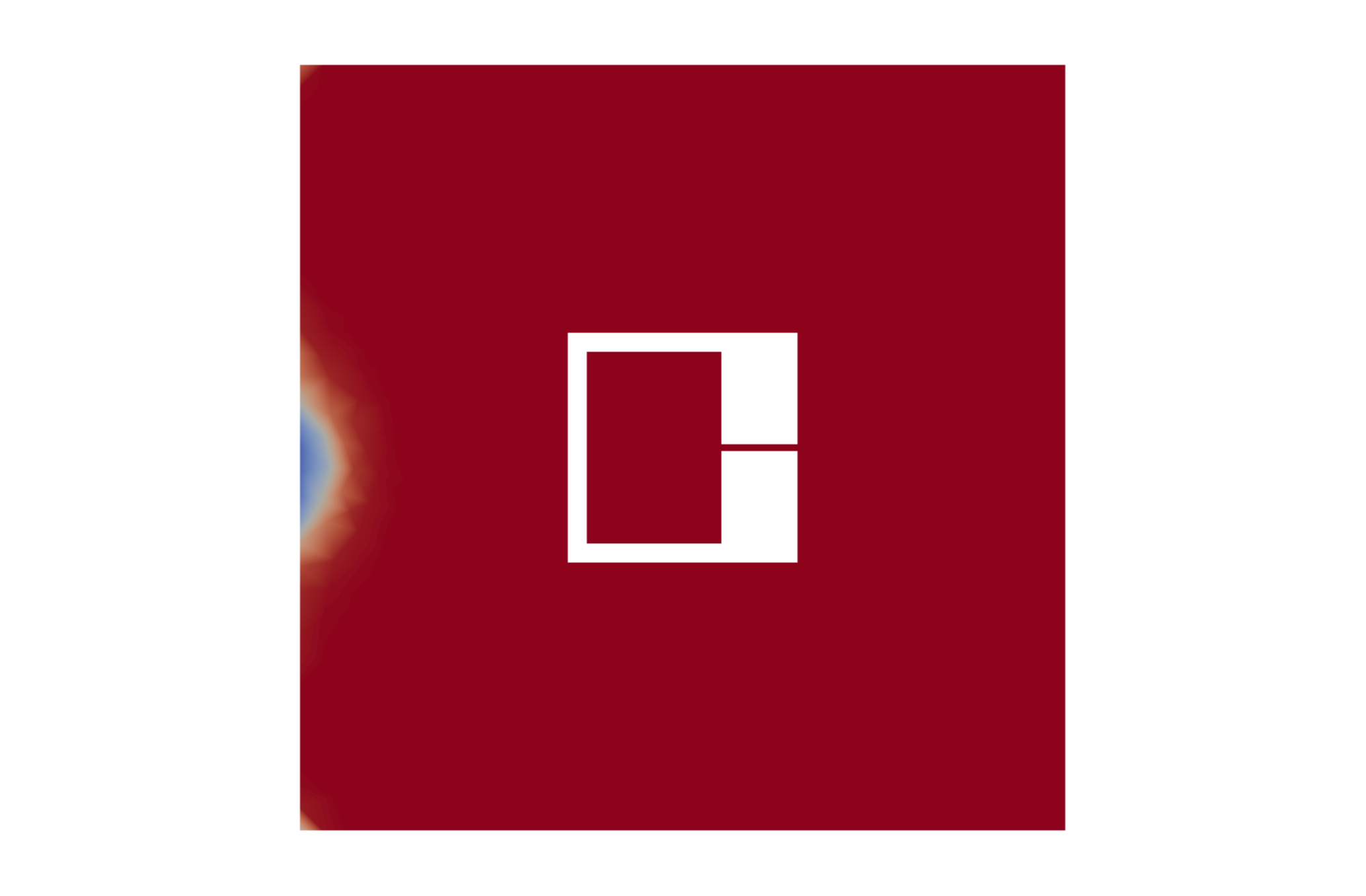}}
\subfloat[$t=2000$]{\includegraphics[height=2.45cm]{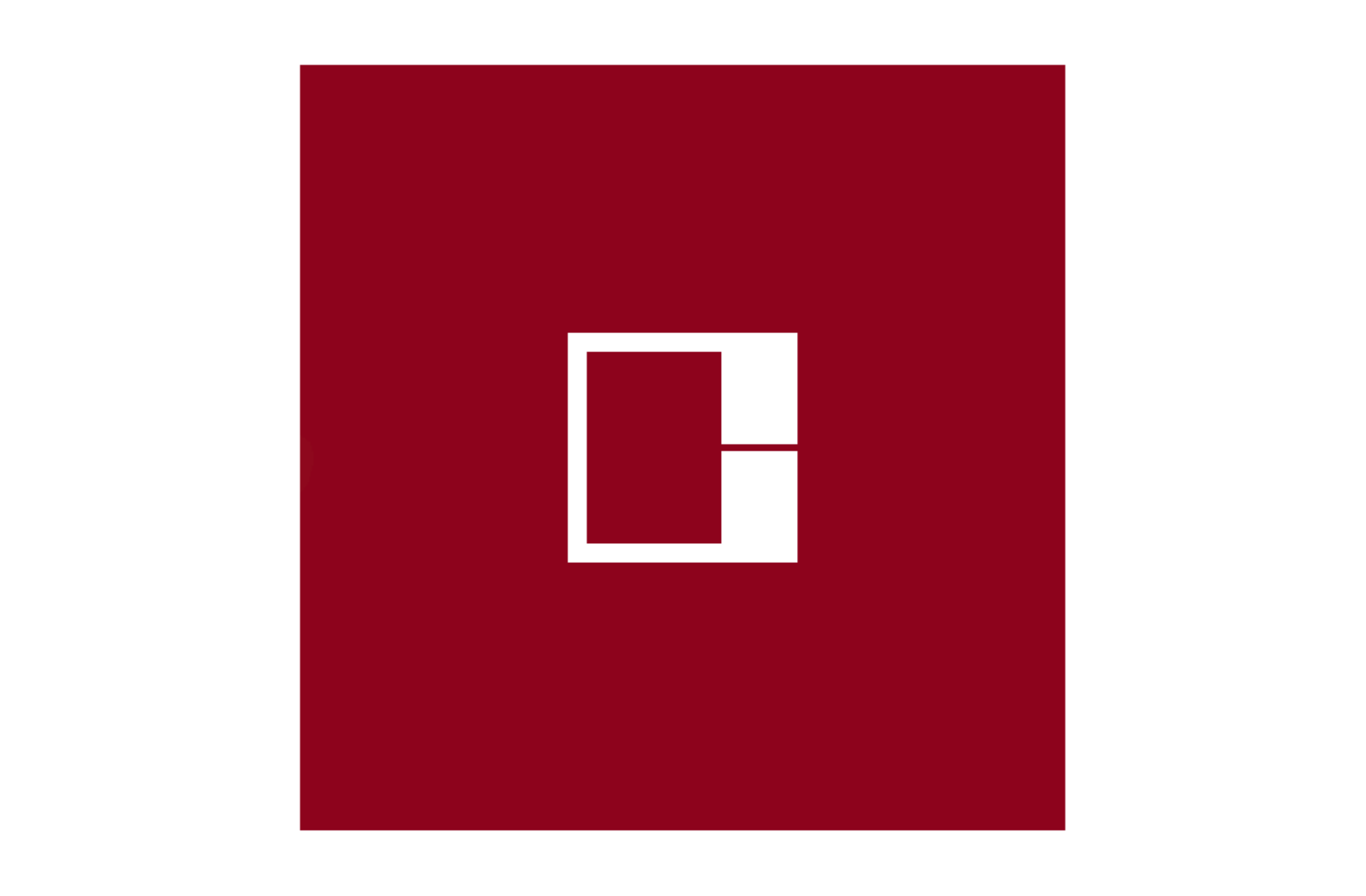}}
\end{tabular}
\caption{\footnotesize Numerical approximation of the solution of problem \eqref{P} at different times, starting from a Heaviside type initial density. For the simulation, $J(x)\sim e^{-|x|^2}\mathds{1}_{B_1}(x)$, the distance $\delta$ is either the \textbf{geodesic distance} (left rows (A, B, E, F, I, J, M, N)) or the \textbf{Euclidean distance} (right rows (C, D, G, H, K, L, O, P)) and the obstacle $K$ is a ``square annulus" (the difference between two axis-parallel rectangles) to which we have removed a small channel to make its complement connected. On the domain $\O:=[-5,5]^2\setminus K$, we perform an IMEX Euler scheme in time combined with a finite element method in space with a time step of $0.1$. The scheme is implemented in Python relying on the FEM library DOLFIN \cite{Logg2012} and the VisiLibity library \cite{Obermeyer2008} in order to evaluate the geodesic distance. We observe that the respective solutions converge to different asymptotic profiles as $t\to\infty$ with a significative difference in the qualitative behaviour of their dynamics. These simulations clearly highlight the importance of the distance in the final outcome of the propagation and on the transition behaviour.}\label{FIG:CONTREEXG}
\end{figure}

Note that, besides being supported by our numerical simulations, this conjecture is fairly reasonable. Indeed, it is quite natural to expect the solution to \eqref{P} to be more sensitive to the geometry of $K$ when $\delta$ is the geodesic distance.
In the case of the obstacle that is pictured in Figure~ \ref{FIG:CONTREEXG}, the individuals do not reach the rectangle inscribed by the obstacle in the same fashion: when $\delta$ is the Euclidean distance, they can simply ``jump" through $K$, but, when $\delta$ is the geodesic distance, they have no choice but to go through a small channel which considerably penalises the propagation.
So the success of the invasion may not depend on the interplay between $K$, $J$ and $f$ \emph{only} but on the interplay between $K$, $J$, $f$ \emph{and} $\delta$.

A rigorous proof of Conjecture~\ref{CONJ:delta} would provide us with a very interesting result from the point of view of ecology. Indeed, it would suggest that, for a given ecological niche, the success of an invasion may crucially depend on the characteristic trait of the species determining its perception of the environment, namely the ``effective distance" between locations.

Lastly, as in \cite{Berestycki2009d}, it would be interesting to determine whether the convergence in \eqref{asympt:large:time} is uniform, namely whether it holds that $\lim_{t\to+\infty}\sup_{x\in\overline{\Omega}}|u(t,x)-u_\infty(x)\hspace{0.1em}\phi(x_1+ct)|=0$.
Based on our numerical simulations and on the analogy with the local case, we conjecture that this holds true in general.
%From our numerical simulations, we conjecture that this result should actually hold true  in general.
We have been recently aware of some recent works \cite{Qiao2020} in this direction dealing with convex obstacles in the particular case where $\delta$ is the Euclidean distance.

\subsection{Organization of the paper}

In the following Section~\ref{SE:CAUCHY}, we focus on the properties of the Cauchy problem associated to \eqref{P}. This will pave the way towards the construction of an entire solution to \eqref{P}. There, we will establish various comparison principles, existence and uniqueness results as well as some parabolic-type estimates. Section~\ref{SE:APRIORI} deals with the a priori regularity of entire solutions. Indeed, it is not clear whether parabolic-type estimates hold for entire solutions, but we prove that, in some circumstances, such estimates can be shown to hold. In Section~\ref{SE:BEFORE}, relying on the results collected in the previous sections and on a sub- and super-solution technique, we prove the existence and uniqueness of an entire solution converging uniformly to $\phi(x_1+ct)$ as $t\to-\infty$. Next, in Section~\ref{SE:LARGETIME}, we study the local behaviour of the entire solution in the large time limit. In Section~\ref{SE:GEOMETRY}, we study the influence of the geometry of $K$ on the large time behaviour of the entire solution. Finally, in Sections~\ref{SE:LargeTimeSLS} and \ref{SE:EntireSol} we respectively study the long time behaviour of the entire solution and prove that this entire solution is actually a generalised transition wave.

%%%%%%%%%%%%%%%%%%%%%
%%%%%%%%%%%%%%%%%%%%%

\section{The Cauchy problem}\label{SE:CAUCHY}

%Throughout this section we assume \eqref{C1}, \eqref{C3} and that $f\in C_{\mathrm{loc}}^{0,1}(\R)$ function.
This section is devoted to the study of the Cauchy problem
\begin{align}
\left\{
\begin{array}{rll}
\partial_tu\!\!\!&=Lu+f(u) & \text{a.e. in }(t_0,\infty)\times\Omega, \vspace{3pt}\\
u(t_0,\cdot)\!\!\!&=u_0(\cdot) & \text{a.e. in }\Omega,
\end{array}
\right. \label{CauchyPB}
\end{align}
where $t_0\in\R$ and $u_0$ is a given data. The study of \eqref{CauchyPB} is essential to our purposes in that it shall pave the way towards the construction of an entire solution to \eqref{P}.

We will establish various comparison principles, existence and uniqueness results for \eqref{CauchyPB} as well as some a priori estimates under appropriate assumptions on the datum $(J,f)$ and the initial datum $u_0$.

\subsection{Some comparison principles}

In this section, we prove several comparison principle that fit for our purposes.

\begin{lemma}[Comparison principle]\label{LE:COMPARISON}
Assume \eqref{C1}, \eqref{C3} and suppose that $f\in C_{\mathrm{loc}}^{0,1}(\R)$. Let $t_0,t_1\in\R$ with $t_0<t_1$ and let $u_1$ and $u_2$ be two bounded measurable functions defined in $[t_0,t_1]\times\Omega$ and such that, for all $i\in\{1,2\}$,
\begin{align*}
u_i(t,\cdot),\,\partial_tu_i(t,\cdot)\in C(\Omega) \text{ for all }t\in(t_0,t_1] \text{ and } u_i(t_0,\cdot)\in C(\Omega),  %\label{H:Cx}
\end{align*}
that
\begin{align}
u_i(\cdot,x)\in C([t_0,t_1])\cap C^1((t_0,t_1]) \text{ for all }x\in \Omega,  \label{H:Ct}
\end{align}
and that
\begin{align}
\sup_{(t,x)\in (t_0,t_1]\times\Omega}\,|\partial_tu_i(t,x)|<\infty. \label{unif:bdd:deriv}
\end{align}
%In addition, we assume that $\partial_tu$ and $\partial_tv$ are uniformly bounded in $x$ and $t$.
Suppose that
\begin{align}
\left\{
\begin{array}{rll}
\partial_tu_1-Lu_1-f(u_1)\!\!&\geq \,\partial_tu_2-Lu_2-f(u_2) & \mbox{in }(t_0,t_1]\times\Omega, \vspace{3pt}\\
u_1(t_0,\cdot)\!\!&\geq \,u_2(t_0,\cdot) & \mbox{in }\Omega.
\end{array}
\right. \label{COMPARAISON:LEMME}
\end{align}
Then,
$$u_1(t,x)\geq u_2(t,x)\ \text{ for all }(t,x)\in [t_0,t_1]\times\Omega.$$
\end{lemma}
\begin{remark}
For related results in similar contexts, the reader may consult \cite{Berestycki2017a,Chen2002}.
\end{remark}
\begin{proof}
We set $w:=u_1-u_2$. Readily, we notice that
\begin{align}
C_1:=\sup_{(t,x)\in(t_0,t_1]\times\Omega}\big(|w(t,x)|+|\partial_tw(t,x)|\big)<\infty. \label{def:C1}
\end{align}
(Remember \eqref{unif:bdd:deriv} and the boundedness assumption on $u_1$ and $u_2$.)

Moreover, we let $\mu\in L^\infty([t_0,t_1]\times\Omega)$ be any function so that
$$ f(u_1(t,x))-f(u_1(t,x))=\mu(t,x)\hspace{0.1em}(u_1(t,x)-u_2(t,x)) \text{ for all }(t,x)\in[t_0,t_1]\times\Omega. $$
%\begin{align*}
%\mu(t,x):=\left\{\begin{array}{cl}\frac{f(u_1(t,x))-f(u_2(t,x))}{u_1(t,x)-u_2(t,x)} &\text{if }u_1(t,x)\neq u_2(t,x), \vspace{3pt} \\ 0 &\text{else. } \end{array} \right. %\text{ and }\mu(t,x):=0\text{ else.} %\label{DEF:mu}
%\end{align*}
Note that such a function always exists since $u_1$ and $u_2$ are bounded and since $f\in C_{\mathrm{loc}}^{0,1}(\R)$. %, it follows that $\mu$ is a well-defined bounded function.

Now, using the hypotheses made on $u_1$ and $u_2$, we have
\begin{align}
\partial_tw(t,x)-Lw(t,x)&\geq f(u_1(t,x))-f(u_2(t,x))= \mu(t,x)\hspace{0.1em} w(t,x), \nonumber
\end{align}
for any $(t,x)\in [t_1,t_2]\times\Omega$. Next, we let $\kappa>0$ be so large that
$$ \kappa\geq \|\mu\|_\infty+\|\mathcal{J}^\delta\|_\infty+1, $$
and we let $\widetilde{w}$ be the function given by $\widetilde{w}(t,x):=e^{\kappa(t-t_0)}w(t,x)$ for all $(t,x)\in[t_0,t_1]\times\Omega$. By a straightforward calculation, we have that
\begin{align}
\partial_t\widetilde w(t,x)&=e^{\kappa t}\partial_tw(t,w)+ \kappa\, \widetilde w(t,x) \nonumber\\
&\geq e^{\kappa t}Lw(t,x) +(\mu(t,x)+\kappa)\hspace{0.1em}\widetilde w(t,x) \nonumber\\
&= \int_{\Omega}J(\delta(x,y))\widetilde w(t,y)\,\mathrm{d}y +(\mu(t,x)+\kappa-\j^{\delta})\hspace{0.1em}\widetilde w(t,x). \label{est:prelim:wtilde}
\end{align}
Furthermore, recalling \eqref{def:C1} and using that $w(\cdot,x)\in C([t_0,t_1])$ (remember \eqref{H:Ct}), we have
\begin{align}
|\widetilde{w}(t,x)-\widetilde{w}(t',x)|&=|e^{\kappa(t-t_0)}w(t,x)-e^{\kappa(t'-t_0)}w(t',x)| \nonumber\\
&=|(e^{\kappa(t-t_0)}-e^{\kappa(t'-t_0)})\hspace{0.1em}w(t,x)+e^{\kappa(t'-t_0)}(w(t,x)-w(t',x))| \nonumber\\
&\leq C_1\left(|e^{\kappa(t-t_0)}-e^{\kappa(t'-t_0)}|+e^{\kappa(t'-t)}|t-t'|\right) \nonumber\\
&\leq C_1(\kappa+1)\hspace{0.1em}e^{\kappa(t_1-t_0)}|t-t'|, \label{Lip:wtilde}
\end{align}
for all $t,t'\in[t_0,t_1]$ and all $x\in\Omega$.

Now, for all $s\geq0$, we define the perturbation $\widetilde{w}_s$ of $\widetilde{w}$ given by $\widetilde{w}_s(t,x)=\widetilde{w}(t,x)+s\hspace{0.1em}e^{2\kappa(t-t_0)}$ for all $(t,x)\in[t_0,t_1]\times\Omega$.
Observe that $\partial_t\widetilde{w}_s(t,x)=\partial_t\widetilde{w}(t,x)+2\hspace{0.05em}\kappa\hspace{0.1em} s\hspace{0.1em} e^{2\kappa(t-t_0)}$. So, using \eqref{est:prelim:wtilde}, by a short computation we find that
\begin{align*}
\partial_t\widetilde{w}_s(t,x)\geq\int_\Omega J(\delta(x,y))\hspace{0.05em}\widetilde{w}_s(t,y)\hspace{0.1em}\mathrm{d}y+\gamma_1(t,x)\hspace{0.1em}\widetilde{w}_s(t,x)+\gamma_2(t,x)\hspace{0.1em}s\hspace{0.1em}e^{2\kappa(t-t_0)}.
\end{align*}
where $\gamma_1$ and $\gamma_2$ denote the following expressions
$$ \gamma_1(t,x):=\mu(t,x)+\kappa-\mathcal{J}^\delta(x) \text{ and } \gamma_2(t,x):=\kappa-\mu(t,x). $$
Observe that, by construction of $\kappa$, we have $\gamma_1(t,x)>0$ and $\gamma_2(t,x)>0$ for all $(t,x)\in[t_0,t_1]\times\Omega$. In particular, we have
\begin{align}
\partial_t\widetilde{w}_s(t,x)>0 \text{ for all }x\in\Omega, \text{ as soon as } \widetilde{w}_s(t,x)>0 \text{ for all }x\in\Omega. \label{coeur:demo}
\end{align}
Since $\widetilde{w}_s(t,x)=\widetilde{w}(t,x)+s\hspace{0.1em}e^{2\kappa(t-t_0)}$ and since $\widetilde{w}(t_0,x)=w(t_0,x)\geq0$, we have
$$ \widetilde{w}_s(t,x)\geq \widetilde{w}(t,x)-\widetilde{w}(t_0,x)+\widetilde{w}(t_0,x)+s\hspace{0.1em}e^{2\hspace{0.1em}\kappa(t-t_0)}\geq -|\widetilde{w}(t,x)-\widetilde{w}(t_0,x)|+s. $$
Using \eqref{Lip:wtilde} with $t'=t_0$, we obtain
$$ \widetilde{w}_s(t,x)\geq -C_2|t-t_0|+s, $$
where $C_2:=C_1(\kappa+1)\hspace{0.1em}e^{\kappa(t-t_0)}$. In turn, this implies that
\begin{align*}
\widetilde{w}_s(t,x)>0 \text{ for all }(t,x)\in\left[t_0,t_0+\frac{s}{2C_2}\right)\times\Omega.
\end{align*}
In particular, the following quantity is well-defined
$$ t_*:=\sup\Big\{t\in(t_0,t_1)\ ;\ \widetilde{w}_s(\tau,x)>0 \text{ for all } (\tau,x)\in(t_0,t)\times\Omega\Big\}. $$
Clearly, $t_*> t_0+s/(4C_2)$. Suppose, by contradiction, that $t_*<t_1$. Then, by definition of $t_*$,
we must have $\widetilde{w}_s(t_*,x)\geq0$ and $\widetilde{w}_s(t,x)>0$ for all $t\in(t_0,t_*)$ and all $x\in\Omega$.
From the latter and \eqref{coeur:demo}, we deduce that $\widetilde{w}_s(t,x)$ is monotone increasing in $(t_0,t_*)$. Hence,
we have
\begin{align*}
\widetilde{w}_s(t,x)\geq \widetilde{w}_s\left(t_0+\frac{s}{4C_2},x\right)\geq\frac{3s}{4}>0 \text{ for all }(t,x)\in\left[t_0+\frac{s}{4C_2},t_*\right)\times\Omega.
\end{align*}
Letting $t\to t_*^-$, we get $\widetilde{w}_s(t_*,x)\geq3s/4$. Thus, recalling the definition of $\widetilde{w}_s$, we have
$$ \widetilde{w}_s(t_*+\eps,x)\geq \widetilde{w}(t_*+\eps,x)-\widetilde{w}(t_*,x)+\widetilde{w}_s(t_*,x)\geq -C_2\hspace{0.1em}\eps+\frac{3s}{4}, $$
for all $0<\eps<t_1-t_*$, where we have used \eqref{Lip:wtilde}. This implies that $\widetilde{w}_s(t_*+\eps,x)>0$ for all $x\in\Omega$ and all $0<\eps<\min\{t_1-t_*,3s/(4C_2)\}$, which contradicts the maximality of $t_*$. Therefore, $t_*=t_1$ which enforces that $\widetilde{w}_s(t,x)>0$ for all $(t,x)\in(t_0,t_1]\times\Omega$. Recalling \eqref{coeur:demo}, we further obtain that $\partial_t\widetilde{w}_s(t,x)>0$ for all $(t,x)\in(t_0,t_1]\times\Omega$, so that $\widetilde{w}_s$ is an increasing function of time for all $x\in\Omega$. In particular, we have
$$ \widetilde{w}_s(t,x)> \widetilde{w}_s(t_0,x)=\widetilde{w}(t_0,x)+s \text{ for all }(t,x)\in(t_0,t_1]\times\Omega. $$
Letting now $s\to0^+$, we obtain that
$$ e^{\kappa(t-t_0)}w(t,x)=\widetilde{w}(t,x)\geq \widetilde{w}(t_0,x)\geq0 \text{ for all }(t,x)\in(t_0,t_1]\times\Omega. $$
Therefore, we have $w(t,x)\geq0$ for all $(t,x)\in[t_0,t_1]\times\Omega$, as desired.
\end{proof}
\begin{remark}\label{RE:COMPARISON}
It turns out that our proof also yield a version of Lemma~\ref{LE:COMPARISON} on half-spaces $H\subset\Omega$.
Namely, if $R_J>0$ is such that $\mathrm{supp}(J)\subset[0,R_J]$, if one replaces \eqref{COMPARAISON:LEMME} by
\begin{align*}
\left\{
\begin{array}{rll}
\partial_tu_1-Lu_1-f(u_1)\!\!&\geq \,\partial_tu_2-Lu_2-f(u_2) & \mbox{in }(t_0,t_1]\times H, \vspace{3pt}\\
u_1(t_0,\cdot)\!\!&\geq \,u_2(t_0,\cdot) & \mbox{in }H.
\end{array}
\right.
\end{align*}
and if one further assume that
\begin{align}
u_1\geq u_2 \text{ in } [t_0,t_1]\times\left\{x\in\Omega\setminus H;\ \inf_{y\in H}\hspace{0.1em}\delta(x,y)\leq R_J\right\}, \label{further:comp}
\end{align}
then, it still holds that $u_1(t,x)\geq u_2(t,x)$ for all $(t,x)\in [t_0,t_1]\times H$.
Notice that it is because of the nonlocality of the operator $L$ that we need to assume \eqref{further:comp}.
\end{remark}
\begin{lemma}\label{LE:COMPARISON2}
Assume \eqref{C1}, \eqref{C3} and suppose that $f\in C^1(\R)$. Let $t_0,t_1\in\R$ with $t_0<t_1$ and let $u:[t_0,t_1]\times\Omega\to\R$ be a measurable function such that $u(t,\cdot)\in C(\Omega)$ for each fixed $t\in[t_0,t_1]$, and that $u(\cdot,x)\in C^1([t_0,t_1])\cap C^2((t_0,t_1])$ for each fixed $x\in\Omega$. Suppose, in addition, that $u$, $\partial_tu$ and $\partial_t^2u$ are uniformly bounded (in $x$ and $t$) and that
\begin{align*}
\left\{
\begin{array}{rll}
\partial_tu\!\!&=Lu+f(u) & \mbox{in }(t_0,t_1]\times\Omega, \vspace{3pt}\\
\partial_tu(t_0,\cdot)\!\!&\geq 0 & \mbox{in }\Omega.
\end{array}
\right.
\end{align*}
Then, $\partial_tu(t,x)\geq0$ in $[t_0,t_1]\times\Omega$.
\end{lemma}
\begin{proof}
Letting $v(t,x):=\partial_tu(t,x)$ we have $v(t_0,\cdot)\geq0$ in $\Omega$ and
$$ \partial_tv(t,x)-Lv(t,x)=v(t,x)\hspace{0.1em}f'(u(t,x))=:\mu(t,x)\hspace{0.1em} v(t,x) {\mbox{ in }}(t_0,t_1]\times\Omega, $$
where $\mu(t,x)$ is a bounded function (because $f\in C^1(\R)$ and $u$ is bounded). From here, we may apply the same strategy as in Lemma \ref{LE:COMPARISON}.
\end{proof}

\subsection{Existence of a unique solution}

In this section, we will establish the existence and uniqueness of a solution to \eqref{CauchyPB}.  For the sake of convenience, for $f\in C^{0,1}\cap C^1(\R)$, we set %$\omega$ to be the following number
\begin{align}
\omega:=\sup_{\R}|f'|+2\hspace{0.1em}\sup_{\Omega}\hspace{0.1em}\mathcal{J}^\delta. \label{DEF:OMEGA}
\end{align}
Then, we have the following result:
\begin{prop}[Existence and uniqueness]\label{PROP:EXIST}
Let $t_0\in\R$ and let $u_0\in C_b(\Omega)$. Assume \eqref{C1}, \eqref{C3} and suppose that $f\in C^{0,1}\cap C^1(\R)$. Then, there exists a unique solution $u\in C^2([t_0,\infty),C(\Omega))$ to \eqref{CauchyPB}. Moreover, for all $T> t_0$, the following estimates hold:
\begin{align}
\begin{array}{rl}
\omega^{-1}\|\partial_{tt}u\|_{L^\infty([t_0,T]\times\Omega)}\leq\|\partial_tu\|_{L^\infty([t_0,T]\times\Omega)}\!\!\!&\leq \big(\omega+|f(0)|\big)\|u\|_{L^\infty([t_0,T]\times\Omega)}.
\end{array} \label{apriori777}
\end{align}
%\begin{align} \sup_{x\in\Omega}\,\|\partial_t^{\ell+1}u(\cdot,x)\|_{\infty}\leq\omega^{\ell+1}\sup_{x\in\Omega}\,\|\partial_t^{\ell}u(\cdot,x)\|_{\infty},  \label{apriori777}
%\end{align}
%for any $\ell\in\{0,1\}$.
\end{prop}
\begin{proof}
The proof is rather standard but we nevertheless outline the main ingredients. First of all, we observe that the a priori estimates \eqref{apriori777} follow directly by using \eqref{CauchyPB} and the equation obtained when differentiating \eqref{CauchyPB} with respect to $t$. %Now, let us extend $f$ by $f'(0)\,r$ for $r<0$ and by $f'(1)(r-1)$ for $r>1$. For simplicity, we still denote by $f$ this extension.
Now, let us define
\begin{align}
\mathcal{L}[u](t,x):=\int_\Omega J(\delta(x,y))\hspace{0.05em}u(t,y)\hspace{0.1em}\mathrm{d}y. \label{Lcrochet}
\end{align}
Observe that, thanks to \eqref{C3}, we have that $\mathcal{J}^\delta\in C_b(\Omega)$ and that the operator $\mathcal{L}[\cdot]$ maps $C_b(\Omega)$ into itself.
In fact, by our assumptions on $J$, $\mathcal{L}[\cdot]$ is a well-defined continuous linear operator in $C_b(\Omega)$ (endowed with the sup-norm) and we have $\|\mathcal{L}\|\leq\|\mathcal{J}^\delta\|_\infty$.

Next, multiplying \eqref{CauchyPB} by $e^{\omega \tau}$, where $\omega$ is given by \eqref{DEF:OMEGA}, and integrating over $\tau\in[t_0,t]$, we arrive at the following integral equation
\begin{align}
u(t,x)\!&=\!e^{-\omega (t-t_0)}u_0(x)\!+\!\!\int_{t_0}^t \! e^{-\omega(t-\tau)}\big(\mathcal{L}[u](\tau,x)\!+\!(\omega\!-\!\mathcal{J}^\delta(x)) u(\tau,x)\!+\!f(u(\tau,x))\big)\mathrm{d}\tau. \label{EQ:INTEGRAL}
\end{align}
Since \eqref{CauchyPB} and \eqref{EQ:INTEGRAL} are equivalent, it suffices to establish the existence and uniqueness of a solution to \eqref{EQ:INTEGRAL}. For the sake of clarity, we subdivide the proof of this into three steps.
\vskip 0.3cm

\noindent \emph{Step 1. A preliminary a priori bound on }$\|u(t,\cdot)\|_\infty$
\smallskip

Prior to proving the existence of a solution $u$ to \eqref{CauchyPB} (or, equivalently, to \eqref{EQ:INTEGRAL}), let us first establish a preliminary a priori bound on $\|u(t,\cdot)\|_\infty$. For it, we observe that
$$ \left|\mathcal{L}[u](\tau,x)+(\omega-\mathcal{J}^\delta(x))\hspace{0.1em}u(\tau,x)+f(u(\tau,x))\right|\leq 2\hspace{0.1em}\omega\|u(\tau,\cdot)\|_\infty+|f(0)|. $$
Now, plugging this into \eqref{EQ:INTEGRAL}, we obtain
$$e^{\omega t}\hspace{0.05em}\|u(t,\cdot)\|_\infty\leq e^{\omega t_0}\hspace{0.05em}\|u_0\|_\infty+2\hspace{0.1em}\omega\int_{t_0}^t e^{\omega\tau}\hspace{0.05em}\|u(\tau,\cdot)\|_\infty\hspace{0.1em}\mathrm{d}\tau+|f(0)|\int_{t_0}^t e^{\omega\tau}\hspace{0.1em}\mathrm{d}\tau. $$
Letting $v(t):=e^{\omega t}\hspace{0.05em}\|u(t,\cdot)\|_\infty$ and $g(t):=|f(0)|\int_{t_0}^t e^{\omega\tau}\hspace{0.1em}\mathrm{d}\tau$, this becomes
$$ v(t)\leq v(t_0)+2\hspace{0.1em}\omega\int_{t_0}^t v(\tau)\hspace{0.1em}\mathrm{d}\tau+g(t). $$
Applying now Gr\"onwall's lemma, we arrive at $v(t)\leq (g(t)+v(t_0))\hspace{0.1em}e^{\omega(t-t_0)}$.
Developping this expression using the definition of $v$ and $g$, we obtain
\begin{align}
\|u(t,\cdot)\|_\infty\leq e^{\omega(t-t_0)}\left(\frac{|f(0)|}{\omega}(e^{\omega(t-t_0)}-1)+\|u_0\|_\infty\right), \label{apriori}
\end{align}
for any $t\geq t_0$. In particular, $\|u(t,\cdot)\|_\infty$ is locally bounded in $t\in[t_0,\infty)$.
\vskip 0.3cm

\noindent\emph{Step 2. Construction of a micro-solution in a small window of time}
\smallskip

Let $T_0\in(t_0,t_0+\omega^{-1}\log(2))$ be arbitrary and let $(u^n)_{n\geq0}$ be the sequence of functions defined on $(t,x)\in[t_0,T_0]\times\Omega$ by
$$ u^0(t,x)=e^{-\omega (t-t_0)}\hspace{0.1em}u_0(x), $$
and, for $n\geq 0$,
$$
u^{n+1}(t,x)=u^0(t,x)\!+\!\!\int_{t_0}^t \! e^{-\omega(t-\tau)}\big(\mathcal{L}[u^n](\tau,x)\!+\!(\omega\!-\!\mathcal{J}^\delta(x)) u^n(\tau,x)\!+\!f(u^n(\tau,x))\big)\mathrm{d}\tau.
$$
Remark that, since $f$ is continuous, $\mathcal{J}^\delta\in C_b(\Omega)$, $u^0\in C_b([t_0,T_0]\times\Omega)$ and $\mathcal{L}[\cdot]$ is a continuous linear operator in $C_b(\Omega)$, it follows that $(u^n)_{n\geq0}\subset C_b([t_0,T_0]\times\Omega)$.

Now, for any $n\geq 1$, it holds that
\begin{align*}
|u^{n+1}(t,x)-u^n(t,x)|&\leq 2\hspace{0.1em}\omega\int_{t_0}^te^{-\omega(t-\tau)}\hspace{0.1em}\mathrm{d}\tau\sup_{(\tau,x)\in[t_0,T_0]\times\Omega}|u^n(\tau,x)-u^{n-1}(\tau,x)| \\
&\leq2\left(1-e^{-\omega(T_0-t_0)}\right)\sup_{(\tau,x)\in[t_0,T_0]\times\Omega}|u^n(\tau,x)-u^{n-1}(\tau,x)|,
\end{align*}
where we have used the definition of $\omega$. We therefore arrive at
$$ \sup_{(t,x)\in[t_0,T_0]\times\Omega}|u^{n+1}(t,x)-u^n(t,x)|\leq H \sup_{(t,x)\in[t_0,T_0]\times\Omega}|u^n(t,x)-u^{n-1}(t,x)|, $$
where we have set
$$ H:=2\left(1-e^{-\omega(T_0-t_0)}\right). $$
Notice that, since $T_0<t_0+\omega^{-1}\log(2)$, we have that $H\in(0,1)$.
Thus,
$$ \sup_{(t,x)\in[t_0,T_0]\times\Omega}|u^{n+1}(t,x)-u^n(t,x)|\leq H^n\sup_{(t,x)\in[t_0,T_0]\times\Omega}|u^{1}(t,x)-u^0(t,x)|\to0 \quad\text{as }n\to\infty. $$
Hence $(u^n)_{n\geq0}$ is a Cauchy sequence in the topology of $C_b([t_0,T_0]\times\Omega)$ (equipped with the sup-norm). Since $(C_b([t_0,T_0]\times\Omega),\left\|\cdot\right\|_{\infty})$ is complete, it follows that $u_n$ converges towards a function $u^{T_0}\in C_b([t_0,T_0]\times\Omega)$ which, by dominated convergence, solves the equation on $[t_0,T_0]\times\Omega$. (Notice that, since $f\in C^1(\R)$, a straightforward bootstrap argument shows that $u^{T_0}\in C^2([t_0,T_0],C(\Omega))$.) Using \eqref{apriori} together with \eqref{apriori777} we may apply the comparison principle Lemma~\ref{LE:COMPARISON} to deduce that $u^{T_0}$ is the \emph{unique} solution to \eqref{CauchyPB} in $[t_0,T_0]$.
\vskip 0.3cm

\noindent \emph{Step 3. Conclusion}
\smallskip

The solution to \eqref{CauchyPB} in the whole $[t_0,\infty)$ is obtained by a classical ``analytic continuation" type argument by concatenating micro-solutions $u^{T_k}$ on time intervals of the form $[T_{k-1},T_{k}]$ with $k\geq0$, where $T_k:=T_0+k(T_0-t_0)$ for any $-1\leq k\in \Z$. This is indeed possible because the micro-solutions $u^{T_k}$ are uniquely determined, continuous up to $T_{k}$ and they satisfy $\partial_tu^{T_k}(T_{k}^{-},\cdot)=\partial_tu^{T_{k+1}}(T_{k}^+,\cdot)$. Hence, using again \eqref{apriori777}, the comparison principle Lemma~\ref{LE:COMPARISON}, the fact that $f$ is $C^1$ and that $u^{T_k}$ is bounded for any $k\geq0$, we may easily check that the so-constructed solution is unique and has the claimed regularity in both space and time.
The proof is thereby complete.
\end{proof}

\begin{remark}
Although this is a standard fact, we recall that a micro-solution on a time interval of length at most $\omega^{-1}\log(2)$ is necessarily continuous in space provided the initial data is continuous (the proof of this fact follows closely the arguments of Step 2).
By induction, it follows that a solution to the Cauchy problem~\eqref{CauchyPB} is also necessarily space continuous provided $u_0\in C(\Omega)$.
In particular, this justifies why we could use the comparison principle Lemma~\ref{LE:COMPARISON} (that requires space continuity) to derive the uniqueness of the solution.
%this means that the uniqueness part of Proposition~\ref{PROP:EXIST} -- which follows from the comparison principle Lemma~\ref{LE:COMPARISON} that requires space continuity -- does not hold only in the class of space continuous solutions.
\end{remark}
\begin{remark}\label{RK:Ext:Cont}
If the initial datum $u_0$ can be extended as a continuous function up to the boundary (for example if it is uniformly continuous), then the solution to the Cauchy problem \eqref{CauchyPB} can also be extended so that $u\in C^2([t_0,\infty),C(\overline{\Omega}))$. Moreover, this extension is a solution of the equation in $\overline{\Omega}$.
\end{remark}

\subsection{Parabolic type estimates}

%We begin with a regularity result for solutions to \eqref{CauchyPB}.
%\begin{prop}
%Let $t_0\in\R$ and let $u_0\in C(\Omega,[0,1])$. Assume \eqref{C1} and \eqref{C3}. Suppose that there exist a solution $u$ to \eqref{CauchyPB}. Then $C(\Omega,C^2([t_0,\infty),[0,1]))$. Moreover,
%\begin{align} \sup_{x\in\Omega}\,\|\partial_t^{\ell+1}u(x,\cdot)\|_{\infty}\leq\omega^{\ell+1}\sup_{x\in\Omega}\,\|\partial_t^{\ell}u(x,\cdot)\|_{\infty},  \label{apriori777}
%\end{align}
%for any $\ell\in\{0,1\}$.
%\end{prop}
Let us now complete this section with a time-global parabolic estimate for the Cauchy problem \eqref{CauchyPB}. For it, we will require the additional assumption
\begin{align}
\max_{\R}f'<\inf_{\Omega}\mathcal{J}^\delta. \label{fjcond}
\end{align}
Precisely, we prove
\begin{prop}[Parabolic estimates]\label{PROP:PARABOLIK2}
%Let $\alpha\in(0,1]$ and assume that $J\in\mathbb{B}_{1,\infty}^\alpha(\Omega;\delta)$.
Assume \eqref{C1} and \eqref{C3}. Suppose, in addition, that $f\in C^{0,1}\cap C^1(\R^N)$, that $J\in\mathbb{B}_{1,\infty}^\alpha(\Omega;\delta)$ for some $\alpha\in(0,1)$ and that \eqref{fjcond} holds.
Let $t_0\in\R$ and let $u_0\in C^{0,\alpha}(\overline{\Omega})$.
Let $u\in C^2([t_0,\infty),C(\overline{\Omega},[0,1]))$ be the unique solution to \eqref{CauchyPB}.
Suppose that $u$ is uniformly bounded by some constant $M_0>0$.
Then, there exists a constant $M>0$ (depending on $J$, $f'$, $M_0$, $[u_0]_{C^{0,\alpha}(\overline{\Omega})}$, $\Omega$ and $\delta$) such that
\begin{align*}
\sup_{t\geq t_0}\,\left([u(t,\cdot)]_{C^{0,\alpha}(\overline{\Omega})}+[\partial_tu(t,\cdot)]_{C^{0,\alpha}(\overline{\Omega})}\right)\leq M. %\label{loc:unif:holder2}
\end{align*}
\end{prop}
\begin{remark}
Notice that, in addition to \eqref{C1} and \eqref{C3}, it is further required that $J\in\mathbb{B}_{1,\infty}^\alpha(\Omega;\delta)$ and that $(J,f)$ satisfies \eqref{fjcond}. These extra assumptions are essentially the same as those which were shown in \cite[Lemma 3.2]{Brasseur2019} (see also \cite[Remark 2.5]{Brasseur2018}) to be sufficient for the stationary solution to be (at least) H\"older continuous (remember Remark~\ref{RE:BESOV}). The estimate we derive for $[u(t,\cdot)]_{C^{0,\alpha}(\overline{\Omega})}$ (see \eqref{HOL:U2} below) is actually very similar to the one obtained in \cite[Lemma 3.2]{Brasseur2019} for the stationary problem. Also, as we already pointed out in \cite{Brasseur2018}, this is a sort of ``nondegeneracy condition" which is somehow necessary to ensure global parabolic regularity. Indeed, if $\delta$ is the Euclidean distance and $K=\emptyset$, this condition reads $\max_\R f'<1$ and, when this condition is not satisfied, it is known that there exists kernels $J\in L^1(\R^N)$ such that the equation $\partial_tu=J\ast u-u+f(u)$ admits \emph{discontinuous} standing fronts \cite{Bates1997,Wang2002a}. In this situation, the solution of the Cauchy problem \eqref{CauchyPB} starting from a smooth Heaviside type initial datum is expected to converge towards a discontinuous front (in some weak topology), making thus the above estimate impossible.
\end{remark}
\begin{proof}
Let $u$ be a solution of \eqref{CauchyPB}. By Proposition \ref{PROP:EXIST}, we know that $u$ is continuous, therefore it is well-defined for all $t\in[t_0,\infty)$ and all $x\in\Omega$. Actually, since $u_0\in C^{0,\alpha}(\overline{\Omega})$, the function $u$ is also space continuous in the whole of $\overline{\Omega}$ (remember Remark~\ref{RK:Ext:Cont}) and, hence, is also well-defined for all $t\in[t_0,\infty)$ and all $x\in\overline{\Omega}$. Let us fix some $x_1,x_2\in\overline{\Omega}$ with $x_1\neq x_2$, define $\Psi_u(t):=u(t,x_1)-u(t,x_2)$ and set
$$ H(t,x_1,x_2):=\displaystyle\int_\Omega \,\hspace{0.1em}(u(t,y)-u(t,x_1))(J(\delta(x_1,y))-J(\delta(x_2,y)))\hspace{0.1em}\mathrm{d}y. $$
Observe immediately that, since $|u|\leq M_0$ and since $J\in \mathbb{B}_{1,\infty}^\alpha(\Omega;\delta)$, we have
\begin{align*}
|H(t,x_1,x_2)|\leq 2\hspace{0.05em}M_0\hspace{0.05em}[J]_{\mathbb{B}_{1,\infty}^\alpha(\Omega;\delta)}|x_1-x_2|^\alpha=:\beta. %\label{trickparab}
\end{align*}
Since $f\in C^1(\R)$ and $u(\cdot,x)\in C(\R)$ for all $x\in\overline{\Omega}$, it follows from the mean value theorem that there exists a function $\Lambda$, ranging between $u(t,x_1)$ and $u(t,x_2)$, such that $f'(\Lambda(t))\Psi_u(t)=f(u(t,x_1))-f(u(t,x_2))$ and that $f'(\Lambda)$ is continuous.
Letting $\gamma(t):=\mathcal{J}^\delta(x_2)-f'(\Lambda(t))$ and using the function $H$, we can write the equation satisfied by $\Psi_u$ as
%Then, $\Psi_u$ solves
$$ \left\{\begin{array}{rcl} \Psi_u'(t)\!\!\!& =&\!\!\! H(t,x_1,x_2)+\gamma(t)\hspace{0.1em}\Psi_u(t) \,\text{ for }t>t_0, \vspace{3pt} \\ \Psi_u(t_0)\!\!\!&=&\!\!\! u_0(x_1)-u_0(x_2), \end{array} \right. $$
%for $t>t_0$ with initial datum $\Psi_u(t_0)=u_0(x_1)-u_0(x_2)$;
%where $\Lambda$ is a function ranging between $u(t,x_1)$ and $u(t,x_2)$ such that $f'(\Lambda(t))\Psi_u(t)=f(u(t,x_1))-f(u(t,x_2))$ and that $f'(\Lambda)$ is continuous. Notice that $\Lambda$ is well-defined (this follows from the mean value theorem and the fact that $f\in C^1(\R)$ and $u(\cdot,x)\in C(\R)$).
Observe that, since $f'(\Lambda)$ is continuous, $\gamma$ is also continuous.

Next, we let $v(t)$ be the unique solution of
\begin{align}
\left\{
\begin{array}{rll}
v'(t)\!\!&\!=\beta-\gamma(t)\hspace{0.1em} v(t) & \text{for }t>t_0, \vspace{3pt}\\
v(t_0)\!\!&\!=d_0, &
\end{array}
\right. \label{ODE}
\end{align}
where we have set $d_0:=[u_0]_{C^{0,\alpha}(\overline{\Omega})}|x_1-x_2|^\alpha$.
Now, since \eqref{ODE} is a linear ordinary differential linear equation, we can compute $v$ explicitly. Namely, we have
\begin{align*}
v(t)=d_0\,\mathrm{exp}\left(-\int_{t_0}^t\gamma(\tau)\hspace{0.1em}\mathrm{d}\tau\right)+\beta\int_{t_0}^t\mathrm{exp}\left(-\int_{T}^t\gamma(\tau)\hspace{0.1em}\mathrm{d}\tau\right)\mathrm{d}T.
\end{align*}
%\begin{align*}
%v(t)&= e^{-\gamma(t-t_0)}d_0+\gamma^{-1}\beta\,\big(1-e^{-\gamma(t-t_0)}\big)  \\
%&=e^{-\gamma(t-t_0)}|x_1-x_2|^\alpha[u_0]_{C^{0,\alpha}(\Omega)}+2\hspace{0.05em}M_0\hspace{0.05em}\gamma^{-1}[J]_{\mathbb{B}_{1,\infty}^\alpha(\Omega;\delta)}|x_1-x_2|^\alpha\big(1-e^{-\gamma(t-t_0)}\big). \end{align*}
By assumption \eqref{fjcond}, we have $\gamma\geq\inf_{\Omega}\mathcal{J}^\delta-\max_{\R}f'=:\gamma_*>0$. In particular,
\begin{align*}
v(t)&\leq d_0\hspace{0.1em}e^{-\gamma_*(t-t_0)}+\beta\int_{t_0}^te^{-\gamma_*(t-T)}\hspace{0.1em}\mathrm{d}T = d_0\hspace{0.1em}e^{-\gamma_*(t-t_0)}+\beta\hspace{0.05em}\gamma_*^{-1}(1-e^{-\gamma_*(t-t_0)}).
\end{align*}
%$$ 0<v(t)\leq \left([u_0]_{C^{0,\alpha}(\Omega)}+2\hspace{0.05em}M_0\hspace{0.05em}\gamma^{-1}[J]_{\mathbb{B}_{1,\infty}^\alpha(\Omega;\delta)}\right)|x_1-x_2|^\alpha. $$
Recalling the definition of $\beta$ and $d_0$, we obtain that
$$ 0<v(t)\leq \left([u_0]_{C^{0,\alpha}(\overline{\Omega})}+2\hspace{0.05em}M_0\hspace{0.05em}\gamma_*^{-1}[J]_{\mathbb{B}_{1,\infty}^\alpha(\Omega;\delta)}\right)|x_1-x_2|^\alpha. $$
Notice, furthermore, that if $\psi$ is either $\Psi_u$ or $-\Psi_u$, then we have
\begin{align*}
\left\{
\begin{array}{rll}
v'(t)-\beta+\gamma(t)\hspace{0.1em} v(t)\!\!&\!\geq \psi'(t)-\beta+\gamma(t)\hspace{0.1em} \psi(t) & \text{for }t>t_0, \vspace{3pt}\\
v(t_0)\!\!&\!\geq \psi(t_0). &
\end{array}
\right.
\end{align*}
Hence, by the comparison principle for ordinary differential equations, we have
\begin{align}
|u(t,x_1)-u(t,x_2)|=|\Psi_u(t)|\leq v(t)\leq  \left([u_0]_{C^{0,\alpha}(\overline{\Omega})}+2\hspace{0.05em}M_0\hspace{0.05em}\gamma_*^{-1}[J]_{\mathbb{B}_{1,\infty}^\alpha(\Omega;\delta)}\right)|x_1-x_2|^\alpha. \label{HOL:U2}
\end{align}
Thus, $[u(t,\cdot)]_{C^{0,\alpha}(\overline{\Omega})}\leq ([u_0]_{C^{0,\alpha}(\overline{\Omega})}+2\hspace{0.05em}M_0\hspace{0.05em}\gamma_*^{-1}[J]_{\mathbb{B}_{1,\infty}^\alpha(\Omega;\delta)})$.
Let us now establish the corresponding inequality for $\partial_tu$.
Using \eqref{CauchyPB}, we have
\begin{align}
|\partial_tu(t,x_1)&-\partial_tu(t,x_2)|\leq \|u(t,\cdot)\|_\infty\int_{\Omega}\hspace{0.1em}|J(\delta(x_1,y))-J(\delta(x_2,y))|\mathrm{d}y \nonumber \\
&+|\mathcal{J}^\delta(x_1)\hspace{0.05em}u(t,x_1)-\mathcal{J}^\delta(x_2)\hspace{0.05em}u(t,x_2)|+|f(u(t,x_1))-f(u(t,x_2))|=: A+B+C. \label{ABC}
\end{align}
Since $J\in \mathbb{B}_{1,\infty}^\alpha(\Omega;\delta)$ and $|u|\leq M_0$ we have
\begin{align}
A\leq M_0\hspace{0.05em}[J]_{\mathbb{B}_{1,\infty}^\alpha(\Omega;\delta)}|x_1-x_2|^\alpha. \label{S001}
\end{align}
Now, using the trivial relation
\begin{align}
\mathcal{J}^\delta(x_1)\hspace{0.05em}u(t,x_1)\!-\!\mathcal{J}^\delta&(x_2)\hspace{0.05em}u(t,x_2)=\mathcal{J}^\delta(x_1)(u(t,x_1)\!-\!u(t,x_2)) \!+\!u(t,x_2)\big(\mathcal{J}^\delta(x_1)\!-\!\mathcal{J}^\delta(x_2)\big), \nonumber
\end{align}
together with the fact that $J\in \mathbb{B}_{1,\infty}^\alpha(\Omega;\delta)$ and that $|u|\leq M_0$, we further have
\begin{align}
B\leq \|\mathcal{J}^\delta\|_\infty|u(t,x_1)-u(t,x_2)|+M_0\hspace{0.05em}[J]_{\mathbb{B}_{1,\infty}^\alpha(\Omega;\delta)}|x_1-x_2|^\alpha.  \label{S002}
\end{align}
Plugging \eqref{S001} and \eqref{S002} in \eqref{ABC}, we get
\begin{align*}
|\partial_tu(t,x_1)-\partial_tu(t,x_2)|\leq \widetilde{\omega}\hspace{0.1em}|u(t,x_1)-u(t,x_2)|+2\hspace{0.05em}M_0\hspace{0.05em}[J]_{\mathbb{B}_{1,\infty}^\alpha(\Omega;\delta)}|x_1-x_2|^\alpha,
\end{align*}
where we have set $\widetilde{\omega}:=\|f'\|_\infty+\|\mathcal{J}^\delta\|_\infty$.
Recalling \eqref{HOL:U2}, we thus obtain
\begin{align}
\frac{|\partial_tu(t,x_1)-\partial_tu(t,x_2)|}{|x_1-x_2|^\alpha}&\leq \widetilde{\omega}\left([u_0]_{C^{0,\alpha}(\overline{\Omega})}\!+\!2M_0\gamma_*^{-1}[J]_{\mathbb{B}_{1,\infty}^\alpha(\Omega;\delta)}\right)\!+\!2M_0[J]_{\mathbb{B}_{1,\infty}^\alpha(\Omega;\delta)}. \nonumber
\end{align}
The proof is thereby complete.
\end{proof}
\begin{remark}\label{rem-f}
If the datum $(J,f)$ satisfies \eqref{C2} and \eqref{C3} (with $f$ being defined only on $[0,1]$), then Proposition~\ref{PROP:EXIST} guarantees the existence of a unique solution, $u(t,x)$, to the Cauchy problem \eqref{CauchyPB} for an initial datum ranging in $[0,1]$. Indeed, it suffices to apply Proposition~\ref{PROP:EXIST} to $\widetilde{f}$, where $\widetilde{f}\in C^{0,1}\cap C^1(\R)$ is the extension of $f$ given by
\begin{align}
\widetilde{f}(s):= \left\{\begin{array}{cl} f'(0)\hspace{0.1em}s & \text{if }s<0, \vspace{3pt}\\ f(s) & \text{if }0\leq s\leq1, \vspace{3pt}\\ f'(1)(s-1) & \text{if }s>1. \end{array} \right. \label{extension:f}
\end{align}
The comparison principle Lemma~\ref{LE:COMPARISON} then guarantees that $0\leq u(t,x)\leq1$ so that \eqref{CauchyPB} (with $f$ being defined only on $[0,1]$) makes sense.
Moreover, if $(J,f)$ also satisfies \eqref{fJdelta-pos}, then $(J,\widetilde{f})$ satisfies \eqref{fjcond}. Indeed, this is because
$$\inf_{\O}\mathcal{J}^{\delta}-\max_\R\widetilde f'=\inf_{\O}\mathcal{J}^{\delta}-\max_{[0,1]}f'>0. $$
In particular, Proposition~\ref{PROP:PARABOLIK2} applies. Therefore, the unique solution to the Cauchy problem \eqref{CauchyPB} with $(J,f)$ satisfying \eqref{C2}, \eqref{C3}, \eqref{fJdelta-pos} and $J\in \mathbb{B}_{1,\infty}^\alpha(\Omega;\delta)$ for some $\alpha\in(0,1)$ enjoys parabolic type estimates.
\end{remark}
%\begin{remark}
%On another note, we point out that the unique solution to the Cauchy problem \eqref{CauchyPB} can be extended so that $u\in C^2([t_0,\infty),C(\overline{\Omega}))$. Of course, the same straightening also applies to the conclusion of Proposition~\ref{PROP:PARABOLIK2}.
%\end{remark}

\section{A priori bounds for entire solutions}\label{SE:APRIORI}

There are no a priori regularity estimates for entire solutions to \eqref{P}. In absence of specific assumptions on the datum $(f,J)$, entire solutions may even not be continuous at all. In this section, we provide some results which show that, under some circumstances, a parabolic-type estimate holds true.

%Namely, we will establish the following

\begin{lemma}[A priori estimates]\label{apriori:entire}
Assume \eqref{C1} and \eqref{C3}. Suppose that $f\in C^{0,1}\cap C^1(\R^N)$, that $J\in\mathbb{B}_{1,\infty}^\alpha(\Omega;\delta)$ for some $\alpha\in(0,1)$ and that \eqref{fjcond} holds.
Let $\phi\in C^{0,\alpha}(\R)$ and $c>0$.
Suppose that there exists an uniformly bounded measurable function $u:\R\times\Omega\to\R$ satisfying
\begin{numcases}{~}
\,\partial_tu=Lu+f(u)\, \mbox{ for a.e. }(t,x)\in\R\times\Omega, & \vspace{3pt} \label{P_ae}\\
\lim_{t\to-\infty}\underset{x\in\Omega}{\mathrm{ess}\,\mathrm{sup}}\,|u(t,x)-\phi(x_1+ct)|=0. & \label{unique2}
\end{numcases}
Then, there exists a constant $M>0$ (depending on $J$, $f'$, $\phi$, $\|u\|_{L^\infty(\R\times\Omega)}$, $\Omega$ and $\delta$) such that
\begin{align*}
\sup_{x\in\Omega}\hspace{0.1em}\|u(\cdot,x)\|_{C^{1,1}(\R)}+\sup_{t\in\R}\,\Big([u(t,\cdot)]_{C^{0,\alpha}(\overline{\Omega})}+[\partial_tu(t,\cdot)]_{C^{0,\alpha}(\overline{\Omega})}\Big)\leq M.
\end{align*}
\end{lemma}
\begin{remark}
As it was already observed by Berestycki, Hamel and Matano in the local case \cite{Berestycki2009d}, the condition \eqref{unique2} plays the role of an ``initial condition" at $-\infty$.
\end{remark}
\begin{proof}
Let $u:\R\times\Omega\to\R$ be an uniformly bounded solution of \eqref{P_ae} with \eqref{unique2}, and let $M_0>0$ be such that ${\mathrm{ess}\,\mathrm{sup}}_{(t,x)\in\R\times\Omega}|u(t,x)|\leq M_0$.
Using the equation \eqref{P_ae} satisfied by $u$, the fact that $f$ is $C^1$ and the boundedness assumption on $u$, it follows directly using the equation \eqref{P_ae} satisfied by $u$ and the one obtained by differentiating \eqref{P_ae} with respect to $t$, that
${\mathrm{ess}\,\mathrm{sup}}_{x\in\Omega}\,\|u(\cdot,x)\|_{C^{1,1}(\R)}\leq M_0(1+\omega+|f(0)|+\omega(\omega+|f(0)|)),$
where $\omega$ is as in \eqref{DEF:OMEGA}. Thus, up to redefine $u$ in a set of measure zero, we may assume that $u(\cdot,x)$ is a $C^{1,1}(\R)$ function for a.e. $x\in\Omega$. Then, $u$ is defined for all $x\in\Omega\setminus \mathcal{N}$ and for all $t\in\R$ where $\mathcal{N}\subset\Omega$ is a set of Lebesgue measure zero. Notice that
\begin{align}
\partial_tu(t,x)\text{ is well-defined whenever }u(t,x)\text{ is,} \label{WellDefNess}
\end{align}
as follows from the equation satisfied by $u$. Let $(t_n)_{n\geq0}\subset(-\infty,0)$ be a decreasing sequence with $t_n\to-\infty$ as $n\to\infty$. Let us now fix some $n\geq0$, let $t>t_n$, let $z,z'\in\Omega\setminus\mathcal{N}$ with $z\neq z'$ and define $\Psi_u(t):=u(t,z)-u(t,z')$. At this stage, using \eqref{fjcond} and recalling \eqref{WellDefNess}, we may apply the same trick as in Proposition \ref{PROP:PARABOLIK2}, to get
\begin{align}
|\Psi_u(t)|\leq  \left(\frac{|u(t_n,z)-u(t_n,z')|}{|z-z'|^\alpha}+2M_0\gamma_*^{-1}[J]_{\mathbb{B}_{1,\infty}^\alpha(\Omega;\delta)}\right)|z-z'|^\alpha, \label{HOL:U3}
\end{align}
for all $t>t_n$, where $\gamma_*:=\inf_{\Omega}\mathcal{J}^\delta-\max_{\R}f'>0$. Next, using \eqref{unique2}, we have
$$ \limsup_{n\to\infty}\,\frac{|u(t_n,z)-u(t_n,z')|}{|z-z'|^\alpha}=\limsup_{n\to\infty}\,\frac{|\phi(z_1+c\hspace{0.05em}t_n)-\phi(z_1'+c\hspace{0.05em}t_n)|}{|z-z'|^\alpha}\leq [\phi]_{C^{0,\alpha}(\R)}. $$
Therefore, letting $n\to\infty$ in \eqref{HOL:U3} and recalling that $\Psi_u(t)=u(t,z)-u(t,z')$, we obtain
$$ |u(t,z)-u(t,z')|\leq  \big([\phi]_{C^{0,\alpha}(\R)}+2M_0\gamma_*^{-1}[J]_{\mathbb{B}_{1,\infty}^\alpha(\Omega;\delta)}\big)|z-z'|^\alpha. $$
Hence, $(u(t,\cdot))_{t\in\R}$ is uniformly H\"older continuous.
The corresponding inequality for $\partial_tu$ follows from the same arguments as in the proof of Proposition \ref{PROP:PARABOLIK2}.
\end{proof}
\begin{remark}
If $(J,f)$ satisfy \eqref{C2}, \eqref{C3}, \eqref{C4} and \eqref{fJdelta-pos}, then Lemma~\ref{apriori:entire} implies that every solution to \eqref{P_ae} ranging in $[0,1]$ and satisfying \eqref{unique2} (where $(\phi,c)$ is as in \eqref{C4}) satisfy parabolic type estimates. To see this it suffices to argue as in Remark~\ref{rem-f} by extending $f$ linearly outside $[0,1]$ and to recall that $\phi\in C^2(\R)$ (remember Remark~\ref{rem:boot}).
\end{remark}

%%%%%%%%%%%%%%%%%%%%%

\section{Time before reaching the obstacle}\label{SE:BEFORE}

In this section we prove the existence of an entire solution to \eqref{P} that is monotone increasing with $t$ and which converges to a planar wave $\phi(x_1+ct)$ as $t\to-\infty$. In addition, we show that this limit condition at $-\infty$ is somehow comparable to an initial value problem in that it determines a unique bounded entire solution.

More precisely, we prove the following
\begin{theo}\label{timebefore}
Assume \eqref{C1}, \eqref{C2}, \eqref{C3}, \eqref{C4} and \eqref{fJdelta-pos}. Suppose that $J\in \mathbb{B}_{1,\infty}^\alpha(\Omega;\delta)$ for some $\alpha\in(0,1)$. Then, there exists an entire solution $u\in C^{2}(\R,C^{0,\alpha}(\overline{\Omega}))$ to \eqref{P} such that
\begin{align}
0<u(t,x)<1 \, \text{ and }\,\partial_tu(t,x)>0 \, {\mbox{ for all }}(t,x)\in\R\times\overline{\Omega}. \label{qualit-u}
\end{align}
Moreover,
\begin{align}
\lim_{t\to-\infty}|u(t,x)-\phi(x_1+ct)|=0 \, {\mbox{ uniformly in }}x\in\overline{\Omega}, \label{origintimes}
\end{align}
and \eqref{origintimes} determines a unique bounded entire solution to \eqref{P}.
\end{theo}
We will rely on a strategy already used in \cite{Berestycki2009d}. That is, we will construct a continuous subsolution $w^-$ and a continuous supersolution $w^+$ to \eqref{P} satisfying $w^-\leq w^+$ and we will use these functions to construct an entire solution to \eqref{P} satisfying the desired requirements.

\subsection{Preliminaries} Let us start by collecting some known facts on the travelling waves defined at \eqref{C4}. Let $(\phi,c)$ be the unique (up to shifts) \emph{increasing} solution of
\begin{align}\left\{
\begin{array}{c}
c\,\phi'=J_1\ast \phi-\phi+f(\phi) \text{ in }\R, \vspace{3pt}\\
\displaystyle\lim_{z\to +\infty}\phi(z)=1,\, \lim_{z\to -\infty}\phi(z)=0,
\end{array}
\right. \label{couplephic}
\end{align}
where $J_1$ is given by \eqref{J1}. In the remaining part of the paper we shall assume, for simplicity, that the function $\phi$ is normalized by
\begin{align}
\phi(0)=\theta. \label{NORMALISATION:PHI}
\end{align}
Notice that \eqref{couplephic} and \eqref{NORMALISATION:PHI} determine $\phi$ uniquely.

Let us now introduce two numbers which will play an important role in the sequel. We define $\lambda,\mu>0$ as the respective \emph{positive} solutions of
\begin{align}
\int_{\R}J_1(h)\hspace{0.05em}e^{\lambda h}\hspace{0.1em}\mathrm{d}h -1-c\lambda +f'(0)=0, \label{EQ:CHARAC}
\end{align}
and
\begin{align}
\int_{\R}J_1(h)\hspace{0.05em}e^{\mu h}\hspace{0.1em}\mathrm{d}h-1-c\mu +f'(1)=0. \label{EQ:CHARAC2}
\end{align}
Since $f$ and $J$ satisfy \eqref{C2} and \eqref{C3}, respectively, and since $J$ is compactly supported, the existence of such $\lambda$ and $\mu$ is a simple exercise (see e.g. \cite[Lemma 2.5]{Sun2011}). We will sometimes refer to \eqref{EQ:CHARAC} and \eqref{EQ:CHARAC2} as the \emph{characteristic equation} satisfied by $\lambda$ and $\mu$.

An important property of $\lambda$ and $\mu$ is that they ``encode" the asymptotic behaviour of $\phi$ and $\phi'$. More precisely:
\begin{lemma}\label{PHI:ASYM}
Assume \eqref{C2}, \eqref{C3} and \eqref{C4}. Let $(\phi,c)$ be a solution to \eqref{couplephic} and let $\lambda,\mu>0$ be the respective positive solutions to \eqref{EQ:CHARAC} and \eqref{EQ:CHARAC2}. Then, it holds that
\begin{align*}
A_0:=\lim_{z\to-\infty}\,e^{-\lambda z}\phi(z)=\lim_{z\to-\infty}\,\frac{e^{-\lambda z}\phi'(z)}{\lambda}\in(0,\infty), %\label{DE:A0}
\end{align*}
and
\begin{align*}
A_1:=\lim_{z\to\infty}\,e^{\mu z}(1-\phi(z))=\lim_{z\to\infty}\,\frac{e^{\mu z}\phi'(z)}{\mu}\in(0,\infty). %\label{DE:A1}
\end{align*}
Moreover,
\begin{align*}
\lim_{z\to-\infty}\,e^{-\lambda z}J_1\ast\phi(z)=A_0\int_{\R}J(h)\hspace{0.05em}e^{\lambda h}\hspace{0.1em}\mathrm{d}h.
\end{align*}
%and
%\begin{align*}
%\lim_{z\to-\infty}\,e^{-\lambda z}\Delta_{h}^2\phi(z)=2A_0\left(\frac{e^{\lambda h}+e^{-\lambda h}}{2}-1\right),
%\end{align*}
%for all $h\in \R$.
\end{lemma}
\begin{proof}
See e.g. Li \emph{et al.} \cite[Theorem 2.7]{Sun2011} for the proof of the  behaviour of $\phi$ and $\phi'$. %The asymptotic behavior of $\Delta_h^2\phi(z)$ is a trivial consequence of that of $\phi$.
To obtain the asymptotic of $J_1\ast \phi(z)$, it suffices to observe that
$$e^{-\lambda z}J_1\ast \phi(z)=\int_{\R}J_1(h)\hspace{0.05em}e^{\lambda h}\hspace{0.05em}e^{-\lambda (z+h)}\hspace{0.05em}\phi(z+h)\hspace{0.1em}\mathrm{d}h. $$
Now, since, for all $h\in\R$, we have $e^{-\lambda(z+h)}\phi(z+h)\to A_0$ as $z\to-\infty$ and since $J_1$ is compactly supported, the asymptotic behaviour of $J_1\ast\phi(z)$ follows by a simple application of the Lebesgue dominated convergence theorem.
\end{proof}
A rather direct consequence of Lemma \ref{PHI:ASYM} is that it ensures the existence of numbers $\alpha_0,\beta_0,\gamma_0,\delta_0>0$ such that
\begin{align}
\alpha_0 e^{\lambda z}\leq \phi(z)\leq \beta_0 e^{\lambda z}\text{ and }\gamma_0e^{\lambda z}\leq \phi'(z)\leq \delta_0 e^{\lambda z} {\mbox{ if }}z\leq0, \label{EST:PHI:NEG}
\end{align}
and numbers $\alpha_1,\beta_1,\gamma_1,\delta_1>0$ such that
\begin{align}
\alpha_1 e^{-\mu z}\leq 1-\phi(z)\leq \beta_1 e^{-\mu z}\text{ and }\gamma_1 e^{-\mu z}\leq \phi'(z)\leq \delta_1 e^{-\mu z} {\mbox{ if }}z>0. \label{EST:PHI:POS}
\end{align}

Finally, let us state a lemma that guarantees that $\phi$ is convex near $-\infty$.
\begin{lemma}\label{LE:CONVEXITE:PHI}
Let $(\phi,c)$ be a solution to \eqref{couplephic}. Then, there exists some $z_*<0$ such that
$$ \phi''(z)\geq \frac{\lambda}{8}\,\phi'(z) \text{ for any } z\leq z_*, $$
where $\lambda$ is the positive solution to \eqref{EQ:CHARAC}.
In particular,
$\phi$ is convex in $(-\infty,z_*]$ and we have
$$ \phi\left(\frac{z_1+z_2}{2}\right)\leq\frac{\phi(z_1)+\phi(z_2)}{2} \text{ for any }z_1,z_2\leq z_*. $$
\end{lemma}
\begin{proof}
Let us first observe that, since $f\in C^{1,1}([0,1])$, by a classical bootstrap argument we automatically get that $\phi\in C^2(\R)$ and that
\begin{align}
c\,\phi''(z)=J_1\ast\phi'(z)-\phi'(z)+\phi'(z)f'(\phi(z)) \text{ for any }z\in\R. \label{Cvx:1}
\end{align}
The assumption that $f\in C^{1,1}([0,1])$ further gives that $|f'(\phi(z))-f'(0)|\leq C\,\phi(z)$ for some $C>0$ (depending on $f$) and for any $z\in\R$. In particular, we have
\begin{align}
f'(\phi(z))\geq f'(0)-C\hspace{0.1em}\phi(z) \text{ for any } z\in\R. \label{Cvx:2}
\end{align}
By Lemma~\ref{PHI:ASYM}, we know that, for all $\eps>0$, there exists $R_\eps>0$ such that
\begin{align}
\lambda\hspace{0.1em}(A_0-\eps)\hspace{0.1em}e^{\lambda z}\leq \phi'(z)\leq \lambda\hspace{0.1em}(A_0+\eps)\hspace{0.1em}e^{\lambda z},  \label{A0epsilon}
\end{align}
for all $z\leq -R_\eps$.
%Thanks of the asymptotic behavior of $\phi$ and $\phi'$, for $\eps>0$ small then there exists $R_\eps>0$  such that for $z\leq -R_\eps$,
%where $(X_j)_{j\geq0}$ is as in Remark~\ref{RK:PHI}. Then, by \eqref{EQ:CHARAC}, \eqref{Cvx:1}, \eqref{Cvx:2}, Lemma~\ref{PHI:ASYM} and Remark~\ref{RK:PHI},
Hence, using \eqref{Cvx:1}, \eqref{Cvx:2} and \eqref{A0epsilon}, we obtain that
\begin{align*}
c\,\phi''(z)&\geq \lambda\hspace{0.05em}(A_0-\eps)\hspace{0.05em}e^{\lambda z}\int_{\R}J_1(h)\hspace{0.05em}e^{\lambda h}\hspace{0.1em}\mathrm{d}h+\lambda\hspace{0.05em}(A_0+\eps)\hspace{0.05em}e^{\lambda z}\left(f'(0)-C(A_0+\eps)\,e^{\lambda z}-1\right),
\end{align*}
where we have used that $J_1$ is even.
By rearranging the terms we may rewrite this as
\begin{align*}
c\,\phi''(z)&\geq \lambda\hspace{0.05em}A_0\hspace{0.05em} e^{\lambda z}\left(f'(0)-1+\int_{\R}J_1(h)\hspace{0.05em}e^{\lambda h}\hspace{0.1em}\mathrm{d}h\right)+\lambda\hspace{0.05em}\eps\hspace{0.05em}e^{\lambda z} \left(f'(0)-1-\int_{\R}J_1(h)\hspace{0.05em}e^{\lambda h}\hspace{0.1em}\mathrm{d}h\right) \\
&\qquad -C\hspace{0.05em}\lambda(A_0+\eps)^2e^{2\lambda z}.
\end{align*}
Using now the characteristic equation \eqref{EQ:CHARAC}, we find that
\begin{align*}
c\,\phi''(z)&\geq c\hspace{0.05em}\lambda^2\hspace{0.05em}A_0\hspace{0.05em} e^{\lambda z}+\lambda\hspace{0.05em}\eps\hspace{0.05em}e^{\lambda z} \left(c\hspace{0.05em}\lambda-2\int_{\R}J_1(h)\hspace{0.05em}e^{\lambda h}\hspace{0.1em}\mathrm{d}h\right)-C\hspace{0.05em}\lambda(A_0+\eps)^2e^{2\lambda z}.
\end{align*}
Choosing $\eps$ small enough, say $0<\eps<\eps_0$, where
$$ \eps_0:=A_0\,\min\left\{1, \frac{c\hspace{0.05em}\lambda}{2}\left|c\hspace{0.05em}\lambda-2\int_{\R}J_1(h)\hspace{0.05em}e^{\lambda h}\hspace{0.1em}\mathrm{d}h\right|^{-1}\right\}, $$
we obtain that
\begin{align*}
c\,\phi''(z)&\ge \frac{c\hspace{0.05em}\lambda^2}{2}\hspace{0.05em}A_0\hspace{0.05em} e^{\lambda z}-C\hspace{0.05em}\lambda(A_0+\eps)^2e^{2\lambda z}\geq \lambda\hspace{0.05em}(A_0+\eps)\hspace{0.05em} e^{\lambda z}\Big(\frac{c\hspace{0.05em}\lambda}{4}-2\hspace{0.05em}CA_0\hspace{0.05em}e^{\lambda z}\Big), %\label{quasi-cvx}
\end{align*}
for all $z\leq-R_\eps$.
Up to choose $R_\eps>0$ larger, we may assume that $2\hspace{0.05em}CA_0e^{\lambda z}\leq c\lambda/8$ for all $z\leq-R_\eps$.
Therefore, recalling \eqref{A0epsilon}, we finally obtain that
$$ \phi''(z)\ge \frac{\lambda^2}{8}\hspace{0.1em}(A_0+\eps)\hspace{0.1em}e^{\lambda z}\geq \frac{\lambda}{8}\,\phi'(z) \text{ for any }z\leq -R_\eps, $$
which thereby completes the proof.
\end{proof}
\begin{remark}
Observe that the same arguments also yield the existence of some $z^*>0$ such that $\phi$ is concave in $[z^*,\infty)$.
\end{remark}

\subsection{Construction of sub- and supersolutions} Let us introduce some necessary notations. Let $k>0$ be a positive number to be fixed later on. We set
\begin{align}
\xi(t):=\frac{1}{\lambda}\log\left(\frac{1}{1-c^{-1}ke^{\lambda ct}}\right) \mbox{ for }t\in(-\infty,T), \label{xi(t)}
\end{align}
where $c$ is the speed of the travelling wave $\phi$, $\lambda$ is given by \eqref{EQ:CHARAC} and
\begin{align}
T:=\frac{1}{\lambda c}\log\left(\frac{c}{k}\right) . \label{T:xi}
\end{align}
To shorten our notations it will be convenient to set
\begin{align}
M^{\pm}(t):=c\hspace{0.1em}t\pm\xi(t). \label{M(t)}
\end{align}
Readily, we observe that $\xi(-\infty)=0$ and $\dot{\xi}(t)=k\hspace{0.05em}e^{\lambda M^+(t)}$. We now define two functions, $w^+$ and $w^-$, in $\R^N\times(-\infty,T_1]$ for some $T_1\in(-\infty,T)$, by
\begin{align}
w^+(t,x)=\left\{
\begin{array}{l l}
\phi(x_1+M^+(t))+\phi(-x_1+M^+(t)) & (x_1\geq0), \vspace{3pt}\\
2\phi(M^+(t)) & (x_1<0),
\end{array}
\right. \label{wplus}
\end{align}
and
\begin{align}
w^-(t,x)=\left\{
\begin{array}{l l}
\phi(x_1+M^-(t))-\phi(-x_1+M^-(t)) & (x_1\geq0), \vspace{3pt}\\
0 & (x_1<0).
\end{array}
\right. \label{wmoins}
\end{align}
Notice that, if $-\infty<T_1\ll T$, then $w^+$ and $w^-$ satisfy
\begin{align*}
0\leq w^-< w^+\leq 1 \mbox{ for any }(t,x)\in(-\infty,T_1]\times\R^N, %\label{ordonne}
\end{align*}
the last inequality being a consequence of the weak maximum principle \cite[Lemma 4.1]{Brasseur2019}.

We now claim the following
\begin{lemma}\label{subsuper}
Let $R_J>0$ be such that $\mathrm{supp}(J)\subset[0,R_J]$. Assume \eqref{C1}, \eqref{C2}, \eqref{C3}, \eqref{C4} and suppose that $K\subset\R^N$ is such that
\begin{align}
K\subset\left\{x_1<-R_J\right\}.    \label{loinK}
\end{align}
Then, for $k>0$ sufficiently large, $w^+$ and $w^-$ are, respectively, a supersolution and a subsolution to \eqref{P} in the time range $t\in(-\infty,T_1]$ for some $T_1\in(-\infty,T)$.
\end{lemma}
\begin{remark}
Just as in the local case, the boundedness assumption on $K$ in \eqref{C1} can be relaxed since one only need \eqref{loinK} to hold. In particular, this still holds when $K$ is, say, an infinite wall with one or several holes pierced in it.
\end{remark}
\begin{proof}

For the sake of convenience, we introduce the operator $\mathcal{P}$ given by
\begin{align*}
\opp{w}(t,x):=\partial_tw(t,x)-Lw(t,x)-f(w(t,x)). %\label{EQ:OP:L}
\end{align*}
Notice that if $T_1\in(-\infty,T)$ is sufficiently negative, then $M^{\pm}(t)<0$ for any $t\in(-\infty,T_1]$.

\vskip 0.3cm
\noindent\emph{Step 1. Supersolution}
\smallskip

We aim to prove that the function $w^+$ given by \eqref{wplus} is a \emph{supersolution} to \eqref{P}. More precisely, we want to show that
\begin{align*}
\opp{w^+}(t,x)\geq0 \,{\mbox{ for any }}(t,x)\in(-\infty,T_1]\times\Omega, %\label{GOALsuper}
\end{align*}
and some $T_1\in(-\infty,T]$. We consider the cases $x\in\{x_1\geq0\}$ and $x\in \{x_1<0\}$ separately. \\

\noindent \underline{\textsc{Case} $x_1\geq0$}. A straightforward calculation gives
\begin{equation}
\partial_tw^+(t,x)-f(w^+(t,x))=(c+\dot{\xi}(t))(\phi'(z_+)+\phi'(z_-))-f\big(\phi(z_+)+\phi(z_-)\big), \label{plug1}
\end{equation}
where $z_+:=x_1+M^+(t)$ and $z_-:=-x_1+M^+(t)$. Furthermore, using \eqref{loinK} and the fact that $\mathrm{supp}(J)\subset[0,R_J]$, we have
$$
Lw^+(t,x)= \int_{\Omega}J(\delta(x,y))(w^+(t,y)-w^+(t,x))\hspace{0.1em}\mathrm{d}y = \int_{\R^N}J(|x-y|)(w^+(t,y)-w^+(t,x))\hspace{0.1em}\mathrm{d}y. $$
Consequently,
\begin{multline*}
-Lw^+(t,x)= -\int_{\R^N}J(|x-y|)\big(\phi(y_1+M^+(t))-\phi(x_1+M^+(t))\big)\mathrm{d}y  \\
-\int_{\R^N}J(|x-y|)\big(\phi(-y_1+M^+(t))-\phi(-x_1+M^+(t))\big)\mathrm{d}y+I_0(t,x),
\end{multline*}
where we have set
\begin{align*}
I_0(t,x)&:=%-\int_{\{y_1<0\}}J(|x-y|)\big(2\phi(M^+(t))-\phi(y_1+M^+(t))-\phi(-y_1+M^+(t))\big)\mathrm{d}y\\
-\int_{\{y_1<0\}}J(|x-y|)\hspace{0.05em}\Delta_{y_1}^2\phi(M^+(t))\hspace{0.1em}\mathrm{d}y,
\end{align*}
where the operator $\Delta_{y_1}^2$ is as defined in Section~\ref{SE:NOTATIONS}.
Notice that, since $x_1\geq0$ and since $\mathrm{supp}(J)\subset[0,R_J]$, the integral over $\{y_1<0\}$ can be replaced by an integral over $\{-R_J\leq y_1<0\}$. But given that $M^+(t)\to-\infty$ as $t\to-\infty$ and that $\phi$ is convex near $-\infty$ (by Lemma~\ref{LE:CONVEXITE:PHI}), we have $\Delta_{y_1}^2\phi(M^+(t))\le 0$ for all $t\le T_1$ and all $-R_J\le y_1\le 0$ (up to take $T_1$ sufficiently negative). Thus, we have that
$$
I_0(t,x)\geq0.
$$
Hence, using the equation satisfied by $\phi$, we obtain
\begin{align*}
-Lw^+(t,x)&\geq-c\hspace{0.1em}(\phi'(z_+)+\phi'(z_-))+f(\phi(z_+))+f(\phi(z_-)).
\end{align*}
Plugging this in \eqref{plug1}, we get
\begin{equation}
\opp{w^+}(t,x)\ge k\hspace{0.1em}e^{\lambda M^+(t)}\left(\phi'(z_+)+\phi'(z_-)\right)+f(\phi(z_+))+f(\phi(z_-))-f(\phi(z_+)+\phi(z_{-})).\label{plug2}
\end{equation}
Using the fact that $f$ is of class $C^{1,1}$, we may find a constant $\varrho>0$ such that
\begin{equation}
|f(a)+f(b)-f(a+b)|\leq \varrho\, ab. \label{hypcf}
\end{equation}
Hence, \eqref{plug2} becomes
\begin{equation}
\opp{w^+}(t,x)\ge k\hspace{0.05em}e^{\lambda M^+(t)}\left(\phi'(z_+)+\phi'(z_-)\right)-\varrho\,\phi(z_+)\hspace{0.1em}\phi(z_{-}).\label{plug3}
\end{equation}
Let us now treat the cases $x\in \{x_1> -M^+(t)\}$ and $x\in \{0\le x_1\le -M^+(t)\}$ separately.
In the latter case, we have $z_-\le z_+\le 0$. Hence, using \eqref{EST:PHI:NEG}, \eqref{plug3} and the fact that $\phi'>0$, we get
\begin{equation*}
\opp{w^+}(t,x)\ge \gamma_0\hspace{0.05em}k\hspace{0.1em}e^{\lambda x_1+2\lambda M^+(t)}-\varrho\hspace{0.05em}\beta_0^2\hspace{0.1em}e^{2\lambda M^+(t)}\ge e^{2\lambda M^+(t)}\left(\gamma_0\hspace{0.05em} k\hspace{0.1em}e^{\lambda x_1}-\varrho\hspace{0.05em}\beta_0^2\right).
\end{equation*}
Thus, we have $\opp{w^+}(t,x)\geq0$ for all $x\in \{0\le x_1\le -M^+(t)\}$ as soon as $k$ is chosen so that
\begin{align}
k\geq\frac{\varrho\beta_0^2}{\gamma_0}. \label{stepeins}
\end{align}
Let us now treat the case $x\in \{x_1>-M^+(t)\}$. In this case, we have $z_-<0<z_+$ and, again, we treat two situations independently, depending on whether $\lambda<\mu$ or $\lambda\ge \mu$.

Assume first that $\lambda\ge \mu$. Then, using \eqref{EST:PHI:NEG}, \eqref{EST:PHI:POS}, \eqref{plug3} and the fact that $\phi'>0$ and $\phi\le 1$, we deduce that
%in this situation, by \eqref{EST:PHI:NEG} and \eqref{EST:PHI:POS}, and since $\phi'>0$ and $\phi\le 1$, we deduce from \eqref{plug3} that
\begin{align*}
\opp{w^+}(t,x)&\geq k\hspace{0.05em}\gamma_1\hspace{0.1em}e^{\lambda M^+(t)}e^{-\mu z_+}-\varrho\hspace{0.05em}\beta_0\hspace{0.05em}e^{\lambda z_-}\\
& \geq \,e^{\lambda M^+(t)}k\hspace{0.05em}\gamma_1\hspace{0.1em}e^{-\lambda (x_1+M^+(t))}-\varrho\hspace{0.05em}\beta_0\hspace{0.1em}e^{-\lambda x_1 +\lambda M^+(t)}\\
& \geq \,e^{-\lambda x_1}\big(k\hspace{0.05em}\gamma_1-\varrho\hspace{0.05em}\beta_0\hspace{0.1em}e^{\lambda M^+(t)}\big).
\end{align*}
Since $M^+(t)\le 0$ for all $t\le T_1$, we then have $\opp{w^+}(t,x)\geq0$ as soon as $k$ is chosen so that
\begin{align}
k\geq\frac{\varrho\beta_0}{\gamma_1}. \label{lambdageqmu}
\end{align}
The remaining case $\lambda<\mu$ is treated using the same trick as in \cite{Berestycki2009d}. Namely, we notice that, if $\lambda<\mu$, then, thanks to the characteristic equations \eqref{EQ:CHARAC} and \eqref{EQ:CHARAC2}, we must necessarily have $f'(0)>f'(1)$ and
$$ f(a)+f(b)-f(a+b)=(f'(0)-f'(1))\hspace{0.1em}b+\mathcal{O}(b^2)+\mathcal{O}\big(|b(1-a)|\big), $$
for $a$ and $b$ close to $1$ and $0$, respectively. In particular, if $z_+\gg 1$ and $z_-\ll -1$, then
$$  f(\phi(z_+))+f(\phi(z_-))-f(\phi(z_+)+\phi(z_-))\geq 0.  $$
Now, by definition of $z_+$ and $z_-$, there is some $L_0>0$ such that the above inequality holds true for all $t\le T_1$ and all $x_1\in[-M^+(t)+L_0,\infty)$ (up to take $T_1$ sufficiently negative).
Consequently, using \eqref{plug2} and the fact that $\xi'$ and $\phi'$ are positive quantities, we infer that $\opp{w^+}(t,x)\geq0$ for all $t\leq T_1$ and all $x\in \{x_1\ge -M^+(t)+L_0\}$.

Lastly, let us treat the case $x\in \{-M^+(t)< x_1 < -M^+(t)+L_0\}$. Using again \eqref{EST:PHI:NEG}, \eqref{EST:PHI:POS}, \eqref{plug3} and the fact that $\phi'>0$ and $\phi\le 1$, we obtain that
\begin{align*}
\opp{w^+}(t,x)&\geq k\hspace{0.05em}\gamma_1\hspace{0.1em}e^{\lambda M^+(t)}e^{-\mu z_+}-\varrho\hspace{0.05em}\beta_0\hspace{0.1em}e^{\lambda z_-}\\
%&\ge  k\gamma_1e^{\lambda M^+(t)}e^{-\mu L_0}-\varrho\beta_0e^{-\lambda x_1+\lambda M^+(t)}\\
&\ge e^{\lambda M^+(t)}\big(k\hspace{0.05em}\gamma_1\hspace{0.1em}e^{-\mu L_0}-\varrho\hspace{0.05em}\beta_0\hspace{0.1em}e^{-\lambda x_1}\big).
\end{align*}
Therefore, we have $\opp{w^+}(t,x)\geq0$ as soon as $k$ is chosen so that
\begin{align}
k\geq\frac{\varrho\beta_0}{\gamma_1}\,e^{\mu L_0}. \label{stepdrei}
\end{align}
Finally, by \eqref{stepeins}, \eqref{lambdageqmu}, and \eqref{stepdrei}, we have
\begin{align*}
\opp{w^+}(t,x)\geq0  \,{\mbox{ whenever }}k\geq\max\left\{\frac{\varrho\beta_0^2}{\gamma_0}, \frac{\varrho\beta_0}{\gamma_1}\,e^{\mu L_0} \,\right\},
\end{align*}
in the set $(t,x)\in(-\infty,T_1]\times\{x_1\ge 0\}$, provided $T_1$ is sufficiently negative. \\

\noindent \underline{\textsc{Case} $x_1<0$}.
Readily, we see that
\begin{align*}
\partial_tw^+(t,x)-f(w^+(t,x))&=2(c+\dot{\xi}(t))\hspace{0.05em}\phi'(M^+(t))-f\left(2\phi(M^+(t))\right). %\label{AA1}
\end{align*}
Now, since $\phi(0)=\theta$ and $\phi'>0$, we have $f(2\phi(M^+(t)))\le 0$ as soon as $\phi(M^+(t))\le \theta/2$. Thus, since $M^+(t)$ is increasing, since $\lim_{t\to -\infty }M(t)= -\infty$ and since $\lim_{z\to-\infty}\phi(z)=0 $, up to decrease further $T_1$, we can assume that $\phi(M^+(t))\le \theta/2$ for all $t\le T_1$. Hence, we have
\begin{equation}
\partial_tw^+(t,x)-f(w^+(t,x))\ge 2\hspace{0.05em}(c+\dot{\xi}(t))\hspace{0.05em}\phi'(M^+(t))\geq0. \label{AA1'}
\end{equation}
Let us now estimate $Lw^+(t,x)$. For it, let us denote by $H^+$ and $H^-$ the half-spaces given by
$$ H^+:=\{x\in\R^N; x_1>0\} \text{ and } H^-:=\{x\in\R^N; x_1\leq0\}, $$
respectively.
%Let $H^+$ and $H^-$ denotes respectively the following half-space  $H^+:=\{x\in \R^N\,|\, x_1>0\}$ respectively $H^-:=\{x\in \R^N\, |\, x_1\le 0\}$.
By definition of $w^+(t,x)$ we have
\begin{align*}
Lw^+(t,x)&=\int_{\O}J(\delta(x,y))(w^+(t,y) - w^+(t,x))\hspace{0.1em}\mathrm{d}y\\
&= \int_{\O\cap H^-}J(\delta(x,y))(2\phi(M^+(t)) - 2\phi(M^+(t)))\hspace{0.1em}\mathrm{d}y \\
&\qquad\quad+ \int_{\O\cap H^+}J(\delta(x,y))(w^+(t,y) - 2\phi(M^+(t)))\hspace{0.1em}\mathrm{d}y\\
&=\int_{\O\cap H^+}J(\delta(x,y))(w^+(t,y) - 2\phi(M^+(t)))\hspace{0.1em}\mathrm{d}y.
\end{align*}
Now since $K\subset\left\{x_1<-R_J\right\}$, we have $\O\cap H^+=H^+\setminus K=H^+$, and so
\begin{equation}\label{AA2}
Lw^+(t,x)=\int_{H^+}J(\delta(x,y))(w^+(t,y) - 2\phi(M^+(t)))\hspace{0.1em}\mathrm{d}y.
\end{equation}
Observe that $\delta(x,y)\ge R_J$ for all $x\in H_{R_J}^-:=\{x_1<-R_J\}$ and all $y\in H^+$. But since $\mathrm{supp}(J)\subset[0,R_J]$, we then have that $J(\delta(x,y))=0$ for all $(x,y)\in H_{R_J}^-\setminus K\times H^+$.
%Thus, $J(\delta(x,y))=0$ (since $\mathrm{supp}(J)\subset[0,R_J]$). Therefore, in  $\O\cap H_{R_J}$ from \eqref{AA2} we have
Therefore, recalling \eqref{AA2}, we have $Lw^+(t,x)=0$ for all $(t,x)\in (-\infty,T_1]\times H_{R_J}^-\setminus K$.
%$$Lw^+(x,t)= \int_{H^+}J(\delta(x,y))(w^+(y,t) - 2\phi(M^+(t)))\hspace{0.1em}\mathrm{d}y=0.$$
Combining this with \eqref{AA1'}, we obtain that $\opp{w^+}(t,x)\ge 0$ for all $(t,x)\in (-\infty,T_1]\times (H_{R_J}^-\setminus K)$.

Let us now treat the case $x\in \{-R_J\leq x_1<0\}$. For it, we observe that $\delta(x,y)=|x-y|$ for all $(x,y)\in [-R_J,0)\times H^+$. Consequently,  \eqref{AA2} rewrites
\begin{align*}
Lw^+(t,x)&=\int_{H^+}J(|x-y|)\big(\phi(y_1+M^+(t))+\phi(-y_1+M^+(t))-2\phi(M^+(t))\big)\mathrm{d}y  \\
&= \int_{0}^{+\infty}J_1(x_1-y_1)\hspace{0.1em}\Delta_{y_1}^2\phi(M^+(t))\hspace{0.1em}\mathrm{d}y_1.
\end{align*}
Since $\mathrm{supp}(J_1) \subset [0,R_J]$ and $-R_J\le x_1<0$, the above equality may be rewritten as
$$Lw^+(t,x)= \int_{0}^{R_J}J_1(x_1-y_1)\hspace{0.1em}\Delta_{y_1}^2\phi(M^+(t))\hspace{0.1em}\mathrm{d}y_1. $$
But given that $M^+(t)\to-\infty$ as $t\to-\infty$ and that $\phi$ is convex near $-\infty$ (by Lemma~\ref{LE:CONVEXITE:PHI}), we have $\Delta_{y_1}^2\phi(M^+(t))\le 0$ for all $t\le T_1$ and all $0\le y_1\le R_J$ (up to take $T_1$ sufficiently negative). Thus, we have
$$Lw^+(t,x)= \int_{0}^{R_J}J_1(x_1-y_1)\hspace{0.1em}\Delta_{y_1}^2\phi(M^+(t))\hspace{0.1em}\mathrm{d}y_1\le 0, $$
for all $t\leq T_1$ and all $x\in \{-R_J\le x_1<0\}$. Hence, recalling \eqref{AA1'}, we obtain that
$$ \opp{w^+}(t,x)\geq 2(c+\dot{\xi}(t))\phi'(M^+(t))\ge 0, $$
for all $t\leq T_1$ and all $x\in\{-R_J\le x_1<0\}$.
Summing up, we have shown that, for every $(t,x)\in (-\infty,T_1]\times\O$ and $T_1\in(-\infty,T)$ sufficiently negative, it holds that
\begin{align*}
\opp{w^+}(t,x)\geq0 \, {\mbox{ whenever }}k\geq\max\left\{\frac{\varrho\beta_0^2}{\gamma_0}, \frac{\varrho\beta_0}{\gamma_1}\,e^{\mu L_0}\,\right\}.
\end{align*}
This proves that $w^+$ is indeed a supersolution to \eqref{P}.
\vskip 0.3cm

\noindent \emph{Step 2. Subsolution}
\smallskip

We will follow the same strategy as above. We aim to prove that the function $w^-$ given by \eqref{wmoins} is a \emph{subsolution} to \eqref{P}. More precisely, we want to show that
\begin{align}
\opp{w^-}(t,x)\leq0  \,{\mbox{ for any }}(t,x)\in (-\infty,T_1]\times\Omega, \nonumber
\end{align}
and some $T_1\in(-\infty,T)$. A direct calculation gives
\begin{align}
\partial_tw^-(t,x)\! -\!f(w^-(t,x))\!&=\!\left\{
\begin{array}{l l}
\!(c\!-\!\dot{\xi}(t))(\phi'(\zeta_+)\!-\!\phi'(\zeta_-))\!-\!f\big(\phi(\zeta_+)\!-\!\phi(\zeta_-)\big) & \!(x_1\geq0), \vspace{3pt}\\
\!0 & \!(x_1<0).
\end{array}
\right. \label{STEPUN}
\end{align}
where $\zeta_+=x_1+M^-(t)$, $\zeta_-=-x_1+M^-(t)$. Let us now estimate $Lw^-(t,x)$. \\

\noindent \underline{\textsc{Case} $x_1<0$}. This case is straightforward. Indeed, as above, we can check that
\begin{align}
Lw^-(t,x)&=\int_{H^+}J(\delta(x,y))\big(\phi(y_1+M^-(t))-\phi(-y_1+M^-(t))\big)\mathrm{d}y. \nonumber
\end{align}
But, since $\phi$ is increasing, the integrand above is nonnegative, and so $Lw^-(t,x)\ge 0$.
Hence, recalling \eqref{STEPUN}, we find that $\opp{w^-}(t,x)\leq 0$ for any $x\in \{x_1<0\}$. \\
%we have $Lw^-(x,t)\ge 0$ and so, by \eqref{STEPUN}, we find that $\mathcal{L}w^-(x,t)\leq 0$ for any $x_1<0$. \\

\noindent \underline{\textsc{Case} $x_1\geq0$}. Observe that, since $\mathrm{supp}(J)\subset [0,R_J]$ and since $K\subset \{x_1\le -R_J\}$, we have
$$Lw^-(t,x)=\int_{\R^N}J(|x-y|)(w^-(t,y)-w^-(t,x)))\hspace{0.1em}\mathrm{d}y,$$
for all $x\in \{x_1\ge 0\}$.
Using the definition of $w^-$, we have
\begin{multline*}
Lw^-(t,x)= \int_{\R^N}J(|x-y|)\big(\phi(y_1+M^-(t))-\phi(x_1+M^-(t))\big)\mathrm{d}y  \\
-\int_{\R^N}J(|x-y|)\big(\phi(-y_1+M^-(t))-\phi(-x_1+M^-(t))\big)\mathrm{d}y -I_1(t,x),
\end{multline*}
where we have set
$$ I_1(t,x):=\int_{\{-R_J\leq y_1\leq 0\}}J(|x-y|)\big(\phi(y_1+M^-(t))-\phi(-y_1+M^-(t))\big)\mathrm{d}y. $$
Since $y_1+M^-(t)\le -y_1+M^-(t)$ for all $ -R_J\leq y_1\leq 0$ and since $\phi$ is increasing, it holds that $-I_1(t,x)\ge 0$. Therefore, by using \eqref{couplephic}, we get
$$
Lw^-(t,x)\geq  c\hspace{0.1em}(\phi'(\zeta_+)-\phi'(\zeta_-))-(f(\phi(\zeta_+))-f(\phi(\zeta_-))).
$$
Recalling \eqref{STEPUN}, we obtain
\begin{equation}
\opp{w^-}(t,x)\leq -\dot{\xi}(t)(\phi'(\zeta_+)-\phi'(\zeta_-))+f(\phi(\zeta_+))-f(\phi(\zeta_-))-f(\phi(\zeta_+)-\phi(\zeta_-)). \label{jc-subsol}
\end{equation}
Let us suppose that $x\in \{x_1\ge -M^-(t)\}$. Then, using \eqref{hypcf} and \eqref{jc-subsol}, we have
\begin{equation}
\opp{w^-}(t,x)\leq -\dot{\xi}(t)(\phi'(\zeta_+)-\phi'(\zeta_-))+\varrho\,\phi(\zeta_-)(\phi(\zeta_+)-\phi(\zeta_-)).  \label{zpm}
\end{equation}
We consider the cases $\lambda\geq\mu$ and $\lambda<\mu$ separately. Let us suppose that $\lambda\geq\mu$. Then, since $\zeta_-\le 0\le \zeta_+$, using \eqref{EST:PHI:NEG} and \eqref{EST:PHI:POS}, we deduce from \eqref{zpm} that
\begin{align}
\opp{w^-}(t,x)&\leq -k\hspace{0.1em}e^{\lambda M^+(t)}\big(\gamma_0\hspace{0.1em}e^{-\mu(x_1+M^-(t))}-\delta_0\hspace{0.1em}e^{\lambda(-x_1+M^-(t))}\big)+\varrho\beta_0\hspace{0.1em}e^{\lambda(-x_1+M^-(t))} \nonumber \\
&= -e^{\lambda (-x_1+M^+(t))}\big(k\gamma_0\hspace{0.1em}e^{-\mu M^-(t)+(\lambda-\mu)x_1}-\delta_0\hspace{0.1em}e^{\lambda M^-(t)}-\varrho\beta_0\hspace{0.1em}e^{-2\lambda\xi(t)}\big) \label{third} \\
&\leq -e^{\lambda (-x_1+M^+(t))}\big(k\gamma_0-\delta_0-\varrho\beta_0\big), \nonumber
\end{align}
since $\lambda,\mu>0$, $M^{-}(t)\le 0$ and $\xi(t)\ge 0$ for all $t \le T_1$.
Whence, $\opp{w^-}(t,x)\leq0$ for $x\in \{x_1\ge -M^-(t)\}$ as soon as $k$ is chosen so that
\begin{equation*}
k\geq \frac{\delta_0+\varrho \beta_0}{\gamma_0}. %\label{cond-ksub-1}
\end{equation*}
Let us now consider the case $\lambda<\mu$. Arguing as in the Step 1, i.e. using the characteristic equations \eqref{EQ:CHARAC} and \eqref{EQ:CHARAC2}, we deduce that $f'(0)>f'(1)$ and that
$$ f(a+b)-f(a)-f(b)=-(f'(0)-f'(1))\,b+\mathcal{O}(b^2)+\mathcal{O}\big(|b(1-a)|\big), $$
for $a$ and $b$ close to $1$ and $0$, respectively. %$a\approx 1$ and $b\approx 0$
Hence, we have
\begin{align*}
f(\phi(\zeta_+))-f(\phi(\zeta_-))-&f\big(\phi(\zeta_+)-\phi(\zeta_-)\big) \\
&=-(f'(0)-f'(1))\hspace{0.1em}\phi(\zeta_-)+\mathcal{O}\big(\phi^2(\zeta_-)\big)+\mathcal{O}\big(\phi(\zeta_-)(1-\phi(\zeta_+))\big),
\end{align*}
provided $\zeta_-\ll-1$ and $\zeta_+\gg1$.  Thanks to the definition of $\zeta_{\pm}$ and since $\phi$ satisfies \eqref{couplephic}, we can then find a constant $L_1>0$ such that
\begin{equation*}
f(\phi(\zeta_+))-f(\phi(\zeta_-))-f\big(\phi(\zeta_+)-\phi(\zeta_-)\big)\le -\kappa\hspace{0.1em}\phi(\zeta_-), %\label{kappaK}
\end{equation*}
for all $x\in \{x_1\ge -M^-(t)+L_1\}$, where we have set $\kappa:=(f'(0)-f'(1))/2$.
This, together with \eqref{zpm} and \eqref{EST:PHI:NEG}, implies that
$$
\opp{w^-}(t,x)\leq e^{\lambda \zeta_-}(k\hspace{0.05em}\delta_0\hspace{0.1em}e^{\lambda M^+(t)}-\kappa\hspace{0.1em}\alpha_0).
$$
It follows that $\opp{w^-}(t,x)\leq0$ in the set $\{x_1\ge -M^-(t)+L_1\}$ provided that $T_1\in(-\infty,T]$ is chosen sufficiently negative so that
$$ k\hspace{0.05em}\delta_0\hspace{0.1em}e^{\lambda M^+(t)}\leq\kappa\hspace{0.1em}\alpha_0 \,{\mbox{ for any }}-\infty<t\leq T_1. $$
Now, suppose that $x\in\{-M^-(t)\le x_1< -M^-(t)+L_1 \}$. Then, it follows from \eqref{third} that
\begin{align*}
\opp{w^-}(t,x)&\leq -e^{\lambda (-x_1+M^+(t))}\big(k\gamma_0\hspace{0.1em}e^{-\mu M^-(t)-(\mu-\lambda)L_1}-\delta_0\hspace{0.1em}e^{\lambda M^-(t)}-\varrho\beta_0\hspace{0.1em}e^{-2\lambda\xi(t)}\big)\\
&\leq -e^{\lambda (-x_1+M^+(t))}\big(k\gamma_0\hspace{0.1em}e^{-\mu M^-(t)-(\mu-\lambda)L_1}-\delta_0-\varrho\beta_0\big).
\end{align*}
Thus, $\opp{w^-}(t,x)\leq0$ in the set $x_1\in\{-M^-(t)\le x_1< -M^-(t)+L_1\}$ provided that $T_1\in(-\infty,T]$ is chosen sufficiently negative so that
\begin{align*}
\gamma_0\hspace{0.05em}k\hspace{0.05em} e^{-\mu M^-(t)-(\mu-\lambda)L_0}-\delta_0-\varrho\hspace{0.05em}\beta_0\geq0 \,{\mbox{ for }}-\infty<t\leq T_1.
\end{align*}
Next, suppose that $x\in \{x_1< -M^-(t)\}$. Then, $\zeta_-\leq \zeta_+\leq0$ and by \eqref{EST:PHI:NEG}, \eqref{EST:PHI:POS} and \eqref{zpm} we have that
\begin{align}
\opp{w^-}(t,x)&\leq -k\hspace{0.1em}e^{\lambda M^+(t)}\big(\gamma_0\hspace{0.1em}e^{\lambda \zeta_+}-\delta_0\hspace{0.1em}e^{\lambda \zeta_-}\big)+\varrho \beta_0^2\hspace{0.1em} e^{\lambda\zeta_-}e^{\lambda\zeta_+} \nonumber \\
&\leq -k\hspace{0.1em}e^{\lambda M^+(t)}\big(\gamma_0\hspace{0.1em}e^{\lambda(x_1+M^-(t))}-\delta_0\hspace{0.1em}e^{\lambda(-x_1+M^{-}(t)}\big)+\varrho \beta_0^2 e^{\lambda(-x_1+M^{-}(t))}e^{\lambda(x_1+M^{-}(t))} \nonumber \\
& \le e^{\lambda (M^+(t)+M^{-}(t))}\big(-k\big(\gamma_0\hspace{0.1em}e^{\lambda x_1}-\delta_0\hspace{0.1em}e^{-\lambda x_1}\big)+\varrho \beta_0^2\hspace{0.1em}e^{2\lambda M^{-}(t)}\big) \nonumber \\
& \le e^{2\lambda ct}\left(-k\big(\gamma_0\hspace{0.1em}e^{\lambda x_1}-\delta_0\hspace{0.1em}e^{-\lambda x_1}\big) +\varrho \beta_0^2\hspace{0.1em}e^{-2\lambda \xi(t)}\right) \nonumber \\
& \le e^{2\lambda ct}\left(-k\gamma_0\hspace{0.1em}e^{\lambda x_1}+k\delta_0 +\varrho \beta_0^2\right). \label{est:m:m:m}
\end{align}
Let $R_0>0$ be the number given by
\begin{equation*}
R_0:=\frac{1}{\lambda}\log\left(\frac{\delta_0}{\gamma_0}+2\right).
\end{equation*}
Choosing $k$ large enough so that $k\geq\varrho\hspace{0.05em}\beta_0^2/\gamma_0$, we have
\begin{equation}
\label{defR0}
-k\hspace{0.05em}\gamma_0\hspace{0.1em}e^{\lambda R_0}+k\hspace{0.05em}\delta_0 +\varrho\hspace{0.05em} \beta_0^2\leq-\varrho\hspace{0.05em}\beta_0^2<0.
\end{equation}
Now, since $\lim_{t\to -\infty}M^-(t)=-\infty$, up to decrease further $T_1$ if necessary, we may assume that $-M^-(t)>R_0+1$. Hence, recalling \eqref{est:m:m:m} and \eqref{defR0}, we have
\begin{equation*}
\opp{w^-}(t,x)\le e^{2\lambda ct}\left(-k\gamma_0\hspace{0.1em}e^{\lambda R_0}+k\delta_0 +\varrho \beta_0^2\right)\leq -\beta_0^2\varrho\hspace{0.1em} e^{2\lambda ct}<0. %\label{kop}
\end{equation*}
for all $x\in \{R_0\le x_1<-M^{-}(t)\}$ and all $t\le T_1$.

Lastly, let us consider the case $x\in \{0\le x_1< R_0\}$. Then, up to take $T_1$ sufficiently negative, we have $\zeta_-<\zeta_+\leq z_*$ (where $z_*$ is as in Lemma~\ref{LE:CONVEXITE:PHI}), which then gives
$$ \phi'(\zeta_+)-\phi'(\zeta_-)=\int_{\zeta_-}^{\zeta_+}\phi''(z)\hspace{0.1em}\mathrm{d}z\geq \frac{\lambda}{8}\int_{\zeta_-}^{\zeta_+}\phi'(z)\hspace{0.1em}\mathrm{d}z=\frac{\lambda}{8}\hspace{0.1em}\big(\phi(\zeta_+)-\phi(\zeta_-)\big). $$
Going back to \eqref{zpm} and recalling that $\dot{\xi}(t)=k\hspace{0.05em}e^{\lambda M^{+}(t )}$, we obtain
\begin{align*}
\opp{w^-}(t,x)&\leq \left(\varrho\hspace{0.1em}\phi(\zeta_-)-\frac{\lambda\hspace{0.05em}k}{8}\hspace{0.1em}e^{\lambda M^{+}(t)}\right)\big(\phi(\zeta_+)-\phi(\zeta_-)\big) \\
&\leq \left(\varrho\beta_0\hspace{0.1em}e^{-\lambda x_1+\lambda M^{-}(t)}-\frac{\lambda\hspace{0.05em}k}{8}\hspace{0.1em}e^{\lambda M^+(t)}\right)\big(\phi(\zeta_+)-\phi(\zeta_-)\big) \\
&\leq e^{\lambda M^+(t)}\left(\varrho\beta_0\hspace{0.1em} e^{-\lambda x_1 -2\lambda \xi(t)}-\frac{\lambda\hspace{0.05em}k}{8}\hspace{0.1em}\right) \big(\phi(\zeta_+)-\phi(\zeta_-)\big) \\
&\le e^{\lambda M^+(t)}\left(\varrho\beta_0-\frac{\lambda\hspace{0.05em}k}{8}\hspace{0.1em}\right) \big(\phi(\zeta_+)-\phi(\zeta_-)\big).
\end{align*}
Therefore, $\opp{w^-}(t,x)\leq 0$ in the set $\{0\le x_1< R_0\}$ provided that $k\geq 8\hspace{0.05em}\lambda^{-1}\varrho\hspace{0.1em}\beta_0$ and that $T_1$ is sufficiently negative. This completes the proof.
\end{proof}

\subsection{Construction of the entire solution} In this subsection, we will use the subsolution and the supersolution constructed above to prove Theorem \ref{timebefore}.

\begin{proof}[Proof of Theorem \ref{timebefore}]
For the clarity of the exposure, we split the proof into four steps.
\vskip 0.3cm

\noindent \emph{Step 1. Construction of an entire solution}
\smallskip

Let $w^+$ and $w^-$ be the functions defined by \eqref{wplus} and \eqref{wmoins}, respectively. By Lemma \ref{subsuper}, we know that $w^+$ and $w^-$ are respectively a supersolution and a subsolution to \eqref{P} in the range $(t,x)\in(-\infty,T_1]\times\Omega$ for some $T_1\in(-\infty,T)$ where $T$ is given by \eqref{T:xi}. We will construct a solution to \eqref{P} using a monotone iterative scheme starting from $w^-$ and using $w^+$ as a barrier.

Let $n\geq0$ be so large that $-n<T_1-1$. By Proposition \ref{PROP:EXIST} and Remark~\ref{RK:Ext:Cont}, we know that there exists a unique solution $u_n(t,x)\in C^1([-n,\infty),C(\overline{\Omega}))$ to
\begin{align*}
\left\{
\begin{array}{rll}
\partial_tu_n\!\!\!&=Lu_n+f(u_n) & \text{in }(-n,\infty)\times\overline{\Omega}, \vspace{3pt}\\
u(-n,\cdot)\!\!\!&=w^-(-n,\cdot) & \text{in }\overline{\Omega}.
\end{array}
\right.
\end{align*}
In particular, we have
$$ w^-(-n,x)=u_n(-n,x)\leq w^+(-n,x) \, {\mbox{ for any }}x\in\overline{\Omega}. $$
In virtue of Proposition \ref{PROP:EXIST}, the functions $u_n$, $w^-$ and $w^+$ satisfy the regularity requirements of Lemma \ref{LE:COMPARISON} in the time segment $[-n,T_1]$. Therefore, by the comparison principle (Lemma \ref{LE:COMPARISON}), we deduce that
\begin{align}
w^-(t,x)\leq u_n(t,x)\leq w^+(t,x) \, {\mbox{ for any }}(t,x)\in(-n,T_1)\times\overline{\Omega}. \label{wplusoumoins}
\end{align}
Note that, by assumption, $-n+1\in(-n,T_1)$. In particular,
$$ u_{n-1}(-n+1,x):=w^-(-n+1,x)\leq u_n(-n+1,x)\leq w^+(-n+1,x) \, {\mbox{ for any }}x\in\overline{\Omega}. $$
Let $\tau>T_1$ be arbitrary. Using again the comparison principle Lemma \ref{LE:COMPARISON}, we obtain
\begin{align}
0\leq u_{n-1}(t,x)\leq u_n(t,x)\leq 1 \, {\mbox{ for any }}(t,x)\in(1-n,\tau)\times\overline{\Omega}. \label{0-1}
\end{align}
Since $\tau$ is arbitrary this still holds for any $(t,x)\in(1-n,\infty)\times\overline{\Omega}$. In particular, $(u_n)_{n> \lfloor1-T_1\rfloor}$ is monotone increasing with $n$. Hence, $u_n$ converges pointwise to some entire function $\bar{u}(t,x)$ defined in $\R\times\overline{\Omega}$. Moreover, by \eqref{0-1} and estimate \eqref{apriori777} in Proposition \ref{PROP:EXIST}, we have
\begin{align}
\|u_n(\cdot,x)\|_{C^{1,1}([-n,\infty))}\leq 1+\omega+\omega^2=:C_0 \, {\mbox{ for any }}x\in\overline{\Omega}, \label{C11:temps}
\end{align}
where $\omega=\sup_{[0,1]}|f'|+2\hspace{0.1em}\sup_{\Omega}\hspace{0.1em}\mathcal{J}^\delta$. Also, given \eqref{0-1} and since $[w^-(-n,\cdot)]_{C^{0,\alpha}(\overline{\Omega})}$ is independent of $n$, we may apply Proposition \ref{PROP:PARABOLIK2} and deduce that
$$ [u_n(t,\cdot)]_{C^{0,\alpha}(\overline{\Omega})}+[\partial_tu_n(t,\cdot)]_{C^{0,\alpha}(\overline{\Omega})}\leq C_1 \, {\mbox{ for any }}t\geq-n, $$
for some constant $C_1>0$.
Passing to the limit as $n\to\infty$ we obtain that
\begin{align}
\sup_{x\in\overline{\Omega}}\,\|\bar{u}(\cdot,x)\|_{C^{1,1}(\R)}+\sup_{t\in\R}\,\Big([\bar{u}(t,\cdot)]_{C^{0,\alpha}(\overline{\Omega})}+[\partial_t\bar{u}(t,\cdot)]_{C^{0,\alpha}(\overline{\Omega})}\Big)\leq C_2, \label{ubar:parabolic}
\end{align}
where $C_2:=C_0+C_1$.
Therefore, $\bar{u}\in C^{1,1}(\R,C^{0,\alpha}(\overline{\Omega}))$. Furthermore, by \eqref{0-1}, we have
\begin{align}
0\leq\bar{u}(t,x)\leq1 \, {\mbox{ for all }}(t,x)\in\R\times\overline{\Omega}. \label{0u1}
\end{align}
Let us now check that $\bar{u}$ solves \eqref{P}. Clearly, $f(u_n)\to f(u)$ as $n\to\infty$. Now, let $k\geq -T_1+1$ and $n\geq k$. Then, by Dini's theorem, for any $(t,x)\in(-k,\infty]\times\overline{\Omega}$, we have
\begin{align}
|Lu_n(t,x)-L\bar{u}(t,x)|&\leq 2\hspace{0.1em}\|\mathcal{J}^\delta\|_\infty\!\sup_{z\in B_{R_J}(x)} |\bar{u}(t,z)-u_n(t,z)|\underset{n\to\infty}{\longrightarrow}0, \label{LjLl}
\end{align}
where $\mathcal{J}^\delta$ is as in \eqref{C3}. Furthermore, using \eqref{C11:temps} we obtain that, up to extract a subsequence, $\partial_tu_n(\cdot,x)\to\partial_t\bar{u}(\cdot,x)$ in $C_{\mathrm{loc}}^{1,\alpha}(\R)$ for any $\alpha\in(0,1)$. Therefore, recalling \eqref{LjLl} and since $k$ can be taken arbitrarily large, we deduce that $\bar{u}$ is indeed an entire solution to \eqref{P} in $\overline{\Omega}\times\R$. Notice that a consequence of this and the fact that $f\in C^1([0,1])$ is that
\begin{align}
\bar{u}\in C^{1,1}(\R,C^{0,\alpha}(\overline{\Omega}))\cap C^2(\R,C^{0,\alpha}(\overline{\Omega})), \label{REG:ubar:2}
\end{align}
as can be seen by a standard bootstrap argument.
\vskip 0.3cm
\noindent \emph{Step 2. Asymptotic behaviour as} $t\to-\infty$
\smallskip

Letting $n\to\infty$ in \eqref{wplusoumoins} we obtain
\begin{align}
w^-(t,x)\leq \bar{u}(t,x)\leq w^+(t,x) \, {\mbox{ for any }} (t,x)\in(-\infty,T_1]\times\overline{\Omega}. \label{wuwENC}
\end{align}
Consequently, if $x_1<0$ and $t\leq T_1$, we have
\begin{align}
|\bar{u}(t,x)\!-\!\phi(x_1\!+\!ct)|\!&\leq \!|\bar{u}(t,x)\!-\!2\hspace{0.1em}\phi(M^+(t))|\!+\!|2\hspace{0.1em}\phi(M^+(t))\!-\!\phi(x_1\!+\!ct)|\leq \!4\hspace{0.1em}\phi(M^+(t)), \label{unifC1}
\end{align}
where $M^\pm(t)$ has the same meaning as in \eqref{M(t)}. Similarly, if $x_1\geq0$ and $t\leq T_1$, then
\begin{align}
|\bar{u}(t,x)\!-\!\phi(x_1\!+\!ct)|&\leq |w^+(t,x)-\phi(x_1\!+\!ct)|+|w^+(t,x)-\bar{u}(t,x)|. \nonumber
\end{align}
Using \eqref{wuwENC} we get
\begin{align}
|\bar{u}(t,x)-\phi(x_1+ct)|&\leq |w^+(t,x)-\phi(x_1+ct)|+|w^+(t,x)-w^-(t,x)| \nonumber \\
&\leq \big\{\|\phi'\|_\infty\hspace{0.1em}\xi(t)+\phi(M^+(t))\big\}+\big\{\phi(M^+(t))+\phi(M^-(t))+2\hspace{0.1em}\|\phi'\|_\infty\hspace{0.1em}\xi(t)\big\} \nonumber \\
&=3\hspace{0.1em}\|\phi'\|_\infty\hspace{0.1em}\xi(t)+2\hspace{0.1em}\phi(M^+(t))+\phi(M^-(t)). \label{unifC2}
\end{align}
By \eqref{unifC1} and \eqref{unifC2}, we obtain
$$ |\bar{u}(t,x)-\phi(x_1+ct)|\underset{t\to-\infty}{\longrightarrow} 0 \, {\mbox{ uniformly in }}x\in\overline{\Omega}, $$
since $\xi(t)\to0$ and $\phi(M^\pm(t))\to0$ as $t\to-\infty$.
\vskip 0.3cm

\noindent \emph{Step 3. Monotonicity of the entire solution}
\smallskip

Let us now prove that $\bar{u}$ is monotone increasing in $t\in\R$. Note that, once this is done, we automatically get the following sharpening of \eqref{0u1}:
$$ 0<u(t,x)<1 \, {\mbox{ for any }}(t,x)\in\R\times\overline{\Omega}. $$
To show that $\partial_t\bar{u}(t,x)>0$ we first notice that
\begin{align*}
\left\{
\begin{array}{rll}
u_n(t,x)\!\!\!&\geq w^-(t,x)  &\text{in }(-n,T_1]\times\overline{\Omega}, \vspace{3pt}\\
u_n(-n,\cdot)\!\!\!&=w^-(-n,\cdot) &\text{in }\overline{\Omega},  \vspace{3pt}\\
\partial_tw^-(-n,\cdot)\!\!\!&\geq0 &\text{in }\overline{\Omega}.
\end{array}
\right.
\end{align*}
In particular, we obtain that $\partial_tu_n(-n,x)\geq0$. By \eqref{REG:ubar:2}, we may apply Lemma \ref{LE:COMPARISON2} to obtain that $\partial_tu_n\geq0$ in $t\in[-n,\infty)$. By the uniform boundedness of $\partial_tu_n(\cdot,x)$ in $C^{0,1}([-n,\infty))$ (remember \eqref{C11:temps}), we may take the limit as $n\to\infty$ to obtain
\begin{align}
\partial_t\bar{u}(t,x)\geq0 \, {\mbox{ for all }}(t,x)\in\R\times\overline{\Omega}. \label{positivite}
\end{align}
Let us now set $\mu:=\inf_{s\in[0,1]}\,f'(s)$. By \eqref{REG:ubar:2}, we can differentiate with respect to $t$ the equation satisfied by $\bar{u}$ to get
\begin{align}
\partial_t^2\bar{u}=L\left(\partial_t\bar{u}\right)+f'(\bar{u})\partial_t\bar{u}\geq L\left(\partial_t\bar{u}\right)+\mu\,\partial_t\bar{u}, \label{acceleration}
\end{align}
which makes sense everywhere. We conclude by contradiction. Suppose that there exists $(T_0,x_0)\in\R\times\overline{\Omega}$ such that $\partial_t\bar{u}(T_0,x_0)=0$. Choose any $t\leq T_0$ and let $\lambda>0$ be some large number to be fixed later on. Multiplying \eqref{acceleration} by $e^{\lambda\tau}$ and integrating over $\tau\in[t,T_0]$, we come up with
$$ e^{\lambda T_0}\partial_t\bar{u}(T_0,x_0)\geq e^{\lambda t}\partial_t\bar{u}(t,x_0)+\int_{t}^{T_0} e^{\lambda\tau}\left(\mathcal{L}\left[\partial_t\bar{u}\right](\tau,x_0)+(\lambda-\mathcal{J}^\delta(x_0)+\mu)\partial_t\bar{u}(\tau,x_0)\right)\mathrm{d}\tau, $$
where the operator $\mathcal{L}[\cdot]$ is given by \eqref{Lcrochet}. We now choose $\lambda>0$ large enough so that $\lambda>\|\mathcal{J}^\delta\|_\infty-\mu$. Then, on account of \eqref{positivite}, we obtain
\begin{align*}
0=\partial_t\bar{u}(T_0,x_0)\geq e^{\lambda(t-T_0)}\partial_t\bar{u}(t,x_0)\geq0 {\mbox{ for any }}t\leq T_0. %\label{CDNiter}
\end{align*}
As a result we infer that $\partial_t\bar{u}(t,x_0)=0$ for any $t\leq T_0$. In particular, $L\bar{u}(t,x_0)+f(\bar{u}(t,x_0))=0$ for any $t\in(-\infty,T_0]$. Differentiating this with respect to $t$ and using again that $\partial_t\bar{u}(t,x_0)=0$ for any $t\leq T_0$ together with the dominated convergence theorem, we arrive at
$$ \int_\Omega J(\delta(x_0,y))\hspace{0.05em}\partial_t\bar{u}(t,y)\hspace{0.1em}\mathrm{d}y=0 \text{ for any } t\in(-\infty,T_0]. $$
In turn this implies that $\partial_t\bar{u}(t,y)=0$ for all $(t,y)\in(-\infty,T_0]\times\Pi_1(J,x_0)$ where $\Pi_1(J,x_0)$ is as in Definition \ref{DE:CovProp}. Applying the same arguments to the new set of stationary points $\Pi_1(J,x_0)$, we obtain that $\partial_t\bar{u}(t,y)=0$ for all $(t,y)\in(-\infty,T_0]\times\Pi_2(J,x_0)$. Iterating this procedure over again implies that $\partial_t\bar{u}(t,y)=0$ for all $(t,y)\in(-\infty,T_0]\times\Pi_j(J,x_0)$ and all $j\in\N$. Since $(\Omega,\delta)$ has the $J$-covering property, we therefore obtain that $\partial_t\bar{u}(t,y)=0$ for every $(t,y)\in(-\infty,T_0]\times\overline{\Omega}$. In particular, this is true for every $y\in\overline{\Omega}$ with $y_1=y\cdot e_1>0$ and, for any such fixed $y$ and any $t<\min\{T_0,T_1\}$, it holds
$$ 0<w^-(t,y)\leq \bar{u}(t,y)\equiv \lim_{\tau\to-\infty}\,\bar{u}(\tau,y)=\lim_{\tau\to-\infty}\,\phi(y_1+c\tau)=0, $$
a contradiction. Therefore, $\partial_t\bar{u}(t,x)>0$ for all $(t,x)\in\R\times\overline{\Omega}$.
\vskip 0.3cm

\noindent \emph{Step 4. Uniqueness of the entire solution}
\smallskip

The proof is almost identical to that given in \cite[Section 3]{Berestycki2009d}. The only difference with the local case is that the solution does no longer satisfy parabolic estimates. However, this is compensated by Lemma \ref{apriori:entire} and \eqref{ubar:parabolic}.
\end{proof}

\subsection{Further properties of the entire solution} In this section, we prove that the unique entire solution to \eqref{P} satisfying \eqref{qualit-u} and \eqref{origintimes} shares the same limit as $x_1\to\pm\infty$ than the planar wave $\phi(x_1+ct)$. Precisely,
\begin{prop}\label{PROP:LIMITE:X1}
Assume \eqref{C1}, \eqref{C2}, \eqref{C3}, \eqref{C4} and \eqref{fJdelta-pos}. Suppose that $J\in \mathbb{B}_{1,\infty}^\alpha(\Omega;\delta)$ for some $\alpha\in(0,1)$. Let $u(t,x)$ be the unique entire solution to \eqref{P} satisfying \eqref{qualit-u} and \eqref{origintimes}. Then, denoting a point $x\in\overline{\Omega}$ by $x=(x_1,x')\in\R\times\R^{N-1}$, we have
$$ \lim_{x_1\to-\infty}\hspace{0.1em}u(t,x)=0 \text{ and } \lim_{x_1\to\infty}\hspace{0.1em}u(t,x)=1\ \text{ for all }(t,x')\in\R\times\R^{N-1}. $$
\end{prop}
\begin{proof}
Let us first prove that $\lim_{x_1\to\infty}\hspace{0.1em}u(t,x)=1$ for all $(t,x')\in\R\times\R^{N-1}$. To see this, it suffices to observe that $u(t,x)\geq w^-(t,x)$ for all $(t,x)\in(-\infty,T_1]\times\overline{\Omega}$.
Hence, using \eqref{couplephic} and the definition of $w^-$ (remember \eqref{wmoins}), we deduce that
$$ 1\geq \limsup_{x_1\to\infty}\hspace{0.1em}u(t,x)\geq \liminf_{x_1\to\infty}\hspace{0.1em}u(t,x) \geq \lim_{x_1\to\infty}\left\{\phi(x_1+M^-(t))-\phi(-x_1+M^-(t))\right\}=1, $$
for all $(t,x')\in(-\infty,T_1]\times\R^{N-1}$, where $M^-(t)$ has the same meaning as in \eqref{M(t)}. Now, since $\partial_tu(t,x)>0$ for all $(t,x)\in\R\times\overline{\Omega}$, we have
$$ 1\geq \limsup_{x_1\to\infty}\hspace{0.1em}u(t,x)\geq \liminf_{x_1\to\infty}\hspace{0.1em}u(t,x) \geq \lim_{x_1\to\infty} u(T_1,x)=1, $$
for all $(t,x')\in(T_1,\infty)\times\R^{N-1}$.
Therefore, $\lim_{x_1\to\infty}\hspace{0.1em}u(t,x)=1$ for all $(t,x')\in\R\times\R^{N-1}$.

To complete the proof, it remains to show that $\lim_{x_1\to-\infty}\hspace{0.1em}u(t,x)=0$ for all $(t,x')\in\R\times\R^{N-1}$. The proof of this is slightly more involved and we need to compare $u$ with the solution of an auxiliary problem.
To this end, we let $g\in C^1([0,2])$ be a nonlinearity of ``ignition" type, namely such that the following properties hold:
$$ \restriction{g}{[0,\theta/4]}\equiv0,\ \restriction{g}{(\theta/4,2)}>0,\ g(2)=0 \text{ and } g'(2)<0. $$
Let us assume, in addition, that $g(s)\geq \max_{[0,1]}\hspace{0.1em}f$ for all $s\in[\theta/2,1+\theta/2]$.
Now, using the existence result \cite[Theorems 1.2-1.3]{Coville2007d} (see in particular \cite[Lemma 5.1]{Coville2007d} and the remarks in \cite[Section 1.2]{Coville2007d} on page 5), we know that there exists a unique monotone increasing front $\varphi\in C(\R)$ with speed $c'>0$, satisfying $\varphi(0)=1$ and such that
\begin{align}
\left\{
\begin{array}{c}
c'\,\varphi'=J_1\ast \varphi-\varphi+g(\varphi) \text{ in }\R, \vspace{3pt}\\
\displaystyle\lim_{z\to +\infty}\varphi(z)=2,\, \lim_{z\to -\infty}\varphi(z)=0,
\end{array}
\right.\label{def:varlphi}
\end{align}
where $J_1$ is as in \eqref{J1}.
Now, let us define $g_\varrho(s):=g(s-\varrho)$, for all $\varrho>0$ and all $s\in[\varrho,2+\varrho]$.
By definition of $g_\varrho$, we can check that the function $\varphi_\varrho(x):=\varrho+\varphi(x)$ solves
\begin{align}
\left\{
\begin{array}{c}
c'\,\varphi_\varrho'=J_1\ast \varphi_\varrho-\varphi_\varrho+g_\varrho(\varphi_\varrho) \text{ in }\R, \vspace{3pt}\\
\displaystyle\lim_{z\to +\infty}\varphi_\varrho(z)=2+\varrho,\, \lim_{z\to -\infty}\varphi_\varrho(z)=\varrho.
\end{array}
\right. \label{def:varlphi:varrho}
\end{align}
Next, for all $\varrho\in(0,\theta/4]$ and all $A>0$, we let $w_{\varrho,A}(t,x):=\varphi_\varrho(x_1+A+c't)$. We claim that
\begin{claim}\label{Claim:1}
For all $\varrho\in(0,\theta/4]$, there exist $A_\varrho>0$ and $t_\varrho\in\R$ such that
$$ u(t,x)\leq w_{\varrho,A_\varrho}(t,x) \text{ for all } (t,x)\in[t_\varrho,\infty)\times\overline{\Omega}. $$
\end{claim}
Note that, by proving Claim~\ref{Claim:1}, we end the proof of Proposition~\ref{PROP:LIMITE:X1}. To see this, fix some $\eps>0$ and let $\varrho=\eps/2$. Also, for $R>0$, let $H_R^+$ and $H_R^-$ be the half-spaces given by
\begin{align}
H_R^+:=\left\{x\in\R^N; x_1>-R\right\} \text{ and } H_R^-:=\left\{x\in\R^N; x_1\leq-R\right\}, \label{Hplusoumoins}
\end{align}
respectively. Assume, for the moment, that $t\in[t_\varrho,\infty)$. By \eqref{def:varlphi}, we know that there exists some $R_\varrho>0$ such that $\varphi(z+A_\varrho)\leq\varrho$ for all $z\leq-R_\varrho$. In particular, we have
$$ w_{\varrho,A_\varrho}(t,x)=\varrho+\varphi(x_1+A_\varrho+c't)\leq2\varrho=\eps \text{ for all }(t,x)\in[t_\varrho,\infty)\times H_{R_\varrho+c't}^-. $$
Applying now Claim~\ref{Claim:1}, we then deduce that $u(t,x)\leq \eps$ for all $(t,x)\in[t_\varrho,\infty)\times\overline{\Omega}\cap H_{R_\varrho+c't}^-$, which, in turn, automatically implies that
\begin{align}
\limsup_{x_1\to-\infty}\,u(t,x)\leq\eps \text{ for all }(t,x')\in[t_\varrho,\infty)\times\R^{N-1}. \label{limsup:tdelta1}
\end{align}
The analogue of this for $t\in(-\infty,t_\varrho)$ is a simple consequence of the monotonicity of $u(t,x)$. Indeed, using that $u(\cdot,x)$ is increasing for all $x\in\overline{\Omega}$, we obtain $u(t,x)\leq u(t_\varrho,x)\leq w_{\varrho,A_\varrho}(t_\varrho,x)\leq\eps$, for all $(t,x)\in(-\infty,t_\varrho)\times\overline{\Omega}\cap H_{R_\varrho+c't}^-$, which, again, implies that
\begin{align}
\limsup_{x_1\to\infty}\,u(t,x)\leq\eps \text{ for all }(t,x')\in(-\infty,t_\varrho)\times\R^{N-1}.  \label{limsup:tdelta2}
\end{align}
Hence, collecting \eqref{limsup:tdelta1}, \eqref{limsup:tdelta2}, recalling that $u(t,x)>0$ for all $(t,x)\in\R\times\overline{\Omega}$ and that $\eps>0$ is arbitrary, we conclude that
$$ \lim_{x_1\to-\infty}u(t,x)=0 \text{ for all }(t,x')\in\R\times\R^{N-1}, $$
which thereby proves Proposition~\ref{PROP:LIMITE:X1}.
\end{proof}
To complete the proof of Proposition~\ref{PROP:LIMITE:X1} it remains to establish Claim~\ref{Claim:1}.
\begin{proof}[Proof of Claim~\ref{Claim:1}]
First of all, we notice that, since $K\subset\R^N$ is compact, we may always find some $R_K>0$ so that $K\subset H_{R_K}^+$ (we use the same notation as in \eqref{Hplusoumoins}). Furthermore, we observe that, by construction of $g_\varrho$, there holds $g_\varrho\geq \widetilde{f}^+\geq\widetilde{f}$ for all $s\in[\varrho,2+\varrho]$ and all $0<\varrho\leq\theta/4$, where $\widetilde{f}\in C^1(\R)$ is the extension of $f$ given by \eqref{extension:f}.
In particular, this implies that the function $w_{\varrho,A}$ satisfies
\begin{align}
\partial_tw_{\varrho,A}\geq J_{\mathrm{rad}}\ast w_{\varrho,A}-w_{\varrho,A}+\widetilde{f}(w_{\varrho,A}) \text{ in }\R\times\R^N. \label{Eq:de:W}
\end{align}
Now, let $R_1\geq R_J$, where $R_J>0$ is any number such that $\mathrm{supp}(J)\subset[0,R_J]$.
Since $u(t,x)$ satisfies \eqref{unique}, there is then some $t_\varrho\in\R$ such that
$$ u(t,x)\leq \phi(x_1+ct)+\frac{\varrho}{2} \text{ for all }(t,x)\in(-\infty,t_\varrho]\times\overline{\Omega}. $$
Since $\phi$ is increasing, we may assume that $\phi(-R_1-R_K+ct_\varrho)\leq\varrho/2$ (up to take $R_1$ larger if necessary).
Consequently, for all $A>0$, we have
\begin{align}
u(t_\varrho,x)\leq\varrho \leq w_{\varrho,A}(t_\varrho,x) \text{ for all }x\in H_{R_1+R_K}^-. \label{Init:Hmoins}
\end{align}
On the other hand, since $\varphi(0)=1$ and $\varphi'>0$, by taking $A_\varrho=R_1+R_K-c't_\varrho$, we get
$$ w_{\varrho,A_\varrho}(t_\varrho,x)=\varrho+\varphi(x_1+R_1+R_K)>\varrho+\varphi(0)=\varrho+1 \text{ for all }x\in H_{R_1+R_K}^+. $$
Since $u< 1$ in $\R\times\overline{\Omega}$ and since $w_{\varrho,A_\varrho}(\cdot,x)$ is increasing for all $x\in\overline{\Omega}$, we deduce that
\begin{align}
w_{\varrho,A_\varrho}(t,x)>\varrho+1> u(t,x) \text{ for all }(t,x)\in[t_\varrho,\infty)\times \overline{\Omega}\cap H_{R_1+R_K}^+. \label{Init:Complement}
\end{align}
On the other hand, since $K\subset H_{R_K}^+$ and since $\mathrm{supp}(J)\subset[0,R_J]$, it follows that
\begin{align}
\partial_tu=J_{\mathrm{rad}}\ast u-u+\widetilde{f}(u) \text{ in }\R\times H_{R_1+R_K}^-. \label{Eq:de:U}
\end{align}
Hence, collecting \eqref{Eq:de:W}, \eqref{Init:Hmoins}, \eqref{Init:Complement} and \eqref{Eq:de:U}, we find that
\begin{align*}
\left\{
\begin{array}{cl}
\partial_tw_{\varrho,A_\varrho}\geq J_{\mathrm{rad}}\ast w_{\varrho,A_\varrho}-w_{\varrho,A_\varrho}+\widetilde{f}(w_{\varrho,A_\varrho}) & \text{in }(t_\varrho,\infty)\times H_{R_1+R_K}^-, \vspace{3pt}\\
\partial_tu=J_{\mathrm{rad}}\ast u-u+\widetilde{f}(u) & \text{in }(t_\varrho,\infty)\times H_{R_1+R_K}^-, \vspace{3pt}\\
w_{\varrho,A_\varrho}>u & \text{in }[t_\varrho,\infty)\times \overline{\Omega}\cap H_{R_1+R_K}^+, \vspace{3pt}\\
w_{\varrho,A_\varrho}(t_\varrho,\cdot)\geq u(t_\varrho,\cdot) & \text{in }H_{R_1+R_K}^-.
\end{array}
\right.
\end{align*}
By a straightforward adaptation of the parabolic comparison principle Lemma~\ref{LE:COMPARISON}, we
%Using the parabolic comparison principle (or, alternatively, by a straightforward adaptation of the proof of Lemma~\ref{LE:COMPARISON}), we
deduce that $u(t,x)\leq w_{\varrho,A_\varrho}(t,x)$ for all $(t,x)\in[t_\varrho,\infty)\times H_{R_1+R_K}^-$ and, hence, this holds for all $(t,x)\in[t_\varrho,\infty)\times \overline{\Omega}$, which thereby establishes Claim~\ref{Claim:1}.
\end{proof}

%%%%%%%%%%%%%%%%%%%%%

\section{Local behaviour after the encounter with $K$}\label{SE:LARGETIME}

In this section, we study how the entire solution $u(t,x)$ to \eqref{P} with \eqref{qualit-u} and \eqref{origintimes} behaves after hitting the obstacle $K$. We will first show that it converges to $u_\infty(x)\hspace{0.1em}\phi(x_1+ct)$, locally uniformly in $x\in\overline{\Omega}$ as $t\to\infty$, where $u_\infty\in C(\overline{\Omega})$ solves
\begin{align}
\left\{
\begin{array}{r l}
Lu_\infty+f(u_\infty)=0 & \text{in } \overline{\Omega}, \vspace{3pt}\\
0\leq u_\infty\leq 1 & \text{in } \overline{\Omega}, \vspace{3pt}\\
u_\infty(x)\to1 & \text{as } |x|\to\infty.
\end{array} \nonumber
\right.
\end{align}
What is more, we will prove that $u(t,x)$ converges to the planar wave $\phi(x_1+ct)$ as $|x'|\to\infty$ when $(t,x_1)$ stays in some compact set or, otherwise said, that {the encounter with the obstacle does not much deform $u(t,x)$ in hyperplanes which are orthogonal to the $x_1$-direction}.
%
%uniformly (in $x\in\Omega$) as $t\to\infty$ and uniformly (in $t\in\R$) as $|x|\to\infty$; where $u_\infty\in C(\overline{\Omega})$ solves
%\begin{align}
%\left\{
%\begin{array}{r l}
%Lu_\infty+f(u_\infty)=0 & \text{in } \overline{\Omega}, \vspace{3pt}\\
%0\leq u_\infty\leq 1 & \text{in } \overline{\Omega}, \vspace{3pt}\\
%u_\infty(x)\to1 & \text{as } |x|\to\infty.
%\end{array} \nonumber
%\right.
%\end{align}
%This together with Theorem \ref{timebefore} will complete the proof of Theorem \ref{TH:EXIST:ENTIRE}. Our strategy involves two ingredients. First, we will show that the convergence towards $u_\infty(x)\hspace{0.1em}\phi(x_1+ct)$ holds locally uniformly in $x\in\Omega$ as $t\to\infty$. Then, we prove that $u(x,t)$ converges to a planar wave $\phi(x_1+ct)$ as $|x'|\to\infty$ when $(x_1,t)$ stays in some compact set; or, otherwise said, that the encounter with the obstacle does not much deform $u(x,t)$ in hyperplanes which are orthogonal to the $x_1$-direction.

The results in this section are somehow independent of the geometry of $K$. The influence of the latter is in fact ``encoded" in the function $u_\infty$ as will be shown in the next section.

\subsection{Local uniform convergence to the stationary solution}

In this sub-section, we prove the local uniform convergence of $u(t,x)$ towards $u_\infty(x)\hspace{0.1em}\phi(x_1+ct)$ as $t\to\infty$.

%More precisely:
\begin{prop}\label{TH:ainf}
Assume \eqref{C1}, \eqref{C2}, \eqref{C3}, \eqref{C5} and \eqref{fJdelta-pos}.
Suppose that $J\in \mathbb{B}_{1,\infty}^\alpha(\Omega;\delta)$ for some $\alpha\in(0,1)$. Let $u(t,x)$ be the unique entire solution to \eqref{P} satisfying \eqref{qualit-u} and \eqref{origintimes}. Then, there exists a solution $u_\infty\in C(\overline{\Omega})$ to \eqref{Pinf} such that
$$ |u(t,x)-u_\infty(x)|\underset{t\to+\infty}{\longrightarrow} 0 \mbox{  locally uniformly in  }x\in\overline{\Omega}. $$
\end{prop}
\begin{remark}
Since the convergence is local uniform, we also have
\begin{align}
|u(t,x)-\phi(x_1+ct)\hspace{0.1em}u_\infty(x)|\underset{t\to+\infty}{\longrightarrow} 0 \mbox{ locally uniformly in  }x\in\overline{\Omega}. \label{asympt:large:time2}
\end{align}
\end{remark}
The proof of Proposition~\ref{TH:ainf} relies on the following lemma:
\begin{lemma}\label{LE:sup-lim}
Assume \eqref{C1}, \eqref{C2}, \eqref{C3} and \eqref{C5}. Let $u\in C(\overline{\Omega},[0,1])$ be a solution to the stationary equation $Lu+f(u)=0$ in $\overline{\Omega}$ satisfying $\sup_{\overline{\Omega}}u=1$. Then,
$$ \lim_{|x|\to\infty}\hspace{0.1em}u(x)=1. $$
%\begin{align}
%\left\{
%\begin{array}{rcl}
%Lu+f(u) & \!\!=\!\! & 0 \hbox{ in }\overline{\Omega},\vspace{3pt}\\
%\displaystyle{\sup_{\Omega}}\,u &  \!\!=\!\! & 1,
%\end{array}
%\right.  \label{LIM:cu9}
%\end{align}
%Then, $u(x)\to1$ as $|x|\to \infty$.
\end{lemma}
\begin{proof}
Let us first consider the case when $\delta\in\mathcal{Q}(\overline{\Omega})$ is the Euclidean distance. Then, Lemma~\ref{LE:sup-lim} is exactly \cite[Lemma 7.2]{Brasseur2019} \emph{without} the extra assumption that $J_{\text{rad}}\in L^2(\R^N)$ (that was required in \cite{Brasseur2019}). However, it turns out that the same arguments given there also yield Lemma~\ref{LE:sup-lim} with only minor changes. As a matter of fact, the only place where the assumption that $J_{\text{rad}}\in L^2(\R^N)$ comes into play is when showing the existence of a maximal solution $w$ to
\begin{align}
\int_{B_R(x_0)}J_{\text{rad}}(x-y)\hspace{0.05em}w(x)\hspace{0.1em}\mathrm{d}y-w(x)+f(w(x))=0 \text{ for all }x\in B_R(x_0), \label{AUX:JL2}
\end{align}
and for any $x_0\in\R^N$ (provided that $R\geq d_0$ for some $d_0=d_0(f,J)>0$ large enough). This technical assumption is here only to ensure that the equation satisfies some compactness property which, in turn, is needed to establish the existence of nontrivial solutions.\par

The strategy of proof used in \cite{Brasseur2019}, consists in using this function $w$ to construct a family of sub-solutions to \eqref{AUX:JL2} and to notice that any solution $u$ to $Lu+f(u)=0$ in $\Omega$ is a super-solution to \eqref{AUX:JL2} on balls $B_R(x_0)$ that are sufficiently far away from $K$. Then, using the sweeping-type principle \cite[Lemma 4.3]{Brasseur2019}, it can be shown that the so-constructed sub-solutions yield lower bounds on $u$ which can be propagated in a way that yields that $u(x)\to1$ as $|x|\to\infty$. This strategy still works if we replace $J_{\text{rad}}$ by the truncation, $J_{\psi}$, defined by $J_{\psi}(z)=J_{\text{rad}}(z)\hspace{0.1em}\psi(z)$, where $\psi\in C_c^\infty(\R^N,[0,1])$ is a radial cut-off function such that
\begin{center}
$|\mathrm{supp}(J_{\psi})|>0$ and $J_{\psi}\in L^2(\R^N)$.
\end{center}
Indeed, since $J_{\psi}\in L^2(\R^N)$ there will then exist a solution $w_\psi$ to \eqref{AUX:JL2} with $J_{\psi}$ instead of $J_{\text{rad}}$. Moreover, $u$ is also a super-solution to \eqref{AUX:JL2} with $J_{\psi}$ instead of $J_{\text{rad}}$ on balls $B_R(x_0)$ that are sufficiently far away from $K$ (since $J_{\psi}\leq J_{\text{rad}}$). We may then simply work with the kernel $J_{\psi}$ instead of $J_{\text{rad}}$. Of course, $J_{\psi}$ has no longer unit mass, but we still have that $0<\int_{\R^N}J_{\psi}\leq1$ which is enough to make the proof given in \cite{Brasseur2019} work, including that of the sweeping-type principle (notice that all the other properties of $J_{\text{rad}}$ are preserved). Arguing in this way, we may then remove the assumption that $J_{\text{rad}}$ is square integrable. In the case when  $\delta\in\mathcal{Q}(\overline{\Omega})$ is not the Euclidean distance, this strategy still works: indeed, as it was already explained in \cite[Remark 2.5]{Brasseur2018}, the proof requires only to work on convex regions far away from the obstacle $K$ in which it trivially holds that $\delta(x,y)=|x-y|$.
\end{proof}
We are now ready to prove Proposition~\ref{TH:ainf}.
\begin{proof}[Proof of Proposition \ref{TH:ainf}]
By \eqref{qualit-u}, one has that $u(t,x)\to u_\infty(x)\in(0,1]$ as $t\to\infty$ for all $x\in\overline{\Omega}$. Furthermore, using Lemma \ref{apriori:entire} (or \eqref{ubar:parabolic}) one has that the convergence is (at least) locally uniform and that $u_\infty$ is a continuous solution of $Lu_\infty+f(u_\infty)=0$ in $\overline{\Omega}$ (the continuity of $u_\infty$ follows straightforwardly from \eqref{fJdelta-pos} and the arguments in \cite[Lemma 3.2]{Brasseur2019}).

Let us now show that $u_\infty(x)\to1$ as $|x|\to\infty$.
By Proposition~\ref{PROP:LIMITE:X1} we know that $u(t,x)\to1$ as $x_1\to\infty$, for any fixed $(t,x)\in\R\times\R^{N-1}$.
But since $0<u(t,x)<1$ and since $\partial_tu(t,x)>0$ for all $(t,x)\in\R\times\overline{\Omega}$, we have $u(t,x)\leq u_\infty(x)\leq1$ for all $(t,x)\in\R\times\overline{\Omega}$. Hence letting $x_1\to\infty$, we
deduce that $u_\infty(x)\to1$ as $x_1\to\infty$. In particular, it holds that $\sup_{\overline{\Omega}}u_\infty=1$. The conclusion now follows from Lemma~\ref{LE:sup-lim}.
%By \eqref{unifC1} and \eqref{unifC2}, we know that there exists a nonnegative function $\Lambda$ with $\Lambda(t)\to0$ as $t\to-\infty$ and such that
%$$ \phi(x_1+ct)-\Lambda(t)\leq u(t,x)\leq \phi(x_1+ct)+\Lambda(t) \mbox{  for any } t\leq T_1 \mbox{ and } x\in\overline{\Omega}. $$
%In particular, this implies that
%$\lim_{t\to-\infty}\sup_{x\in\overline{\Omega}}\,u(t,x)=1$.
%Now, since $0<u<1$ and since $\partial_tu>0$, we have that $\sup_{\Omega}u(t,\cdot)=1$ for every $t\in\R$. Therefore, it holds that $\sup_{\overline{\Omega}} u_\infty=1$ (because $u(t,\cdot)\leq u_\infty\leq 1$). The conclusion now follows from Lemma~\ref{LE:sup-lim}.
%It remains to treat the case when $J$ is neither non-increasing nor in $L^2((0,\infty),t^{N-1}\mathrm{d}t)$.
\end{proof}

\subsection{Convergence near the horizon}

Here, we shall prove that the encounter with $K$ does not alter too much the entire solution $u(t,x)$ to \eqref{P} with \eqref{qualit-u} and \eqref{origintimes} in hyperplanes orthogonal to the $x_1$-direction, in the sense that it remains close to the planar wave $\phi(x_1+ct)$ locally uniformly in $(t,x_1)$ when $|x'|\to\infty$.

\begin{prop}\label{PROP:HORIZON}
Assume \eqref{C1}, \eqref{C2}, \eqref{C3}, \eqref{C4} and \eqref{fJdelta-pos}.
Suppose that $J\in \mathbb{B}_{1,\infty}^\alpha(\Omega;\delta)$ for some $\alpha\in(0,1)$. Let $u(t,x)$ be the unique entire solution to \eqref{P} satisfying \eqref{qualit-u} and \eqref{origintimes}. Then, for any sequence $(x_n')_{n\geq0}\subset\R^{N-1}$ such that $|x_n'|\to\infty$ as $n\to\infty$, there holds
$$ |u(t,x_1,x'+x_n')-\phi(x_1+ct)|\underset{n\to\infty}{\longrightarrow}0 \text{ locally uniformly in }(t,x)=(t,x_1,x')\in\R\times\R^N. $$
\end{prop}
\begin{proof}
The proof works essentially as in the local case, see \cite[Proposition 4.1]{Berestycki2009d}. Let us, however, outline the main ingredients of the proof. For each $n\geq0$, we set $\Omega_n:=\Omega-(0,x_n')$ and, for $(t,x)\in\R\times\Omega_n$, we let $u_n(t,x):=u(t,x_1,x'+x_n')$. By Lemma \ref{apriori:entire} and the boundedness assumption on $u$, up to extraction of a subsequence, we have that $u_n$ converges locally uniformly in $(t,x)\in\R\times\R^N$ to a solution $V$ of
\begin{align*}
\left\{
\begin{array}{rl}
\partial_tV=J_{\mathrm{rad}}\ast V-V+f(V) & \text{ in }\R\times\R^N, \vspace{3pt}\\
0\leq V\leq 1 & \text{ in }\R\times\R^N. \\
\end{array}
\right.
\end{align*}
By \eqref{origintimes}, $V$ inherits from the limit behaviour of $u$ as $-\infty$, namely:
$$ \lim_{t\to-\infty}\,\sup_{x\in\R^N}\,|V(t,x)-\phi(x_1+ct)|=0. $$
From here, we may reproduce the arguments in \cite{Berestycki2009d} using the trick of Fife and McLeod \cite{Fife1977}, to prove that $V(t,x)\equiv\phi(x_1+ct)$ which then completes the proof. Notice that the arguments in \cite{Berestycki2009d} adapt with no difficulty since the local structure of the operator $\Delta u$ does not come into play and can easily be replaced by $J_{\mathrm{rad}}\ast u-u$.
\end{proof}

%%%%%%%%%%%%%%%%%%%%%%%%%
%%%%%%%%%%%%%%%%%%%%%%%%%

\section{On the impact of the geometry}\label{SE:GEOMETRY}

%As we have seen in the last section, the main information on the large time behavior of the entire solution $u(x,t)$ to \eqref{P} satisfying \eqref{unique} is contained in the properties of the solution, $u_\infty$, to the stationary problem \eqref{Pinf}.
So far, the geometry of $K$ has not played any role in our analysis. The main purpose of this section is to understand how the geometry of $K$ impacts the asymptotic behaviour of $u(t,x)$ as $t\to\infty$. In a nutshell, we will show that the main information on the large time behaviour is contained in the properties of the solution, $u_\infty$, to the stationary problem \eqref{Pinf}.

We will first discuss the validity of the Liouville-type property for \eqref{Pinf} depending on the geometry of $K$ (namely whether its only possible solution is $u_\infty\equiv1$). In particular, we extend some previous results of Hamel, Valdinoci and the authors to the case of a general $\delta\in\mathcal{Q}(\overline{\Omega})$ and we prove that, when $K$ is a convex set, then the Liouville-type property is satisfied (at least if $J$ is non-increasing). Second, the prove that whether $u(t,x)$ recovers the shape of the planar front $\phi(x_1+ct)$ as $t\to\infty$ is \emph{equivalent} to whether \eqref{Pinf} satisfies the Liouville-type property.

\subsection{A Liouville type result}

We establish a Liouville type result which extends the results obtained by Hamel, Valdinoci and the authors in \cite{Brasseur2019} to arbitrary quasi-Euclidean distances.
\begin{prop}[Liouville type result]\label{Liouville}
Let $K\subset\R^N$ be a compact convex set and let $\delta\in\mathcal{Q}(\overline{\Omega})$.
Assume \eqref{C2}, \eqref{C3}, \eqref{C4} and \eqref{fJdelta-pos}. If $\delta(x,y)\not\equiv|x-y|$ suppose, in addition, that $J$ is non-increasing. Let $u_\infty:\Omega\to[0,1]$ be a measurable function
satisfying
\begin{align}
\left\{
\begin{array}{rl}
Lu_\infty+f(u_\infty)=0 & \text{a.e. in }\Omega, \vspace{3pt}\\
u_\infty(x)\to1 & \text{as }|x|\to\infty.
\end{array}
\right. \label{PB:STATION}
\end{align}
Then, $u_\infty\equiv1$ a.e. in $\Omega$.
\end{prop}

\begin{proof}
If $\delta$ is the Euclidean distance, then Proposition~\ref{Liouville} is covered by \cite[Theorem 2.2]{Brasseur2019} together with \cite[Lemma 3.2]{Brasseur2019}. So it remains only to address the case when $\delta$ is not the Euclidean distance.
It turns out that this case follows from the same arguments as in the case of the Euclidean distance, with only minor changes that we now explain in detail. First of all,
%let us observe that, if $\delta(x,y)\not\equiv|x-y|$, then thanks to Proposition~\ref{COV:CVX} and the fact that $K$ is a compact convex set and that $J$ is non-increasing, it follows that $(\Omega,\delta)$ has the $J$-covering property. Similarly, if $\delta$ is the Euclidean distance (thanks to Proposition~\ref{CovEucl}). Moreover,
as we already pointed out in~\cite[Remark~2.5]{Brasseur2018}, we note that the proof of \cite[Lemma 3.2]{Brasseur2019} can be adapted to prove that~\eqref{fJdelta-pos} still implies that $u_\infty$ has a uniformly continuous representative in its class of equivalence. Hence, we may assume, without loss of generality, that $u_\infty\in C(\overline{\Omega})$.\par

The strategy of proof used in \cite{Brasseur2019} to show that $u_\infty$ is necessarily identically $1$ in the whole of $\overline{\Omega}$, consists in comparing a solution $u_\infty$ to \eqref{PB:STATION} to some planar function of the type $\phi(x\cdot e-r_0)$ with $e\in\S^{N-1}$, $r_0\in\R$ and where $\phi$ is as in \eqref{C4}. This is done using a sliding type method by letting $r$ vary from $+\infty$ to $-\infty$.

To implement this method, two ingredients are needed: first, we need to establish appropriate comparison principles and, second, we need to be able to compare a given solution $u_\infty$ to the planar function $\phi(x\cdot e-r_0)$ in half-spaces of the form
$$H_e:=x_0+\{x\in\R^N; x\cdot e>0\} \text{ with } \overline{H_e}\subset\Omega.$$

It turns out that these two ingredients adapt to our generalized setting with no difficulty. Indeed, the proof of the comparison principles \cite[Lemmata 4.1, 4.2]{Brasseur2019} require only that $(\Omega,\delta)$ has the $J$-covering property, that $L$ maps continuous functions to continuous functions and that $\mathcal{J}^\delta\in C(\overline{\Omega})$. But all these requirements are guaranteed by assumption \eqref{C3}.

\begin{figure}
\centering
\includegraphics[scale=0.5]{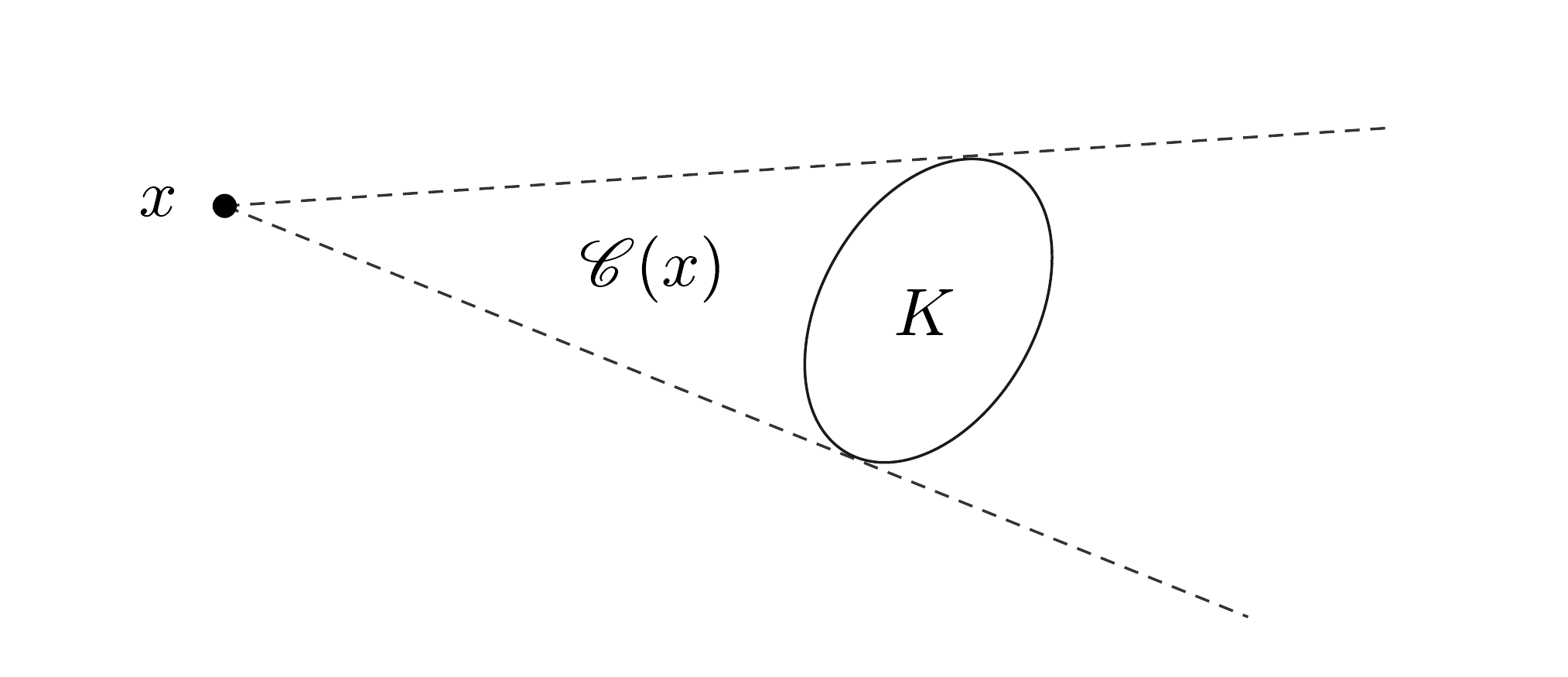} \\
\caption{The cone $\mathscr{C}(x)$ with boundary tangent to $K$. } \label{figCONE}
\end{figure}

On the other hand, to be able to compare $u_\infty$ with $\phi_{r_0,e}(x):=\phi(x\cdot e-r_0)$ in $H_e$, it suffices to make sure that $\phi_{r_0,e}$ is a sub-solution to $Lw+f(w)=0$ in $H_e$. For it, we notice that
\begin{align*}
L\phi_{r_0,e}(x)&=\int_{\Omega}J(|x-y|)(\phi_{r_0,e}(y)-\phi_{r_0,e}(x))\hspace{0.1em}\mathrm{d}y \\
&\qquad+\int_{\mathscr{C}(x)\setminus K}(J(\delta(x,y))-J(|x-y|))(\phi_{r_0,e}(y)-\phi_{r_0,e}(x))\hspace{0.1em}\mathrm{d}y,
\end{align*}
where $\mathscr{C}(x)$ is the cone with vertex $x$ tangent to $\partial K$ (see Figure~\ref{figCONE}).

Since $\phi_{r_0,e}(y)\leq\phi_{r_0,e}(x)$ for any $x\in H_e$ and any $y\in\mathscr{C}(x)\setminus K$ (because $\phi'>0$) and since $J$ is non-increasing, it then holds that
$$ L\phi_{r_0,e}(x)\geq\int_{\Omega}J(|x-y|)(\phi_{r_0,e}(y)-\phi_{r_0,e}(x))\hspace{0.1em}\mathrm{d}y \text{ for all }x\in H_e. $$
In other words, the problem reduces to the case $\delta(x,y)\equiv|x-y|$. At this stage, the arguments of \cite{Brasseur2019} can be adapted without modification.
\end{proof}

\subsection{The stationary solution encodes the geometry}

In this section, we prove that the large time behaviour of $u(t,x)$ is determined by the Liouville-type property of \eqref{Pinf}.

In fact, we will prove a bit more than what we stated above: we will prove that the stationary solution $u_\infty$ which arises in the large time limit is the \emph{minimal solution} to \eqref{Pinf}. More precisely, we prove the following result:

\begin{prop}\label{NonLiouville}
Assume that \eqref{C1}, \eqref{C2}, \eqref{C3}, \eqref{C5} and \eqref{fJdelta-pos} hold.
Suppose that $J\in \mathbb{B}_{1,\infty}^\alpha(\Omega;\delta)$ for some $\alpha\in(0,1)$.
%Assume that \eqref{C1}, \eqref{C2}, \eqref{C3}, \eqref{C4} and \eqref{fJdelta-pos} hold.
Let $u(t,x)$ be the unique bounded entire solution to \eqref{P} satisfying \eqref{unique}. Let $u_\infty\in C(\overline{\Omega})$ be the solution to \eqref{Pinf} such that \eqref{asympt:large:time2} holds, i.e. such that
$$ |u(t,x)-u_\infty(x)\hspace{0.1em}\phi(x_1+ct)|\underset{t\to+\infty}{\longrightarrow} 0  {\mbox{ locally uniformly in }}x\in\overline{\Omega},$$
and let $\widetilde{u}_\infty\in C(\overline{\Omega})$ be any solution to \eqref{Pinf}. Then, $u_\infty\leq\widetilde{u}_\infty$ in $\overline{\Omega}$.
%Suppose that \eqref{Pinf} does not satisfy the Liouville property, i.e. that \eqref{Pinf} admits a solution $\widetilde{u}_\infty\in C(\overline{\Omega})$ with $0<\widetilde{u}_\infty<1$ in $\overline{\Omega}$. Then, $0<u_\infty<1$ in $\overline{\Omega}$.
\end{prop}
Let us explain why proving Proposition~\ref{NonLiouville} is indeed sufficient to establish Theorem~\ref{NonLiouville2}.
\begin{proof}[Proof of Theorem~\ref{NonLiouville2}]
If \eqref{Pinf} satisfies the Liouville property, then, since the trivial solution is the only possible one, we clearly have that $u_\infty\equiv1$. On the other hand, if $u_\infty\equiv 1$, then either \eqref{Pinf} satisfies the Liouville property or it does not. Suppose, by contradiction, that \eqref{Pinf} does not satisfy the Liouville property, namely that there exists a solution $\widetilde{u}_\infty$ to \eqref{Pinf} with $0<\widetilde{u}_\infty<1$ a.e. in $\overline{\Omega}$. Because of assumption \eqref{fJdelta-pos}, by \cite[Lemma 3.2]{Brasseur2019}, we know that every solution to \eqref{Pinf} admits a representative in its class of equivalence that is uniformly continuous. Hence, we may always assume that $\widetilde{u}_\infty\in C(\overline{\Omega})$. Applying now Proposition~\ref{NonLiouville}, we find that $1\equiv u_\infty\leq \widetilde{u}_\infty<1$ in $\overline{\Omega}$, a contradiction.
\end{proof}
Let us now prove Proposition~\ref{NonLiouville}.
\begin{proof}[Proof of Proposition~\ref{NonLiouville}]
For the convenience of the reader, the proof is split into three parts. After a preparatory step, where we collect some preliminary observations, we show that any solution to \eqref{Pinf} bounds $u(\tau,x)$ from above, for some time $\tau\in\R$ in a neighborhood of $-\infty$.
Lastly, we show that this estimate holds for all $t\in(\tau,\infty)$ using the comparison principle and we conclude using the convergence result obtained in Proposition~\ref{TH:ainf}.
\vskip 0.3cm

\noindent\emph{Step 1. Preliminary observations}
\smallskip

Let $\widetilde{u}_\infty\in C(\overline{\Omega})$ be any solution to \eqref{Pinf} and let $s_0,s_1>0$ be such that $f'\leq-s_1$ in $[1-s_0,1]$ (note that $s_0,s_1$ are well-defined since $f'(1)<0$).
Observe that, since $\widetilde{u}_\infty$ is independent of $t$, it also satisfies
\begin{align}
\partial_t\widetilde{u}_\infty-L\widetilde{u}_\infty-f(\widetilde{u}_\infty)=0 \text{ in }\R\times\overline{\Omega}. \label{stat-evol}
\end{align}
Furthermore, since $\inf_{\overline{\Omega}}\widetilde{u}_\infty>0$ (by the strong maximum principle \cite[Lemma 4.2]{Brasseur2019}) we may apply \cite[Lemma 5.1]{Brasseur2019} which yields the existence of a number $r_0>0$ such that
\begin{align}
\phi(|x|-r_0)\leq \widetilde{u}_\infty(x) \text{ for any }x\in\overline{\Omega}, \label{philequinfini}
\end{align}
where $\phi$ is as in \eqref{C4}. (Note that the use of \cite[Lemmata 4.2, 5.1]{Brasseur2019} in the case of a general $\delta\in\mathcal{Q}(\overline{\Omega})$ is licit, as can easily be seen by reasoning as in the proof of Proposition~\ref{Liouville}.)

Lastly, we recall that, by construction of $u(t,x)$, we have that
\begin{align}
u(t,x)\leq w^+(t,x) \text{ for any }(t,x)\in(-\infty,T_1]\times\overline{\Omega}, \label{uleqwplus}
\end{align}
where $w^+$ is given by \eqref{wplus} (remember \eqref{wuwENC}).
\vskip 0.3cm

\noindent\emph{Step 2. A first upper bound}
\smallskip

%Let us now prove that $u(t,x)$ is bounded from above by $\widetilde{u}_\infty(x)$ for some $t\in\R$ in a neighborhood of $-\infty$.
%To this end, we
Let $R_K>0$ be such that $K\subset B_{R_K}$ and $\widetilde{u}_\infty\geq1-s_0$ in $\R^N\setminus B_{R_K}$. Also, let $\tau\in(-\infty,T_1]$ be sufficiently negative so that $c\tau+\xi(\tau)+r_0\leq0$ and
$$ \max\Big\{2\hspace{0.1em}\phi(c\tau\!+\!\xi(\tau))-\phi(-r_0),\ \beta_0+(c\tau\!+\!\xi(\tau)\!+\!r_0)\min\{\gamma_0,\gamma_1\hspace{0.1em}e^{-\mu R_K}\}\Big\}\leq0, $$
where $\mu$, $\beta_0$, $\gamma_0$, $\gamma_1$ and $\xi(t)$ are given by \eqref{EQ:CHARAC2}, \eqref{EST:PHI:NEG}, \eqref{EST:PHI:POS} and \eqref{xi(t)}, respectively. (Note that $\tau$ is well-defined since $\xi(t)\to0$ as $t\to-\infty$, since $\phi(z)\to0$ as $z\to-\infty$, since $\phi'(z)>0$ for all $z\in\R$ and since $\gamma_0,\gamma_1>0$.)

Now, we notice that, if $x_1<0$, then, by \eqref{philequinfini} and \eqref{uleqwplus}, we have
\begin{align*}
u(\tau,x)-\widetilde{u}_\infty(x)&\leq w^+(\tau,x)-\phi(|x|-r_0)\leq 2\hspace{0.1em}\phi(c\tau+\xi(\tau))-\phi(-r_0)\leq0,
\end{align*}
where we have used that $\phi$ is increasing.
%where we have used that $t\leq\tau$.
In other words, we have that
\begin{align}
u(\tau,x)\leq \widetilde{u}_\infty(x) \text{ for any }x\in\overline{\Omega} \text{ with } x_1<0. \label{DEF:TAU1}
\end{align}
Similarly, if $x_1\geq0$, then $|x|\geq x_1$ and we have
\begin{align}
u(\tau,x)-\widetilde{u}_\infty(x)&\leq u(\tau,x)-\phi(|x|-r_0) \nonumber \\
&\leq \phi(x_1+c\tau+\xi(\tau))+\phi(-x_1+c\tau+\xi(\tau))-\phi(x_1-r_0) \nonumber \\
&\leq \phi(c\tau+\xi(\tau))+\phi'(x_1+\Theta)(c\tau+\xi(\tau)+r_0) \text{ for some }\Theta\in[c\tau+\xi(\tau),-r_0]. \nonumber
\end{align}
Let us now consider three subcases. First, if $0\leq x_1\leq-\Theta$, then, by \eqref{EST:PHI:NEG} and \eqref{EST:PHI:POS}, we have
\begin{align*}
u(\tau,x)-\widetilde{u}_\infty(x)&\leq e^{\lambda\Theta}\big(\beta_0+\gamma_0(c\tau+\xi(t)+r_0)\big)\leq0.
\end{align*}
%for all $t\in(-\infty,\tau]$.
%where, again, we have used that $t\leq\tau$.
Therefore, we have that
%Since $ct+\xi(t)\to-\infty$ as $t\to-\infty$ and $\beta_0,\gamma_0>0$, we may find some $-\infty<\tau_2\leq T_1$ so that
\begin{align}
u(\tau,x)\leq \widetilde{u}_\infty(x) \text{ for any }x\in\overline{\Omega} \text{ with } 0\leq x_1\leq-\Theta. \label{DEF:TAU2}
\end{align}
Now, if $-\Theta<x_1\leq R_K-\Theta$, then $\phi'(x_1+\Theta)\geq \gamma_1\hspace{0.1em} e^{-\mu(x_1+\Theta)}$ (by \eqref{EST:PHI:NEG} and \eqref{EST:PHI:POS}). Hence,
\begin{align*}
u(\tau,x)-\widetilde{u}_\infty(x)&\leq \beta_0\hspace{0.1em} e^{\lambda(c\tau+\xi(\tau))}+\gamma_1(c\tau+\xi(\tau)+r_0)e^{-\mu(x_1+\Theta)} \\
&\leq \beta_0+\gamma_1(c\tau+\xi(\tau)+r_0)\hspace{0.1em}e^{-\mu R_K}\leq0.
\end{align*}
Thus, we have that
%Thus, we may find some $-\infty<\tau_3\leq T_1$ so that
\begin{align}
u(\tau,x)\leq \widetilde{u}_\infty(x) \text{ for any }x\in\overline{\Omega} \text{ with }-\Theta<x_1\leq R_K-\Theta. \label{DEF:TAU3}
\end{align}
%Let us define $\tau$ to be the following number:
%$$ \tau:=\min\{\tau_1,\tau_2,\tau_3\}. $$
Finally, let us consider the case $x_1>R_K-\Theta$. Let $H$ be the half-space given by
$$H:=\left\{x\in\R^N; x_1>R_K-\Theta\right\}\subset\R^N\setminus B_{R_K}. $$
Since $\partial_tu>0$ in $\R\times\overline{\Omega}$, using \eqref{DEF:TAU1}, \eqref{DEF:TAU2} and \eqref{DEF:TAU3}, we have
\begin{align*}
\left\{
\begin{array}{rl}
L\widetilde{u}_\infty+f(\widetilde{u}_\infty)= 0 & {\mbox{in }}\overline{H},\vspace{3pt}\\
Lu(\tau,\cdot)+f(u(\tau,\cdot))\ge 0 & {\mbox{in }}\overline{H},\vspace{3pt}\\
u(\tau,\cdot)\leq\widetilde{u}_\infty & {\mbox{in }}\overline{\Omega}\setminus H.
\end{array}
\right.
\end{align*}
%for all $t\in(-\infty,\tau]$.
Since, in addition, it holds that $\limsup_{|x|\to\infty}\left(u(\tau,x)-\widetilde{u}_\infty(x)\right)\leq0$ and that  $\widetilde{u}_\infty\geq1-s_0$ in $\overline{H}$ (remember the definition of $R_K$), we may then apply the weak maximum principle \cite[Lemma 4.1]{Brasseur2019} (which we can do as pointed out in the proof of Proposition~\ref{Liouville}) to obtain that
\begin{align}
u(\tau,x)\leq \widetilde{u}_\infty(x) \text{ for any }x\in\overline{\Omega}. \label{CDN:init:liouv}
\end{align}
It remains to show that this estimate holds for all $t\in(\tau,\infty)$.
\vskip 0.3cm

\noindent\emph{Step 3. Conclusion}
\smallskip

Let $T_*>\tau$ be arbitrary. Then, using \eqref{CDN:init:liouv} and recalling \eqref{stat-evol}, we arrive at
\begin{align*}
\left\{
\begin{array}{rll}
\partial_t\widetilde{u}_\infty-L\widetilde{u}_\infty-f(\widetilde{u}_\infty)\!\!&\geq \,\partial_tu-Lu-f(u) & \mbox{in }(\tau,T_*]\times\overline{\Omega}, \vspace{3pt}\\
\widetilde{u}_\infty(\cdot)\!\!&\geq \,u(\tau,\cdot) & \mbox{in }\overline{\Omega}.
\end{array}
\right.
\end{align*}
Hence, by the comparison principle Lemma~\ref{LE:COMPARISON}, we obtain that $u(t,x)\leq \widetilde{u}_\infty(x)$ for any $(t,x)\in[\tau,T_*]\times\overline{\Omega}$. But since $T_*>\tau$ is arbitrary, we find that
$$ u(t,x)\leq \widetilde{u}_\infty(x) \text{ for any }(t,x)\in[\tau,\infty)\times\overline{\Omega}. $$
Using Proposition~\ref{TH:ainf}, we obtain $u_\infty\leq \widetilde{u}_\infty$ in $\overline{\Omega}$, which completes the proof.
\end{proof}

%%%%%%%%%%%%%%%%%%%%%%%%%
%%%%%%%%%%%%%%%%%%%%%%%%%

\section{Large time behaviour of super-level sets}\label{SE:LargeTimeSLS}

In this section, we characterise further the large time behaviour of the entire solution $u(t,x)$. Precisely, we will prove that its super-level sets, namely the sets given by
$$ E_\lambda(t):=\left\{x\in\overline{\Omega}\hspace{0.1em};\ u(t,x)\geq\lambda\right\} \text{ for }\lambda\in(0,1)\text{ and }t\in\R, $$
are trapped between two moving frames which move at speed is $c$ when $t$ is large enough. In addition to the intrinsic interest of the results contained in this section, they will allow us to pave the way for the characterisation of the entire solution $u(t,x)$ as a generalised transition front with global mean speed $c$. The main result of this section is the following:
\begin{prop}[Large time behaviour of super-level sets]\label{PROP:SUPERLEVEL}
Assume \eqref{C1}, \eqref{C2}, \eqref{C3}, \eqref{C5} and \eqref{fJdelta-pos}.
Suppose that $J\in \mathbb{B}_{1,\infty}^\alpha(\Omega;\delta)$ for some $\alpha\in(0,1)$. Let $u(t,x)$ be the unique entire solution to \eqref{P} satisfying \eqref{qualit-u} and \eqref{origintimes}. Then, the following properties hold:
\begin{enumerate}
\item[(i)] (Upper bound) For all $\lambda\in(0,1)$, there exists a time $t_1=t_1(\lambda)\geq0$ and a constant $\Gamma_0=\Gamma_0(\lambda)\in\R$ such that, for all $t\geq t_1$, there holds
$$ E_\lambda(t)\subset\left\{x\in\overline{\Omega}\hspace{0.1em};\ x_1\geq \Gamma_0-ct\right\}. $$
\item[(ii)] (Lower bound) For all $\lambda\in(0,1)$, there are some $t_2=t_2(\lambda)\geq0$, $\Gamma_1=\Gamma_1(\lambda)\in\R$ and $\Gamma_2=\Gamma_2(\lambda)\in\R$ such that, for all $t\geq t_2$, there holds
$$ E_\lambda(t)\supset\left\{x\in\overline{\Omega}\hspace{0.1em};\ \Gamma_1\geq x_1\geq \Gamma_2-ct\right\}. $$
\end{enumerate}
\end{prop}
Let us make some comments on the strategy of the proof of Proposition~\ref{PROP:SUPERLEVEL}.
The heart of the proof consists in comparing $u(t,x)$ with a sub- and a supersolution of the problem
\begin{align}
\partial_tv=L_{\R^N}v+f(v), \label{EQ:CONVOLUTION}
\end{align}
where $L_{\R^N}v:=J_{\mathrm{rad}}\ast v-v$. %The most difficult part of the proof is the construction of an appropriate subsolution.
The construction of the subsolution is slightly more involved than that of the supersolution due to the fact that the large time limit of $u(t,x)$, namely $u_\infty(x)$, may not be identically $1$ in the whole of $\overline{\Omega}$. This forces us to consider subsolutions to \eqref{EQ:CONVOLUTION} in some strips of $\R^N$ which move at speed $c'<c$. This will allow us to obtain a preliminary lower bound showing that the super-level sets move asymptotically at speed $c$. The desired lower bound on $E_\lambda(t)$ will follow using a contradiction argument. For the convenience of the reader, the proof of Proposition~\ref{PROP:SUPERLEVEL} is divided into three subsections, each of which corresponds to a step of the proof of Proposition~\ref{PROP:SUPERLEVEL}.

\subsection{An upper bound: the super-level sets move at speed at most $c$}
In this subsection, we establish the upper bound on the super-level sets of $u(t,x)$ stated at Proposition~\ref{PROP:SUPERLEVEL} (namely assertion $(i)$). To this end, we need to construct a supersolution to an auxiliary problem. This supersolution and the comparison principle will provide an upper bound on $u(t,x)$, from which Proposition~\ref{PROP:SUPERLEVEL}(i) will follow.
So let us first prove the following lemma:
\begin{lemma}\label{LE:SUPERSOL:LEVEL}
There exists $\alpha_0,\beta_0>0$ such that, for any $0\leq\alpha\leq\alpha_0$ and any $0\leq\beta\leq\beta_0$, there exists some $\kappa_0=\kappa_0(\alpha)>0$ such that, for all $\kappa\geq\kappa_0$ and all $\rho\in\R$, the function
\begin{align}
w_\rho^+(t,x):=\phi\left(x_1+ct+\rho+\kappa\hspace{0.1em}\beta\hspace{0.1em}(1-e^{-\alpha t})\right)+\beta\hspace{0.1em}e^{-\alpha t}, \label{wplus:level}
\end{align}
is a supersolution to \eqref{EQ:CONVOLUTION} in $[0,\infty)\times\R^N$.
\end{lemma}
\begin{proof}
Let $\rho>0$ be arbitrary and set $\xi(t):=ct+\rho+\kappa\hspace{0.1em}\beta\hspace{0.1em}(1-e^{-\alpha t})$.
Let us extend $f$ linearly outside $[0,1]$ by the function $\widetilde{f}$ given by \eqref{extension:f}. For the sake of simplicity, let us still denote by $f$ this extension. Now, let $\beta_0>0$ be small enough so that
$$ -f'(s)\geq-\max\big\{f'(0),f'(1)\big\}=:\alpha_0 \text{ for all }s\in[-2\beta_0,2\beta_0]\cup[1-2\beta_0,1+2\beta_0]. $$
Furthermore, let $A>0$ be large enough so that
$$
\left\{
\begin{array}{ll}
\phi(z)\leq\beta_0 &\text{if }z<-A, \vspace{3pt}\\
\phi(z)\geq1-\beta_0&\text{if }z> A.
\end{array}
\right.
$$
Note that $A$ is well-defined since $\lim_{z\to-\infty}\phi(z)=1$ and $\lim_{z\to+\infty}\phi(z)=1$.
Now, let $(t,x)\in[0,\infty)\times\R^N$ be arbitrary. By a short computation, we obtain that
$$ \partial_tw_\rho^+(t,x)=\left(c+\alpha\beta\kappa\hspace{0.1em}e^{-\alpha t}\right)\phi'(x_1+\xi(t))-\alpha\beta\hspace{0.1em}e^{-\alpha t}. $$
Recalling \eqref{C4}, this rewrites
\begin{align}
\partial_tw_\rho^+(t,x)&=L_{\R^N}w_\rho^+(t,x)+f\left(w_\rho^+(t,x)-\beta\hspace{0.1em}e^{-\alpha t}\right)+\alpha\beta\kappa\hspace{0.1em}e^{-\alpha t}\phi'(x_1+\xi(t))-\alpha\beta\hspace{0.1em}e^{-\alpha t}. \label{eq-w+1}
\end{align}
Let us now distinguish between two different cases.
\vskip 0.3cm

\noindent\emph{Case 1: $|x_1+\xi(t)|> A$.}
\smallskip

Given the definition of $w_\rho^+(t,x)$ and $A$, we must have that either $w_\rho^+(t,x)-\beta\hspace{0.1em}e^{-\alpha t}\leq\beta_0$ or $w_\rho^+(t,x)-\beta\hspace{0.1em}e^{-\alpha t}\geq 1-\beta_0$. In both situations, we have
$$ f\left(w_\rho^+(t,x)-\beta\hspace{0.1em}e^{-\alpha t}\right)-f(w_\rho^+(t,x))\geq \alpha_0\beta\hspace{0.1em}e^{-\alpha t}, $$
for all $0\leq\beta\leq\beta_0$ (recall that $t\geq0$ so that $\beta\hspace{0.1em} e^{-\alpha t}\leq\beta$). Since $\phi'>0$, using \eqref{eq-w+1}, we get
$$ \partial_tw_\rho^+(t,x)\geq L_{\R^N}w_\rho^+(t,x)+f(w_\rho^+(t,x))+\beta(\alpha_0-\alpha)\hspace{0.1em}e^{-\alpha t}. $$
Thus, $w_\rho^+$ is a supersolution in the range $|x_1+\xi(t)|> A$ provided $0\leq\alpha\leq\alpha_0$ and $0\leq\beta\leq\beta_0$.
\vskip 0.3cm

\noindent\emph{Case 2: $|x_1+\xi(t)|\leq A$.}
\smallskip

Let $\zeta:=\inf_{z\in[-A,A]}\phi'(z)>0$. Since $f\in C^{0,1}(\R)$, it follows from \eqref{eq-w+1} that
$$ \partial_tw_\rho^+(t,x)\geq L_{\R^N}w_\rho^+(t,x)+f(w_\rho^+(t,x))+\beta\hspace{0.1em}(\zeta\alpha\kappa-\|f'\|_\infty-\alpha)\hspace{0.1em}e^{-\alpha t}. $$
Thus, we, again, have that $w_\rho^+$ is a supersolution in the range $|x_1+\xi(t)|\leq A$ provided $\kappa\geq\kappa_0(\alpha)$, where we have set $\kappa_0(\alpha):=(\|f'\|_\infty+\alpha)/(\zeta\alpha)$.

Summing up, we have shown that $w_\rho^+$ is a supersolution to \eqref{EQ:CONVOLUTION} for all $(t,x)\in[0,\infty)\times\R^N$ provided that $0\leq\alpha\leq\alpha_0$, that $0\leq\beta\leq\beta_0$ and that $\kappa\geq\kappa_0(\alpha)$, as desired.
\end{proof}
\begin{remark}
Observe that $\alpha_0$ and $\beta_0$ depend only on $f'$.
\end{remark}
We are now in position to prove Proposition~\ref{PROP:SUPERLEVEL}(i).
\begin{proof}[Proof of Proposition~\ref{PROP:SUPERLEVEL}(i)]
Since $K$ is compact, up to immaterial translations, we may assume, without loss of generality, that $K\subset\{x_1\geq-R_J\}$ where $R_J>0$ is such that $\mathrm{supp}(J)\subset[0,R_J]$. By Lemma~\ref{LE:SUPERSOL:LEVEL}, we know that there exists $\alpha_0>0$ and $\beta_0>0$ such that, for all $0\leq\alpha\leq\alpha_0$ and all $0\leq\beta\leq\beta_0$, there is some $\kappa_0=\kappa_0(\alpha)>0$ such that, for all $0\leq\kappa\leq\kappa_0$ and all $\rho\in\R$, the function $w_\rho^+$ given by \eqref{wplus:level}, namely
$$ w_\rho^+(t,x):=\phi\left(x_1+ct+\rho+\kappa\hspace{0.1em}\beta\hspace{0.1em}(1-e^{-\alpha t})\right)+\beta\hspace{0.1em}e^{-\alpha t}, $$
is a supersolution to \eqref{EQ:CONVOLUTION} for all $(t,x)\in[0,\infty)\times\R^N$.
Since $u(0,x)\to0$ as $x_1\to-\infty$ (by Proposition~\ref{PROP:LIMITE:X1}), since $w(0,x)\geq\beta$ for all $x\in\overline{\Omega}$ and since $\lim_{z_1\to\infty}w_\rho^+(0,z)=1+\beta>1>u(0,x)$ for all $x\in\overline{\Omega}$, up to take $\rho>0$ sufficiently large, we may always assume that
\begin{align}
w_\rho^+(0,x)\geq u(0,x) \text{ for all }x\in\overline{\Omega}. \label{super:lev:init}
\end{align}
On the other hand, by exploiting the asymptotic properties of $\phi$ (Lemma~\ref{PHI:ASYM}), we have
\begin{align*}
w_\rho^+(t,x)&\geq 1-\gamma_1 e^{-\mu(x_1+ct+\rho +\kappa\beta(1-e^{-\alpha t}))} +\beta\hspace{0.1em} e^{-\alpha t}\\
&= 1 +e^{-\alpha t}\left(\beta-\gamma_1e^{\mu(2R_J-\rho)}e^{-\mu\kappa\beta(1-e^{-\alpha t})}e^{-(\mu c-\alpha)t} \right),
\end{align*}
for all $(t,x)\in[0,\infty)\times\{x_1\geq-2R_J\}$.
Hence, choosing $0\leq\alpha\leq\min\left\{\alpha_0,\mu c\right\}$ and $\rho\geq 2R_J-\mu^{-1}\log(\beta/\gamma_1)$,
we have that
\begin{align}
w_\rho^+(t,x)\geq 1>u(t,x) \text{ for all }(t,x)\in[0,\infty)\times\{x_1\geq-2R_J\}. \label{super:lev:ext}
\end{align}
Since $K\subset\{x_1\geq-R_J\}$ it follows that $u(t,x)$ is also a solution to \eqref{EQ:CONVOLUTION} in $(t,x)\in[0,\infty)\times\{x_1\leq-2R_J\}$. Hence, recalling \eqref{super:lev:init} and \eqref{super:lev:ext}, it follows from the comparison principle Lemma~\ref{LE:COMPARISON} (remember Remark~\ref{RE:COMPARISON}) that
\begin{align}
u(t,x)\leq w_\rho^+(t,x)\leq \phi\left(x_1+ct+\rho+\kappa\hspace{0.1em}\beta\right)+\beta\hspace{0.1em}e^{-\alpha t}, \label{superlevel:upper:est}
\end{align}
for all $(t,x)\in[0,\infty)\times\overline{\Omega}$.
Now, let $\lambda\in(0,1)$ and let $t_1(\lambda):=\max\{0,\alpha^{-1}\log(\lambda/(2\beta))\}$. Then, for all $t\geq t_1(\lambda)$ and all $x\in E_\lambda(t)$, we have $\Gamma_0(\lambda):=\phi^{-1}(\lambda/2)-(\rho+\kappa\hspace{0.1em}\beta)\leq  x_1+ct$.
Therefore, the following inclusion holds
$$ E_\lambda(t)\subset \left\{x\in\overline{\Omega}\hspace{0.1em};\ x_1\geq \Gamma_0(\lambda)-ct\right\}, $$
for all $t\geq t_1(\lambda)$, which thereby proves Proposition~\ref{PROP:SUPERLEVEL}(i).
\end{proof}

\subsection{A lower bound: the super-level sets move asymptotically at speed $c$}
In this subsection, we establish a preliminary lower bound for the super-level sets of $u(t,x)$ which will be useful to complete the proof of Proposition~\ref{PROP:SUPERLEVEL}(ii) and, hence, of Proposition~\ref{PROP:SUPERLEVEL}. This preliminary lower bound will allow us to obtain an asymptotic version of Proposition~\ref{PROP:SUPERLEVEL}(ii). Namely, we will prove the following intermediary result:
\begin{lemma}\label{LE:intermediaire}
For all $\lambda\in(0,1)$, there is some $n_0=n_0(\lambda)\geq1$ such that, for all $n\geq n_0$, there exists $\Gamma_{1,n}\in\R$, $\Gamma_{2,n}(\lambda)\in\R$ and $t_{2,n}(\lambda)\geq0$ such that, for all $t\geq t_{2,n}(\lambda)$, there holds
$$ E_\lambda(t)\supset\left\{x\in\overline{\Omega}\hspace{0.1em};\ \Gamma_{1,n}\geq x_1\geq \Gamma_{2,n}(\lambda)-c\left(1-\frac{1}{n}\right)t\right\}. $$
\end{lemma}

For the convenience of the reader, let us first introduce a few notations which will be extensively used in this subsection.
For any fixed $R_0\in\R$, we let $\chi_{R_0}\in C^2(\R,[0,1])$ be a cut-off function such that
$$ \chi_{R_0}\equiv0 \text{ in }\{x_1\geq R_0+2R_J\} \text{ and } \chi_{R_0}\equiv 1 \text{ in }\{x_1\leq R_0+R_J\}. $$
Then, we have the following result:
\begin{lemma}\label{LE:sup:intermediaire}
There exists $\alpha_0,\beta_0>0$ such that, for any $c'<c$, there exists $\eps_0>0$ with the property that, for all $0<\eps\leq\eps_0$, all $0\leq\alpha\leq\alpha_0$ and all $0\leq\beta\leq\beta_0$, there is some constant $\kappa=\kappa(\alpha)>0$ for which the function
$$ w_\rho^-(t,x):=(1-\eps)\hspace{0.1em}\phi\left(x_1+c't-\rho-\kappa\hspace{0.1em}\beta(1-e^{-\alpha t})\right)\chi_{R_0}(x_1)-\beta\hspace{0.1em}e^{-\alpha t}, $$
is a subsolution to \eqref{EQ:CONVOLUTION} in $[0,\infty)\times\{x_1\leq R_0\}$ for all $\rho\in\R$ and all $R_0\in\R$.
\end{lemma}
\begin{proof}
Let $\rho>0$ and $c'<c$ be arbitrary and set $\xi(t):=c't-\rho-\kappa\hspace{0.1em}\beta\hspace{0.1em}(1-e^{-\alpha t})$.
Let us extend $f$ linearly outside $[0,1]$ by the function $\widetilde{f}$ given by \eqref{extension:f}. For the sake of simplicity, let us still denote by $f$ this extension. Now, let $0<\beta_0<\min\{1-\theta,\theta\}/4$ be small enough so that
\begin{align}
\left\{
\begin{array}{cl}
\frac{5}{4}\hspace{0.1em}f'(0)\leq f'(s)\leq \frac{3}{4}\hspace{0.1em}f'(0) &\text{for }s\in[-4\beta_0, 4\beta_0], \vspace{3pt}\\
f'(s)\leq \frac{1}{2}\hspace{0.1em}f'(1) &\text{for }s\in[1-4\beta_0,1].
\end{array}
\right. \label{bc-eq-f}
\end{align}
As in the proof of Lemma~\ref{LE:SUPERSOL:LEVEL}, we may find some number $A>0$ such that
\begin{align}
\left\{
\begin{array}{ll}
\phi(z)\leq\beta_0/2 &\text{if }z<-A, \vspace{3pt}\\
\phi(z)\geq1-\beta_0/2&\text{if }z> A.
\end{array}
\right. \label{bc-eq-phi:A}
\end{align}
Now, let $(t,x)\in[0,\infty)\times\{x_1\leq R_0\}$ be arbitrary. A short computation shows that
$$ \partial_tw_\rho^-(t,x)=(1-\eps)(c'-\alpha\beta\kappa\hspace{0.1em}e^{-\alpha t})\hspace{0.1em}\phi'(x_1+\xi(t))+\alpha\beta\hspace{0.1em}e^{-\alpha t}. $$
Recalling \eqref{C4}, this rewrites
\begin{align*}
\partial_tw_\rho^-(t,x)&=(1-\eps)L_{\R^N}\phi(x_1+\xi(t))+(1-\eps)\hspace{0.1em}f(\phi(x_1+\xi(t))) \\
&\qquad +(1-\eps)(c'-c-\alpha\beta\kappa\hspace{0.1em}e^{-\alpha t})\hspace{0.1em}\phi'(x_1+\xi(t))+\alpha\beta\hspace{0.1em}e^{-\alpha t}.
\end{align*}
Since $x_1\leq R_0$, since $\mathrm{supp}(J)\subset[0,R_J]$, since $\chi_{R_0}\equiv 1$ in $\{x_1\leq R_0+R_J\}$ and since $L_{\R^N}(\phi+v)=L_{\R^N}\phi$ for all functions $v$ independent of $x$, we have $(1-\eps)L_{\R^N}\phi(x_1+\xi(t))=L_{\R^N}w_\rho^-(t,x)$ for all $x\in\{x_1\leq R_0\}$.
Therefore, we have
$$ \partial_tw_\rho^-(t,x)=L_{\R^N}w_\rho^-(t,x)+f(w_\rho^-(t,x))+\mathcal{I}(t,x) \text{ for all }x\in\{x_1\leq R_0\}, $$
where $\mathcal{I}(t,x)$ denotes the following expression
\begin{align*}
\mathcal{I}(t,x)&:=(1-\eps)(c'-c-\alpha\beta\kappa\hspace{0.1em}e^{-\alpha t})\hspace{0.1em}\phi'(x_1+\xi(t))+\alpha\beta\hspace{0.1em}e^{-\alpha t} \\
&\qquad+(1-\eps)\hspace{0.1em}f(\phi(x_1+\xi(t)))-f(w_\rho^-(t,x)).
\end{align*}
Hence, to prove that $w_\rho^-(t,x)$ is a subsolution to \eqref{EQ:CONVOLUTION} for all $(t,x)\in[0,\infty)\times\{x_1\leq R_0\}$ it suffices to show that $\mathcal{I}(t,x)\leq0$ for all $(t,x)\in[0,\infty)\times\{x_1\leq R_0\}$. To this end, we distinguish between three different cases.
\vskip 0.3cm

\noindent\emph{Case 1: $x_1+\xi(t)<-A$.}
\smallskip

Let us first observe that
\begin{align*}
(1-\eps)\hspace{0.1em}f(\phi(x_1+\xi(t)))-\!f(w_\rho^-(t,x))=(1-\eps)\left[f(\phi(x_1+\xi(t)))-f(w_\rho^-(t,x))\right]-\eps\hspace{0.1em}f(w_\rho^-(t,x)).
\end{align*}
Since $x_1+\xi(t)<-A$, we have $\phi(x_1+\xi(t))\leq\beta_0/2$ (by definition of $A$). Since, in addition, $t\geq0$ and $\eta\equiv1$ in $\{x_1\leq R_0\}$, we also have $-3\beta_0/2\leq w_\rho^-(t,x)\leq \beta_0/2$. Hence, recalling the definition of $w_\rho^-(t,x)$ and of $\beta_0$ (remember \eqref{bc-eq-f}), we obtain that
\begin{align}
f(\phi(x_1+\xi(t)))-f(w_\rho^-(t,x))&\leq \frac{3}{4}f'(0)\left(\eps\hspace{0.1em}\phi(x_1+\xi(t))+\beta\hspace{0.1em}e^{-\alpha t}\right). \label{fxi:cond1:cas}
\end{align}
Now, if $w_\rho^-(t,x)=\phi(x_1+\xi(t))-\beta\hspace{0.1em}e^{-\alpha t}\leq0$, then, since $f(s)=f'(0)\hspace{0.1em}s\geq0$ for all $s\leq0$ (by construction), we have $-\eps\hspace{0.1em}f(w_\rho^-(t,x))\leq0$. In this situation, since $c>c'$, $\phi>0$ and $\phi'>0$, it follows that
$$ \mathcal{I}(t,x)\leq \alpha\beta\hspace{0.1em}e^{-\alpha t}+\frac{3}{4}f'(0)\left(\eps\hspace{0.1em}\phi(x_1+\xi(t))+\beta\hspace{0.1em}e^{-\alpha t}\right)\leq \beta\hspace{0.1em}e^{-\alpha t}\left(\alpha+\frac{3}{4}f'(0)\right). $$
Therefore, $\mathcal{I}(t,x)\leq0$ provided $0\leq\alpha\leq -3f'(0)/4$.

On the other case, if $w_\rho^-(t,x)=\phi(x_1+\xi(t))-\beta\hspace{0.1em}e^{-\alpha t}>0$, then, since $\phi(x_1+\xi(t))\leq\beta_0/2$, we also have $w_\rho^-(t,x)\leq\beta_0/2$ and, by \eqref{bc-eq-f}, there holds
\begin{align}
-\eps\hspace{0.1em}f(w_\rho^-(t,x))=\eps\left[f(0)-f(w_\rho^-(t,x))\right]\leq -\eps\hspace{0.1em}\frac{5}{4}f'(0)\hspace{0.1em}w_\rho^-(t,x). \label{fxi:cond2:cas}
\end{align}
Since the right-hand side of \eqref{fxi:cond1:cas} (resp. \eqref{fxi:cond2:cas}) is nonpositive (resp. nonnegative), we have
\begin{align*}
(1-\eps)\hspace{0.1em}&f(\phi(x_1+\xi(t)))-\!f(w_\rho^-(t,x))\leq \frac{f'(0)}{2}\left(\eps\hspace{0.1em}\phi(x_1+\xi(t))+\beta\hspace{0.1em}e^{-\alpha t}\right)-\eps\hspace{0.1em}\frac{3}{2}f'(0)\hspace{0.1em}w_\rho^-(t,x) \\
&=-(1-\eps)\hspace{0.1em}\eps\hspace{0.1em}\phi(x_1+\xi(t))f'(0)+2\eps\hspace{0.1em}f'(0)\beta\hspace{0.1em}e^{-\alpha t}+\eps^2\frac{f'(0)}{2}\phi(x_1+\xi(t))+\frac{f'(0)}{2}\beta\hspace{0.1em}e^{-\alpha t} \\
&\leq -(1-\eps)\hspace{0.1em}\eps\hspace{0.1em}\phi(x_1+\xi(t))f'(0)+\frac{f'(0)}{2}\hspace{0.1em}\beta\hspace{0.1em} e^{-\alpha t}\left(1+4\eps\right).
\end{align*}
Plugging this in the definition of $\mathcal{I}(t,x)$, we obtain
\begin{align*}
\mathcal{I}(t,x)&\leq \beta\hspace{0.1em}e^{-\alpha t}\left(\frac{1+4\eps}{2}f'(0)+\alpha-\alpha\kappa\hspace{0.1em}(1-\eps)\hspace{0.1em}\phi'(x_1+\xi(t))\right) \\
&\qquad+(1-\eps)\Big((c'-c)\hspace{0.1em}\phi'(x_1+\xi(t))-\eps\hspace{0.05em}\phi(x_1+\xi(t))f'(0)\Big).
\end{align*}
Recalling \eqref{EST:PHI:NEG}, we know that there is a constant $\sigma_0>0$ such that $\phi(z)\leq \sigma_0\hspace{0.1em}\phi'(z)$ for all $z<0$. Hence, using that $c'<c$ and that $\phi'>0$, we obtain
\begin{align*}
\mathcal{I}(t,x)&\leq \beta\hspace{0.1em}e^{-\alpha t}\left(\frac{1+4\eps}{2}f'(0)+\alpha\right)+(1-\eps)\left(\frac{c'-c}{\sigma_0}-\eps f'(0)\right)\phi(x_1+\xi(t)).
\end{align*}
Therefore, $\mathcal{I}(t,x)\leq0$ provided $0\leq\alpha\leq-f'(0)/2$ and $0<\eps\leq\min\{1/2,(c'-c)/(\sigma_0f'(0))\}$.
\vskip 0.3cm

\noindent\emph{Case 2: $x_1+\xi(t)>-A$.}
\smallskip

In this situation, we have $\phi(x_1+\xi(t))\geq 1-\beta_0/2$ (by construction of $A$). Since $\beta_0/2<1-\theta$ (by construction of $\beta_0$), we further have that $f(\phi(x_1+\xi(t)))>0$. Hence, using \eqref{bc-eq-f}, we find that, for all $0<\eps\leq\beta_0$, there holds
\begin{align}
(1-\eps)f(\phi(x_1+\xi(t)))-f((1-\eps)\phi(x_1+\xi(t)))&\leq \eps\hspace{0.1em}\frac{f'(0)}{2}\hspace{0.1em}\phi(x_1+\xi(t))\leq0. \label{cas:x1:suppp1}
\end{align}
%for all $0<\eps<\beta_0$.
On the other hand, since $1-\beta_0/2\leq\phi(x_1+\xi(t))\leq 1$, since $0\leq\beta\leq\beta_0$ and since $t\geq0$, for all $0<\eps<\beta_0$, we have
$$w_\rho^-(t,x)=\phi(x_1+\xi(t))-\eps\hspace{0.1em}\phi(x_1+\xi(t))-\beta\hspace{0.1em}e^{-\alpha t}\geq 1-\frac{\beta_0}{2}-2\beta_0>1-4\beta_0.$$
Otherwise said, we have
$$1-4\beta_0\leq w_\rho^-(t,x)\leq(1-\eps)\hspace{0.1em}\phi(x_1+\xi(t))\leq1,$$ which, together with \eqref{bc-eq-f}, implies that, for all $0<\eps<\beta_0$, there holds
\begin{align}
f((1-\eps)\hspace{0.1em}\phi(x_1+\xi(t)))-f(w_\rho^-(t,x))\leq \beta\hspace{0.1em}\frac{f'(0)}{2}\hspace{0.1em}e^{-\alpha t}. \label{cas:x1:suppp2}
\end{align}
Plugging \eqref{cas:x1:suppp1} and \eqref{cas:x1:suppp2} in the definition of $\mathcal{I}(t,x)$, we obtain
$$ \mathcal{I}(t,x)\leq \beta\left(\alpha+\frac{f'(0)}{2}\right)e^{-\alpha t}. $$
Therefore, $\mathcal{I}(t,x)\leq0$ provided $0\leq\alpha\leq -f'(0)/2$ and $0<\eps\leq\beta_0$.
\vskip 0.3cm

\noindent\emph{Case 3: $-A\leq x_1+\xi(t)\leq A$.}
\smallskip

Let us decompose the last two terms in the definition of $\mathcal{I}(t,x)$ as follows
\begin{align*}
(1-\eps)\hspace{0.1em}f(\phi(x_1+\xi(t&)))-f(w_\rho^-(t,x))=(1-\eps)\Big(f(\phi(x_1+\xi(t)))-f((1-\eps)\phi(x_1+\xi(t)))\Big) \\
&+\Big(f((1-\eps)\phi(x_1+\xi(t)))-f(w_\rho^-(t,x))\Big)-\eps\hspace{0.1em}f((1-\eps)\phi(x_1+\xi(t))).
\end{align*}
Since $f\in C^{0,1}(\R)$ and since $0<\phi<1$, we then obtain the following estimate
\begin{align*}
(1-\eps)\hspace{0.1em}f(\phi(x_1+\xi(t)))-f(w_\rho^-(t,x))\leq [f]_{C^{0,1}(\R)}\Big(\eps(1-\eps)+\beta\hspace{0.1em}e^{-\alpha t}\Big)+\eps\hspace{0.1em}\|f\|_{L^\infty([0,1])}.
\end{align*}
Letting $\zeta:=\inf_{z\in[-A,A]}\phi'(z)>0$ and recalling the definition of $\mathcal{I}(t,x)$, we get
$$ \mathcal{I}(t,x)\leq \beta\hspace{0.1em}e^{-\alpha t}\left(\alpha+[f]_{C^{0,1}(\R)}-\zeta\alpha\kappa(1-\eps)\right)+(1-\eps)\left(\zeta(c'-c)+\eps\hspace{0.1em}[f]_{C^{0,1}(\R)}\right)+\eps\hspace{0.1em}\|f\|_{L^\infty([0,1])}. $$
Therefore, $\mathcal{I}(t,x)\leq0$ provided we choose $\eps$ and $\alpha$ such that
$$0<\eps\leq\zeta(c-c')/(2\|f\|_{L^\infty([0,1])}+2[f]_{C^{0,1}(\R)})\ \text{ and }\  \kappa\geq\kappa(\alpha):=(\alpha+[f]_{C^{0,1}(\R)})/(\zeta\alpha).$$

Summing up, we have shown that $w_\rho^-$ is a subsolution to \eqref{EQ:CONVOLUTION} in $[0,\infty)\times\{x_1\leq R_0\}$ provided that $0\leq\alpha\leq\alpha_0:=-f'(0)/2$, that $0\leq\beta\leq\beta_0$, that $\kappa\geq\kappa(\alpha)$ and that $0<\eps\leq\eps_0$, where $\eps_0>0$ depends only on $c$, $c'$, $f$, $\beta_0$ and $\phi$, as desired.
\end{proof}
\begin{remark}\label{RK:intermediaire}
The constants $\alpha_0$ and $\beta_0$ depend only on $f'$ and the constant $\kappa(\alpha)$ depends only on $f$, $\phi$ and $\alpha$. Moreover, by taking $c':=(1-1/n)c$ for $n\in\N\setminus\{0\}$, it can easily be seen from the proof that there exists $n_0\geq1$ such that $\eps_0$ reads
$$ \eps_0=\frac{\vartheta\hspace{0.05em}c}{n}\ \text{ where }\ \vartheta:=\frac{\min\{\sigma_0^{-1},\zeta\}}{2([f]_{C^{0,1}(\R)}+\|f\|_{L^\infty([0,1])})}, $$
for all $n\geq n_0$. Therefore, the function
$$ w_{n,\rho}^-(t,x):=\left(1-\frac{\vartheta c}{n}\right)\phi\left(x_1+\left(1-\frac{1}{n}\right)ct-\rho-\kappa(\alpha_0)\beta_0\hspace{0.1em}(1-e^{-\alpha_0 t})\right)\chi_{R_0}(x_1)-\beta_0\hspace{0.1em}e^{-\alpha_0t}, $$
is a subsolution to \eqref{EQ:CONVOLUTION} in $[0,\infty)\times\{x_1\leq R_0\}$ for all $n\geq n_0$, all $\rho\in\R$ and all $R_0\in\R$.
\end{remark}
\begin{remark}
Notice that the same proof also works for the classical problem \eqref{eqlocale}.
\end{remark}
We are now in position to prove Lemma~\ref{LE:intermediaire}.
\begin{proof}[Proof of Lemma~\ref{LE:intermediaire}]
Let $\lambda\in(0,1)$ be arbitrary. Since $K$ is compact, up to immaterial translations, we may assume, without loss of generality, that $K\subset\{x_1\geq-R_J\}$. Let $n\geq n_0$ and $\eps_n:=\vartheta c/n$, where $n_0\geq1$ and $\vartheta>0$ have the same meaning as in Remark~\ref{RK:intermediaire}. Let $n_\lambda\geq n_0$ be such that $1-\eps_n>\lambda$ for all $n\geq n_\lambda$. Since $u_\infty(x)\to1$ as $|x|\to\infty$, there exists then an increasing sequence $(R_n)_{n\geq n_0}\subset[2R_J,\infty)$ such that
$$ u_\infty(x)\geq 1-\frac{\eps_n}{2} \text{ for all }x_1\leq-R_n\text{ and all }n\geq n_0. $$
Let us now fix some $n\geq n_\lambda\,(\geq n_0)$. Since $\lim_{z\to +\infty}\phi(z)=1$, there is some $t_n\geq0$ such that
$$\phi(x_1+ct)\geq 1-\eps_n/2  \text{ for all } (t,x)\in[t_n,\infty)\times\{-R_n\leq x_1\leq -R_n+2R_J\}.$$
Since $u(t_n,x_1,x')\to\phi(x_1+ct_n)$ for all $-R_n\leq x_1\leq-R_n+2R_J$ as $|x'|\to\infty$ (by Proposition~\ref{PROP:HORIZON}), there exists $R_n'>0$ such that %, for all $|x'|>R_n'$
$$u(t_n,x)\geq \phi(x_1+ct_n)-\eps_n/2>1-\eps_n, $$
for all $-R_n\leq x_1\leq -R_n+2R_J$ and all $|x'|\geq R_n'$.
Since $\partial_tu>0$, this implies that %for all $t\geq t_n,$  and all $|x'|\geq R_n'$
$$u(t,x)\geq 1-\eps_n, $$
for all $t\geq t_n,$ all $|x'|\geq R_n'$ and all $-R_n\leq x_1\leq -R_n+2R_J$.
Since, in addition, $u(t,x)$ converges locally uniformly towards $u_\infty(x)$ as $t\to\infty$, we may then find some $t_n'\geq0$ such that, for all $|x'|\leq R_n'$ and all $-R_n\leq x_1\leq -R_n+2R_J$, there holds
$$u(t_n',x)\geq u_\infty(x)-\eps_n/2\geq 1-\eps_n. $$
Therefore, letting $t_n^*:=\max\{t_n,t_n'\}$, we have
$$u(t,x)\geq 1-\eps_n \text{ for all } (t,x)\in[t_n^*,\infty)\times\{-R_n\leq x_1\leq -R_n+2R_J\}.$$
In particular, for all $\rho\in\R$, we have
$$u(t,x)\geq w_{n,\rho}^-(t,x) \text{ for all } (t,x)\in[t_n^*,\infty)\times\{-R_n\leq x_1\leq -R_n+2R_J\}, $$
where $w_{n,\rho}^-(t,x)$ is as in Remark~\ref{RK:intermediaire} with $R_0:=-R_n$. But since $w_{n,\rho}^-(t,x)\leq 0\leq u(t,x)$ for all $x_1\geq -R_n+2R_J$ and all $t\geq0$ (recall that $\chi_{\text{-}\!\text{-}R_n}\equiv0$ in $\{x_1\geq-R_n+2R_J\}$), we have
\begin{align}
u(t,x)\geq w_{n,\rho}^-(t,x) \text{ for all }t\geq t_n^*,\text{ all }x_1\geq-R_n\text{ and all }\rho\in\R. \label{asympt:level:1}
\end{align}
%for all $t\geq t_n^*$, all $x_1\geq-R_n$ and all $\rho\in\R$.
Moreover, since $\lim_{z\to -\infty}\phi(z)=0$, there is some $\rho_n>0$ such that
$$ (1-\eps_n)\hspace{0.1em}\phi\left(-R_n+\left(1-\frac{1}{n}\right)ct_n^*-\rho_n-\kappa(\alpha_0)\beta_0\hspace{0.1em}(1-e^{-\alpha_0 t_n^*})\right)\leq\beta_0\hspace{0.1em}e^{-\alpha_0t_n^*}. $$
Therefore, $w_{n,\rho_n}^-(t_n^*,x)\leq 0$ for all $x_1\leq-R_n$. But since $u$ is positive, we obtain that
\begin{align}
u(t_n^*,x)\geq w_{n,\rho_n}^-(t_n^*,x) \text{ for all }x\in\overline{\Omega}. \label{asympt:level:2}
\end{align}
Collecting \eqref{asympt:level:1}, \eqref{asympt:level:2}, recalling Lemma~\ref{LE:sup:intermediaire}, Remark~\ref{RK:intermediaire}, that $K\subset\{x_1\geq-R_J\}$ and that $-R_n+R_J\leq-R_J$ (because $R_n\geq2R_J$, by construction of $R_n$), we have
\begin{align*}
\left\{
\begin{array}{rll}
\partial_tu\!\!\!&\geq L_{\R^N}u+f(u)& \mbox{in }[t_n^*,\infty)\times\{x_1\leq-R_n\}, \vspace{3pt}\\
\partial_tw_{n,\rho_n}^-\!\!\!&\leq L_{\R^N}w_{n,\rho_n}^-\!+f(w_{n,\rho_n}^-) & \mbox{in }[t_n^*,\infty)\times\{x_1\leq-R_n\}, \vspace{3pt}\\
u\!\!\!&\geq w_{n,\rho_n}^- & \mbox{in }[t_n^*,\infty)\times\{x_1\geq-R_n\}, \vspace{3pt}\\
u(t_n^*,\cdot)\!\!\!&\geq w_{n,\rho_n}^-(t_n^*,\cdot) & \mbox{in }\overline{\Omega}.
\end{array}
\right.
\end{align*}
By the comparison principle Lemma~\ref{LE:COMPARISON} (whose application is licit as explained in Remark~\ref{RE:COMPARISON}), we deduce that $u(t,x)\geq w_{n,\rho_n}(t,x)$
%$$ u(t,x)\geq \left(1-\frac{\vartheta c}{n}\right)\phi\left(x_1+\left(1-\frac{1}{n}\right)ct-\rho_n-\kappa(\alpha_0)\beta_0\hspace{0.1em}(1-e^{-\alpha_0 t})\right)-\beta_0\hspace{0.1em}e^{-\alpha_0t}, $$
for all $(t,x)\in[t_n^*,\infty)\times\overline{\Omega}$.

Now, since $n\geq n_\lambda$, we may find some $\tau_n\geq t_n^*$ large enough so that $\lambda+\beta_0\hspace{0.1em}e^{-\alpha_0 \tau_n}< 1-\eps_n$ (by construction of $n_\lambda$). Since we also have $\varkappa_n(\lambda):=(\lambda+\beta_0\hspace{0.1em}e^{-\alpha_0 \tau_n})/(1-\eps_n)>0$, the number $\Theta_n(\lambda):=\phi^{-1}(\varkappa_n(\lambda))$ is well-defined.
%$$ \Theta_n(\lambda):=\phi^{-1}\left(\frac{\lambda+\beta_0\hspace{0.1em}e^{-\alpha_0 \tau_n}}{1-\eps_n}\right), $$
%is well-defined.
Now, let $\Gamma_{n}(\lambda)$ and $T_n$ be given by
$$ \Gamma_{n}(\lambda):=\rho_n+\beta_0\kappa(\alpha_0)+\Theta_n(\lambda)\ \text{ and }\ T_n:=\max\left\{\tau_n,\frac{\Gamma_n(\lambda)+R_n}{c(1-1/n)}\right\}. $$
Suppose that $t\in[T_n,\infty)$ and that
$$ x\in\left\{x\in\overline{\Omega}\hspace{0.1em};\hspace{0.1em}-R_n\geq x_1\geq \Gamma_{n}(\lambda)-c\left(1-\frac{1}{n}\right)t\right\}. $$
%-\left(1-\frac{1}{n}\right)ct+\rho_n+\beta\kappa+\Phi_n\leq x_1\leq -R_n. $$
Then, by construction, we have that
\begin{align*}
u(t,x)\geq -\beta_0\hspace{0.1em}e^{-\alpha_0 t}+(1-\eps_n)\hspace{0.1em}\phi\left(\Theta_n(\lambda)\right)\geq \lambda+\beta_0\hspace{0.1em}e^{-\alpha_0 \tau_n}-\beta_0\hspace{0.1em}e^{-\alpha_0 t}\geq \lambda.
\end{align*}
Therefore, $x\in E_\lambda(t)$ for all $t\geq T_n$, which thereby proves Lemma~\ref{LE:intermediaire}.
\end{proof}

\subsection{Conclusion: the super-level sets move exactly at speed $c$}
This subsection is devoted to the proof of Proposition~\ref{PROP:SUPERLEVEL}(ii). We will rely on some estimates obtained in the previous section together with a contradiction argument. So let us proceed.

\begin{proof}[Proof of Proposition~\ref{PROP:SUPERLEVEL}(ii)]
Let us extend $f$ linearly outside $[0,1]$ by the function $\widetilde{f}$ given by \eqref{extension:f}. For the sake of simplicity, let us still denote by $f$ this extension. Now, let $0<\beta_0<\min\{1-\theta,\theta\}/4$ and $A>0$ be as in the previous section, namely as in \eqref{bc-eq-f}-\eqref{bc-eq-phi:A}.
Let us first remark that the arguments in the proof of Lemma~\ref{LE:intermediaire} also yield the existence of some $\eps_*, t_\eps, \rho_\eps, R_\eps>0$ such that, for all $0<\eps\leq\eps_*$, all $t\geq t_\eps$ and all $x\in\overline{\Omega}$, there holds
$$ u(t,x)\geq w_{\eps}^-(t,x):=\left(1-\frac{\eps}{4}\right)\phi\left(x_1+\frac{ct}{2}-\rho_\eps-\kappa(\alpha_0)\hspace{0.1em}(1-e^{-\alpha_0 t})\right)\chi_{\text{-}\!\text{-}R_\eps}(x_1)-\beta_0\hspace{0.1em}e^{-\alpha_0 t}, $$
where $\chi_{\text{-}\!\text{-}R_\eps}$ has the same meaning as in the previous section.

Now, since $\lim_{z\to+\infty}\phi(z)=1$ and $\phi'>0$, there exists $\widetilde{R}_\eps>0$ such that $\phi(z)\geq1-\eps/4$ for all $z\geq\widetilde{R}_\eps$.
Moreover, if we let $t_*:=\max\{t_\eps,t_\eps'\}$, where $t_\eps'\geq0$ is such that $\beta_0\hspace{0.1em}e^{-\alpha_0t_\eps'}\leq\eps^2/16$ and $\widetilde{R}_\eps-ct_\eps'/4+\rho_\eps+\kappa(\alpha_0)<-R_\eps$, we have
\begin{align*}
w_\eps^-(t,x)> \left(1-\frac{\eps}{4}\right)\phi\left(\widetilde{R}_\eps+\frac{c}{4}\hspace{0.1em}(t-t_*)\right)-\beta_0\hspace{0.1em}e^{-\alpha_0 t}\geq \left(1-\frac{\eps}{4}\right)^2-\frac{\eps^2}{16}=1-\frac{\eps}{2},
\end{align*}
for all $t\geq t_*$ and all $x\in\{-R_\eps-ct/4<x_1<-R_\eps+R_J\}$. Consequently, we have
\begin{align}
u(t,x)>1-\frac{\eps}{2} \text{ for all }(t,x)\in[t_*,\infty)\times\left\{-R_\eps-\frac{ct}{4}<x_1<-R_\eps+R_J\right\}. \label{bc-eq-fe2}
\end{align}
%for all $(t,x)\in[t_*,\infty)\times\{-R_\eps-ct/4<x_1<-R_\eps\}$.
Now, let $\eps_0:=\min\{\eps_*,\beta_0/4\}$ and, for all $\sigma>0$, all $\rho>0$ and all $\kappa>0$, let
$$ \widetilde{w}_{\sigma,\rho}(t,x):=\left(1-\eps_0\hspace{0.1em}e^{\sigma(x_1+R_{\eps_0})}\right)\phi\left(x_1+ct-\rho-2\hspace{0.05em}\eps_0\hspace{0.05em}\kappa\left(1-e^{-\frac{\sigma ct}{4}}\right)\right)-2\hspace{0.05em}\eps_0\hspace{0.1em}e^{-\frac{\sigma ct}{4}}, $$
where $R_{\eps_0}$ has the same meaning as $R_\eps$ with $\eps_0$ instead of $\eps$. We claim that
\begin{claim}\label{Claim:super}
There exists $\kappa_*>0$, $\sigma_*>0$ and $\rho_{\sigma_*}>0$ such that
$$ u(t,x)\geq \widetilde{w}_{\sigma_*,\rho_{\sigma_*}}(t,x) \text{ for all }(t,x)\in[t_*,\infty)\times\overline{\Omega}. $$
\end{claim}
Note that by proving Claim~\ref{Claim:super}, we end the proof of Proposition~\ref{PROP:SUPERLEVEL}(ii). To see this, fix some $\lambda\in(0,1)$. Set $\varrho:=(1-\lambda)/2$ and $R_\lambda:=\phi^{-1}(\lambda+\varrho)$. Also, let us set
$$
\Gamma_1:=-R_{\eps_0}-\displaystyle\frac{1}{\sigma_*}\log\left(\frac{\varrho}{2\eps_0(\lambda+\varrho)}\right)\ \text{ and }\ \Gamma_2:=R_\lambda+\rho_{\sigma_*}+2\hspace{0.05em}\eps_0\hspace{0.05em}\kappa_*.
$$
Let $t\geq \max\{t_*,(\Gamma_2-\Gamma_1)/c\}$ and let $x\in\left\{x\in\overline{\Omega}\hspace{0.1em};\ \Gamma_1\geq x_1\geq \Gamma_2-ct\right\}$. Then, we have
\begin{align*} \widetilde{w}_{\sigma_*,\rho_{\sigma_*}}(t,x)&\geq \frac{2\lambda+\varrho}{2(\lambda+\varrho)}\,\phi(R_\lambda)-2\hspace{0.05em}\eps_0\hspace{0.1em}e^{-\frac{\sigma_* ct}{4}}=\lambda+\frac{\varrho}{2}-2\hspace{0.05em}\eps_0\hspace{0.1em}e^{-\frac{\sigma_* ct}{4}}.
\end{align*}
If also $t\geq t_\lambda:=\max\{t_*,(\Gamma_2-\Gamma_1)/c,4(\sigma_* c)^{-1}\log(4\eps_0/\varrho)\}$, then, by Claim~\ref{Claim:super}, we have
$$ u(t,x)\geq \widetilde{w}_{\sigma_*,\rho_{\sigma_*}}(t,x)\geq \lambda+\frac{\varrho}{2}-2\hspace{0.05em}\eps_0\hspace{0.1em}e^{-\frac{\sigma_* ct}{4}}\geq\lambda. $$
Therefore, we obtain that $\left\{x\in\overline{\Omega}\hspace{0.1em};\ \Gamma_1\geq x_1\geq \Gamma_2-ct\right\}\subset E_\lambda(t)$ for all $t\geq t_\lambda$.
\end{proof}
To complete the proof of Proposition~\ref{PROP:SUPERLEVEL}(ii), it remains to establish Claim~\ref{Claim:super}.
\begin{proof}[Proof of Claim~\ref{Claim:super}]
Define
\begin{align}
M(J):=\int_{\R}J_1(z_1)\hspace{0.05em}|z_1|\hspace{0.1em}e^{|z_1|}\hspace{0.1em}\mathrm{d}z_1. \label{DE:MJ}
\end{align}
Moreover, set
\begin{align*}
\alpha:=\frac{\sigma c}{4} \text{ and } \kappa_*:=\frac{4\hspace{0.1em}\zeta^{-1}}{c(1-\eps_0)}\hspace{0.1em}\left(M(J)+\frac{c}{2}+\frac{3\hspace{0.1em}[f]_{C^{0,1}(\R)}}{\sigma}+\frac{\|f\|_{L^\infty([0,1])}}{\sigma}\right),
\end{align*}
where $\zeta:=\inf_{z\in[-A,A]}\phi'(z)$ (remember that $A$ is given by \eqref{bc-eq-phi:A}) and $\sigma$ is given by
$$ \sigma:= \min\left\{1,\frac{|f'(0)|}{2\hspace{0.1em}M(J)+c},\frac{|f'(1)|}{c},\frac{|f'(1)|(1-\frac{\beta_0}{2})}{4\hspace{0.1em}M(J)}\right\}, $$
Recall that $R_J$ is an arbitrary constant such that $\mathrm{supp}(J)\subset[0,R_J]$, hence we have the freedom to chose $R_J$ arbitrarily large. So even if it means increasing $R_J$ (which is always possible and which does not impact the value of the constants involved in the definition of $\sigma$ and $\eps_0$), we may always assume that
\begin{align}
R_J\geq\frac{1}{\sigma}\hspace{0.1em}\log\left(\frac{1}{\eps_0}\right). \label{HYP:sur:Rj}
\end{align}
Furthermore, since $K$ is compact, up to immaterial translations, we may assume, without loss of generality, that $K\subset\{x_1\geq-R_J\}$. Even if it means increasing $R_{\eps_0}$ (and, therefore, increasing $t_*$), we may further assume that $K\subset\{x_1\geq -R_{\eps_0}+2R_J\}$ (notice that this does not affect the validity of any of the previous estimates).

Now,  since $u(t_*,x)>0$ and since $\lim_{x_1\to\infty}u(t_*,x)= 1$ (by Proposition~\ref{PROP:LIMITE:X1}), thanks to following asymptotic behaviour of $\phi$, namely
\begin{align*}
&\lim_{x_1\to-\infty}\phi(x_1+ct_*-\rho)-2\hspace{0.05em}\eps_0\hspace{0.1em}e^{-\alpha t_*}=-2\hspace{0.05em}\eps_0\hspace{0.1em}e^{-\alpha t_*}<0,\\
&\lim_{x_1\to+\infty}\phi(x_1+ct_*-\rho)-2\hspace{0.05em}\eps_0\hspace{0.1em}e^{-\alpha t_*}=1-2\hspace{0.05em}\eps_0\hspace{0.1em}e^{-\alpha t_*}<1,
\end{align*}
we may find some $\rho_{\sigma}>0$ such that
$$ u(t_*,x)\geq \phi(x_1+ct_*-\rho_{\sigma})-2\hspace{0.05em}\eps_0\hspace{0.1em}e^{-\alpha t_*}>\widetilde{w}_{\sigma,\rho_{\sigma}}(t_*,x) \text{ for all }x\in\overline{\Omega}. $$
The goal is to show that this inequality remains true for all $t\geq t_*$. Thanks to the uniform continuity of $u$ and $\phi$, there exists $t_0>t_*$ such that
$$ u(t,x)>\widetilde{w}_{\sigma,\rho_{\sigma}}(t,x) \text{ for all } (t,x)\in[t_*,t_0]\times\overline{\Omega}. $$
Observe also that
\begin{align}
u(t,x)>0\geq\widetilde{w}_{\sigma,\rho_{\sigma}}(t,x) \text{ for all }(t,x)\in[t_*,\infty)\times\left\{x_1\geq-R_{\eps_0}+\frac{1}{\sigma}\hspace{0.1em}\log\left(\frac{1}{\eps_0}\right)\right\}. \label{eq-ref-1}
\end{align}
Now, let us define the following set
$$ E:=\Big\{t\geq t_*\text{ such that }u(t,x)< \widetilde{w}_{\sigma,\rho_\sigma}(t,x)\text{ for some }x\in\overline{\Omega}\Big\}. $$
If we can prove that $E=\emptyset$, then Claim~\ref{Claim:super} will automatically follow. So let us assume, by contradiction, that
$E\neq\emptyset$ and  set $\overline{t}:=\inf E<+\infty$.
Readily, we observe that $\overline{t}\geq t_0>t_*$.

Next, we claim that
\begin{claim}
There exists a point $\overline{x}\in\overline{\Omega}$ such that $u(\overline{t},\overline{x})=\widetilde{w}_{\sigma,\rho_\sigma}(\overline{t},\overline{x})$.
\end{claim}
\begin{proof}  Let $(t_n)_{n\in\N}\subset E$ be such that $t_n\to\overline{t}$ as $n\to\infty$. Then, for all $n\in\N$, there is some point $x_n\in\overline{\Omega}$ such that $u(t_n,x_n)\leq \widetilde{w}_{\sigma,\rho_n}(t_n,x_n)$. We claim that the sequence $(x_n)_{n\in\N}$ stay in a compact set.
If not, then $|x_n|\to\infty$ as $n\to\infty$.
Since
$$\lim_{x_1\to-\infty}\widetilde{w}_{\sigma,\rho_\sigma}(t,x)=-2\hspace{0.05em}\eps_0\hspace{0.1em}e^{-\alpha t}<0,$$
and since (by \eqref{eq-ref-1}) we have %for all $t\ge t_*$
$$u(t,x)>\widetilde{w}_{\sigma,\rho_\sigma}(t,x) \text{ for all } (t,x)\in [t_*,\infty)\times\left\{x_1\geq-R_{\eps_0}+\frac{1}{\sigma}\hspace{0.1em}\log\left(\frac{1}{\eps_0}\right)\right\},$$
we must have that $-R'\leq x_{1,n}< -R_{\eps_0}+\sigma^{-1}\log(1/\eps_0)$ for some $R'>0$. Hence, we must have $|x_n'|\to\infty$. However, by Proposition~\ref{PROP:HORIZON}, we have $|u(t_n,x_{1,n},x_n')-\phi(x_{1,n}+ct_n)|\to0$ as $n\to\infty$, but
$$\phi(x_{1,n}+ct_n) >2\hspace{0.05em}\eps_0\hspace{0.1em}e^{-\alpha t_n}+\widetilde{w}_{\sigma,\rho_\sigma}(t_n,x_n),$$
and so we have
$$ \limsup_{n\to\infty}\hspace{0.1em}u(t_n,x_n)\geq 2\hspace{0.05em}\eps_0\hspace{0.1em}e^{-\alpha\overline{t}}+\limsup_{n\to\infty}\hspace{0.1em}\widetilde{w}_{\sigma,\rho_\sigma}(t_n,x_n)>\limsup_{n\to\infty}\hspace{0.1em}\widetilde{w}_{\sigma,\rho_\sigma}(t_n,x_n), $$
a contradiction. Therefore, $(x_n)_{n\in\N}$ cannot be unbounded and stays in a compact set.

As a consequence, we may find a subsequence, still denoted by $x_n$, which converges towards some $\overline{x}\in\overline{\Omega}$. By \eqref{eq-ref-1}, we deduce that $\overline{x}_{1}<-R_{\eps_0}+\sigma^{-1}\log(1/\eps_0)\leq-R_{\eps_0}+R_J$.

In addition, we have $u(\overline{t},\overline{x})\leq\widetilde{w}_{\sigma,\rho_\sigma}(\overline{t},\overline{x})$ and $u(t,x)>\widetilde{w}_{\sigma,\rho_\sigma}(t,x)$ for all $t<\overline{t}$ and all $x\in\overline{\Omega}$ (since otherwise this would contradict the definition of $\overline{t}$). In fact, the latter implies that $u(\overline{t},\overline{x})=\widetilde{w}_{\sigma,\rho_\sigma}(\overline{t},\overline{x})$ (because $u$ and $\widetilde{w}_{\sigma,\rho_\sigma}$ are both continuous).
\end{proof}

By definition of $\overline{t}$ and $\overline{x}$, the function $z(\overline{t},x):= u(\overline{t}, x)-w_{\sigma,\rho_\sigma}(\overline{t},x)\ge 0$ achieves its global minimum at $\overline{x}$.
% Also, we may suppose that $\overline{x}$ is a global maximum of $z(\overline{t},\cdot):=\widetilde{w}_{\sigma,\rho_\sigma}(\overline{t},\cdot)-u(\overline{t},\cdot)$.
Hence,
$$ L_{\R^N}z(\overline{t},\overline{x})>0,\ z(\overline{t},\overline{x})=0 \text{ and } \partial_t z(\overline{t},\overline{x})\le 0. $$
(Remember that $K\subset\{x_1\geq-R_{\eps_0}+2R_J\}$ and that $\overline{x}_1<-R_{\eps_0}+R_J$.) Thus, we deduce that
\begin{align*}
\partial_t\widetilde{w}_{\sigma,\rho_\sigma}(\overline{t},\overline{x})-L_{\R^N}\widetilde{w}_{\sigma,\rho_\sigma}(\overline{t},\overline{x})-f(\widetilde{w}_{\sigma,\rho_\sigma}(\overline{t},\overline{x}))&>\partial_tu(\overline{t},\overline{x})-L_{\R^N}u(\overline{t},\overline{x})-f(u(\overline{t},\overline{x}))=0.  \end{align*}
The whole game is now to obtain a contradiction with this by showing that
\begin{align} \mathcal{I}(\overline{t},\overline{x}):=\partial_t\widetilde{w}_{\sigma,\rho_\sigma}(\overline{t},\overline{x})-L_{\R^N}\widetilde{w}_{\sigma,\rho_\sigma}(\overline{t},\overline{x})-f(\widetilde{w}_{\sigma,\rho_\sigma}(\overline{t},\overline{x}))<0. \label{super:contradiction} \end{align}
Set $\eta(x_1):=1-\eps_0\hspace{0.1em}e^{\sigma(x_1+R_{\eps_0})}$ and $\xi(\overline{t}):=ct-\rho-2\hspace{0.05em}\eps_0\hspace{0.05em}\kappa_*(1-e^{-\frac{\sigma ct}{4}})$. A short computation (using the equation satisfied by $\phi$) shows that
\begin{align*}
\mathcal{I}(\overline{t},\overline{x})&=\int_{\R}J_1(\overline{x}_1-y_1)\hspace{0.05em}\phi(y_1+\xi(\overline{t}))\hspace{0.05em}(\eta(\overline{x}_1)-\eta(y_1))\hspace{0.1em}\mathrm{d}y_1+2\hspace{0.05em}\eps_0\hspace{0.05em}\alpha\hspace{0.1em}e^{-\alpha \overline{t}} \\
&\qquad +\eta(\overline{x}_1)f(\phi(\overline{x}_1+\xi(\overline{t})))-f\big(\eta(\overline{x}_1)\phi(\overline{x}_1+\xi(\overline{t}))-2\hspace{0.05em}\eps_0\hspace{0.1em}e^{-\alpha\overline{t}}\big) \\
&\qquad -2\hspace{0.05em}\eps_0\hspace{0.05em}\alpha\hspace{0.05em}\kappa_*\hspace{0.1em}e^{-\alpha \overline{t}}\phi'(\overline{x}_1+\xi(\overline{t}))\hspace{0.1em}\eta(\overline{x}_1).
\end{align*}
Using the definition of $\eta$, we have
\begin{align*}
\int_{\R}J_1(\overline{x}_1-y_1)\hspace{0.05em}\phi(y_1+\xi(\overline{t}))\hspace{0.05em}&(\eta(\overline{x}_1)-\eta(y_1))\hspace{0.1em}\mathrm{d}y_1 \\
&\leq\eps_0\hspace{0.1em}e^{\sigma(\overline{x}_1+R_{\eps_0})}\int_{\R}J_1(z_1)\hspace{0.05em}\phi(\overline{x}_1+z_1+\xi(\overline{t}))\hspace{0.05em}(e^{\sigma |z_1|}-1)\hspace{0.1em}\mathrm{d}z_1. %\\
%&\leq \sigma\hspace{0.1em}\eps_0\hspace{0.1em}e^{\sigma(\overline{x}_1+R_{\eps_0})}\int_{\R}J_1(z_1)\hspace{0.05em}|z_1|\hspace{0.05em}e^{\sigma z_1}\hspace{0.1em}\mathrm{d}z_1\leq\sigma\hspace{0.1em}\eps_0\hspace{0.1em}M(J)\hspace{0.1em}e^{\sigma(\overline{x}_1+R_{\eps_0})},
\end{align*}
Using that $e^X-1\leq Xe^X$ for all $X\geq0$, together with $\sigma\leq1$ and $0<\phi<1$ %and that
%$$ e^{\sigma|z_1|}-1=\sigma|z_1|\sum_{k=1}^\infty\frac{(\sigma|z_1|)^{k-1}}{k!}\leq \sigma|z_1|\sum_{k=1}^\infty\frac{(\sigma|z_1|)^{k-1}}{(k-1)!}=\sigma|z_1|\hspace{0.05em}e^{\sigma|z_1|}\leq \sigma|z_1|\hspace{0.1em}e^{|z_1|}, $$
we obtain
\begin{align*}
\int_{\R}J_1(\overline{x}_1-y_1)\hspace{0.05em}\phi(y_1+\xi(\overline{t}))\hspace{0.05em}&(\eta(\overline{x}_1)-\eta(y_1))\hspace{0.1em}\mathrm{d}y_1\leq \sigma\hspace{0.1em}\eps_0\hspace{0.1em}M(J)\hspace{0.1em}e^{\sigma(\overline{x}_1+R_{\eps_0})},
\end{align*}
where $M_J>0$ is as in \eqref{DE:MJ}.
Thus, we have that
\begin{align}
\mathcal{I}(\overline{t},\overline{x})&\leq \sigma\hspace{0.1em}\eps_0\hspace{0.1em}M(J)\hspace{0.1em}e^{\sigma(\overline{x}_1+R_{\eps_0})}+2\hspace{0.05em}\eps_0\hspace{0.05em}\alpha\hspace{0.1em}e^{-\alpha \overline{t}}-2\hspace{0.05em}\eps_0\hspace{0.05em}\alpha\hspace{0.05em}\kappa_*\hspace{0.1em}e^{-\alpha \overline{t}}\phi'(\overline{x}_1+\xi(\overline{t}))\hspace{0.1em}\eta(\overline{x}_1) \nonumber \\
&\qquad +\eta(\overline{x}_1)f(\phi(\overline{x}_1+\xi(\overline{t})))-f\big(\eta(\overline{x}_1)\phi(\overline{x}_1+\xi(\overline{t}))-2\hspace{0.05em}\eps_0\hspace{0.1em}e^{-\alpha\overline{t}}\big). \label{eq-ref-4}
\end{align}
We now distinguish between three different cases.
\vskip 0.3cm

\noindent\emph{Case 1: $\overline{x}_1+\xi(\overline{t})<-A$.}
\smallskip

By definition of $A$, $\beta_0$ and $\eps_0$, we have
$$ \phi(\overline{x}_1+\xi(\overline{t}))<\phi(-A)\leq\frac{\beta_0}{2}<1-\frac{\beta_0}{2}<1-\eps_0. $$
Therefore, we have that $\widetilde{w}_{\sigma,\rho_\sigma}(\overline{t},\overline{x})<1-\eps_0$. Thanks to \eqref{bc-eq-fe2},  \eqref{HYP:sur:Rj}, \eqref{eq-ref-1} and the definition of $\sigma$, we have $\overline{x}_1\leq -R_{\eps_0}-c\overline{t}/4$. This, in turn, implies that
\begin{align}
e^{\sigma(\overline{x}_1+R_{\eps_0})}\leq e^{-\sigma c\overline{t}/4}=e^{-\alpha\overline{t}}. \label{eq-ref-5}
\end{align}
Combining \eqref{eq-ref-5} with \eqref{eq-ref-4}, using that $\eta(\overline{x}_1)>0$, that $\alpha=\sigma c/4$ and that $\phi'>0$, we get
\begin{align}
\mathcal{I}(\overline{t},\overline{x})&\leq \sigma\hspace{0.1em}\eps_0\left(M(J)+\frac{c}{2}\right)e^{-\alpha\overline{t}}+\eta(\overline{x}_1)f(\phi(\overline{x}_1+\xi(\overline{t}))) \nonumber\\
&\qquad-f\big(\eta(\overline{x}_1)\phi(\overline{x}_1+\xi(\overline{t}))-2\hspace{0.05em}\eps_0\hspace{0.1em}e^{-\alpha\overline{t}}\big). \label{eq:ref-CAS1}
\end{align}
Let us rewrite the last two terms in the right-hand side as follows
\begin{align*}
\eta(\overline{x}_1)f(\phi(\overline{x}_1+\xi(\overline{t})))-&f\big(\eta(\overline{x}_1)\phi(\overline{x}_1+\xi(\overline{t}))-2\hspace{0.05em}\eps_0\hspace{0.1em}e^{-\alpha\overline{t}}\big) \\
&\qquad=f(\phi(\overline{x}_1+\xi(\overline{t})))-f\big(\eta(\overline{x}_1)\phi(\overline{x}_1+\xi(\overline{t}))-2\hspace{0.05em}\eps_0\hspace{0.1em}e^{-\alpha\overline{t}}\big) \\
&\qquad\qquad+(\eta(x_1)-1)f(\phi(\overline{x}_1+\xi(\overline{t}))).
\end{align*}
Since $\phi(\overline{x}_1+\xi(\overline{t}))\leq\beta_0/2$ and since $\overline{x}_1\leq -R_{\eps_0}-c\overline{t}/4$, we have
\begin{align*}
f(\phi(\overline{x}_1+\xi(\overline{t})))-&f\big(\eta(\overline{x}_1)\phi(\overline{x}_1+\xi(\overline{t}))-2\hspace{0.05em}\eps_0\hspace{0.1em}e^{-\alpha\overline{t}}\big) \\
&\leq \eps_0\hspace{0.1em}\frac{3f'(0)}{4}\left(e^{\sigma(\overline{x}_1+R_{\eps_0})}\phi(\overline{x}_1+\xi(\overline{t}))+2\hspace{0.1em}e^{-\alpha\overline{t}}\right),
\end{align*}
In addition, by \eqref{bc-eq-f}, we have
\begin{align*}
(\eta(x_1)\!-\!1)f(\phi(\overline{x}_1\!+\!\xi(\overline{t})))&=-\eps_0\hspace{0.1em}e^{\sigma(x_1+R_{\eps_0})}f(\phi(\overline{x}_1\!+\!\xi(\overline{t})))\leq -\eps_0\hspace{0.1em}\frac{5f'(0)}{4}\hspace{0.1em}e^{\sigma(x_1+R_{\eps_0})}\phi(\overline{x}_1\!+\!\xi(\overline{t})).
\end{align*}
Hence, recalling \eqref{eq-ref-5} and using that $\phi\leq1$, we find that
\begin{align*}
\eta(\overline{x}_1)f(\phi(\overline{x}_1+\xi(\overline{t})))-&f\big(\eta(\overline{x}_1)\phi(\overline{x}_1+\xi(\overline{t}))-2\hspace{0.05em}\eps_0\hspace{0.1em}e^{-\alpha\overline{t}}\big)\leq\eps_0\hspace{0.1em}f'(0)\hspace{0.1em}e^{-\alpha\overline{t}}.
\end{align*}
Therefore, recalling \eqref{eq:ref-CAS1} and the definition of $\sigma$, we have
\begin{align}
\mathcal{I}(\overline{t},\overline{x})\leq \eps_0\left[\sigma\hspace{0.1em}\left(M(J)+\frac{c}{2}\right)+f'(0)\right]e^{-\alpha\overline{t}}< 0. \label{super:case1:contr}
\end{align}
%\vskip 0.3cm

\noindent\emph{Case 2: $-A\leq\overline{x}_1+\xi(\overline{t})\leq A$.}
\smallskip

In this case, we have $\phi(\overline{x}_1+\xi(\overline{t}))<1-\beta_0/2<1-\eps_0$. Hence, the estimate \eqref{eq-ref-5} remains true. Moreover, $\phi'(\overline{x}_1+\xi(\overline{t}))\geq\inf_{z\in[-A,A]}\phi'(z)=:\zeta>0$ and, since $\overline{x}_1<-R_{\eps_0}-c\overline{t}/4<-R_{\eps_0}$, we also have $\eta(\overline{x}_1)\geq 1-\eps_0$. Hence, using \eqref{eq-ref-4}, we obtain that
\begin{align*}
\mathcal{I}(\overline{t},\overline{x})&\leq \sigma\hspace{0.1em}\eps_0\left(M(J)+\frac{c}{2}-\frac{c}{2}(1-\eps_0)\hspace{0.1em}\zeta\hspace{0.1em}\kappa_*\right)e^{-\alpha\overline{t}} \\
&\qquad+\eta(\overline{x}_1)f(\phi(\overline{x}_1+\xi(\overline{t})))-f\big(\eta(\overline{x}_1)\phi(\overline{x}_1+\xi(\overline{t}))-2\hspace{0.05em}\eps_0\hspace{0.1em}e^{-\alpha\overline{t}}\big).
\end{align*}
Using that $f\in C^{0,1}(\R)$, we further have
\begin{align*}
\eta(\overline{x}_1)f(\phi(\overline{x}_1+\xi(\overline{t})))-&f\big(\eta(\overline{x}_1)\phi(\overline{x}_1+\xi(\overline{t}))-2\hspace{0.05em}\eps_0\hspace{0.1em}e^{-\alpha\overline{t}}\big) \\
&\qquad\leq \eps_0\hspace{0.1em}[f]_{C^{0,1}(\R)}\left(2\hspace{0.1em}e^{-\alpha\overline{t}}+e^{\sigma(x_1+R_{\eps_0})}\phi(\overline{x}_1+\xi(\overline{t}))\right) \\
&\qquad\qquad-\eps_0\hspace{0.1em}e^{-\sigma(x_1+R_{\eps_0})}f(\phi(\overline{x}_1+\xi(\overline{t}))) \\
&\qquad\leq \eps_0\hspace{0.1em}e^{-\alpha\overline{t}}\left(3\hspace{0.1em}[f]_{C^{0,1}(\R)}+\|f\|_{L^\infty([0,1])}\right).
\end{align*}
Therefore, recalling the definition of $\kappa_*$, we have
\begin{align}
\mathcal{I}(\overline{t},\overline{x})&\leq \sigma\hspace{0.1em}\eps_0\left(M(J)+\frac{c}{2}-\frac{c}{2}(1-\eps_0)\hspace{0.1em}\zeta\hspace{0.1em}\kappa_*+\frac{3\hspace{0.1em}[f]_{C^{0,1}(\R)}}{\sigma}+\frac{\|f\|_{L^\infty([0,1])}}{\sigma}\right)e^{-\alpha\overline{t}}<0. \label{super:case2:contr}
\end{align}

\noindent\emph{Case 3: $\overline{x}_1+\xi(\overline{t})>A$.}
\smallskip

In this case, we have $\phi(\overline{x}_1+\xi(\overline{t}))\geq1-\beta_0/2>\theta$, so that $f(\phi(\overline{x}_1+\xi(\overline{t})))\geq0$. Thus, using the fact that $\eta(\overline{x}_1)\leq1$ and the definition of $\eps_0$, we find that
\begin{align*}
\eta(\overline{x}_1)f(\phi(\overline{x}_1+\xi(\overline{t})))-&f\big(\eta(\overline{x}_1)\phi(\overline{x}_1+\xi(\overline{t}))-2\hspace{0.05em}\eps_0\hspace{0.1em}e^{-\alpha\overline{t}}\big) \\
&\qquad\leq f(\phi(\overline{x}_1+\xi(\overline{t})))-f\big(\eta(\overline{x}_1)\phi(\overline{x}_1+\xi(\overline{t}))-2\hspace{0.05em}\eps_0\hspace{0.1em}e^{-\alpha\overline{t}}\big) \\
&\qquad\leq \eps_0\hspace{0.1em}\frac{f'(1)}{2}\left(e^{-\sigma(x_1+R_{\eps_0})}\left(1-\frac{\beta_0}{2}\right)+2\hspace{0.1em}e^{-\alpha\overline{t}}\right).
\end{align*}
Plugging this in \eqref{eq-ref-4}, using that $\phi'>0$ and recalling the definition of $\sigma$, we obtain that
\begin{align}
\mathcal{I}(\overline{t},\overline{x})&\leq \eps_0\hspace{0.1em}e^{-\sigma(x_1+R_{\eps_0})}\left(\sigma\hspace{0.1em}M(J)+\frac{f'(1)}{2}\left(1-\frac{\beta_0}{2}\right)\right)+2\hspace{0.1em}\eps_0\hspace{0.1em}e^{-\alpha\overline{t}}\left(\frac{\sigma c}{4}+\frac{f'(1)}{2}\right)<0. \label{super:case3:contr}
\end{align}

Collecting \eqref{super:case1:contr}, \eqref{super:case2:contr} and \eqref{super:case3:contr} we obtain that \eqref{super:contradiction} holds, which is the desired contradiction. Therefore, $E=\emptyset$ and so %$\overline{t}=\infty$ which which implies that
$\widetilde{w}_{\sigma,\rho_\sigma}(t,x)\leq u(t,x)$ for all $(t,x)\in[t_*,\infty)\times\overline{\Omega}$, as desired. The proof is thereby complete.
\end{proof}

%%%%%%%%%%%%%%%%%%%%%%%%%
%%%%%%%%%%%%%%%%%%%%%%%%%

\section{The entire solution $u(t,x)$ is a generalised transition wave}\label{SE:EntireSol}

This last section is devoted to the proof of Theorem~\ref{TH:FRONT}. Namely, we will prove that $u(t,x)$ is a generalised transition almost-planar invasion front between $0$ and $u_\infty$ with global mean speed $c$. That is, we will prove that
\begin{align}
\sup_{(t,x)\in\R\times\overline{\Omega},\,x_1+ct\geq A}|u(t,x)-u_\infty(x)|\underset{A\to\infty}{\longrightarrow} 0,
\label{front:uinfini}
\end{align}
and that
\begin{align}
\sup_{(t,x)\in\R\times\overline{\Omega},\,x_1+ct\leq -A} u(t,x)\underset{A\to\infty}{\longrightarrow} 0.
\label{front:0}
\end{align}
We will adapt to our situation the arguments developed in \cite{Berestycki2009d} using the characterisation of the large time behaviour of the super-level sets of $u(t,x)$ derived at the previous section. For the convenience of the reader, we prove separately \eqref{front:uinfini} and \eqref{front:0}. %So let us first prove \eqref{front:0}.
\begin{proof}[Proof of \eqref{front:uinfini}]
Assume, by contradiction, that \eqref{front:uinfini} does not hold. Then, there exists $\eps>0$ and a sequence $(t_n,x_n)_{n\in\N}$ such that $x_{1,n}+ct_n\to\infty$ as $n\to\infty$ and
$$ u_\infty(x_n)-u(t_n,x_n)>\eps \text{ for all }n\in\N. $$
Up to extraction of a subsequence, two cases may occur: either $|x_n|\to\infty$ or $x_n\to\overline{x}$. In the latter case, we must have that $t_n\to\infty$ and, since $u(t,x)\to u_\infty(x)$ locally uniformly in $\overline{\Omega}$ as $t\to\infty$, we deduce that
$$ 0<\eps\leq \lim_{n\to\infty}\left(u_\infty(x_n)-u(t_n,x_n)\right)=0, $$
a contradiction. In the former case, we have $|x_n|\to\infty$ and so $u_\infty(x_n)\to1$ as $n\to\infty$.
Therefore, up to extraction of a subsequence, we have
\begin{align}
1-u(t_n,x_n)>\eps \text{ for all }n\in\N. \label{contra:front:1}
\end{align}
At this stage, up to extraction of a subsequence, three different subcases may occur.
\vskip 0.3cm

\noindent\emph{Subcase 1: $t_n\to-\infty$.}
\smallskip

Since $|u(t,x)-\phi(x_1+ct)|\to 0$ uniformly in $x\in\overline{\Omega}$ as $t\to-\infty$, we have $u(t_n,x_n)-\phi(x_{1,n}+ct_n)\to 0$ as $n\to\infty$. But since $x_{1,n}+ct_n\to\infty$, we must have that $\phi(x_{1,n}+ct_n)\to1$ which, in turn, implies that $u(t_n,x_n)\to1$ as $n\to\infty$, contradicting \eqref{contra:front:1}.
\vskip 0.3cm

\noindent\emph{Subcase 2: $t_n\to\overline{t}$.}
\smallskip

Since $x_{1,n}+ct_n\to\infty$, we must have $x_{1,n}\to\infty$. Since $|u(t,x)-\phi(x_1+ct)|\to 0$ uniformly in $x\in\overline{\Omega}$ as $t\to-\infty$, we may then find some $T\ll-1$ such that
$$ \sup_{x\in\overline{\Omega}}\hspace{0.1em}|u(T,x)-\phi(x_{1}+cT)|\leq\frac{\eps}{2}. $$
Otherwise said, we have that $u(T,x)\geq \phi(x_{1}+cT)-\eps/2$ for all $x\in\overline{\Omega}$. Up to decrease further $T$, we may assume that $T<\overline{t}$. Now, using that $\partial_t u>0$, we have $u(t,x)\geq u(T,x)\geq \phi(x_1+cT)-\eps/2$ for all $t\geq T$ and all $x\in\overline{\Omega}$. Thus, recalling \eqref{contra:front:1}, we obtain that
$$ 1-\eps>\limsup_{n\to\infty}\hspace{0.1em}u(t_n,x_n)\geq \limsup_{n\to\infty}\hspace{0.1em}\phi(x_{1,n}+cT)-\frac{\eps}{2}=1-\frac{\eps}{2}, $$
a contradiction.
\vskip 0.3cm

\noindent\emph{Subcase 3: $t_n\to\infty$.}
\smallskip

Up to extraction of a subsequence, three situations may occur: either $x_{1,n}\to-\infty$, or $x_{1,n}\to\infty$ or $x_{1,n}\to\overline{x}$. If $x_{1,n}\to\infty$, we readily get a contradiction (by repeating the arguments of Subcase~2), so this situation is ruled out. Next, if $x_{1,n}\to\overline{x}$, then, since $|x_n|\to\infty$, we must have that $|x_n'|\to\infty$. Hence, by Proposition~\ref{PROP:HORIZON}, we have
$$ \liminf_{n\to\infty}\hspace{0.1em}u(T,x_n)=\phi(\overline{x}_1+cT) \text{ for all }T\in\R. $$
Choose $T>0$ large enough so that $\phi(\overline{x}_1+cT)\geq1-\eps/2$. Since $t_n\geq T$ for $n$ large enough and since $\partial_tu>0$, we then have
$$ \eps\leq \limsup_{n\to\infty}\hspace{0.1em}\left(1-u(t_n,x_n)\right)\leq \limsup_{n\to\infty}\hspace{0.1em}\left(1-u(T,x_n)\right)\leq 1-\phi(\overline{x}_1+cT)\leq \frac{\eps}{2}, $$
which, again, yields a contradiction. Lastly, let us consider the case when $x_{1,n}\to-\infty$.
Since $t_n\to\infty$, $x_{1,n}+ct_n\to\infty$ and $x_{1,n}\to-\infty$ as $n\to\infty$, it follows immediately from Claim~\ref{Claim:super} and \eqref{contra:front:1} that, for $n$ large enough, there holds
$$ 1-\eps>u(t_n,x_n)\geq 1-\frac{\eps}{2}, $$
a contradiction. Hence, this case is ruled out too. The proof of \eqref{front:uinfini} is thereby complete.
\end{proof}
Finally, let us now prove \eqref{front:0}.
\begin{proof}[Proof of \eqref{front:0}]
As above, we argue by contradiction and we assume that there exist some $\eps>0$ and a sequence $(t_n,x_n)_{n\in\N}$ such that $x_{1,n}+ct_n\to-\infty$ as $n\to\infty$ and
\begin{align}
u(t_n,x_n)>\eps \text{ for all }n\in\N. \label{contra:front:uinfini}
\end{align}
Up to extraction, three situations may occur.
\vskip 0.3cm

\noindent\emph{Subcase 1: $t_n\to-\infty$.}
\smallskip

Since $|u(t,x)-\phi(x_1+ct)|\to0$ uniformly in $x\in\overline{\Omega}$ as $t\to-\infty$, we have $|u(t_n,x_n)-\phi(x_{1,n}+ct_n)|\to0$ as $n\to\infty$. But since $x_{1,n}+ct_n\to-\infty$ as $n\to\infty$, we have $\phi(x_{1,n}+ct_n)\to0$ as $n\to\infty$ which, in turn, implies that $u(t_n,x_n)\to0$ as $n\to\infty$, contradicting \eqref{contra:front:uinfini}.
\vskip 0.3cm

\noindent\emph{Subcase 2: $t_n\to\overline{t}$.}
\smallskip

Since $x_{1,n}+ct_n\to-\infty$ as $n\to\infty$, we must have $x_{1,n}\to-\infty$. By Proposition~\ref{PROP:LIMITE:X1}, we further have that $\lim_{x_1\to-\infty}u(\overline{t}+1,x)=0$. Hence, there exists $R_0>0$ such that
$$ u(\overline{t}+1,x)\leq\frac{\eps}{2} \text{ for all }x\in\{x_1\leq R_0\}. $$
But since $\partial_tu>0$ and since $x_{1,n}\to-\infty$, we infer that
$$ \eps\leq\limsup_{n\to\infty}\hspace{0.1em}u(t_n,x_n)\leq\lim_{n\to\infty}u(\overline{t}+1,x_n)\leq\frac{\eps}{2}, $$
a contradiction.
\vskip 0.3cm

\noindent\emph{Subcase 3: $t_n\to\infty$.}
\smallskip

By \eqref{superlevel:upper:est}, there exists $\alpha,\beta,\rho>0$ such that
$$ u(t,x)\leq \phi(x_1+ct+\rho)+\beta\hspace{0.1em}e^{-\alpha t} \text{ for all }(t,x)\in[0,\infty)\times\overline{\Omega}. $$
Since $x_{1,n}+ct_n\to-\infty$ and since $t_n\to\infty$, we obtain that
$$ \eps\leq\limsup_{n\to\infty}\hspace{0.1em}u(t_n,x_n)\leq\lim_{n\to\infty}\left(\phi(x_{1,n}+ct_n+\rho)+\beta\hspace{0.1em}e^{-\alpha t_n}\right)\leq\frac{\eps}{2}, $$
a contradiction. Hence, this case is ruled out too. The proof of \eqref{front:0} is thereby complete.
\end{proof}

%%%%%%%%%%%%%%%%%%%%%%%%%

\section*{Acknowledgments}
The authors warmly thank \emph{Guillaume Legendre} (CEREMADE, Universit\'e Paris-Dauphine) for its great help in the implementation of the numerics involving the geodesic distance. This work has been carried out in the framework of the Archim\`ede Labex (ANR-11-LABX-0033) and of the A*MIDEX project (ANR-11-IDEX-0001-02), funded by the ``Investissements d’Avenir" French Government program managed by the French National Research Agency (ANR). The research leading to these results has also received funding from the ANR DEFI project NONLOCAL (ANR-14-CE25-0013) and the ANR JCJC project MODEVOL (ANR-13-JS01-0009).

%%%%%%%%%%%%%%%%%%%%%%%%%

\appendix

\section{The $J$-covering property}

In this Appendix, we list some additional results regarding the properties of quasi-Euclidean distances. Precisely, we prove the assertions made in Remark~\ref{RE:COV:PROP}. Incidentally, this will justify that the first assumption in \eqref{C3} is satisfied in a wide range of situations (and is, therefore, not an empty assumption). Firstly, we show that if $\delta$ is the Euclidean distance, then the $J$-covering property always holds.
\begin{prop}\label{CovEucl}
Let $E\subset\R^N$ be a connected set and let $\delta\in\mathcal{Q}(\overline{E})$ be the Euclidean distance. Let $J:[0,\infty)\to[0,\infty)$ be a measurable function with $|\mathrm{supp}(J)|>0$. Then, $(E,\delta)$ has the $J$-covering property.
\end{prop}
\begin{proof}
Let $x_0\in\overline{E}$. By definition of $\Pi_2(J,x_0)$, we have
$$ \Pi_2(J,x_0)=\left(x_0+\mathrm{supp}(J_{\mathrm{rad}})+\mathrm{supp}(J_{\mathrm{rad}})\right)\cap \overline{E}. $$
Let $R>0$ be such that $\Lambda:=\mathrm{supp}(J_{\mathrm{rad}})\cap \overline{B_R}$ has positive Lebesgue measure.
Since the function $G:\R^N\to[0,\infty)$ given by $G(x):=\mathds{1}_{\Lambda}\ast\mathds{1}_{\Lambda}(x)$ is continuous and since, on the other hand, $G(0)=|\Lambda|>0$, we deduce that there is some $\tau>0$ such that
$$ \overline{B_\tau}\subset\mathrm{supp}(G)\subset \Lambda+\Lambda\subset \mathrm{supp}(J_{\mathrm{rad}})+\mathrm{supp}(J_{\mathrm{rad}}). $$
Hence, $\overline{B_\tau(x_0)\cap E}\subset \Pi_2(J,x_0)$. Since $x_0\in\overline{E}$ was chosen arbitrarily, we may apply the same reasoning to any boundary point $z_0$ of $\overline{B_\tau(x_0)\cap E}$ and we have $\overline{B_\tau(z_0)\cap E}\subset \Pi_2(J,z_0)$. But since $z_0\in\Pi_2(J,x_0)$, we have $\Pi_2(J,z_0)\subset\Pi_4(J,z_0)$ and so $\overline{B_\delta(z_0)\cap E}\subset \Pi_4(J,x_0)$. This being true for any boundary point of $\overline{B_\tau(x_0)\cap E}$, we then obtain that $\overline{B_{2\tau}(x_0)\cap E}\subset\Pi_2(J,x_0)\cup\Pi_4(J,x_0)$. By induction, we find that
$$ \overline{B_{\tau k}(x_0)\cap E}\subset\bigcup_{j=1}^k\,\Pi_{2j}(J,x_0) \text{ for all }k\in\N^\ast. $$
In turn, this implies that the following chain of inclusions hold:
$$ \overline{E}=\bigcup_{k\geq0}\overline{B_{\tau k}(x_0)\cap E}\subset \bigcup_{k\geq0} \,\Pi_{2k}(J,x_0)\subset\bigcup_{j\geq0}\,\Pi_j(J,x_0)\subset\overline{E}. $$
Therefore, $(\Omega,\delta)$ has the $J$-covering property.
\end{proof}
Lastly, we prove that $(\Omega,\delta)$ has the $J$-covering property for \emph{all} $\delta\in\mathcal{Q}(\overline{\Omega})$, whenever $\Omega$ is the complement of a compact convex set with $C^{2}$ boundary and $J$ satisfies some mild additional assumptions.
\begin{prop}\label{COV:CVX}
Let $K\subset\R^N$ be a compact convex set with nonempty interior and $C^2$ boundary, let $\Omega:=\R^N\setminus K$ and let $\delta\in\mathcal{Q}(\overline{\Omega})$. Suppose that $J:[0,\infty)\to[0,\infty)$ is such that $[r_1,r_2]\subset\mathrm{supp}(J)$ for some $0\leq r_1<r_2$. Then, $(\Omega,\delta)$ has the $J$-covering property.
\end{prop}
\begin{proof}
The proof follows roughly the same structure as the one of Proposition~\ref{CovEucl}. However, it is slightly more involved due to the presence of an arbitrary quasi-Euclidean distance, which forces us to ``secure" starshaped regions in which it behaves like the Euclidean distance. To keep the proof as clear as possible, we split it into three main steps. First, we introduce some useful notations and terminology. Then, we make some preliminary geometric observations and, finally, we complete the proof by estimating the sets $\Pi_j(J,\cdot)$.
\vskip 0.3cm

\noindent\emph{Step 1. Some preparatory definitions}
\smallskip

Prior to proving Proposition~\ref{COV:CVX}, we will need to introduce a few definitions and notations.
For any $x\in\overline{\Omega}$, we define $\widetilde{\Pi}_1(x,r_1,r_2):=\{x\}$ and, for $j\geq0$, we set
$$ \widetilde{\Pi}_{j+1}(x,r_1,r_2):=\bigcup_{z\in\widetilde{\Pi}_j(x,r_1,r_2)}\!\!\mathrm{supp}\left(\mathds{1}_{[r_1,r_2]}(\delta(\cdot,z))\right). $$ %\left\{y\in\overline{\Omega};\ r_1\leq\delta(z,y)\leq r_2\right\}. $$
Clearly, $\widetilde{\Pi}_j(x,r_1,r_2)\subset\Pi_j(J,x)$ for all $j\geq1$.  Also, for all $x\in\overline{\Omega}$, we set
$$ \mathrm{star}(x):=\big\{y\in\overline{\Omega} \text{ s.t. } [x,y]\subset\overline{\Omega}\big\}. $$
Roughly speaking, $\mathrm{star}(x)$ is the set of all points which are reachable from $x$ without ``jumping" through $K$. By definition, it is the largest subset of $\overline{\Omega}$ which is starshaped with respect to $x$. In addition, for any $x\in\overline{\Omega}$, we let $\mathscr{C}(x)$ be the closed cone with vertex $x$ whose boundary $\partial(\mathscr{C}(x))$ is tangent to $\partial K$. Notice that, since $K$ is a compact convex set, $\mathscr{C}(x)$ is always well-defined and we have $K\subset\mathscr{C}(x)$ for any $x\in\overline{\Omega}$.
For later purposes, it will be useful to denote by $\mathscr{C}^+(x):=\mathscr{C}(x)\cap\mathrm{star}(x)$ the upper part of the cone $\mathscr{C}(x)$.
\vskip 0.3cm

\noindent\emph{Step 2. Preliminary geometric observations}
\smallskip

First of all, we notice that, since $[r_1,r_2]\subset\mathrm{supp}(J)$, we also have $[r_1,\widetilde{r}_2]\subset\mathrm{supp}(J)$ for any $\widetilde{r}_2\in(r_1,r_2)$. Hence, up to replace $r_2$ by some $\widetilde{r}_2\in(r_1,r_2)$ arbitrarily close to $r_1$,
\begin{align}
\text{we have the freedom to choose $\varkappa:=r_1-r_2$ arbitrarily small.} \label{J:freedom}
\end{align}

Let $m\in\partial K$ be arbitrary and let $R_{\min}:=(\max_{\partial K}\gamma)^{-1}$ where $\gamma$ is the maximum principal curvature of $\partial K$. Since, by definition, $R_{\min}$ is the minimum of the radii of curvature of $\partial K$, there is then an osculating open ball $B$ with radius $R_{\min}$ such that $\partial B\cap\partial K=\{m\}$ and that $B\subset \mathrm{int}(K)$.
Although this is classical, we recall that $\max_{\partial K}\gamma>0$ (since $K$ is a compact convex set) and that $\max_{\partial K}\gamma<\infty$ (since the Weingarten map is bounded, as follows from the fact that $K$ has $C^{2}$ boundary), so that $R_{\min}$ and $B$ are well-defined.

Now, let $p:=m+\varkappa\hspace{0.1em}\nu(m)$, where $\nu(m)$ is the outward unit normal to $\partial K$ at $m$. Then, the ball $B_\varkappa(p)$ is tangent to $\partial K$ at $p$, satisfies $\overline{B_\varkappa(p)}\cap K=\{m\}$ and $B_\varkappa(p)\subset\Omega$ (remember that $K$ is convex). Let $q:=m+r_1\hspace{0.1em}\nu(m)$ and let $\mathscr{C}^+(q)$ be the upper part of $\mathscr{C}(q)$.
Also, let $\mathscr{C}_B(q)$ be the closed cone with vertex $q$ and tangent to $B$ and let $\mathscr{C}_B^+(q):=\mathscr{C}_B(q)\cap\mathrm{star}(q)$ be its upper part. Clearly, $\mathscr{C}_B(q)\subset\mathscr{C}(q)$ and $\mathscr{C}_B^+(q)\subset\mathscr{C}^+(q)$.

Now, by Thales' theorem, up to choose $\varkappa$ small (remember \eqref{J:freedom}), say if
$$ 0<\varkappa<\min\left\{\frac{r_1}{3},\frac{r_1R_{\min}}{2R_{\min}+r_1}\right\}, $$
we may assume that $\overline{B_\varkappa(p)}\subset\overline{\mathscr{C}_B^+(q)}$ (regardless of the choice of $m$).
Since $p=m+\varkappa\hspace{0.1em}\nu(m)$ and since the orthogonal cross section of the cone $\mathscr{C}_B(q)$ is increasing in the direction $-\nu(m)$ (in the sense of the inclusion), we also have
\begin{align}
\overline{B_\varkappa(m+\ell\hspace{0.1em}\nu(m))\cap\Omega}\subset\overline{\mathscr{C}_B^+(q)}\subset\overline{\mathscr{C}^+(q)} \text{ for all }\ell\in[0,\varkappa], \label{boule:step2:geod}
\end{align}
see Figure~\ref{Jcovering} for a visual evidence.
\begin{figure}
\centering
\includegraphics[scale=0.4]{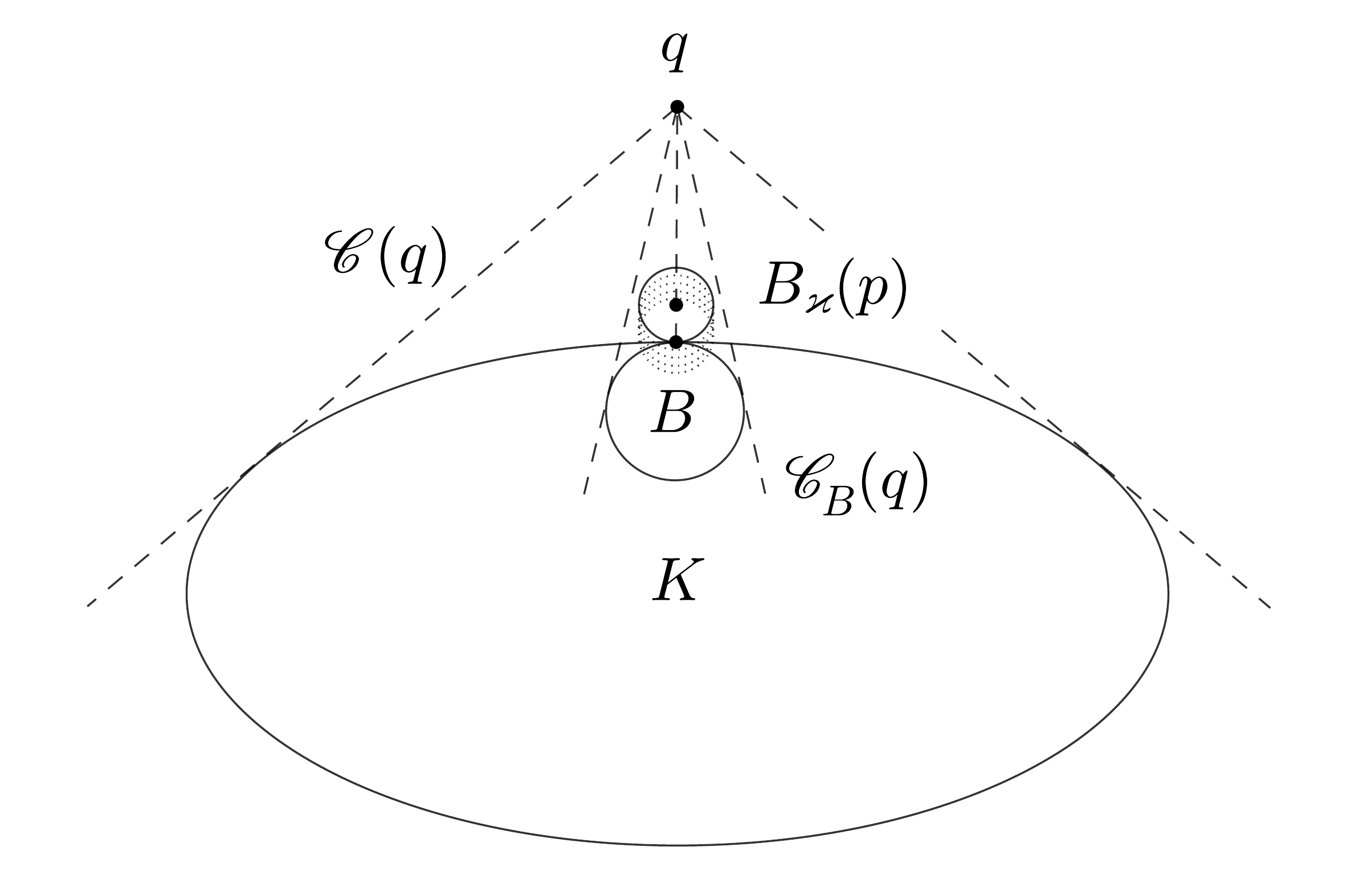} \\
\caption{Illustration of the balls $B$ and $B_{\varkappa}(p)$ and the cones $\mathscr{C}(q)$ and $\mathscr{C}_B(q)$, when $K$ is an ellipse. The upper cone $\mathscr{C}^+(q)$ (resp. $\mathscr{C}_B^+(q)$) correspond to the region of the cone $\mathscr{C}(q)$ (resp. $\mathscr{C}_B(q)$) which lie above $K$. The translates of the ball $B_{\varkappa}(p)$ appearing in \eqref{boule:step2:geod} are represented in thin dashed lines.} \label{Jcovering}
\end{figure}
Moreover, since $\ell+r_1\geq r_1$ for all $\ell\in[0,\varkappa]$, we have the straightforward inclusion
$$ \overline{\mathscr{C}^+(q)}=\overline{\mathscr{C}^+(m+r_1\hspace{0.1em}\nu(m))}\subset \overline{\mathscr{C}^+(m+(\ell+r_1)\hspace{0.1em}\nu(m))}. $$
Recalling \eqref{boule:step2:geod}, we obtain that
$$ \overline{B_{\varkappa}(m+\ell\hspace{0.1em}\nu(m))\cap\Omega}\subset\overline{\mathscr{C}^+(m+(\ell+r_1)\hspace{0.1em}\nu(m))} \text{ for all }\ell\in[0,\varkappa]. $$
Since $\overline{\mathscr{C}^+(m+(\ell+r_1)\hspace{0.1em}\nu(m))}\subset\mathrm{star}(m+(\ell+r_1)\hspace{0.1em}\nu(m))$ (by definition), it follows that
\begin{align}
\overline{B_{\varkappa}(m+\ell\hspace{0.1em}\nu(m))\cap\Omega}\subset \mathrm{star}(m+(\ell+r_1)\hspace{0.1em}\nu(m)) \text{ for all }\ell\in[0,\varkappa], \label{key:inclusion}
\end{align}
and all $m\in\partial K$.
Now that we have \eqref{key:inclusion}, we are in position to complete the proof.
\vskip 0.3cm

\noindent\emph{Step 3. Estimates for $\widetilde{\Pi}_j(\cdot,r_1,r_2)$ and conclusion}
\smallskip

Now, let us fix an arbitrary point $x_0\in\overline{\Omega}$. Since $\delta\in\mathcal{Q}(\overline{\Omega})$, we have that $\delta(x_0,y)=|x_0-y|$ for every $y\in\mathrm{star}(x_0)$. In particular,
\begin{align}
\mathrm{star}(x_0)\cap\overline{\mathcal{A}(x_0,r_1,r_2)}\subset\widetilde{\Pi}_1(x_0,r_1,r_2). \label{hyp:star2}
\end{align}
Since $\overline{\R^N\setminus\mathscr{C}(x_0)}$ is starshaped with respect to $x_0$ and since $(\R^N\setminus\mathscr{C}(x_0))\cap K=\emptyset$, we have
\begin{align}
\R^N\setminus\mathscr{C}(x_0)\subset\mathrm{star}(x_0). \label{hyp:star3}
\end{align}
%Hence, $\overline{\mathcal{A}(x_0,r_1,r_2)\setminus\mathscr{C}(x_0)}\subset\widetilde{\Pi}(x_0,r_1,r_2)$.
Now, let $S(x_0)$ be the set of all $e\in\S^{N-1}$ such that $x_0+e\hspace{0.1em}t\in\R^N\setminus\mathscr{C}(x_0)$ for all $t\geq0$ (note that $S(x_0)$ is well-defined because $\R^N\setminus\mathscr{C}(x_0)$ is also a cone).
Since $K$ is convex, it follows that $\mathscr{C}(x_0)$ has a maximum opening angle less than $\pi$. In particular, the cone $\R^N\setminus\mathscr{C}(x_0)$ has a minimum opening angle greater than $\pi$. Hence, $S(x_0)$ contains a half-sphere.

Let $e\in S(x_0)$ and let $q\in[x_0-\varkappa\hspace{0.1em}e,x_0+\varkappa\hspace{0.1em}e]\cap\mathrm{star}(x_0)$ be arbitrary.
Then, there exist $t,\tau\in[r_1,r_2]$ such that $q=x_0+(t-\tau)\hspace{0.1em}e$. Hence, letting $p:=x_0+e\hspace{0.1em}t$, we have
$$ p\in\overline{\mathcal{A}(x_0,r_1,r_2)}\setminus\mathscr{C}(x_0),\quad p-\tau\hspace{0.1em}e=x_0+(t-\tau)\hspace{0.1em}e=q\quad \text{and}\quad |p-q|\in[r_1,r_2]. $$
Recalling \eqref{hyp:star2} and \eqref{hyp:star3}, we have that $p\in\widetilde{\Pi}_1(x_0,r_1,r_2)$.
Moreover, by construction, we further have $\delta(p,q)=|p-q|\in[r_1,r_2]$. Therefore, for all $e\in S(x_0)$ and all $q\in [x_0-\varkappa\hspace{0.1em}e,x_0+\varkappa\hspace{0.1em}e]\cap\mathrm{star}(x_0)$, there exists $p\in\widetilde{\Pi}_1(x_0,r_1,r_2)$ such that $r_1\leq\delta(p,q)\leq r_2$. Consequently,
$$ \bigcup_{e\in S(x_0)} [x_0-\varkappa\hspace{0.1em}e, x_0+\varkappa\hspace{0.1em}e]\cap\mathrm{star}(x_0)\subset\widetilde{\Pi}_2(x_0,r_1,r_2). $$
But since $S(x_0)$ contains a half-sphere, the left-hand side in the above equation is nothing but $\overline{B_{\varkappa}(x_0)}\cap\mathrm{star}(x_0)$. Hence, we have that
\begin{align}
\overline{B_{\varkappa}(x_0)}\cap\mathrm{star}(x_0)\subset \widetilde{\Pi}_2(x_0,r_1,r_2). \label{inter:star(x)}
\end{align}
Let us now prove that $\overline{B_{\varkappa}(x_0)\cap\Omega}\setminus\mathrm{star}(x_0)\subset \widetilde{\Pi}_2(x_0,r_1,r_2)$. We may suppose, without loss of generality, that $\overline{B_{\varkappa}(x_0)\cap\Omega}\setminus\mathrm{star}(x_0)\neq\emptyset$, since otherwise there is nothing to prove. So, we have, in particular, that $\overline{B_{\varkappa}(x_0)}\cap K\neq\emptyset$. Let $m\in\partial K$ be the orthogonal projection of $x_0$ to $\partial K$. Then, by construction, we have $x_0=m+|x_0-m|\hspace{0.1em}\nu(m)$, where $\nu(m)$ denotes the outward unit normal to $\partial K$ at $m$. Set $x_0^\perp:=x_0+r_1\hspace{0.1em}\nu(m)$. Notice that $x_0^\perp\in \overline{\mathcal{A}(x_0,r_1,r_2)}\setminus\mathscr{C}(x_0)$ (by construction of $x_0^\perp$), so that $x_0^\perp\in\widetilde{\Pi}_1(x_0,r_1,r_2)$ (remember \eqref{hyp:star2} and \eqref{hyp:star3}). Moreover, we have $\overline{B_\varkappa(x_0)\cap\Omega}\setminus\mathrm{star}(x_0)\subset\R^N\setminus B_{r_1}(x_0^\perp) \text{ and } \overline{B_\varkappa(x_0)}\subset\overline{B_{r_2}(x_0^\perp)}$. Therefore, we have
\begin{align}
\overline{B_\varkappa(x_0)\cap\Omega}\setminus\mathrm{star}(x_0)\subset \overline{\mathcal{A}(x_0^\perp,r_1,r_2)}. \label{hyp:star4}
\end{align}
Since $x_0=p+\ell\hspace{0.1em}\nu(p)$ and $x_0^\perp=p+(\ell+r_1)\hspace{0.1em}\nu(p)$ for some $\ell\in[0,\varkappa]$ and some $p\in\partial K$, we may apply \eqref{key:inclusion}, which then yields $\overline{B_\varkappa(x_0)\cap\Omega}\subset \mathrm{star}(x_0^\perp)$. Hence, using \eqref{hyp:star4}, it follows that
$$ \overline{B_\varkappa(x_0)\cap\Omega}\setminus\mathrm{star}(x_0)\subset \mathrm{star}(x_0^\perp)\cap \overline{\mathcal{A}(x_0^\perp,r_1,r_2)}. $$
Since $\delta(x_0^\perp,y)=|x_0^\perp-y|$ for all $y\in\mathrm{star}(x_0^\perp)$, this then implies that
$$ \overline{B_\varkappa(x_0)\cap\Omega}\setminus\mathrm{star}(x_0)\subset \widetilde{\Pi}_1(x_0^\perp,r_1,r_2)\subset \widetilde{\Pi}_2(x_0,r_1,r_2), $$
where, in the last inclusion, we have used that $x_0^\perp\in\widetilde{\Pi}_1(x_0,r_1,r_2)$. Together with \eqref{inter:star(x)}, this yields that $\overline{B_\varkappa(x_0)\setminus\Omega}\subset\widetilde{\Pi}_2(x_0,r_1,r_2)$. At this stage, we may conclude exactly as in the proof of Proposition~\ref{CovEucl} (remember that $\widetilde{\Pi}_j(x,r_1,r_2)\subset\Pi_j(J,x)$ for all $x\in\overline{\Omega}$) and we therefore obtain that $(\Omega,\delta)$ has the $J$-covering property, as desired.
\end{proof}

\bibliographystyle{plain}
\bibliography{BC}

\begin{thebibliography}{10}

\bibitem{Adriaensen2003}
F.~Adriaensen, J.~P. Chardon, G.~De Blust, E.~Swinnen, S.~Villalba, H.~Gulinck,
  and E.~Matthysen.
\newblock The application of ‘least-cost’ modelling as a functional
  landscape model.
\newblock {\em Landscape Urban Plan.}, 64(4):233 -- 247, 2003.

\bibitem{Alfaro2017}
M.~Alfaro and J.~Coville.
\newblock Propagation phenomena in monostable integro-differential equations:
  Acceleration or not?
\newblock {\em J. Diff. Eq.}, 263(9):5727 -- 5758, 2017.

\bibitem{Bartumeus2007}
F.~Bartumeus.
\newblock L\'evy processes in animal movement: an evolutionary hypothesis.
\newblock {\em Fractals}, 15(2):151--162, 2007.

\bibitem{Bartumeus2009}
F.~Bartumeus.
\newblock Behavioral intermittence, {L}\'evy patterns, and randomness in animal
  movement.
\newblock {\em Oikos}, 118(4):488--494, 2009.

\bibitem{Bates1997}
P.~W. Bates, P.~C. Fife, X.~Ren, and X.~Wang.
\newblock Traveling waves in a convolution model for phase transitions.
\newblock {\em Arch. Rational Mech. Anal.}, 138(2):105--136, 1997.

\bibitem{Berestycki2016b}
H.~Berestycki, J.~Bouhours, and G.~Chapuisat.
\newblock Front blocking and propagation in cylinders with varying cross
  section.
\newblock {\em Calc. Var. Partial Dif.}, 55(3):44, 2016.

\bibitem{Berestycki2016a}
H.~Berestycki, J.~Coville, and H.-H. Vo.
\newblock Persistence criteria for populations with nonlocal dispersion.
\newblock {\em J. Math. Biol.}, 72(7):1693--1745, 2016.

\bibitem{Berestycki2002}
H.~Berestycki and F.~Hamel.
\newblock Front propagation in periodic excitable media.
\newblock {\em Comm. Pure Appl. Math.}, 55(8):949--1032, 2002.

\bibitem{Berestycki2007b}
H.~Berestycki and F.~Hamel.
\newblock Generalized travelling waves for reaction-diffusion equations.
\newblock In {\em Perspectives in Nonlinear Partial Differential Equations: in
  Honor of Haim Brezis}, 2007.

\bibitem{Berestycki2012a}
H.~Berestycki and F.~Hamel.
\newblock Generalized transition waves and their properties.
\newblock {\em Comm. Pure Appl. Math.}, 65:592--648, 2012.

\bibitem{Berestycki2017}
H.~Berestycki and F.~Hamel.
\newblock {\em Reaction-Diffusion Equations and Propagation Phenomena}.
\newblock Applied Mathematical Sciences. Springer New York, 2017.

\bibitem{Berestycki2009d}
H.~Berestycki, F.~Hamel, and H.~Matano.
\newblock Bistable traveling waves around an obstacle.
\newblock {\em Comm. Pure Appl. Math.}, 62(6):729--788, 2009.

\bibitem{Berestycki2005b}
H.~Berestycki, F.~Hamel, and L.~Roques.
\newblock Analysis of the periodically fragmented environment model: {II} -
  biological invasions and pulsating travelling fronts.
\newblock {\em J. Math. Pures Appl.}, 84(8):1101--1146, 2005.

\bibitem{Berestycki1992}
H.~Berestycki and L.~Nirenberg.
\newblock Travelling fronts in cylinders.
\newblock {\em Ann. I. H. Poincare -- AN}, 9(5):497--572, 1992.

\bibitem{Berestycki2017a}
H.~Berestycki and N.~Rodr{\'\i}guez.
\newblock A non-local bistable reaction-diffusion equation with a gap.
\newblock {\em Discrete \& Continuous Dynamical Systems-A}, 37(2):685--723,
  2017.

\bibitem{Bonnet1999}
A.~Bonnet and F.~Hamel.
\newblock Existence of nonplanar solutions of a simple model of premixed bunsen
  flames.
\newblock {\em SIAM J. Math. Anal.}, 31(1):80--118, 1999.

\bibitem{Bouhours2015}
J.~Bouhours.
\newblock Robustness for a {L}iouville type theorem in exterior domains.
\newblock {\em J. Dyn. Diff. Equat.}, 27(2):297--306, 2015.

\bibitem{Brasseur2018}
J.~{Brasseur} and J.~{Coville}.
\newblock {A counterexample to the Liouville property of some nonlocal
  problems}.
\newblock {\em arXiv:1804.07485}, Apr 2018.

\bibitem{Brasseur2019}
J.~Brasseur, J.~Coville, F.~Hamel, and E.~Valdinoci.
\newblock Liouville type results for a nonlocal obstacle problem.
\newblock {\em Proc. London Math. Soc.}, 119(2):291--328, 2019.

\bibitem{Cain}
M.~L. Cain, B.~G. Milligan, and A.~E. Strand.
\newblock Long-distance seed dispersal in plant populations.
\newblock {\em Am. J. Bot.}, 87(9):1217--1227, 2000.

\bibitem{Cantrell2004}
R.~S. Cantrell and C.~Cosner.
\newblock {\em Spatial ecology via reaction-diffusion equations}.
\newblock John Wiley \& Sons, 2004.

\bibitem{Carr2004}
J.~Carr and A.~Chmaj.
\newblock Uniqueness of travelling waves for nonlocal monostable equations.
\newblock {\em Proc. Amer. Math. Soc.}, 132(8):2433--2439 (electronic), 2004.

\bibitem{Chapman}
D.~S. Chapman, C.~Dytham, and G.~S. Oxford.
\newblock Modelling population redistribution in a leaf beetle: an evaluation
  of alternative dispersal functions.
\newblock {\em J. Anim. Ecol.}, 76(1):36--44, 2007.

\bibitem{Chen2002}
F.~Chen.
\newblock Almost periodic traveling waves of nonlocal evolution equations.
\newblock {\em Nonlinear Anal.}, 50(6):807 -- 838, 2002.

\bibitem{Chen1997}
X.~Chen.
\newblock Existence, uniqueness, and asymptotic stability of traveling waves in
  nonlocal evolution equations.
\newblock {\em Adv. Differential Equations}, 2(1):125--160, 1997.

\bibitem{Clobert2012}
J.~Clobert, M.~Baguette, T.~G. Benton, and J.~M. Bullock.
\newblock {\em Dispersal ecology and evolution}.
\newblock Oxford University Press, 2012.

\bibitem{Cortazar2007}
C.~Cortazar, J.~Coville, M.~Elgueta, and S.~Martinez.
\newblock A nonlocal inhomogeneous dispersal process.
\newblock {\em J. Diff. Eq.}, 241(2):332--358, 2007.

\bibitem{cortazar2012}
C.~Cort{\'a}zar, M.~Elgueta, F.~Quir{\'o}s, and N.~Wolanski.
\newblock Asymptotic behavior for a nonlocal diffusion equation in domains with
  holes.
\newblock {\em Arch. Ration. Mech. Anal.}, 205(2):673--697, Aug 2012.

\bibitem{cortazar2016}
C.~Cortázar, M.~Elgueta, F.~Quirós, and N.~Wolanski.
\newblock Asymptotic behavior for a nonlocal diffusion equation in exterior
  domains: The critical two-dimensional case.
\newblock {\em J. Math. Anal. Appl.}, 436(1):586 -- 610, 2016.

\bibitem{Coville2007d}
J.~Coville.
\newblock Travelling fronts in asymmetric nonlocal reaction diffusion
  equations: The bistable and ignition cases.
\newblock {\em CCSD-Hal e-print}, pages~--, May 2007.

\bibitem{Coville2008a}
J.~Coville, J.~Davila, and S.~Martinez.
\newblock Nonlocal anisotropic dispersal with monostable nonlinearity.
\newblock {\em J. Diff. Eq.}, 244(12):3080--3118, 2008.

\bibitem{Coville2013}
J.~Coville, J.~Davila, and S.~Martinez.
\newblock Pulsating fronts for nonlocal dispersion and {KPP} nonlinearity.
\newblock {\em Ann. I. H. Poincare -- AN}, (30):179--223, 2013.

\bibitem{Coville2010a}
J.~Coville, N.~Dirr, and S.~Luckhaus.
\newblock Non-existence of positive stationary solutions for a class of
  semi-linear {PDE}s with random coefficients.
\newblock {\em NHM}, 5(4):745--763, 2010.

\bibitem{Coville2007a}
J.~Coville and L.~Dupaigne.
\newblock On a non-local equation arising in population dynamics.
\newblock {\em Proc. Roy. Soc. Edinburgh Sect. A}, 137(4):727--755, 2007.

\bibitem{Dirr2006}
N.~Dirr and N.~K. Yip.
\newblock Pinning and de-pinning phenomena in front propagation in
  heterogeneous media.
\newblock {\em Interfaces Free Bound.}, 8(1):79--109, 2006.

\bibitem{Etherington2016}
T.~R. Etherington.
\newblock Least-cost modelling and landscape ecology: Concepts, applications,
  and opportunities.
\newblock {\em Current Landscape Ecology Reports}, 1(1):40--53, Mar 2016.

\bibitem{Fagan2002}
W.~F. Fagan.
\newblock Connectivity, fragmentation, and extinction risk in dendritic
  metapopulations.
\newblock {\em Ecology}, 83(12):3243--3249, 2002.

\bibitem{Fang2015}
J.~Fang and X.-Q. Zhao.
\newblock Bistable traveling waves for monotone semiflows with applications.
\newblock {\em J. Eur. Math. Soc.}, 17(9):2243--2288, 2015.

\bibitem{Fife1977}
P.~C. Fife and J.~B. McLeod.
\newblock The approach of solutions of nonlinear diffusion equations to
  travelling front solutions.
\newblock {\em Arch. Ration. Mech. Anal.}, 65(4):335--361, 1977.

\bibitem{Fisher1937}
R.~A. Fisher.
\newblock The wave of advance of advantageous genes.
\newblock {\em Ann. Eugenics}, 7:335--369, 1937.

\bibitem{Frantz2012}
A.~C. Frantz, S.~Bertouille, M.-C. Eloy, A.~Licoppe, F.~Chaumont, and M.-C.
  Flamand.
\newblock Comparative landscape genetic analyses show a belgian motorway to be
  a gene flow barrier for red deer ({C}ervus elaphus), but not wild boars
  ({S}us scrofa).
\newblock {\em Mol. Ecol.}, 21(14):3445--3457, 2012.

\bibitem{Graves2014}
T.~Graves, R.~Chandler, J.~A. Royle, P.~Beier, and K.~C. Kendall.
\newblock Estimating landscape resistance to dispersal.
\newblock {\em Landscape Ecol.}, 29(7):1201--1211, 2014.

\bibitem{Hamel2016}
F.~Hamel.
\newblock Bistable transition fronts in $\mathbb{R}^{N}$.
\newblock {\em Adv. Math.}, 289:279--344, 2016.

\bibitem{Hutson2003}
V.~Hutson, S.~Martinez, K.~Mischaikow, and G.~T. Vickers.
\newblock The evolution of dispersal.
\newblock {\em J. Math. Biol.}, 47(6):483--517, 2003.

\bibitem{Kagan2000}
L.~Kagan and G.~Sivashinsky.
\newblock Flame propagation and extinction in large-scale vortical flows.
\newblock {\em Combustion and Flame}, 120(1):222 -- 232, 2000.

\bibitem{Kagan1998}
L.~Kagan, G.~Sivashinsky, and G.~Makhviladze.
\newblock On flame extinction by a spatially periodic shear flow.
\newblock {\em Combust. Theor. Model.}, 2(4):399--404, 1998.

\bibitem{Kanel1961}
J.~I. Kanel.
\newblock Certain problem of burning-theory equations.
\newblock {\em Dokl. Akad. Nauk SSSR}, 136:277--280, 1961.

\bibitem{Kawasaki1997}
K.~Kawasaki and N.~Shigesada.
\newblock {\em Biological Invasions: Theory and Practice}.
\newblock Oxford University Press, 1997.

\bibitem{Kinezaki2003}
N.~Kinezaki, K.~Kawasaki, F.~Takasu, and N.~Shigesada.
\newblock Modeling biological invasions into periodically fragmented
  environments.
\newblock {\em Theor. Popul. Biol.}, 64(3):291--302, 2003.

\bibitem{Kolmogorov1937}
A.~N. Kolmogorov, I.~G. Petrovsky, and N.~S. Piskunov.
\newblock {\'E}tude de l'\'equation de la diffusion avec croissance de la
  quantit\'e de mati\`ere et son application \`a un probl\`eme biologique.
\newblock {\em Bulletin Universit\'e d'\'Etat \`a Moscow (Bjul. Moskowskogo
  Gos. Univ)}, S\'erie Internationale(Section A):1--26, 1937.

\bibitem{Langlois2001}
J.~P. Langlois, L.~Fahrig, G.~Merriam, and H.~Artsob.
\newblock Landscape structure influences continental distribution of hantavirus
  in deer mice.
\newblock {\em Landscape Ecol.}, 16(3):255--266, Apr 2001.

\bibitem{deOliveira2014}
F.~Lemes~de Oliveira.
\newblock {\em Eco-cities: The Role of Networks of Green and Blue Spaces},
  pages 165--178.
\newblock Springer Berlin Heidelberg, Berlin, Heidelberg, 2014.

\bibitem{Sun2011}
W.-T. Li, Y.-J. Sun, and Z.-C. Wang.
\newblock Entire solutions in nonlocal dispersal equations with bistable
  nonlinearity.
\newblock {\em J. Diff. Eq.}, 251(3):551 -- 581, 2011.

\bibitem{Lim2016}
T.~S. Lim and A.~Zlato{\v{s}}.
\newblock Transition fronts for inhomogeneous {F}isher-{KPP} reactions and
  non-local diffusion.
\newblock {\em Transactions of the American Mathematical Society},
  368(12):8615--8631, 2016.

\bibitem{Logg2012}
A.~Logg, G.~N. Wells, and J.~Hake.
\newblock {\em DOLFIN: A C++/Python Finite Element Library}, chapter~10.
\newblock Springer, 2012.

\bibitem{Matano2003}
H.~Matano.
\newblock Traveling waves in spatially random media (mathematical economics).
\newblock {\em PRIMS, Kyoto University}, 1337:1--9, 2003.

\bibitem{Matano2006}
H.~Matano, K.~I. Nakamura, and B.~Lou.
\newblock Periodic traveling waves in a two-dimensional cylinder with
  saw-toothed boundary and their homogenization limit.
\newblock {\em Netw. Heterog. Media}, 1(4):537--568 (electronic), 2006.

\bibitem{Murray1993}
J.~D. Murray.
\newblock {\em Mathematical biology}, volume~19 of {\em Biomathematics}.
\newblock Springer-Verlag, Berlin, second edition, 1993.

\bibitem{Nadin2009}
G.~Nadin.
\newblock Travelling fronts in space-time periodic media.
\newblock {\em J. Math. Pures Appl.}, 92:232--262, 2009.

\bibitem{Nadin2018}
G.~Nadin.
\newblock {\em Propagation phenomena in various reaction-diffusion models}.
\newblock Habilitation \`a diriger les recherches, Sorbonne University,
  Doctoral School of Mathematical Science of Paris Centre, 2018.

\bibitem{Nathan2012}
R.~Nathan, E.~K. Klein, J.~J. Robledo-Arnuncio, and E.~Revilla.
\newblock Dispersal kernels.
\newblock In {\em Dispersal ecology and evolution}, pages 187--210. Oxford
  University Press Oxford, 2012.

\bibitem{Ninomiya2005}
H.~Ninomiya and M.~Taniguchi.
\newblock Existence and global stability of traveling curved fronts in the
  {A}llen–{C}ahn equations.
\newblock {\em J. Diff. Eq.}, 213(1):204 -- 233, 2005.

\bibitem{Nolen2009}
J.~Nolen and L.~Ryzhik.
\newblock Traveling waves in a one-dimensional heterogeneous medium.
\newblock In {\em Ann. I. H. Poincare -- AN}, volume~26, pages 1021--1047,
  2009.

\bibitem{Obermeyer2008}
K.~J. Obermeyer and Contributors.
\newblock {VisiLibity}: A {C}++ library for visibility computations in planar
  polygonal environments.
\newblock \texttt{http://www.VisiLibity.org}, 2008.
\newblock R-1.

\bibitem{Okubo2002}
A.~Okubo and S.~A. Levin.
\newblock {\em Diffusion and Ecological Problems -- Modern Perspectives}.
\newblock Second edition, Springer-Verlag, New York, 2002.

\bibitem{Pokorny2015}
T.~Pokorny, D.~Loose, G.~Dyker, J.~J.~G. Quezada-Eu{\'a}n, and T.~Eltz.
\newblock Dispersal ability of male orchid bees and direct evidence for
  long-range flights.
\newblock {\em Apidologie}, 46(2):224--237, Mar 2015.

\bibitem{Qiao2020}
S-X. Qiao, W-T. Li, and J-W. Sun.
\newblock Propagation phenomena for nonlocal dispersal equations in exterior
  domains.
\newblock {\em arXiv preprint arXiv:2005.01307}, 2020.

\bibitem{Rawal2015}
N.~Rawal, W.~Shen, and A.~Zhang.
\newblock Spreading speeds and traveling waves of nonlocal monostable equations
  in time and space periodic habitats.
\newblock {\em Discrete Cont. Dyn-A}, 35(4):1609--1640, 2015.

\bibitem{Ricketts2001}
T.~H. Ricketts.
\newblock The matrix matters: effective isolation in fragmented landscapes.
\newblock {\em The American Naturalist}, 158(1):87--99, 2001.

\bibitem{Robledo-Arnuncio2014}
J.~J. Robledo-Arnuncio, E.~K. Klein, H.~C. Muller-Landau, and L.~Santamaría.
\newblock Space, time and complexity in plant dispersal ecology.
\newblock {\em Movement Ecology}, 2(1):16, August 2014.

\bibitem{Schurr2005}
F.~M. Schurr, W.~J. Bond, G.~F. Midgley, and S.~I. Higgins.
\newblock A mechanistic model for secondary seed dispersal by wind and its
  experimental validation.
\newblock {\em J. Ecol.}, 93(5):1017--1028, 2005.

\bibitem{Shen2004}
W.~Shen.
\newblock Traveling waves in diffusive random media.
\newblock {\em J. Dyn. Diff. Equat.}, 16:1011--1060, 2004.

\bibitem{Shen2017}
W.~Shen and Z.~Shen.
\newblock Regularity and stability of transition fronts in nonlocal equations
  with time heterogeneous ignition nonlinearity.
\newblock {\em J. Diff. Eq.}, 262(5):3390--3430, 2017.

\bibitem{Shen2017a}
W.~Shen and Z.~Shen.
\newblock Transition fronts in nonlocal equations with time heterogeneous
  ignition nonlinearity.
\newblock {\em Discrete Contin. Dyn. Syst. - A}, 37(2):1013--1037, 2017.

\bibitem{Shen2019}
W.~Shen and Z.~Shen.
\newblock Existence, uniqueness and stability of transition fronts of non-local
  equations in time heterogeneous bistable media.
\newblock {\em Eur. J. Appl. Math.}, page 1–45, 2019.

\bibitem{Shen2012a}
W.~Shen and A.~Zhang.
\newblock Traveling wave solutions of monostable equations with nonlocal
  dispersal in space periodic habitats.
\newblock {\em Communications on Applied Nonlinear Analysis}, 19:73--101, 2012.

\bibitem{Shigesada1986}
N.~Shigesada, K.~Kawasaki, and E.~Teramoto.
\newblock Traveling periodic waves in heterogeneous environments.
\newblock {\em Theor. Popul. Biol.}, 30(1):143--160, 1986.

\bibitem{Sivashinsky2002}
G.~I. Sivashinsky.
\newblock Some developments in premixed combustion modeling.
\newblock {\em Proceedings of the Combustion Institute}, 29(2):1737 -- 1761,
  2002.

\bibitem{Sutherland2015}
C.~Sutherland, A.~K. Fuller, and J.~A. Royle.
\newblock Modelling non-{E}uclidean movement and landscape connectivity in
  highly structured ecological networks.
\newblock {\em Methods Ecol. Evol.}, 6(2):169--177, 2015.

\bibitem{Taylor1993}
P.~D. Taylor, L.~Fahrig, K.~Henein, and G.~Merriam.
\newblock Connectivity is a vital element of landscape structure.
\newblock {\em Oikos}, pages 571--573, 1993.

\bibitem{Tischendorf2000}
L.~Tischendorf and L.~Fahrig.
\newblock On the usage and measurement of landscape connectivity.
\newblock {\em Oikos}, 90(1):7--19, 2000.

\bibitem{Wang2002a}
X.~Wang.
\newblock Metastability and stability of patterns in a convolution model for
  phase transitions.
\newblock {\em J. Diff. Eq.}, 183(2):434--461, 2002.

\bibitem{Xin2000}
J.~Xin.
\newblock Front propagation in heterogeneous media.
\newblock {\em SIAM Rev.}, 42(2):161--230 (electronic), 2000.

\bibitem{Xiu2016}
N.~Xiu, M.~Ignatieva, and C.~Konijnendijk van~den Bosch.
\newblock The challenges of planning and designing urban green networks in
  {S}candinavian and {C}hinese cities.
\newblock {\em Journal of Architecture and Urbanism}, 40(3):163--176, 2016.

\bibitem{Yagisita2009}
H.~Yagisita.
\newblock Existence and nonexistence of travelling waves for a nonlocal
  monostable equation.
\newblock {\em Publ. RIMS, Kyoto Univ.}, 45:925--953, 2009.

\end{thebibliography}

\vspace{2mm}

\end{document}